\documentclass[11pt]{amsart}

\newtheorem{thm}{Theorem}[section]
\newtheorem*{thm*}{Theorem}

\newtheorem{claim}[thm]{Claim}

\newtheorem{cor}[thm]{Corollary}

\newtheorem{lem}[thm]{Lemma}
\newtheorem*{lem*}{Lemma}
\newtheorem{mainthm}{Theorem}
\newtheorem*{mainthm*}{Theorem}
\newtheorem{maincor}[mainthm]{Corollary}
\newtheorem{prop}[thm]{Proposition}

\theoremstyle{definition}

\newtheorem*{case*}{Case}

\newtheorem{defn}[thm]{Definition}
\newtheorem*{defn*}{Definition}

\newtheorem*{exmp*}{Example}

\newtheorem{method}{Method}

\newtheorem{step}{Step}\renewcommand{\thestep}{}

\theoremstyle{remark}

\renewcommand{\thecase}{}

\newtheorem{rmk}[thm]{Remark}
\newtheorem*{rmk*}{Remark}

\makeatletter
\def\alphenumi{
  \def\theenumi{\alph{enumi}}
  \def\p@enumi{\theenumi}
  \def\labelenumi{(\@alph\c@enumi)}}
\makeatother

\makeatletter
\def\thecase{\@arabic\c@case}
\makeatother

\makeatletter
\def\thestep{\@arabic\c@step}
\makeatother

\newcount\hh
\newcount\mm
\mm=\time
\hh=\time
\divide\hh by 60
\divide\mm by 60
\multiply\mm by 60
\mm=-\mm
\advance\mm by \time
\def\hhmm{\number\hh:\ifnum\mm<10{}0\fi\number\mm}

\let\oldmarginpar\marginpar
\renewcommand\marginpar[1]{\-\oldmarginpar[\raggedleft\footnotesize #1]%
{\raggedright\footnotesize #1}}

\renewcommand\emptyset{\varnothing}

\newcommand\CC{\mathbb{C}}

\newcommand\HH{\mathbb{H}}

\newcommand\NN{\mathbb{N}}

\newcommand\RR{\mathbb{R}}

\newcommand\ZZ{\mathbb{Z}}

\newcommand\cH{{\mathcal{H}}}

\newcommand\cS{{\mathcal{S}}}

\newcommand\cX{{\mathcal{X}}}

\newcommand\fg{{\mathfrak{g}}}

\newcommand\sA{{\mathscr{A}}}
\newcommand\sB{{\mathscr{B}}}
\newcommand\sC{{\mathscr{C}}}
\newcommand\sD{{\mathscr{D}}}
\newcommand\sE{{\mathscr{E}}}
\newcommand\sF{{\mathscr{F}}}

\newcommand\sH{{\mathscr{H}}}

\newcommand\sN{{\mathscr{N}}}

\newcommand\sT{{\mathscr{T}}}
\newcommand\sU{{\mathscr{U}}}
\newcommand\sV{{\mathscr{V}}}
\newcommand\sW{{\mathscr{W}}}
\newcommand\sX{{\mathscr{X}}}

\newcommand\balpha{{\boldsymbol{\alpha}}}

\newcommand\beps{{\boldsymbol{\varepsilon}}}

\newcommand\blambda{{\boldsymbol{\lambda}}}

\newcommand\bE{{\mathbf{E}}}

\newcommand\bg{{\mathbf{g}}}

\newcommand\bx{{\mathbf{x}}}

\newcommand\by{{\mathbf{y}}}

\newcommand\bz{{\mathbf{z}}}

\newcommand{\cov}{\nabla}

\newcommand\eps{\varepsilon}

\newcommand\SO{\operatorname{SO}}

\newcommand\SU{\operatorname{SU}}
\newcommand\U{\operatorname{U}}

\newcommand\less{\setminus}

\newcommand\ad{{\operatorname{ad}}}
\newcommand\Ad{{\operatorname{Ad}}}

\newcommand\Aut{\operatorname{Aut}}

\DeclareMathOperator{\BCrit}{BCrit}

\newcommand\Center{\operatorname{Center}}

\newcommand\Conf{\operatorname{Conf}}

\DeclareMathOperator{\Crit}{Crit}

\DeclareMathOperator{\cyl}{cyl}

\newcommand\dist{\operatorname{dist}}

\newcommand\End{\operatorname{End}}

\newcommand{\esssup}{\operatornamewithlimits{ess\ sup}}

\newcommand\Fr{\operatorname{Fr}}

\DeclareMathOperator{\Inj}{Inj}

\newcommand\Ker{\operatorname{Ker}}

\DeclareMathOperator{\Mass}{Mass}

\newcommand\rank{\operatorname{rank}}

\newcommand\Riem{\operatorname{Riem}}
\newcommand\Scale{\operatorname{Scale}}

\newcommand\Sym{\operatorname{Sym}}

\DeclareMathOperator{\UCrit}{UCrit}

\newcommand\vol{\operatorname{vol}}

\newcommand\annulus{{\mathrm{annulus}}}

\newcommand\apriori{{\emph{a priori }}}
\newcommand\Apriori{{\emph{A priori }}}
\newcommand\adhoc{{\emph{ad hoc }}}

\newcommand\background{{\mathrm{background}}}
\newcommand\ball{{\mathrm{ball}}}

\newcommand\mycenter{{\mathrm{center}}}

\newcommand\coul{{\mathrm{coul}}}
\newcommand\coulomb{{\mathrm{coulomb}}}

\newcommand\euclid{{\mathrm{euclid}}}

\newcommand\loc{{\mathrm{loc}}}

\newcommand\mutatis{{\emph{mutatis mutandis }}}

\newcommand\myref{{\textsc{ref}}}

\newcommand\round{{\mathrm{round}}}
\newcommand\scale{{\mathrm{scale}}}

\newcommand\sphere{\mathrm{sphere}}

\newcommand\tube{{\mathrm{tube}}}

\newcommand\ym{\textsc{ym}}

\numberwithin{equation}{section}

\usepackage{amscd, amssymb, graphicx, hyperref, mathrsfs, paralist, url, verbatim}
\usepackage[usenames]{color}

\hypersetup{pdftitle={Discreteness for energies of Yang-Mills connections over four-dimensional manifolds}}
\hypersetup{pdfauthor={Paul M. N. Feehan}}

\begin{document}

\title[Discreteness for energies of Yang-Mills connections]{Discreteness for energies of Yang-Mills connections over four-dimensional manifolds}

\author[Paul M. N. Feehan]{Paul M. N. Feehan}
\address{Department of Mathematics, Rutgers, The State University of New Jersey, 110 Frelinghuysen Road, Piscataway, NJ 08854-8019, United States of America}
\email{feehan@math.rutgers.edu}

\date{This version: May 21, 2015}
\date{May 26, 2015}

\begin{abstract}
We generalize our previous results \cite[Theorem 1 and Corollary 2]{Feehan_yangmillsenergygap},
\cite[Theorem 1]{Feehan_yangmillsenergygapflat} on the existence of an $L^2$-energy gap for Yang-Mills connections over closed four-dimensional manifolds and energies near the ground state (occupied by flat, anti-self-dual, or self-dual connections) to the case of Yang-Mills connections with arbitrary energies. We prove that for any principal bundle with compact Lie structure group over a closed, four-dimensional, Riemannian manifold, the $L^2$ energies of Yang-Mills connections on that principal bundle form a discrete sequence without accumulation points. Our proof employs a version of our {\L}ojasiewicz-Simon gradient inequality for the Yang-Mills $L^2$-energy functional from our monograph \cite{Feehan_yang_mills_gradient_flow} and extensions of our previous results on the bubble-tree compactification for the moduli space of anti-self-dual connections \cite{FeehanGeometry} to the moduli space of Yang-Mills connections with a uniform $L^2$ bound on their energies.
\end{abstract}




\thanks{The author is grateful for research support provided by the Department of Mathematics at Rutgers University and hospitality provided by the Department of Mathematics at Columbia University during the preparation of this article.}

\maketitle
\tableofcontents

\section{Introduction}
\label{sec:Introduction}
Suppose that $G$ is a Lie group and $A$ is a connection on a principal $G$-bundle $P$ over an oriented manifold, $X$, endowed with a Riemannian metric, $g$. The \emph{Yang-Mills $L^2$-energy functional} is defined by
\begin{equation}
\label{eq:Yang-Mills_energy_functional}
\sE_g(A)  := \int_X |F_A|^2\,d\vol_g,
\end{equation}
where $F_A \in \Omega^2(X;\ad P)$ is the curvature of $A$. We extend the results of our previous two articles \cite{Feehan_yangmillsenergygap,  Feehan_yangmillsenergygapflat} on $L^2$-energy gap results for Yang-Mills connections in several ways, building on the method of \cite{Feehan_yangmillsenergygapflat} which relies on the {\L}ojasiewicz-Simon gradient inequality \cite{Feehan_yang_mills_gradient_flow, Simon_1983} and compactness of the moduli space of flat connections over a closed Riemannian manifold of arbitrary dimension greater than or equal to two. We specialize here to the case of a base manifold of dimension four but generalize the results of \cite{Feehan_yangmillsenergygap} by removing the constraints described in that article on the compact Lie structure group, genericity or positivity of the Riemannian metric, topology of the base four-manifold, or topology of the principal bundle supporting the Yang-Mills connections. Our approach in \cite{Feehan_yangmillsenergygapflat} relied on the compactness of the moduli space of flat connections to ensure that the {\L}ojasiewicz-Simon gradient inequality constants were uniform with respect to the critical point. More generally, the moduli space of Yang-Mills connections with a uniform $L^2$ bound on their curvatures are non-compact due to the phenomenon of energy bubbling in dimension four, a property better known in the context of moduli spaces of connections attaining the absolute minimum value of the $L^2$-energy functional predicted by topology, namely the moduli spaces of anti-self-dual or self-dual connections \cite{AHS, DK, FU, FrM}. Many examples of non-minimal Yang-Mills connections are now known to exist, as we discuss in Section \ref{subsec:Existence_non-minimal_Yang-Mills_connections}, but they remain mysterious even when the base four-manifold is the sphere with its standard round metric of radius one, and non-compactness makes it difficult to understand their possible energy values, including the fundamental question as to whether those values are discrete or form continua.

\subsection{Main results}
\label{subsec:Main_results}
We refer the reader to Section \ref{sec:Preliminaries} for the definitions and notation employed in the statements of our main results (Theorems \ref{mainthm:Discreteness_Yang-Mills_energies} and \ref{mainthm:Separation_strata_Yang-Mills_connections} and Corollary \ref{maincor:Yang-Mills_energy_near_ground_state}) and ensuing discussion.

\begin{mainthm}[Discreteness of critical values of the Yang-Mills $L^2$-energy functional for connections over four-dimensional manifolds]
\label{mainthm:Discreteness_Yang-Mills_energies}
Let $G$ be a compact Lie group and $P$ be a smooth principal $G$-bundle over a closed, four-dimensional, oriented, smooth manifold, $X$, endowed with a smooth Riemannian metric, $g$. Then the subset of critical values of the $L^2$-energy functional, $\sE_g:\sB(P,g) \to [0,\infty)$, is closed and discrete, depending at most on $g$, $G$, and the homotopy equivalence class, $[P]$. In particular, if $\{c_i\}_{i\in\NN} \subset [0,\infty)$ denotes the strictly increasing sequence of critical values of $\sE_g$ and $A$ is a $g$-Yang-Mills connection on $P$ with
\begin{equation}
\label{eq:Energy_Yang-Mills_connection_near_lower_critical_value}
c_i \leq \sE_g(A) < c_{i + 1},
\end{equation}
for some $i \geq 0$, then $\sE_g(A) = c_i$.
\end{mainthm}

\begin{rmk}[Discreteness of critical values of coupled Yang-Mills $L^2$-energy functionals for pairs of connections and spinor fields over four-dimensional manifolds]
\label{rmk:Discreteness_coupled_Yang-Mills_energies}
A version of Theorem \ref{mainthm:Discreteness_Yang-Mills_energies} should hold more generally for solutions to coupled Yang-Mills equations for pairs, $(A,\Phi)$, consisting of a connection $A$ on $P$ and a section $\Phi$ of an associated vector bundle; see the discussion of such equations by Parker in \cite[Section 2]{ParkerGauge}. We shall consider such generalizations elsewhere.
\end{rmk}

A brief introduction to the main ideas underlying the proof of Theorem \ref{mainthm:Discreteness_Yang-Mills_energies} is given in Section \ref{subsec:Discussion_main_ideas}, especially Sections \ref{subsubsec:Lojasiewicz-Simon_gradient_inequality_moduli_space_flat_connections} and \ref{subsubsec:Lojasiewicz-Simon_gradient_inequality_Yang-Mills_connections_sphere}. Theorem \ref{mainthm:Discreteness_Yang-Mills_energies} yields the following special case, generalizing our previous results \cite[Theorem 1 and Corollary 2]{Feehan_yangmillsenergygap}.

\begin{maincor}[$L^2$-energy gap for Yang-Mills connections over four-dimensional manifolds with energies near the ground state]
\label{maincor:Yang-Mills_energy_near_ground_state}
Let $G$ be a compact Lie group and $P$ be a principal $G$-bundle over a closed, four-dimensional, oriented, smooth manifold, $X$, endowed with a smooth Riemannian metric, $g$. Then there is a positive constant, $\eps \in (0, 1]$ depending at most on $g$, $G$, and the homotopy class, $[P]$ of the principal bundle, with the following significance.
If $A$ is a smooth $g$-Yang-Mills connection on $P$ and the $g$-self-dual component, $F_A^{+,g}$, of its curvature, $F_A$, obeys
\begin{equation}
\label{eq:Curvature_self-dual_L2_small}
\|F_A^{+,g}\|_{L^2(X)} < \eps,
\end{equation}
then $A$ is a $g$-anti-self-dual connection, that is, $F_A^{+,g} = 0$.
\end{maincor}

\begin{rmk}[On the conclusions of Corollary \ref{maincor:Yang-Mills_energy_near_ground_state}]
\label{rmk:Yang-Mills_energy_near_ground_state}
By reversing orientations on $X$, the analogous conclusion in Corollary \ref{maincor:Yang-Mills_energy_near_ground_state} is true for $g$-self-dual rather than $g$-anti-self-dual connections on $P$: If \eqref{eq:Curvature_self-dual_L2_small} holds for $F_A^{-,g}$ in place of $F_A^{+,g}$, then $A$ is $g$-self-dual, that is, $F_A^{-,g} = 0$.
\end{rmk}

Corollary \ref{maincor:Yang-Mills_energy_near_ground_state} does not provide a new method of establishing the existence of anti-self-dual connections extending previous existence results due to Donaldson \cite{DonASD}, Taubes \cite{TauSelfDual, TauIndef}, or Taylor \cite{Taylor_2002} since the problem of proving existence of Yang-Mills connections is considerably harder than that of proving the existence of anti-self-dual connections. For example, see Taubes \cite{TauFrame} for the case of a base manifold of dimension four and Brendle \cite{Brendle_2003arxiv} for the case of base manifolds of dimension greater than four.

Theorem \ref{mainthm:Discreteness_Yang-Mills_energies} acquires further context from the more familiar setting of Morse theory on finite-dimensional manifolds. If $M$ is closed, finite-dimensional, smooth manifold, then its Euler characteristic can be computed with the aid of a Morse function, $f:M\to\RR$, via the identity \cite[Theorem 5.2]{Milnor_morse_theory}
\[
\chi(M) = \sum_i (-1)^i n_i,
\]
where $n_i$ is the number of critical points of $f$ of index $i$. 

The proof of Theorem \ref{mainthm:Discreteness_Yang-Mills_energies} on the separation of critical values also yields a result on the separation of strata of critical points. For $p \geq 1$ and integers $k \geq 0$, we define a $W^{k,p}$ \emph{distance function} on the stratified Banach manifold, $\sB(P,g) = \sA(P)/\Aut(P)$, by
\begin{equation}
\label{eq:Wkp_distance_on_quotient_space_connections}
\dist_{W_g^{k,p}}([A_1],[A_0])
:=
\frac{1}{2}\inf_{u,v\in\Aut(P)}\left(\|u(A_1)-A_0\|_{W_{A_0}^{k,p}(X,g)}
+ \|v(A_0)-A_1\|_{W_{A_1}^{k,p}(X,g)}\right),
\end{equation}
and, for any non-empty subsets $\sU_0, \sU_1 \subset \sB(P,g)$,
\begin{equation}
\label{eq:Wkp_distance_between_subsets_quotient_space_connections}
\dist_{W_g^{k,p}}(\sU_1,\sU_0) := \inf_{[A_i]\in\sU_i, \, i=0,1} \dist_{W_g^{k,p}}([A_1],[A_0]).
\end{equation}
and when $k=0$, we write $\dist_{W_g^{0,p}} = \dist_{L_g^p}$ in \eqref{eq:Wkp_distance_on_quotient_space_connections} and  \eqref{eq:Wkp_distance_between_subsets_quotient_space_connections}. We suppress explicit notation for the Riemannian metric $g$ in \eqref{eq:Wkp_distance_on_quotient_space_connections} and  \eqref{eq:Wkp_distance_between_subsets_quotient_space_connections} when there is no ambiguity.

\begin{mainthm}[Uniform $W^{1,2}$-distance between Yang-Mills connections over four-dimensional manifolds with distinct critical values]
\label{mainthm:Separation_strata_Yang-Mills_connections}
Let $G$ be a compact Lie group and $P$ be a principal $G$-bundle over a closed, four-dimensional, oriented, smooth manifold, $X$, endowed with a smooth Riemannian metric, $g$, and $E > 0$ be a constant. Then there is a positive constant, $\delta \in (0, 1]$, depending at most on $E$, $g$, $G$, and $[P]$, with the following significance. If $c, c' \in [0,E]$ and
\begin{equation}
\label{eq:W12_distance_between_critical_sets_connections_small}
\dist_{W^{1,2}}(\Crit(P,g;c), \Crit(P,g;c')) < \delta,
\end{equation}
then $c' = c$.
\end{mainthm}

\subsection{Discreteness of critical values of the Yang-Mills energy functional over closed Riemannian smooth manifolds of dimension two or three}
\label{subsec:Discreteness_critical_values_Yang-Mills_energy_over_two_or_three-manifolds}
The energies of Yang-Mills connections on a principal $G$-bundle $P$ over a closed Riemann surface are necessarily discrete, just as in Theorem \ref{mainthm:Discreteness_Yang-Mills_energies} for the case of a four-dimensional manifold, when $X$ is a closed, Riemannian smooth manifold of dimension two or three and $G$ is any compact Lie group. Indeed, for $X$ of dimension two or three, the conclusion of Theorem \ref{mainthm:Discreteness_Yang-Mills_energies} is a straightforward consequence of the {\L}ojasiewicz-Simon gradient inequality for the Yang-Mills energy functional \cite[Proposition 7.2]{Rade_1992} (see Simon \cite[Theorem 3]{Simon_1983} for the original version) and Uhlenbeck compactness for Yang-Mills connections with an $L^2$ bound on curvature \cite{UhlLp, UhlRem}, as observed by R\r{a}de \cite[p. 127]{Rade_1992}. This observation will be explained in more detail when we discuss the key ideas in the proof of Theorem \ref{mainthm:Discreteness_Yang-Mills_energies} in Section \ref{subsec:Discussion_main_ideas}. This paradigm underlies the proof of our main result \cite[Theorem 1]{Feehan_yangmillsenergygapflat} concerning the gap between the energies of non-flat Yang-Mills connections and zero (the energy of any flat connection), where we exploit compactness of the moduli space of flat connections over a closed Riemannian smooth manifold of arbitrary dimension greater than or equal to two and our version of the {\L}ojasiewicz-Simon gradient inequality for the Yang-Mills energy functional over such manifolds \cite[Theorem 21.8]{Feehan_yang_mills_gradient_flow}, quoted here as Theorem \ref{thm:Rade_proposition_7-2}. However, to extend those ideas to prove Theorem \ref{mainthm:Discreteness_Yang-Mills_energies} one must address the fundamental problem of non-compactness, due to \emph{bubbling}, of moduli spaces of Yang-Mills connections with bounded energy.

\subsection{Discreteness of critical values of the Yang-Mills energy functional over compact K\"ahler manifolds}
\label{subsec:Discreteness_critical_values_Yang-Mills_energy_over_Kaehler_manifolds}
For a compact Lie group, $G\subset \U(n)$, one can also see that the energies of Yang-Mills connections on a principal $G$-bundle $P$ over a closed Riemann surface, $\Sigma$, are necessarily discrete, just as in Theorem \ref{mainthm:Discreteness_Yang-Mills_energies} for the case of a four-dimensional manifold, as a consequence of the results of Atiyah and Bott \cite{Atiyah_Bott_1983}. (The author is grateful to George Daskalopoulos and Richard Wentworth for their explanations of what is known with regard to discreteness of $L^2$ energies of integrable Yang-Mills connections over K\"ahler manifolds of complex dimension one and two.)

In the case of a closed Riemann surface, the discreteness follows from results of Atiyah and Bott \cite{Atiyah_Bott_1983}. They show that every Yang-Mills connection on Hermitian vector bundle $E$ over $\Sigma$ is a direct sum of Hermitian-Yang-Mills connections, where the key idea is to rewrite the Yang-Mills equation using the K\"ahler identities. The $L^2$ energy of a Hermitian-Yang-Mills connection is the square of the \emph{slope} of $E$, up to some fixed overall constant depending on the volume of the Riemann surface. The slope of $E$ is equal to the degree of $E$ divided by its rank. Hence, the collection of possible slopes that appear in the Harder-Narasimhan decomposition are rational with a denominator bounded above by the rank of the initial bundle, $E$. This proves discreteness in the case of Yang-Mills connections over Riemann surfaces.

A version of this argument applies to compact K\"ahler surfaces, $Z$, but not necessarily in the generality provided by Theorem \ref{mainthm:Discreteness_Yang-Mills_energies}, as we must restrict to holomorphic vector bundles given the possibility that there may be Yang-Mills connections that are not sums of Hermitian-Yang-Mills connections. For K\"ahler surfaces, there could exist Yang-Mills connections whose curvature is not of type $(1,1)$.

Bubbling occurs in real dimension four, so one must consider the energies of Yang-Mills connections and \emph{ideal} Yang-Mills connections arising, for example, as Uhlenbeck limits of Yang-Mills connections. Here one can rely on work of Daskalopoulos and Wentworth \cite{Daskalopoulos_Wentworth_2004, Daskalopoulos_Wentworth_2007}, where they show that the Harder-Narasimhan type of the Uhlenbeck limit is the same as the Harder-Narasimhan type of the original holomorphic bundle, $(E,\bar\partial_{A_0})$. For a bundle $E$ of prescribed topological type, the possible slopes are again discrete. The denominator is bounded by the rank, and the degrees are obtained by  intersecting elements of a lattice $H^2(Z;\ZZ)$ (the first Chern classes) with the K\"ahler class.

It is possible that the preceding argument for K\"ahler surfaces might extend to the case of compact K\"ahler manifolds of arbitrary complex dimension based on work of Sibley \cite{Sibley_2013arxiv}, Tian \cite{TianGTCalGeom}, and others concerning Yang-Mills connections over higher-dimensional manifolds. However, methods based on generalizations of those of Atiyah and Bott \cite{Atiyah_Bott_1983} giving discreteness in the case of Riemann surfaces would not appear to yield the conclusion of Theorem \ref{mainthm:Discreteness_Yang-Mills_energies} in full generality when $Z$ is a compact K\"ahler surface and
$G \subset \U(n)$ or, more generally, when $Z$ is a compact K\"ahler manifold of complex dimension three or higher.

While our \cite[Theorem 1]{Feehan_yangmillsenergygapflat} yields a gap between zero and the energy of non-flat Yang-Mills connections over a base manifold of any dimension, through Uhlenbeck compactness and our {\L}ojasiewicz-Simon gradient inequality, it would be a challenging problem to generalize Theorem \ref{mainthm:Discreteness_Yang-Mills_energies} to the case of dimensions greater than four.

\subsection{Discreteness and non-discreteness of critical values of the Yang-Mills energy functional over complete, non-compact four-dimensional manifolds}
\label{subsec:Discreteness_critical_values_Yang-Mills_energy_over_complete_manifolds}
Dodziuk and Min-Oo \cite{Dodziuk_Min-Oo_1982}, Gerhardt \cite{Gerhardt_2010}, Shen \cite{Shen_1982}, and Xin \cite{Xin_1984} establish $L^2$-energy gap results near the \emph{ground state} (or absolute minimum) energy in the case of four-dimensional, non-compact, smooth manifolds equipped with a \emph{complete} Riemannian metric that is positive in the sense of \eqref{eq:Freed_Uhlenbeck_page_174_positive_metric}. It is possible that Theorem \ref{mainthm:Discreteness_Yang-Mills_energies} might extend, perhaps with a suitable weighting at spatial infinity, to four-dimensional, complete, non-compact manifolds which are not conformally compact and whose Riemannian metric does not obey the positivity condition \eqref{eq:Freed_Uhlenbeck_page_174_positive_metric}.

Our proof of Theorem \ref{mainthm:Discreteness_Yang-Mills_energies} makes essential use of the hypotheses that $(X,g)$ is conformally compact and that $G$ is a compact Lie group. Using the conformal invariance of Yang-Mills equations over four-dimensional manifolds, Chaohao \cite[Theorem 1]{Chaohao_1980} has proved that over a complete, connected, simply-connected manifold of constant negative sectional curvature (namely, the hyperbolic space $\HH^4$ with sectional curvature $-1$), there exist Yang-Mills connections with $L^2$-energy $c$ equal to \emph{any} non-negative real number; see Chaohao \cite{Chaohao_1981} for further discussion. As Sadun points out, some condition at spatial infinity should be required in order to achieve discrete values of the Yang-Mills $L^2$-energy functional in the absence of a positivity condition on the Riemannian metric.

\subsection{Relevance to previous results on Morse theory for the Yang-Mills $L^2$-energy functional over closed four-dimensional manifolds}
\label{subsec:Comparison_results_Taubes}
The primary difficulty one must confront when analyzing the $L^2$-energy functional on a principal $G$-bundle over a four-dimensional manifold is non-compactness of the critical sets, $\Crit(P,g,\sC)$ in Definition \ref{defn:Moduli_space_Yang-Mills_connections_L2-energy_compact_range}, due to the well-known phenomenon of energy \emph{bubbling} \cite{DK, FU} and indeed this is the difficulty addressed by Taubes in \cite{TauFrame} in his approach to Morse theory for the Yang-Mills $L^2$-energy functional; see also \cite{TauPath, TauStable} for related results.

We recall some of the ideas underlying the proof of his main result \cite[Theorem 1.1]{TauFrame}, essentially following his description in \cite[Section 2]{TauFrame}, and describe the relationship between the `approximate critical sets' $\Sigma(E, \delta)$ defined by Taubes in \eqref{eq:Taubes_1988_2-12} and the traditional critical sets, $\Crit(P,g,c)$, that we specify in Definition \ref{defn:Moduli_space_Yang-Mills_connections_L2-energy_c}. Our summary of Taubes' approach in \cite{TauFrame} is highly simplified and we omit mention of many deep technicalities discussed in detail his article.

\subsubsection{Taubes' framework for Morse theory for the Yang-Mills $L^2$-energy functional over closed four-dimensional manifolds}
\label{subsubsec:Taubes_framework_Morse_theory_Yang-Mills_energy functional}
The purpose of Taubes' article is to explain \cite[p. 329]{TauFrame} how the Morse theory for the Yang-Mills $L^2$-energy functional can be recovered by examining the restriction of the functional to a countable set of finite-dimensional, non-compact subvarieties of the Banach manifold of connections with base point modulo gauge transformations, $\sB'(P,g) = (\sA(P,g)\times P|_{x_0})/\Aut(P)$, for some fixed $x_0 \in X$. (The stratified Banach manifold, $\sB(P,g)$, is recovered as the quotient, $\sB'(P,g)/G$; the open subset $\sB^*(P,g) \subset \sB(P,g)$ of points $[A]\in \sB(P,g)$ with stabilizer (isotropy group) of $A$ in $\Aut(P)$ isomorphic to $\Center(G)$ is a Banach manifold \cite[p. 132]{DK}.)

Taubes first considers a smooth function, $\sF:\sX \to \RR$, on a smooth Banach manifold, $\sX$, for which the Palais-Smale Condition C is \emph{not} assumed to hold. Rather than examine the exact critical sets of $\sF$, he considers sets of approximate critical points \cite[Equation (2.1)]{TauFrame},
\begin{equation}
\label{eq:Taubes_1988_2-1}
\Crit(E,\delta) := \left\{x\in \sX: |\sF(x) - E| < \delta \text{ and } \|\sF'(x)\| < \delta \right\},
\end{equation}
where $\sF'(x):T_x\sX \to T_x^*\sX$ is the gradient of $\sF$ at each $x \in \sX$, and also introduces the set
\begin{equation}
\label{eq:Taubes_1988_2-2}
\Crit(E,\delta)^- := \left\{x \in \Crit(E, \delta): \sF(x) < E \right\}.
\end{equation}
When $\delta > \delta' > 0$ are given, there is a natural inclusion of the pairs,
\begin{equation}
\label{eq:Taubes_1988_2-3}
\iota: (\Crit(E, \delta'), \Crit(E, \delta')^-) \to  (\Crit (E, \delta), \Crit(E, \delta)^-).
\end{equation}
It is the effect of the map $\iota$ on the relative homotopy groups of the pairs in \eqref{eq:Taubes_1988_2-3} which determines whether $\Crit(E, \delta)$ contributes to the Morse theory of $\sF$ on $\sX$. Taubes calls the number $E$ a \emph{regular value} of $\sF$ if there exists $\delta > 0$ and $\delta' \in (0, \delta]$ such that the map $\iota$ in \eqref{eq:Taubes_1988_2-3} induces the zero map between the relative homotopy groups. If $E$ is not a regular value (in the preceding sense), then it is called a \emph{critical value}. When $\sX$ is a compact manifold, these definitions coincide with the usual definitions of regular value and critical value.

We emphasize that in Theorem \ref{mainthm:Discreteness_Yang-Mills_energies} and throughout our article, we \emph{only} employ the standard definitions of regular value and critical value for the Yang-Mills $L^2$ energy functional on $\sA(P,g)$ or $\sB'(P,g)$, irrespective of the fact that $\sA(P,g)$ and $\sB'(P,g)$ are infinite-dimensional Banach manifolds.

Taubes constructs certain finite-dimensional subvarieties $\sN \subset \sB'(P,g)$ parameterized by the following \emph{gluing data} and a splicing map, where we restrict our attention here to a special case (one rather than iterated levels of splicing) of his construction outlined in \cite[Equations (2.5) through (2.10)]{TauFrame} for the sake of exposition:
\begin{enumerate}
\item Set of Yang-Mills connections, $A_0$, on a principal $G$-bundle $P_0$ over $X$;

\item Sets of Yang-Mills connections, $A_\alpha$, on principal $G$-bundles $P_\alpha$ over $S^4$ corresponding to points $x_\alpha \in X$, where $1 \leq \alpha \leq l$;

\item Set of oriented $g$-orthonormal frames, $f_\alpha$ for $(TX)_{x_\alpha}$, where $1 \leq \alpha \leq l$;

\item Set of points $q_\alpha \in P_0|_{x_\alpha}$, for $1 \leq \alpha \leq l$;

\item Set of points $p_\alpha \in P_\alpha|_s$, for $1 \leq \alpha \leq l$, where $s \in S^4$ is the south pole; and

\item Set of scales $\lambda_\alpha \in (0,1]$, for $1 \leq \alpha \leq l$.
\end{enumerate}
The automorphism groups which act on the gluing data are described in \cite[Sections 2 and 4]{TauFrame}. In general, as in \cite[Sections 2 and 4]{TauFrame}, the preceding construction must be iterated by prescribing a tree with root vertex for each point $x_\alpha \in X$ and a copy of the preceding gluing data for $S^4$ assigned to each vertex in the tree.

The gluing data (which is subject to certain constraints described in \cite[Sections 2 and 4]{TauFrame}) is used to build spliced connections, $A$, that are approximate solutions to the $g$-Yang-Mills equation on a principal $G$-bundle $P$ over $X$. In the $W^{1,2}$ -Sobolev neighborhood of the embedded, finite-dimensional submanifold, $\sN \subset \sB'(P,g)$, of spliced connections corresponding to the gluing data, the value of the Yang-Mills $L^2$ energy functional, $\sE_g$, is approximately $E$ and the $W^{-1,2}$-norm of the gradient of $\sE_g$ is approximately zero, and so, in this neighborhood, a suitably defined Hessian of $\sE_g$ determines the behavior of $\sE_g$. The Hessian of $\sE_g$ has at most $n^*$ small or negative eigenvalues, and all of the other eigenvalues are bounded from below by a positive constant. It is shown in \cite[Section 7]{TauFrame} that the number, $n*$ has an \apriori bound, given $E$. The eigenvalues of the Hessian are weakly continuous on $\sB'(P,g)$ with respect to the $W^{1,2}$ topology.

Taubes' splicing construction produces, for each $[A] \in \sN$, a vector space of dimension $n^*$ or less which approximates the span of those eigenvectors of the Hessian of $\sE_g$ which have small or negative eigenvalue, called the \emph{obstruction space}. At each point $[A] \in \sN$, let $\Pi_A$ denote the projection onto the span of those eigenvectors of the Hessian of $\sE_g$ which have small or negative eigenvalue. (For the sake of exposition, we shall ignore issues related to approximate reducibility of connections $A$ defined by the splicing construction.) The open set $\sN$ has two locally finite covers by open sets $\{\sN_i\}$ and $\{\sN_i'\}$ such that $\sN_i' \subset \sN_i$ and such that over $\sN_i$, the set of projections $\{\Pi_A: [A] \in \sN_i\}$ is smoothly varying. The set of images of $\Pi_A$, as $[A]$ varies through $\sN_i$, defines a vector subbundle, $V_i \subset T\sN_i$, of rank $n^*$ or less, called the \emph{obstruction bundle}.

At each point in this subset, the Hessian of $\sE_g$ is uniformly convex on the complement of $V_i$ in $T\sN_i$. This fact, together with the Contraction Mapping Theorem, allows one to construct a homotopy of $\sB'(P,g)$ which is $\sE_g$-decreasing, and which deforms $\sN_i'$ onto the $\rank(V_i)$-dimensional submanifold of $\sN_i$ consisting of the set of points $[A] \in \sN_i$ where
\begin{equation}
\label{eq:Taubes_1988_2-11}
(1 - \Pi_{A+a})\sE_g'(A+a) = 0,
\end{equation}
for suitable $a \in \Omega^1(X;\ad P)$. This submanifold is denoted $\Sigma_0(\sN_i)$. (We again ignore issues related to approximate reducibility of connections $A$ defined by the splicing construction.) Equation \eqref{eq:Taubes_1988_2-11} is the Yang-Mills equation for $A+a$ composed with an $L^2$-orthogonal projection onto the $L^2$-orthogonal complement of the finite-dimensional obstruction space.

On $\sN_i \cap \sN_j$, either $V_i \subset V_j$ or $V_i = V_j$ or $V_i \supset V_j$. Thus, either $\Sigma_0(\sN_i)\cap \sN_j$ is a smooth submanifold of $\Sigma_0(\sN_j)$, or vice versa. The union of the manifolds $\{\Sigma_0(\sN_i)\}$ defines a smooth variety, $\Sigma_0 \cap \sN$ with maximum dimension $n^*$.

Given $\delta > 0$, one defines, by analogy with \eqref{eq:Taubes_1988_2-1},
\begin{equation}
\label{eq:Taubes_1988_2-12}
\Sigma(E, \delta) := \left\{[A] \in \Sigma_0\cap\sN: |\sE_g(A) - E| < \delta \text{ and } \|\sE_g'(A)\| < \delta\right\}.
\end{equation}
By analogy with \eqref{eq:Taubes_1988_2-2} and \eqref{eq:Taubes_1988_2-3}, one defines
\begin{equation}
\label{eq:Taubes_1988_2-13}
\Sigma(E, \delta)^- := \{[A] \in \Sigma(E,\delta): \sE_g(A) < E \},
\end{equation}
and for $\delta > \delta' > 0$, one has the inclusion of pairs,
\begin{equation}
\label{eq:Taubes_1988_2-14}
\iota: \left(\Sigma(E, \delta'), \Sigma(E, \delta')^-\right) \to \left(\Sigma(E, \delta), \Sigma(E, \delta)^-\right).
\end{equation}
By construction, the effect of the inclusion map in \eqref{eq:Taubes_1988_2-3} on the relative homology and homotopy groups of the pairs of spaces in \eqref{eq:Taubes_1988_2-3} can be calculated from the effect of the inclusion map in \eqref{eq:Taubes_1988_2-14} on the relative groups for the pairs of spaces in \eqref{eq:Taubes_1988_2-14}. This is Taubes' main result in \cite{TauFrame} and the detailed statement is given by \cite[Theorem 1.1]{TauFrame}. It is in this sense that Morse theory for the Yang-Mills $L^2$-energy functional is recovered by the addition of the finite-dimensional variational problems which are implicit in calculating the effect of the map in \eqref{eq:Taubes_1988_2-14} on the relative homology groups of the pairs of finite dimensional varieties in \eqref{eq:Taubes_1988_2-14}.

\subsubsection{Relationship with Theorem \ref{mainthm:Discreteness_Yang-Mills_energies}}
\label{subsubsec:Relationship_between_Taubes_main_theorem_and_discreteness_Yang-Mills_energies}
As is clear from \eqref{eq:Taubes_1988_2-12} and Definition \ref{defn:Moduli_space_Yang-Mills_connections_L2-energy_c}, the subvariety $\Sigma(E, \delta) \subset \sB'(P,g)$ serves as a finite-dimensional neighborhood of a piece of the set $\Crit(P,g,E)$ of critical points of the Yang-Mills $L^2$-energy functional, $\sE_g:\sB'(P,g)\to \RR$, with critical value $E \in [0,\infty)$.

It is important to note that the conclusions in Taubes' \cite[Theorem 1.1]{TauFrame} concerning Morse theory for $\sE_g:\sB'(P,g) \to [0,\infty)$ and the topology of $\sB'(P,g)$ follow from the existence of a countable collection of finite-dimensional varieties, $\{\Sigma_j\}_{j\in\NN} \subset \sB'(P,g)$, as described in Section \ref{subsubsec:Taubes_framework_Morse_theory_Yang-Mills_energy functional}. However, neither \cite[Theorem 1.1]{TauFrame} nor the discussion in \cite[p. 329]{TauFrame} assert that $L^2$ energies of Yang-Mills connections on $P$ form a discrete sequence.

\subsection{Relevance to quantum Yang-Mills field theory over four-dimensional manifolds}
\label{subsec:Relevance_quantum_Yang-Mills_4_manifolds}
While a mathematically rigorous approach to quantum Yang-Mills field theory over conformally compact four-dimensional manifolds remains a distant vision at present, an analysis of the geometry of the critical sets and discreteness of critical values of the Yang-Mills $L^2$ energy functional values should nonetheless have a role in such endeavours. Research aimed at developing the mathematical foundation of quantum Yang-Mills field theory over Riemann surfaces serves to shed light on the corresponding theory in the case of four-dimensional manifolds. There has been considerable interest in the case of Yang-Mills theory over Riemann surfaces and, rather than review previous research here, we shall mention only the work of Atiyah and Bott \cite{Atiyah_Bott_1983} and Witten \cite{Witten_1992} together with the references included in their bibliographies and the many articles citing their work. Witten notes \cite[p. 303]{Witten_1992} that the contributions to the partition function from the unstable classical solutions are `exponentially small' and we might expect that the same is true in the setting of Yang-Mills theory over four-dimensional manifolds.

\subsection{Examples of Schr\"odinger operators with dense point spectrum and examples of energy functionals with dense critical values}
\label{subsec:Examples_Schrodinger_operators_dense_point spectrum_and_energy_functionals_dense_critical_values}
The $L^2$ energy gap results in \cite{Feehan_yangmillsenergygap} are proved with the aid of spectral gap results for Laplace operator, $d_A^{+,g}d_A^{+,*,g}$ on $\Omega^+(X;\ad P)$. Thus it is instructive to compare our gap results for critical values of the Yang-Mills $L^2$-energy functional with related results for energy functionals defined by linear self-adjoint operators on a Hilbert space. To explain those results, it is useful to recall some basic terminology from the spectral theory for linear operators on Banach and Hilbert spaces \cite{Kato, Reed_Simon_v1, Rudin}.

Let $B(\cX)$ be the Banach algebra of all bounded linear operators on a Banach space $\cX$. The \emph{spectrum} $\sigma(L)$ of an operator $L \in B(\cX)$ is the set of all $\lambda \in \CC$ such that $L - \lambda I$ is \emph{not} invertible \cite[Definition 4:17]{Rudin}. Thus $\lambda \in \sigma(L)$ if and only if at least one of the following two statements is true:
\begin{inparaenum}[\itshape i\upshape)]
\item \label{item:not_onto} The range of $L - \lambda I$ is not all of X.
\item \label{item:not_one_to_one} $L - \lambda I$ is not one-to-one.
\end{inparaenum}
If \eqref{item:not_one_to_one} holds, $\lambda$ is said to be an eigenvalue of $L$; the corresponding eigenspace is $\Ker(L - \lambda I)$; each $\xi \in \Ker(L - \lambda I)$ (except $\xi = 0$) is an eigenvector of $L$; it satisfies the equation
\[
L\xi = \lambda\xi.
\]
The \emph{point spectrum} $\sigma_p(L)$ of an operator $L \in B(\cX)$ is the set of all eigenvalues of $L$ \cite[Definition 10.32]{Rudin}. Thus $\lambda \in \sigma_p(L)$ if and only if the null space $\Ker(L - \lambda I)$ of $L - \lambda I$ has positive dimension.

Naboko \cite{Naboko_1986} constructs self-adjoint one-dimensional Schr\"odinger and Dirac operators with dense point spectrum in $[0, \infty)$ and $\RR$, respectively. For the case of the self-adjoint one-dimensional Schr\"odinger operator, one has the

\begin{thm}
\label{thm:Naboko_1}
\cite[Theorem 1 and remark below Theorem 2]{Naboko_1986}, \cite[Theorem 1]{Simon_1997}
Let $\{\kappa_n\}_{n=1}^\infty$ be a sequence of rationally independent positive real numbers. Let $g(x)$ be a monotone function on $[0,\infty)$ with $g(0) = 1$ and $\lim_{x\to\infty} g(x) = \infty$. Then there exists a $C^\infty$ potential $V(x)$ on $[0,\infty)$ such that
\begin{enumerate}
\item For each $n \geq 1$, the equation $-d^2u/dx^2 + Vu = \kappa_n^2u$ on $[0,\infty)$ has a solution $u_n \in L^2(0,\infty)$ with boundary condition $u_n(0) = 0$;

\item $|V(x)| \leq g(x)/(x+1)$, for all $x \in [0,\infty)$.
\end{enumerate}
\end{thm}

Simon \cite{Simon_1997} uses a different method to construct Schr\"odinger operators with dense point spectrum in $[0,\infty)$.

\begin{thm}
\label{thm:Simon_2}
\cite[Theorem 2]{Simon_1997}
Let $\{\kappa_n\}_{n=1}^\infty$ be a sequence of arbitrary distinct positive real numbers. Let $g(x)$ be a monotone function on $[0,\infty)$ with $g(0) = 1$ and $\lim_{x\to\infty} g(x) = 1$. Let $\{\theta_n\}_{n=1}^\infty$ be a sequence of angles in $[0,\pi)$. Then there exists a potential $V(x)$ on $[0,\infty)$ such that
\begin{enumerate}
\item For each $n \geq 1$, the equation $-d^2u/dx^2 + Vu = \kappa_n^2u$ on $[0,\infty)$ has a solution $u_n \in L^2(0,\infty)$ with boundary condition,
    \[
    \frac{u'(0)}{u(0)} = \cot\theta_n.
    \]
\item $|V(x)| \leq g(x)/(x+1)$, for all $x \in [0,\infty)$.
\end{enumerate}
\end{thm}

In particular, Schr\"odinger operators such as those in Theorems \ref{thm:Naboko_1} (Naboko) or \ref{thm:Simon_2} (Simon) can have dense point spectrum in $[0,\infty)$ and thus no spectral gaps anywhere. See also the articles of Christ, Kiselev, and Remling \cite{Christ_Kiselev_1998, Christ_Kiselev_Remling_1997, Remling_1998}, Deift and Killip \cite{Deift_Killip_1999}, Simon \cite{Simon_1997}, as well as the article by Killip and Simon \cite{Killip_Simon_2009} which extend the results of Naboko \cite{Naboko_1986} and Simon \cite{Simon_1997}.

It is interesting to now contrast Theorems \ref{thm:Naboko_1} and \ref{thm:Simon_2} with Theorem \ref{mainthm:Discreteness_Yang-Mills_energies}. Suppose $L$ is a self-adjoint operator on a Hilbert space $\cH$ with domain $\sD(L) \subset \cH$, let $S_\cH := \{\xi \in \cH: \|\xi\| = 1\}$ be the unit sphere in $\cH$, and define an energy functional, $\sF:\sD(L) \cap S_\cH \to \RR$, by setting
\[
\sF(\xi) := \langle L\xi,\xi\rangle, \quad\forall\, \xi \in \sD(L) \cap S_\cH.
\]
The tangent space at $\xi\in S_\cH$ is given by $(TS_\cH)_\xi = \xi^\perp = \{\eta \in \cH: \langle \xi,\eta\rangle = 0\}$. In particular, the expression for the differential map at $\xi \in \sD(L) \cap S_\cH$ is
\[
(D\sF)_\xi(\eta) = 2\langle L\xi,\eta\rangle, \quad\forall\, \eta \in (TS_\cH)_\xi.
\]
Consequently, $\xi \in \sD(L) \cap S_\cH$ is a critical point of $\sF:\sD(L) \cap S_\cH\to \RR$ if and only if $\xi$ is an eigenvector of $L$, that is $L\xi = \lambda\xi$ for some $\lambda\in\RR$. Letting $\sD(L) = W_0^{1,2}(0,\infty) \subset L^2(0,\infty)$ denote the domain of $L$ on $L^2(0,\infty)$ with homogeneous Dirichlet boundary condition at $x=0$, the work of Naboko \cite{Naboko_1986} and Simon \cite{Simon_1997} allows one to construct energy functionals of the form
\[
\sF(u) := \int_0^\infty (|\nabla u|^2 + Vu^2)\,dx, \quad u \in \sD(L)
\text{ and } \|u\|_{L^2(0,\infty)} = 1,
\]
which do not exhibit gaps between their critical values of the kind asserted by Theorem \ref{mainthm:Discreteness_Yang-Mills_energies} for the Yang-Mills $L^2$-energy functional.

\subsection{Existence of non-minimal Yang-Mills connections}
\label{subsec:Existence_non-minimal_Yang-Mills_connections}
Theorem \ref{mainthm:Discreteness_Yang-Mills_energies} has significance because the existence of non-minimal Yang-Mills connections over closed four-dimensional, Riemannian, smooth manifolds has by now long been well-established. We survey some of the known results.

One of the earliest examples (in 1980) of non-minimal Yang-Mills $\SU(2)$-connections over a closed four-dimensional manifold, namely $S^1\times S^3$, is due to Chaohao \cite[Theorem 2]{Chaohao_1980}. Constructions of non-minimal Yang-Mills $\SU(2)$-connections over $S^1\times S^3$ were later given by Urakawa \cite{Urakawa_1988} (in 1988), Wang \cite{Wang_1991} (in 1991), and Parker \cite{Parker_1992invent} (in 1992). Urakawa \cite{Urakawa_1988} and Wang \cite{Wang_1991} also proved the existence of non-minimal Yang-Mills $\SU(2)$-connections on $S^2\times S^2$.

For many years the question of existence or otherwise of non-minimal Yang-Mills $\SU(2)$-connections over $S^4$ remained open until their existence on the product bundle on $P = S^4\times \SU(2)$ was first proved by Sibner, Sibner, and Uhlenbeck \cite{SibnerSibnerUhlenbeck} (published in 1992) and, more generally (in a series of articles published between 1990 and 1994), principal $\SU(2)$-bundles, $P$, over $S^4$ for any $c_2(P) \geq 2$ by Bor and Montgomery \cite{Bor_1992, Bor_Montgomery_1990} and Sadun and Segert \cite{Sadun_1994, Sadun_Segert_1991, Sadun_Segert_1992cmp, Sadun_Segert_1992cpam, Sadun_Segert_1993}.

More recently Gritsch \cite{Gritsch_2000} (in 2000) has proved existence of non-minimal Yang-Mills connections over a broader family of four-dimensional manifolds. We recall her construction of those manifolds. Let $F_{2g}$ be a Riemann surface of genus $2g$ with involution given by reflection in the mid-circle and $S^2$ the $2$-sphere with the antipodal map. These involutions combine to give an involution $\sigma:F_{2g}\times S^2 \to F_{2g}\times S^2$ with no fixed points. Then $M_{2g} := F_{2g}\times_\sigma S^2$ is a closed, oriented, smooth, spin four-dimensional manifold. The second Chern classes of the $\pm$ spinor bundles $W^\pm$ over $M_{2g}$ are $c_2(W^\pm) = \pm(2g-1)$. Gritsch uses homotopy theory to prove that there exist non-minimal Yang-Mills connections on the bundles $W^\pm$.

Given one isolated non-minimal Yang-Mills connection (where isolated means that the Hessian of Yang-Mills functional is non-degenerate), Wang \cite[Theorem 1]{Wang_1997} has proved using the min-max method that there are then infinitely many other non-minimal Yang-Mills connections.

Finally, geometric questions concerning non-minimal Yang-Mills connections have also been explored by Stern \cite{Stern_2010}.

\subsection{Discussion of the main ideas underlying the proof of Theorem \ref{mainthm:Discreteness_Yang-Mills_energies}}
\label{subsec:Discussion_main_ideas}
In Sections \ref{subsubsec:Lojasiewicz-Simon_gradient_inequality_moduli_space_flat_connections} and \ref{subsubsec:Lojasiewicz-Simon_gradient_inequality_Yang-Mills_connections_sphere} we describe some of the key ideas used in the proof of Theorem \ref{mainthm:Discreteness_Yang-Mills_energies} --- where we impose no conditions on the Riemannian metric $g$, compact Lie group $G$, or topologies of the principal $G$-bundle $P$ and the closed manifold $X$ --- and contrast them in Sections \ref{subsubsec:Spectral_gaps_and_moduli_space_anti-self-dual_connections} and \ref{subsubsec:Lojasiewicz-Simon_gradient_inequality_moduli_space_flat_connections} with earlier results yielding an $L^2$-energy gap near the ground state but only under rather strong hypotheses on $g$, $G$, $P$, and $X$.

\subsubsection{Spectral gap for a Laplace operator and energy gap near the moduli space of anti-self-dual or self-dual connections}
\label{subsubsec:Spectral_gaps_and_moduli_space_anti-self-dual_connections}
In our \cite[Theorem 1 and Corollary 2]{Feehan_yangmillsenergygap}, we establish that the $L^2$ energies of non-minimal $g$-Yang-Mills connections on a principal $G$-bundle, with compact Lie structure group $G$, over a closed, connected, four-dimensional, oriented, Riemannian, smooth manifold $(X,g)$, are separated from the $L^2$ energy of the minimal $g$-Yang-Mills connections by a uniform positive constant depending at most on the Riemannian metric $g$ and the homotopy type $[P]$ of $P$. In particular, rather than require that $g$ be \emph{positive} in the sense of \eqref{eq:Freed_Uhlenbeck_page_174_positive_metric}, as assumed by Bourguignon, Lawson, and Simon \cite{Bourguignon_Lawson_1981, Bourguignon_Lawson_Simons_1979} in their $L^\infty$-energy gap results
or by Min-Oo \cite[Theorem 2]{Min-Oo_1982} and Parker \cite[Proposition 2.2]{ParkerGauge} in their $L^2$-energy gap results (see also Donaldson and Kronheimer \cite[Lemma 2.3.24]{DK}) --- which constrains $X$ to have negative definite intersection form --- we instead assume in \cite[Corollary 2]{Feehan_yangmillsenergygap} that the Lie structure group is $\SU(2)$ or $\SO(3)$ and the Riemannian metric is generic in the sense of \cite{FU}, and impose mild conditions on the topologies of $P$ and $X$ inspired by those employed in the most general definitions of the Donaldson invariants \cite{DonPoly,DK,  KMStructure, MorganMrowkaPoly} of $X$.

It is enlightening to recall the main ideas underlying the proofs of \cite[Theorem 2]{Min-Oo_1982}, \cite[Proposition 2.2]{ParkerGauge}, \cite[Theorem 1 and Corollary 2]{Feehan_yangmillsenergygap}. Let $R_g(x)$ denote the scalar curvature of $g$ at a point $x \in X$ and let $\sW_g^\pm(x) \in \End(\Lambda_x^\pm)$ denote its self-dual and anti-self-dual Weyl curvature tensors at $x$, where $\Lambda_x^2 = \Lambda_x^+\oplus \Lambda_x^-$. Define
\[
w_g^\pm(x) := \text{Largest eigenvalue of } \sW_g^\pm(x), \quad\forall\, x \in X.
\]
We recall the following Bochner-Weitzenb\"ock formula
\cite[Equation (6.26) and Appendix C, p. 174]{FU}, \cite[Equation (5.2)]{GroisserParkerSphere},
\begin{equation}
\label{eq:Freed_Uhlenbeck_6-26}
2d_A^{+,g}d_A^{+,*_g}v = \nabla_A^{*_g}\nabla_Av + \left(\frac{1}{3}R_g - 2w_g^+\right)v + \{F_A^{+,g}, v\},
\quad\forall\, v \in \Omega^{+,g}(X; \ad P).
\end{equation}
We call a Riemannian metric, $g$, on $X$ \emph{positive} if
\begin{equation}
\label{eq:Freed_Uhlenbeck_page_174_positive_metric}
\frac{1}{3}R_g - 2w_g^+ > 0 \quad\hbox{on } X.
\end{equation}
Consequently, if $\|F_A^{+,g}\|_{L^2(X,g)} < \eps(g) \in (0,1]$, then the Laplace operator, $d_A^{+,g}d_A^{+,*_g}$ on $\Omega^{+,g}(X; \ad P)$ exhibits a spectral gap in the sense that its least eigenvalue, $\mu_g(A)$, has a positive lower bound $\mu_0=\mu_0(g)$ that is uniform with respect to $[A] \in \sB(P,g)$ obeying
\[
\|F_A^{+,g}\|_{L^2(X,g)} < \eps(g).
\]
Standard arguments (see the proofs by Min-Oo of \cite[Theorem 2]{Min-Oo_1982} and Parker of \cite[Proposition 2.2]{ParkerGauge}) imply that if $A$ is also $g$-Yang-Mills, then one must have
\[
F_A^{+,g} = 0 \quad\text{on } X,
\]
and so $A$ is a $g$-anti-self-dual connection on $P$, that is, an absolute minimum of the Yang-Mills $L^2$ energy functional on $\sB(P,g)$. By reversing orientations of $X$, one also sees that if $A$ is a $g$-Yang-Mills connection $A$ with $\|F_A^{-,g}\|_{L^2(X,g)} < \eps(g)$, then $F_A^{-,g} = 0$ on $X$ and $A$ is a $g$-self-dual connection on $P$, also an absolute minimum of the Yang-Mills $L^2$ energy functional on $\sB(P,g)$. Results of this kind were first obtained by Bourguignon, Lawson, and Simon \cite{Bourguignon_Lawson_1981, Bourguignon_Lawson_Simons_1979} but under the stronger hypotheses that $\|F_A^{+,g}\|_{L^\infty(X,g)} < \eps(g)$ or $\|F_A^{-,g}\|_{L^\infty(X,g)} < \eps(g)$.

The structure of the Bochner-Weitzenb\"ock formula \eqref{eq:Freed_Uhlenbeck_6-26} reveals that the method of proof of \cite[Theorem 2]{Min-Oo_1982} or \cite[Proposition 2.2]{ParkerGauge} is unlikely to yield any information concerning energy gaps near higher critical values of the Yang-Mills $L^2$-energy functional, where $\|F_A^{+,g}\|_{L^2(X,g)}$ or $\|F_A^{-,g}\|_{L^2(X,g)}$ is large and the positivity condition \eqref{eq:Freed_Uhlenbeck_page_174_positive_metric} can no longer ensure that $\mu_g(A) \geq \mu_0(g) > 0$ for all $[A] \in \sB(P,g)$ with $\|F_A^{+,g}\|_{L^2(X,g)} < \eps(g)$.

The main idea in the proof of our \cite[Corollary 2]{Feehan_yangmillsenergygap} is to show that the positivity hypothesis \eqref{eq:Freed_Uhlenbeck_page_174_positive_metric} for the Riemannian metric $g$ can be replaced by the requirement that $g$ be generic and $G$ have dimension three, together one of the following combinations of hypotheses on $G$ and the topologies of $P$ and $X$:
\begin{enumerate}
\item $b^+(X) = 0$, the fundamental group $\pi_1(X)$ has no non-trivial representations in $G$, and $G = \SU(2)$ or $G = \SO(3)$; or

\item $b^+(X) > 0$, the fundamental group $\pi_1(X)$ has no non-trivial representations in $G$, and $G = \SO(3)$, and the second Stiefel-Whitney class, $w_2(P) \in H^2(X;\ZZ/2\ZZ)$, is non-trivial; or

\item $b^+(X) \geq 0$, and $G = \SO(3)$, and no principal $\SO(3)$-bundle $P_l$ over $X$ appearing in the Uhlenbeck compactification of the moduli space of $g$-anti-self-dual connections on $P$, namely $\bar M(P,g)$, admits a flat connection.
\end{enumerate}
Unfortunately, while the preceding combinations of hypotheses on $g$, $G$, $P$, and $X$ again yield an $L^2$-energy gap result for the ground state,
\[
c_0 < c_1,
\]
where $c_0$ is the $L^2$-energy of a $g$-anti-self-dual (or $g$-self-dual) connection on $P$ and $c_1$ is the minimum $L^2$-energy of a non-minimal Yang-Mills connection on $P$, our proof of \cite[Corollary 2]{Feehan_yangmillsenergygap} still requires that $\|F_A^{+,g}\|_{L^2(X,g)}$ (or $\|F_A^{-,g}\|_{L^2(X,g)}$) be less than $\eps(g,[P]) \in (0,1]$ in order to establish the existence of a spectral gap for the Laplace operator, $d_A^{+,g}d_A^{+,*_g}$, in the sense that $\mu_g(A) \geq \mu_0(g,[P]) > 0$ for all $[A] \in \sB(P,g)$ with $\|F_A^{+,g}\|_{L^2(X,g)} < \eps(g,[P])$. Therefore, this alternative method yields no information either regarding higher-order critical values of the Yang-Mills $L^2$-energy functional.

\subsubsection{{\L}ojasiewicz-Simon gradient inequality and energy gap near the moduli space of flat connections}
\label{subsubsec:Lojasiewicz-Simon_gradient_inequality_moduli_space_flat_connections}
Our sequel \cite{Feehan_yangmillsenergygapflat} to our article \cite{Feehan_yangmillsenergygap} establishes an $L^{d/2}$-energy gap \cite[Theorem 1]{Feehan_yangmillsenergygapflat} for a $g$-Yang-Mills connection $A$ on a principal $G$-bundle, with compact Lie structure group $G$, over a closed, connected, smooth manifold, $X$, of dimension $d \geq 2$ and endowed with a Riemannian metric $g$. Indeed, if
\[
\|F_A\|_{L^{d/2}(X,g)} < \eps,
\]
where $\eps = \eps(d,g,[P]) \in (0,1]$ and $A$ is $g$-Yang-Mills, then \cite[Theorem 1]{Feehan_yangmillsenergygapflat} asserts that $F_A = 0$ on $X$, that is, $A$ is a flat connection. When $d=4$, we obtain an $L^2$-energy gap result for the ground state,
\[
c_0 = 0 < c_1,
\]
where $c_0$ is the $L^2$-energy of a flat connection on $P$ and $c_1$ is the minimum $L^2$-energy of a non-minimal Yang-Mills connection on $P$.

Our proof of \cite[Theorem 1]{Feehan_yangmillsenergygapflat} uses our version of the {\L}ojasiewicz-Simon gradient inequality \cite[Theorem 21.8]{Feehan_yang_mills_gradient_flow} to remove a positivity constraint on a combination of the Ricci and Riemannian curvatures in a previous $L^{d/2}$-energy gap result due to Gerhardt \cite[Theorem 1.2]{Gerhardt_2010} and a previous $L^\infty$-energy gap result due to Bourguignon, Lawson, and Simons \cite[Theorem C]{Bourguignon_Lawson_1981}, \cite[Theorem 5.3]{Bourguignon_Lawson_Simons_1979}. Our proof of \cite[Theorem 1]{Feehan_yangmillsenergygapflat} relies on the fact that the moduli space, $M(P,g)$, of gauge-equivalence classes of flat connections on $P$ is \emph{compact} with respect to the Uhlenbeck topology \cite{UhlLp, UhlChern}. We shall explain the roles of the {\L}ojasiewicz-Simon gradient inequality and Uhlenbeck compactness of $M(P,g)$ --- and also how reliance on Uhlenbeck compactness of $M(P,g)$ might be relaxed ---  in Section \ref{subsubsec:Lojasiewicz-Simon_gradient_inequality_Yang-Mills_connections_sphere}.

\subsubsection{{\L}ojasiewicz-Simon gradient inequality for Yang-Mills connections over the sphere and main ideas underlying the proof of Theorem \ref{mainthm:Discreteness_Yang-Mills_energies}}
\label{subsubsec:Lojasiewicz-Simon_gradient_inequality_Yang-Mills_connections_sphere}
While our proof of Theorem \ref{mainthm:Discreteness_Yang-Mills_energies} formally mirrors that of \cite[Theorem 1]{Feehan_yangmillsenergygapflat}, in so far as that we rely on our {\L}ojasiewicz-Simon gradient inequality (Theorem \ref{thm:Rade_proposition_7-2}) and compactness results (developed in this article) for sequences of Yang-Mills connections with bounded $L^2$ energy, it is the non-compactness due to the energy bubbling phenomenon which makes the proof of Theorem \ref{mainthm:Discreteness_Yang-Mills_energies} difficult in the present context of bounded but otherwise arbitrary $L^2$ energy as opposed to $L^2$ energy near zero.

Our proof of our previous energy gap result, \cite[Theorem 1]{Feehan_yangmillsenergygapflat}, exploits the fact that the {\L}ojasiewicz-Simon radius, $\sigma[\Gamma]$, has a uniform positive lower bound as $\Gamma$ varies over the moduli space, $\Crit(P,g,0)$, of flat connections on a principal $G$-bundle over $X$. Now, however, $[A_\ym]$ varies over the moduli space of Yang-Mills connections, $\Crit(P,g,\sC)$, where $\sC \Subset [0,\infty)$ is any compact subset. Because $\Crit(P,g,\sC)$ is non-compact due to bubbling when $\sC = [0,E]$ and $E>0$ is sufficiently large (for example, greater than or equal to the energy of a non-flat anti-self-dual or self-dual connection on a principal $G$-bundle over $S^4$), it is no longer clear that the {\L}ojasiewicz-Simon radius, $\sigma[A_\ym]$, will have a uniform positive lower bound as $[A_\ym]$ varies over $\Crit(P,g,\sC)$. A direct examination, via the lengthy proof of Theorems \ref{thm:Rade_proposition_7-2} or \ref{thm:Rade_proposition_7-2_d_is_4}, of the dependence of the {\L}ojasiewicz-Simon triple of constants, $(Z,\sigma,\theta)$, on $[A_\ym]$ would be very difficult. However, one can gain insight into why there is still a uniform positive lower bound for $\sigma[A_\ym]$ when $[A_\ym]$ varies over $\Crit(P,g,[0,E])$, for $E$ large, by examining the case where $(X,g)$ is $(S^4,g_\round)$ and $g_\round$ is the standard round metric of radius one over $S^4$.

In our application to the proof of Theorem \ref{mainthm:Discreteness_Yang-Mills_energies}, we shall only need a weaker special case, essentially Corollary \ref{cor:Rade_proposition_7-2_d_is_4_conformally_invariant}, of the full {\L}ojasiewicz-Simon gradient inequality in dimension four, Theorem \ref{thm:Rade_proposition_7-2_d_is_4}, where the $W_A^{-1,2}$ norm on $d_A^*F_A$ in \eqref{eq:Rade_7-1_d_is_4} is replaced by the $L^4$ norm on $d_A^*F_A$ to give \eqref{eq:Rade_7-1_conformally_invariant}. This simplification is too weak for application to the questions of convergence and global existence of solutions to Yang-Mills gradient flow developed by the author in \cite{Feehan_yang_mills_gradient_flow}, but suffices for our application to the proof of Theorem \ref{mainthm:Discreteness_Yang-Mills_energies}; the simplification is attractive because the Yang-Mills equation, $d_A^*F_A = 0$, is conformally invariant (since the Yang-Mills $L^2$ energy functional is conformally invariant) and the $L^4$ norm on the one-form, $d_A^*F_A$, is invariant with respect to the pull-back action on $A$ and elements of $\Omega^1(X,\ad P)$ by conformal diffeomorphisms of $(X,g)$. If $A_\ym$ is a $g$-Yang-Mills connection on a principal $G$-bundle $P$ over $X$ and $A$ is a $W^{2,2}$ connection on $P$ that obeys
\[
\|A - A_\ym\|_{W_{A,g}^{1,2}(X)} < \sigma[A_\ym],
\]
then Corollary \ref{cor:Rade_proposition_7-2_d_is_4_conformally_invariant} implies that $A$ obeys the \emph{{\L}ojasiewicz-Simon gradient inequality},
\[
\|d_A^*F_A\|_{L^4(X,g)} \geq Z|\sE_g(A) - \sE_g(A_\ym)|^\theta.
\]
Now suppose that $A_\ym$ is \emph{centered} in the sense of Definition \ref{defn:Mass_center_scale_connection}, with center at the north pole and scale one, and let $\tilde \delta_\lambda$ be the conformal transformation of $S^4\subset\RR^5$ induced by a stereographic projection from the south pole and the rescaling map on $\RR^4$ given by $\delta_\lambda(x) = x/\lambda$ for $x\in\RR^4$ and $\lambda\in \RR_+$. Then $\tilde\delta_\lambda^*A_\ym$ is a one-parameter family of $g_\round$-Yang-Mills connections on $P$ over $S^4$ whose energy density becomes totally concentrated at the north pole when $\lambda\searrow 0$ and totally concentrated at the south pole when $\lambda\nearrow \infty$. However, while the standard Sobolev norm on $W_A^{1,2}(S^4;\Lambda^1\otimes\ad P)$ (see Section \ref{subsec:Sobolev_spaces_and_connections} for an explanation of the notation for Sobolev norms),
\[
\|a\|_{W_A^{1,2}(S^4)}
\equiv
\|a\|_{W_{A,g_\round}^{1,2}(S^4)}
:=
\left(\int_{S^4} \left(|\nabla_A^{g_\round}a|^2 + |a|^2 \right)\,d\vol_{g_\round} \right)^{1/2},
\quad
\forall\, a \in \Omega^1(S^4,\ad P),
\]
is not invariant with respect to the pull-back action on $a \in \Omega^1(S^4,\ad P)$ and connections $A$ on $P$ by elements, $h$, of the conformal group of transformations of $S^4$, it is \emph{quasi-conformally invariant}, with universal constants as one can see from Taubes \cite[Lemma 3.1]{TauFrame}, quoted here as Lemma \ref{lem:Taubes_1988_3-1}. In particular, there is a universal constant $C \in [1,\infty)$ such that, for any $W^{1,2}$ connection $A$ on $P$ and any element $h$ of the group $\Conf(S^4)$ of conformal transformations of $(S^4,g_\round)$, one has
\[
C^{-1}\|a\|_{W_A^{1,2}(S^4)}
\leq
\|h^*a\|_{W_{h^*A}^{1,2}(S^4)}
\leq
C^{-1}\|a\|_{W_A^{1,2}(S^4)},
\quad\forall\, a \in W_A^{1,2}(S^4).
\]
Hence, if a $W^{2,2}$ connection $A$ on $P$ obeys
\begin{equation}
\label{eq:W12_norm_A_minus_lambda*Aym_lessthan_Cinverse_sigma_Aym}
\|A - h^*A_\ym\|_{W_{h^*A_\ym}^{1,2}(S^4)} < C^{-1}\sigma[A_\ym],
\end{equation}
then $h^{-1,*}A$ obeys
\[
\|h^{-1,*}A - A_\ym\|_{W_{A_\ym}^{1,2}(S^4)} < \sigma[A_\ym],
\]
and the {\L}ojasiewicz-Simon gradient inequality for $A_\ym$ implies that
\[
\|d_{h^{-1,*}A}^*F_{h^{-1,*}A}\|_{L^4(S^4)}
\geq
Z|\sE(h^{-1,*}A) - \sE(A_\ym)|^\theta.
\]
By the invariance of the $L^4$ norm on one-forms and $L^2$ norm on two-forms with respect to the pull-back action on forms by conformal diffeomorphisms, the preceding inequality is equivalent to
\[
\|d_A^*F_A\|_{L^4(S^4)}
\geq
Z|\sE(A) - \sE(h^*A_\ym)|^\theta.
\]
In other words, we have shown that
\[
\sigma[h^*A_\ym] \geq C^{-1}\sigma[A_\ym], \quad\forall\, h \in \Conf(S^4),
\]
and the following choice of {\L}ojasiewicz-Simon triple of constants for $[h^*A_\ym]$,
\[
(Z[h^*A_\ym],\sigma[h^*A_\ym],\theta[h^*A_\ym])
=
(Z[A_\ym],C^{-1}\sigma[A_\ym],\theta[A_\ym])
\]
depends only on the point $[A_\ym] \in \Crit(P,g_\round,[0,E])$ and \emph{not} on the map $h \in \Conf(S^4)$.

Therefore, if $A$ is a connection that obeys
\[
\|A - h^*A_\ym\|_{W_{h^*A_\ym}^{1,2}(S^4)}
< C^{-1}\sigma[A_\ym],
\]
and $A$ is $g_\round$-\emph{Yang-Mills}, then $d_A^*F_A = 0$ on $X$ and the {\L}ojasiewicz-Simon gradient inequality implies that
\[
\sE(A) = \sE(h^*A_\ym) = \sE(A_\ym).
\]
The preceding observations are more significant than they might appear at first glance. Although $A$ is $W^{1,2}$ close to $h^*A_\ym$ in the sense of \eqref{eq:W12_norm_A_minus_lambda*Aym_lessthan_Cinverse_sigma_Aym}, the space $\Crit(P,g_\round,[0,E])$ is not bounded with respect to the distance function defined by the Sobolev $W^{1,2}$ norm, and so it is essential that the {\L}ojasiewicz-Simon radius, $\sigma[h^*A_\ym]$, be comparable to $\sigma[A_\ym]$ with universal constants that are independent of $h \in \Conf(S^4) = \RR^4\times\RR_+\times\SO(4)$, where $\RR_+$ acts on $S^4$ by rescaling as noted above, $\RR^4$ acts on $S^4 = \RR^4\cup\{\infty\}$ by translation, and $\SO(4)$ acts on $S^4\subset\RR^5$ by rotation.

For the general case, when $P$ is a principal $G$-bundle over an arbitrary closed, Riemannian, smooth manifold, $(X,g)$, one has to examine the \emph{bubble-tree compactification} of $\Crit(P,g,\sC)$ in great detail, extending previous work of Parker and Wolfson in the context of pseudo-holomorphic maps \cite{ParkerHarmonic, ParkerWolfson}, Taubes \cite{TauFrame} in the context of sequences of Yang-Mills connections, and the author \cite{FeehanGeometry} in the context of sequences of anti-self-dual connections. By iterating the application of local dilations and local translations to a sequence of $g$-Yang-Mills connections on $P$ near points of curvature concentration in $X$ and then near points of curvature concentration in copies of $S^4$, one obtains a sequence of $g^\#$-Yang-Mills connections over a connected sum of $X$ and copies of $S^4$ associated with vertices of the trees arising in the bubble-tree closure of $\Crit(P,g,\sC)$, where $g^\#$ is a Riemannian metric on the connected sum that is conformally equivalent to $g$.

We prove that $\Crit(P,g,\sC)$ is covered by finitely many suitably-defined \emph{coarse} $W_\loc^{1,2}$ bubble-tree open neighborhoods and hence we prove Theorem \ref{mainthm:Discreteness_Yang-Mills_energies} by showing that the energies of Yang-Mills connections must coincide if they belong to any one such small-enough neighborhood.

\subsubsection{Bubble-tree closures and Sobolev metric completions of moduli spaces of Yang-Mills connections}
\label{subsubsec:Bubble-tree closures and Sobolev metric completions moduli spaces_Yang-Mills connections}
The restriction of the $W^{1,2}$ distance function \eqref{eq:Wkp_distance_on_quotient_space_connections} on $\sB(P,g)$ to the moduli subspace $\Crit(P,g,\sC)$ of Yang-Mills connections with $L^2$ energies in a compact subset $\sC \Subset [0,\infty)$ yields the metric $\dist_{W^{1,2}}(\cdot,\cdot)$ on $\Crit(P,g,\sC)$. The proof of Proposition \ref{prop:W12_continuity_family_connections_wrt_mass_centers_and_scales} suggests that $\Crit(P,g,\sC)$ should have infinite volume and diameter with respect to the infinitesimal $W^{1,2}$ Riemannian metric. The fine bubble-tree closure of $\Crit(P,g,\sC)$ should be equal to the completion of $\Crit(P,g,\sC)$ with respect to $\dist_{W^{1,2}}(\cdot,\cdot)$. See Remark \ref{rmk:Uhlenbeck_compactification_and_L2-metric_completion_moduli_space_anti-self-dual_connections} for a discussion of the corresponding results relating the Uhlenbeck closure of $\Crit(P,g,\sC)$ and the completion of $\Crit(P,g,\sC)$ with respect to $\dist_{L^2}(\cdot,\cdot)$; our article \cite{FeehanGeometry} and those of Groisser and Parker \cite{Groisser_1990, GroisserParkerSphere, GroisserParkerGeometryDefinite} and of Peng \cite{Peng_1995, Peng_1996} and the references cited therein contain further results concerning the relationship between the Uhlenbeck closure and $L^2$ metric completion of $M(P,g)$.

When $P$ is a principal $\SU(2)$ bundle over $S^4$ with its standard metric of radius one and $c_2(P)=1$ and we restrict our attention to the moduli space $M(P,g)$ of anti-self-dual connections, the results of Doi and Kobayashi \cite{Doi_Kobayashi_1990}, and Matumoto \cite{Matumoto_1989} prove that infinitesimal metrics on $M(P,g)$ --- which should be comparable to $\dist_{W^{1,2}}(\cdot,\cdot)$ --- are complete and have infinite volume and diameter. See also the work of Babadshanjan and Habermann \cite{Babadshanjan_Habermann_1991}. Matumoto \cite{Matumoto_1989} considers the metrics
\[
(\bg_{\textrm{II}})_{[A]}(a,b) := (\Pi_A d_Aa, \Pi_A d_Ab), \quad\forall\, a, b \in \Omega^1(X; \ad P),
\]
where $\Pi_A:L^2(X;\Lambda^1\otimes\ad P)$ denotes $L^2$-orthogonal projection onto $(d_A^2\Omega^0(X;\ad P))^\perp$, and
\[
(\bg_{\textrm{I-II}})_{[A]}(a,b) := (d_A\Pi_A a, d_A\Pi_A b), \quad\forall\, a, b \in \Omega^1(X; \ad P),
\]
where $\Pi_A:L^2(X;\Lambda^1\otimes\ad P)$ denotes $L^2$-orthogonal projection onto $\Ker d_A^*\cap \Omega^1(X;\ad P) = (d_A\Omega^0(X;\ad P))^\perp$. For $X = S^4$ and $G = \SU(2)$, the metric $\bg_{\textrm{II}}$ has a constant negative sectional curvature $-5/(32\pi^2)$, so this metric is hyperbolic and complete, while the metric $\bg_{\textrm{I-II}}$ has negative sectional curvature everywhere and is also complete.

\subsection{Outline of the article}
\label{subsec:Outline}
The content of each of the following sections is accompanied by its own introduction, so we shall just briefly indicate the overall structure of the remainder of our article. Section \ref{sec:Preliminaries} provides definitions of analytical and topological concepts in gauge theory which we shall need throughout our article. Section \ref{sec:Uhlenbeck_compactness_Yang-Mills_and_anti-self-dual_connections} develops results concerning Uhlenbeck convergence for a sequence of Yang-Mills connections with a uniform $L^2$ bound on their curvatures. In Section \ref{sec:Mass_center_and_scale_maps_on_space_connections_over_sphere}, we discuss the definitions and properties of the mass center and scale maps on the quotient space of connections over the sphere. Section \ref{sec:Bubble-tree_compactness_Yang-Mills_and_anti-self-dual_connections} initiates our development of bubble-tree convergence properties of a sequence of Yang-Mills connections with a uniform $L^2$ bound on their curvatures. In Section \ref{sec:Bubble_trees_and_Riemannian_metrics_connected_sums}, we describe the construction of smooth Riemannian metrics on connected sums of $X$ and copies of $S^4$, as prescribed by bubble-tree data. In Section \ref{sec:Global W12_metrics_bubble-tree_neighborhoods} we explore the relationship between $W_\loc^{1,2}$ bubble-tree open neighborhoods of the moduli space of Yang-Mills connections with a uniform $L^2$ bound on their curvatures and the global $W^{1,2}$ distance between pairs of connections, separately analyzing the case of \emph{fine} and \emph{coarse} bubble-tree neighborhoods. We conclude in Section \ref{sec:Lojasiewicz-Simon gradient inequality_bubble-tree_compactification} with a discussion of the {\L}ojasiewicz-Simon gradient inequality and show how a version of this inequality for the Yang-Mills $L^2$-energy functional and the bubble-tree compactification for the moduli space of Yang-Mills connections with a uniform $L^2$ bound on their curvatures yield the main results of this article.

\subsection{Acknowledgments}
\label{subsec:Acknowledgments}
I would like to thank George Daskalopoulos and Richard Wentworth for their explanations of results for the Yang-Mills $L^2$-energy functional over Riemann surfaces and K\"ahler surfaces, Tom Leness for allowing me to include sections from our forthcoming monograph on gluing $\SO(3)$ monopoles and anti-self-dual connections \cite{Feehan_Leness_monopolegluingbook}, Tim Nguyen for his comments on a preliminary draft of this article, Lorenzo Sadun for his explanations of results for non-minimal Yang-Mills connections, and Cliff Taubes for explanations of his results on Morse theory for the Yang-Mills $L^2$-energy functional over four-dimensional manifolds. I am grateful to the Department of Mathematics at Rutgers University for research support and the Department of Mathematics at Columbia University for their hospitality during the preparation of this article.

\section{Preliminaries}
\label{sec:Preliminaries}
We shall generally adhere to the now standard gauge-theory conventions and notation of Donaldson and Kronheimer \cite{DK}, Freed and Uhlenbeck \cite{FU}, and Friedman and Morgan \cite{FrM}; those references and our monograph \cite{Feehan_yang_mills_gradient_flow} also provide the necessary background for our article.

\subsection{Sobolev spaces and connections}
\label{subsec:Sobolev_spaces_and_connections}
Throughout our article, $G$ denotes a compact Lie group and $P$ a smooth principal $G$-bundle over a smooth manifold, $X$, of dimension $d \geq 2$ and endowed with Riemannian metric, $g$. We denote $\Lambda^l := \Lambda^l(T^*X)$ for integers $l\geq 1$ and $\Lambda^0 = X\times\RR$, and let $\ad P := P\times_{\ad}\fg$ denote the real vector bundle associated to $P$ by the adjoint representation of $G$ on its Lie algebra,
$\Ad:G \ni u \to \Ad_u \in \Aut\fg$, with fiber metric defined through the Killing form on $\fg$. Given a $C^\infty$ reference connection, $A$, on $P$, we let
\begin{align*}
\nabla_A \text{ or } \nabla_A^g: C^\infty(X;\Lambda^l\otimes\ad P) &\to C^\infty(X; T^*X\otimes \Lambda^l\otimes\ad P),
\\
d_A: C^\infty(X; \Lambda^l\otimes\ad P) &\to C^\infty(X; \Lambda^{l+1}\otimes\ad P), \quad l \in \NN,
\end{align*}
denote the \emph{covariant derivative} \cite[Equation (2.1.1)]{DK} and \emph{exterior covariant derivative} \cite[Equation (2.1.12)]{DK}, respectively, defined by the connection $A$ on $P$ and Levi-Civita connection for the Riemannian metric, $g$, on the tangent bundle, $TX$, and all associated vector bundles. We write the set of non-negative integers as $\NN$ and abbreviate $\Omega^l(X; \ad P) :=  C^\infty(X; \Lambda^l\otimes\ad P)$, the Fr\'echet space of $C^\infty$ sections of $\Lambda^l\otimes\ad P$.

More generally, given a Hermitian or Riemannian vector bundle, $E$, over $X$ and covariant derivative, $\nabla_A$, which is compatible with the fiber metric on $E$, we denote the Banach space of sections of $E$ of Sobolev class $W^{k,p}$, for any $k\in \NN$ and $p \in [1,\infty]$, by $W_A^{k,p}(X; E)$, with norm,
\begin{equation}
\label{eq:Sobolev_norm_WAkp_sections_vector_bundle_over_manifold_finite_p}
\|v\|_{W_A^{k,p}(X)} := \left(\sum_{j=0}^k \int_X |\nabla_A^j v|^p\,d\vol_g \right)^{1/p},
\end{equation}
when $1\leq p<\infty$ and
\begin{equation}
\label{eq:Sobolev_norm_WAkp_sections_vector_bundle_over_manifold_infinite_p}
\|v\|_{W_A^{k,\infty}(X)} := \sum_{j=0}^k \esssup_X |\nabla_A^j v|,
\end{equation}
otherwise, where $v \in W_A^{k,p}(X; E)$.

\subsection{Sobolev embeddings and Kato inequality}
\label{subsec:Sobolev_embeddings_Kato_inequality}
Suppose that $U \subset \RR^d$ is an open subset obeying an interior cone condition \cite[Section 4.6]{AdamsFournier} with cone $\kappa$ and $d \geq 2$. We recall from \cite[Theorem 4.12]{AdamsFournier} that there are continuous embeddings,
\begin{equation}
\label{eq:Sobolev_embedding_domain}
W^{k,p}(U;\RR)
\hookrightarrow
\begin{cases}
L^q(U;\RR), &\quad\text{for } kp < d \text{ and } p > 1 \text{ and } p \leq q \leq p^* := dp/(d-kp),
\\
L^q(U;\RR), &\quad\text{for } kp = d \text{ and } p \leq q < \infty,
\\
C_b(U;\RR), &\quad\text{for } kp > d \text{ and } p > 1,
\end{cases}
\end{equation}
for $p \geq 1$ and integers $k \geq 0$, where the norm of the embedding depends at most on $k$, $p$, and $\kappa$. If $U$ obeys the strong local Lipschitz condition \cite[Section 4.9]{AdamsFournier}, then $C_b(U;\RR)$ may be replaced by $C(\bar U;\RR)$ in \eqref{eq:Sobolev_embedding_domain}.

If $(X,g)$ is Riemannian smooth manifold that is closed, compact with boundary, or complete, and having bounded geometry in the sense of injectivity radius having a uniform positive lower bound, $\varrho_0$, and Riemannian curvature tensor with uniform bounds on its covariant derivatives with respect to the Levi-Civita connection, then the preceding local embeddings take the form
\begin{equation}
\label{eq:Sobolev_embedding_manifold_bounded_geometry}
W_g^{k,p}(X;\RR)
\hookrightarrow
\begin{cases}
L^q(X,g;\RR), &\quad\text{for } kp < d \text{ and } p > 1 \text{ and } p \leq q \leq p^* := dp/(d-kp),
\\
L^q(X,g;\RR), &\quad\text{for } kp = d \text{ and } p \leq q < \infty,
\\
C_b(X,g;\RR), &\quad\text{for } kp > d \text{ and } p > 1,
\end{cases}
\end{equation}
where the norm of the embedding depends at most on $k$, $p$, $\varrho_0$, and $\|\nabla_g^k\Riem_g\|_{C_b(X)}$ and the norm on the Banach spaces $W_g^{k,p}(X;\RR)$ is defined via \eqref{eq:Sobolev_norm_WAkp_sections_vector_bundle_over_manifold_finite_p} or \eqref{eq:Sobolev_norm_WAkp_sections_vector_bundle_over_manifold_infinite_p}, with covariant derivative, $\nabla_A^g$ replaced by $\nabla^g$ (Levi-Civita connection).

The Sobolev Embedding \cite[Theorem 4.12]{AdamsFournier} may be usefully combined with the pointwise \emph{Kato Inequality} \cite[Equation (6.20)]{FU},
\begin{equation}
\label{eq:FU_6-20_first-order_Kato_inequality}
|d|v|| \leq |\nabla_A v| \quad\text{a.e. on } X,
\end{equation}
for a section $v \in W^{1, p}(X; E)$ of a Hermitian or Riemannian vector bundle, $E$, over $X$ and covariant derivative, $\nabla_A$, which is compatible with the fiber metric. The embeddings then yield the estimates,
\begin{equation}
\label{eq:Sobolev-Kato_inequalities}
\|v\|_{L^q(X,g)} \leq C\|v\|_{W_{A,g}^{1,p}(X)},
\quad
\text{for }
\begin{cases}
1 < p < d \text{ and } p \leq q \leq p^* := dp/(d-p),
\\
p = d \text{ and } p \leq q < \infty,
\\
p > d \text{ and } q = \infty,
\end{cases}
\end{equation}
for a constant, $C \in [1,\infty)$, that is \emph{independent} of the connection, $A$.

\subsection{Gauge transformations and quotient space of connections}
\label{subsec:Gauge_transformations_quotient_space_connections}
We let $\sA(P)$ denote the affine space of connections on $P$ of Sobolev class $W^{k,p}$ for $p\geq 2$ and integer $k \geq 1$ and $\Aut P$ denote the Banach Lie group of automorphisms (or gauge transformations) of Sobolev class $W^{k+1,p}$ of the principal $G$-bundle, $P$, for $p\geq 2$ and integer $k \geq 1$ obeying $(k+1)p > d$, so  $W^{k,p}(X) \hookrightarrow C(X)$ is a continuous Sobolev embedding by \eqref{eq:Sobolev_embedding_manifold_bounded_geometry}.

We let $\sB(P,g) := \sA(P)/\Aut P$ denote the quotient of the affine space of connections, $\sA(P)$, of class $W^{k,p}$ modulo the action of the group, $\Aut P$, of automorphisms of $P$ of class $W^{k+1,p}$, for $p\geq 2$ and integer $k \geq 1$ such that $(k+1)p > d$. We refer to \cite[Section 4.2]{DK} or \cite[Chapter 3]{FU} for constructions of the Banach manifold structures on $\Aut P$ and $\sB^*(P,g)$, where we recall that $\sB^*(P,g)\subset \sB(P,g)$ is the open subset gauge-equivalence classes of connections on $P$ whose isotropy group is the center of $G$ \cite[p. 132]{DK}.

\subsection{Curvature of connections}
\label{subsec:Curvature_connections}
Many of our calculations require local expressions for connections or curvature, so we shall review the relevant facts here. For further details regarding connections on principal bundles and curvature, see Bleecker \cite[Chapter 2]{Bleecker_1981}, Donaldson and Kronheimer \cite[Section 2.1]{DK}, Jost \cite[Section 4.1]{Jost_riemannian_geometry_geometric_analysis}, Kobayashi \cite[Section 1.1]{Kobayashi}, and Kobayashi and Nomizu \cite[Section 2.5]{Kobayashi_Nomizu_v1}.

Let $G$ be a Lie group with Lie algebra $\fg$. Following \cite[Chapter 2]{Bleecker_1981}, given a smooth manifold, $Z$, and $a, b \in \Omega^1(Z;\fg)$, one defines $[a,b] := \sum_{\alpha,\beta} a^\alpha\wedge b^\beta\, [E_\alpha, E_\beta]$, where $\{E_\alpha\}$ is a basis for $\fg$ and $a = a^\alpha E_\alpha$ and $b = b^\beta E_\beta$. Let $P$ be a smooth principal $G$-bundle over a smooth manifold, $X$. If $A \in \Omega^1(P;\fg)$ is a (global) connection one-form, then its curvature is \cite[Theorem 2.2.4]{Bleecker_1981}, \cite[Theorem 2.5.2]{Kobayashi_Nomizu_v1}
\[
F_A = dA + \frac{1}{2}[A,A] \in \Omega^2(P;\fg),
\]
where $d:\Omega^l(P,\fg) \to \Omega^{l+1}(P;\fg)$ is exterior derivative for $\fg$-valued forms on $P$. If $G$ is a matrix Lie group with matrix Lie algebra $\fg$, then  $[a,b] = a\wedge b + b\wedge a$, where $a\wedge b$ is defined by matrix multiplication with entries multiplied by wedge product, and then one may alternatively write \cite[Theorem1 2.2.12 and Corollary 2.2.13]{Bleecker_1981}
\[
F_A = dA + A\wedge A \in \Omega^2(P;\fg).
\]
Given a local section $\sigma:U\to P$, where $U \subset X$ is an open subset, and denoting $a = \sigma^*A \in \Omega^1(U;\fg)$, we have
\[
\sigma^*F_A = da + \frac{1}{2}[a,a]
\quad\text{or}\quad
 da + a\wedge a \in \Omega^2(U;\fg),
\]
where the first expression is valid for any Lie group and the second expression is valid when $G$ is a matrix Lie group. We may regard $F_A$ as an element of $\Omega^2(P;\fg)$ or $\Omega^2(X;\ad P)$, due its transformation property with respect to local gauge transformations \cite[Theorem 2.2.14]{Bleecker_1981}, \cite[Equation (1.1.14)]{Kobayashi}.

With respect to local coordinates on $X$, say $x(\cdot) = \varphi^{-1}$ on an open subset $U \subset X$, the local expressions for the curvature $\sigma^*F_A \in \Omega^2(U,\fg)$ in terms of the local connection one-form $\sigma^*A \in \Omega^1(U,\fg)$ are given by \cite[Equation (4.1.27)]{Jost_riemannian_geometry_geometric_analysis},
\begin{align*}
\sigma^*F_A &= \left(\frac{\partial a_\nu}{\partial x^\mu} + a_\mu a_\nu \right)dx^\mu \wedge dx^\nu
\\
&= \frac{1}{2} \left(\frac{\partial a_\nu}{\partial x^\mu} - \frac{\partial a_\mu}{\partial x^\nu}
+ [a_\mu, a_\nu] \right)dx^\mu \wedge dx^\nu
\\
&= \sum_{\mu < \nu}\left(\frac{\partial a_\nu}{\partial x^\mu} - \frac{\partial a_\mu}{\partial x^\nu}
+ [a_\mu, a_\nu] \right)dx^\mu \wedge dx^\nu,
\end{align*}
where the first expression is valid if $G$ is a matrix group. We write (compare \cite[Equation (2.1.17)]{DK},
\[
F_{\mu\nu} = \frac{\partial a_\nu}{\partial x^\mu} - \frac{\partial a_\mu}{\partial x^\nu}
+ [a_\mu, a_\nu], \quad 1 \leq \mu, \nu \leq d.
\]
Of course, the same expressions hold for connections, $A$, of Sobolev class $W^{k,p}$ for $p\geq 2$ and integer $k\geq 1$, albeit only almost everywhere unless $kp>d$, where $X$ has dimension $d$.

\subsection{Yang-Mills energy functional and Yang-Mills connections}
\label{subsec:Yang-Mills_energy_functional}
The Yang-Mills $L^2$-energy functional, $\sE_g:\sA(P,g) \to [0, \infty)$, given by \eqref{eq:Yang-Mills_energy_functional} is gauge-invariant and thus descends to a function on the quotient space, $\sE_g:\sB(P,g) \to [0, \infty)$.

A connection, $A$ on $P$, is a \emph{critical point} of $\sE_g$ --- and by definition a \emph{Yang-Mills connection} with respect to the metric $g$ --- if and only if it obeys the \emph{Yang-Mills equation} with respect to the metric $g$,
\begin{equation}
\label{eq:Yang-Mills_equation}
d_A^{*,g}F_A = 0 \quad\hbox{a.e. on } X,
\end{equation}
since $d_A^{*,g}F_A = \sE_g'(A)$ when the gradient of $\sE = \sE_g$ is defined by the $L^2$ metric\footnote{We omit the customary factor of $1/2$ in the definition \eqref{eq:Yang-Mills_energy_functional} of $\sE_g$ for convenience in this article and compensate by defining the $L^2$ metric to be $2(a,b)_{L^2(X)}$, for $a,b \in \Omega^1(X;\ad P)$.}
\cite[Section 6.2.1]{DK}, \cite{GroisserParkerSphere} and $d_A^* = d_A^{*,g}: \Omega^l(X; \ad P) \to \Omega^{l-1}(X; \ad P)$ is the $L^2$ adjoint of $d_A:\Omega^l(X; \ad P) \to \Omega^{l+1}(X; \ad P)$, for integers $l\geq l$.

\subsection{Classification of principal $G$-bundles}
\label{subsec:Taubes_1982_Appendix}
In this subsection and in Section \ref{subsec:Absolute_minima_Yang-Mills_energy_functional}, we summarize the main points of \cite[Section 9]{Feehan_yang_mills_gradient_flow}, which extends the discussion in Donaldson and Kronheimer \cite[Sections 2.1.3 and 2.1.4]{DK} to the case of compact Lie groups, and is based in turn on Atiyah, Hitchin, and Singer \cite{AHS}, Milnor and Stasheff \cite[Section 15 and Appendix C]{MilnorStasheff}, Sedlacek \cite[Appendix]{Sedlacek}, Steenrod \cite{Steenrod}, and Taubes \cite[Appendix]{TauSelfDual}.

We assume in this subsection that $X$ is a closed, four-dimensional, oriented, topological manifold and that $G$ is a compact \emph{semisimple} Lie group. We recall from \cite[Appendix]{Sedlacek}, \cite[Propositions A.1 and A.2]{TauSelfDual} that a topological principal $G$-bundle, $P$ over $X$, is classified up to isomorphism by a cohomology class $\eta(P) \in H^2(X;\pi_1(G))$ and its \emph{first Pontrjagin classes}, $p_1^j(P) \in H^4(X;\ZZ)$, or equivalently, \emph{first Pontrjagin degrees}, $\langle p_1^j(P), [X] \rangle \in \ZZ$, where $[X] \in H_4(X;\ZZ)$ denotes the fundamental class of $X$ and $1 \leq j \leq j_\fg$, where the integer $j_\fg \geq 1$ is the number of non-trivial simple ideals comprising the semi-simple Lie algebra $\fg$ of $G$. The topological invariant, $\eta \in H^2(X; \pi_1(G))$, is the \emph{obstruction} to the existence of a principal $G$-bundle, $P$ over $X$, with a specified Pontrjagin degrees.

\subsection{Absolute minima of the Yang-Mills $L^2$-energy functional}
\label{subsec:Absolute_minima_Yang-Mills_energy_functional}
In addition to the assumptions of Section \ref{subsec:Taubes_1982_Appendix}, we require that $P$ and $X$ be smooth and that $X$ be equipped with a Riemannian metric, $g$. For any $C^\infty$ connection, $A$, we recall that the curvature $F_A$ splits into its $g$-self-dual and $g$-anti-self-dual components \cite[Equation (2.1.25)]{DK},
\begin{equation}
\label{eq:Donaldson_Kronheimer_2-1-25}
F_A = F_A^{+,g} \oplus F_A^{-,g} \in \Omega^2(X;\ad P) = \Omega^{+,g}(X;\ad P) \oplus \Omega^{-,g}(X;\ad P),
\end{equation}
corresponding to the positive and negative eigenspaces, $\Lambda^{\pm,g}$, of the Hodge star operator $*_g:\Lambda^2 \to \Lambda^2$ defined by the Riemannian metric, $g$, so \cite[Equation (1.3)]{TauSelfDual}
\begin{equation}
\label{eq:Taubes_1982_1-3}
F_A^{\pm,g} = \frac{1}{2}(1 \pm *_g)F_A \in \Omega^{\pm,g}(X; \ad P).
\end{equation}
Of course, similar observations apply more generally to connections, $A$, of Sobolev class $W^{k,p}$ for $p\geq 2$ and integer $k\geq 1$.

If $A$ is $g$-self-dual, that is, $F_A^{-,g} = 0$ on $X$, then the Chern-Weil formula  \cite[Appendix C]{MilnorStasheff}, \cite[Appendix]{TauSelfDual} implies that the components of the Pontrjagin vector for $P$ are non-negative; if $A$ is $g$-anti-self-dual, that is, $F_A^{+,g} = 0$ on $X$, then the Chern-Weil formula implies that the components of the Pontrjagin vector of $P$ are non-positive.

By the Chern-Weil formula, a connection, $A$ on $P$, attains the absolute minimum value of the Yang-Mills $L^2$-energy functional $\sE_g$ in \eqref{eq:Yang-Mills_energy_functional} if and only if $A$ is $g$-self-dual when the components of the Pontrjagin vector of $P$ are non-negative or $A$ is $g$-anti-self-dual when the components of the Pontrjagin vector of $P$ are non-positive. If $A$ is $g$-self-dual or $g$-anti-self-dual, then the Chern-Weil formula gives an expression for the absolute minimum value of $\sE_g$ in terms of the components of the Pontrjagin vector for $P$. Consequently, the $g$-self-dual or $g$-anti-self-dual connections on $P$ comprise the absolute minima or `ground states' for $\sE_g$.

\subsection{Miscellaneous notation}
In the sequel, constants are generally denoted by $C$ (or $C(*)$ to indicate explicit dependencies) and may increase from one line to the next in a series of inequalities. We write $\eps \in (0,1]$ to emphasize a positive constant that is understood to be small or $K \in [1,\infty)$ to emphasize a constant that is understood to be positive but finite. An open ball in $\RR^d$ with center $x$ and radius $r$ may be denoted by $B_r(x)$ or $B(x,r)$, depending on the context. When there is no ambiguity, we suppress explicit mention of the underlying Riemannian metric, $g$, from differential operators or Sobolev spaces and write $\nabla_A^g = \nabla_A$, and $*=*_g$, and $d_A^* = d_A^{*,g}$, and $d_A^\pm = d_A^{\pm,g}$, and $\sE = \sE_g$, and $F_A^\pm = F_A^{\pm,g}$, and $W_g^{k,p}(X) = W^{k,p}(X)$, and similarly elsewhere. Manifolds, principal bundles, and Riemannian metrics are always smooth. For a Riemannian manifold, $(X,g)$, we denote its injectivity radius at a point $x \in X$ by $\Inj_x(X,g)$ and let $\Inj(X,g) := \inf_{x \in X}\Inj_x(X,g)$.

\section{Uhlenbeck convergence for a sequence of Yang-Mills connections with a uniform $L^2$ bound on curvature}
\label{sec:Uhlenbeck_compactness_Yang-Mills_and_anti-self-dual_connections}
In Section \ref{subsec:Uhlenbeck_compactness_Yang-Mills_connections_L2_bounds_curvature}, we develop an Uhlenbeck convergence result (Theorem \ref{thm:Sedlacek_4-3_Yang-Mills}) for a sequence of Yang-Mills connections with a uniform $L^2$ bound on their curvatures. Section \ref{subsec:Uhlenbeck_compactification_moduli_space_ASD_connections} contains a definition of the Uhlenbeck topology and provides Uhlenbeck compactness results for the moduli space of anti-self-dual connections, $M(P,g)$; Section \ref{subsec:Uhlenbeck_compactification_moduli_space_Yang-Mills_connections} contains the corresponding definitions for the moduli spaces of Yang-Mills connections, $\Crit(P,g,\sC)$, with $L^2$-energies in a compact range $\sC \subset [0,\infty)$.

\subsection{Uhlenbeck convergence for a sequence of Yang-Mills connections with a uniform $L^2$ bound on curvature}
\label{subsec:Uhlenbeck_compactness_Yang-Mills_connections_L2_bounds_curvature}
In \cite[Theorem 3.1 and Lemma 3.4]{Sedlacek}, Sedlacek extends Uhlenbeck's Weak Compactness Theorem \cite[Theorem 1.5 = Theorem 3.6]{UhlLp} from the case of a sequence of connections, with a uniform $L^p(X)$-bound on curvature, on a principal $G$-bundle over a closed, four-dimensional, Riemannian, smooth manifold, $X$, from the case of $p > 2$ to $p = 2$. Moreover, in \cite[Proposition 4.2 and Theorem 4.3]{Sedlacek}, Sedlacek applies his Weak Compactness Theorem to the case of an $L^2$-energy minimizing sequence of $W^{k,p}$ connections, $\{A_m\}_{m\in\NN}$ on $P$, where $k\geq 1$ and $p\geq 2$ obey $(k+1)p>4$. Our purpose in this section is to describe a simple refinement (Theorem \ref{thm:Sedlacek_4-3_Yang-Mills}) of Sedlacek's Weak Compactness Theorem for the case where $\{A_m\}_{m\in\NN}$ is a sequence of Yang-Mills connections, thus obtaining an analogue of \cite[Theorem 4.4.4]{DK}, which provides weak compactness for sequences of anti-self-dual connections on $P$. A result similar to Theorem \ref{thm:Sedlacek_4-3_Yang-Mills} was proved by Taubes as \cite[Proposition 5.1]{TauFrame}, \cite[Proposition 4.4]{TauPath} and the main ideas involved in their proofs are essentially equivalent. One may also view Theorem \ref{thm:Sedlacek_4-3_Yang-Mills} below as a sequential analogue of \cite[Theorems 29.3 and 29.4]{Feehan_yang_mills_gradient_flow}, an Uhlenbeck compactness result for Yang-Mills gradient flow. Analogues of Theorem \ref{thm:Sedlacek_4-3_Yang-Mills} for closed Riemannian manifolds, $X$, of dimension $d \geq 5$ are given by Nakajima \cite{Nakajima_1988}, Tian \cite{TianGTCalGeom}, and Zhang \cite{Zhang_2004cmb}.

\begin{thm}[Uhlenbeck convergence for a sequence of Yang-Mills connections with a uniform $L^2$ bound on curvature ]
\label{thm:Sedlacek_4-3_Yang-Mills}
Let $G$ be a compact Lie group and $P$ be a principal $G$-bundle over a closed, connected, four-dimensional, smooth manifold, $X$, and endowed with a Riemannian metric, $g$. If $\{A^m\}_{m\in\NN}$ is a sequence of $g$-Yang-Mills connections on $P$ of class $W^{k,p}$, with $k\geq 1$ and $p\geq 2$ obeying $(k+1)p>4$, with bounded $L^2$-energy,
\[
\sup_{m\in\NN} \|F_{A^m}\|_{L^2(X)} < \infty,
\]
then there exist
\begin{enumerate}
  \item A sequence of gauge transformations, $\{u_m\}_{m\in\NN} \subset \Aut(P)$, of class $W^{k+1,p}$;

  \item An integer, $l \geq 1$, and, if $l \geq 1$, a finite set of points, $\bx = \{x_1,\ldots,x_l\} \subset X$;

  \item A $g$-Yang-Mills connection, $\tilde A^\infty$ of class $W^{k,p}$, on a principal $G$-bundle, $\tilde P_\infty$, over $X\less\bx$;

  \item A gauge transformation, $u_\infty \in \Aut(\tilde P_\infty\restriction X\less\bx)$ of class $W^{k+1,p}$, such that $u_\infty(\tilde A^\infty)$
      extends to a $g$-Yang-Mills connection, $A^\infty$ of class $W^{k,p}$, on a principal $G$-bundle, $P_\infty$, over $X$ with $\eta(P_\infty) = \eta(P)$,
\end{enumerate}
such that the following hold:
\begin{enumerate}
  \item The sequence $\{u_m(A^m)\}_{m\in\NN}$ converges as $m\to\infty$ in the sense of $W_\loc^{k,p}(X\less\bx)$ to $\tilde A^\infty$;

  \item If $l\geq 1$, the points, $x_i \in \bx$, are characterized by positivity of the limits\footnote{For the sake of convenience, we omit from the expression for the energy constants, $\sE_i$, the factors of $1/2$ employed in the related expression \eqref{eq:Yang-Mills_energy_functional} for the Yang-Mills energy function, $\sE$.},
\[
\sE_i := \lim_{r\to 0}\limsup_{m\to\infty} \int_{B_r(x_i)} |F_{A^m}|^2\,d\vol_g > 0, \quad 1 \leq i \leq l,
\]
        and the sequence of measures, $|F_{A^m}|^2\,d\vol_g$, converges as $m\to\infty$ in the weak-star topology to the measure, $|F_{A^\infty}|^2\,d\vol_g + \sum_{i=1}^l \sE_i\,\delta_{x_i}$, that is,
\[
|F_{A^m}|^2\,d\vol_g \rightharpoonup |F_{A^\infty}|^2\,d\vol_g + \sum_{i=1}^l \sE_i\,\delta_{x_i}
\quad\text{in } (C(X;\RR))'  \quad\text{as } m \to \infty,
\]
where $\delta_x$ is the Dirac delta measure centered at $x \in X$.

  \item If in addition the sequence $\{A^m\}_{m\in\NN}$ is \emph{absolutely minimizing} in the sense that
\begin{equation}
\label{eq:Absolutely_minimizing_sequence_connections}
\|F_{A^m}^{+,g}\|_{L^2(X)} \to 0 \quad\text{as } m \to \infty,
\end{equation}
then $A^\infty$ is a $g$-anti-self-dual connection.
\end{enumerate}
\end{thm}

\begin{rmk}[Comparison with previous results of Taubes]
\label{rmk:Sedlacek_4-3_Yang-Mills}
Theorem \ref{thm:Sedlacek_4-3_Yang-Mills} has antecedents in \cite[Proposition 5.1]{TauFrame}, \cite[Proposition 4.4]{TauPath} where, rather than assume that $\{A^m\}_{m\in\NN}$ is a sequence of Yang-Mills connections (with bounded $L^2$-energy) on $P$, Taubes instead assumes that $\{A^m\}_{m\in\NN}$ is a sequence of arbitrary ($W^{k,p}$) connections on $P$ such that
\[
\sE_g(A^m) \to E \quad\text{and}\quad \|d_{A^m}^*F_{A^m}\|_{W_{A^m}^{-1,2}(X)} \to 0, \quad\text{as } m \to \infty,
\]
where $E$ is a finite constant and we recall that $\sE'(A) = d_A^{*,g}F_A$ is the gradient of the Yang-Mills energy functional at a connection, $A$, on $P$. (See the statement of \cite[Proposition 5.1]{TauFrame} and the discussion in \cite[pp. 340--341]{TauFrame}.) The stronger conclusions in Theorem \ref{thm:Sedlacek_4-3_Yang-Mills} are, naturally, a consequence of its stronger hypotheses relative to \cite[Proposition 5.1]{TauFrame}.
\end{rmk}

We shall just indicate the changes required to the proof of \cite[Theorem 4.4.4]{DK}, namely, the enhancements required to strengthen the notion of convergence in \cite[Proposition 4.2 and Theorem 4.3]{Sedlacek} to that described in the conclusions of Theorem \ref{thm:Sedlacek_4-3_Yang-Mills}. We first summarize the main technical ingredients that we shall require from Uhlenbeck \cite{UhlLp, UhlRem}, specialized from the setting of a base manifold of arbitrary dimension to one of dimension four.

\begin{thm}[Existence of a local Coulomb gauge and \apriori estimate for a Sobolev connection with $L^{d/2}$-small curvature]
\label{thm:Uhlenbeck_Lp_1-3}
\cite[Theorem 1.3 or Theorem 2.1 and Corollary 2.2]{UhlLp}
Let $d \geq 2$ be an integer, $G$ be a compact Lie group, and $p \in [d/2, d)$. Then there are constants, $C=C(d,G,p) \in [1,\infty)$ and $\eps=\eps(d, G, p) \in (0,1]$, with the following significance. Let $A$ be a connection of class $W^{1,p}$ on $B\times G$ such that
\begin{equation}
\label{eq:Ldover2_ball_curvature_small}
\|F_A\|_{L^{d/2}(B)} \leq \eps,
\end{equation}
where $B \subset \RR^d$ is the unit ball with center at the origin. Then there is a gauge transformation, $u:B\to G$, of class $W^{2,p}$ such that the following holds. If $A = \Theta + a$, where $\Theta$ is the product connection on $B\times G$, and $u(A) = \Theta + u^{-1}au + u^{-1}d_\Theta u$, then
\begin{align*}
d_\Theta^*(u(A) - \Theta) &= 0 \quad \text{a.e. on } B,
\\
(u(A) - \Theta)(\vec n) &= 0 \quad \text{on } \partial B,
\end{align*}
where $\vec n$ is the outward-pointing unit normal vector field on $\partial B$, and
\begin{equation}
\label{eq:Uhlenbeck_1-3_W1p_norm_connection_one-form_leq_constant_Lp_norm_curvature}
\|u(A) - \Theta\|_{W^{1,p}(B)} \leq C\|F_A\|_{L^p(B)}.
\end{equation}
\end{thm}

\begin{rmk}[Dependencies of the constants in Theorem \ref{thm:Uhlenbeck_Lp_1-3}]
\label{rmk:Uhlenbeck_Lp_1-3_constant_dependencies}
The statements of \cite[Theorem 1.3 or Theorem 2.1 and Corollary 2.2]{UhlLp} imply that the constants, $\eps$ and $C$, in estimate \eqref{eq:Uhlenbeck_1-3_W1p_norm_connection_one-form_leq_constant_Lp_norm_curvature} only depend the dimension, $d$, of the ball, $B$. However, their proofs suggest that these constants may also depend on $G$ and $p$ through the appeal to an elliptic estimate for $d_\Theta+d_\Theta^*$ in the verification of \cite[Lemma 2.4]{UhlLp} and arguments immediately following.
\end{rmk}

\begin{rmk}[Construction of a $W^{k+1,p}$ transformation to Coulomb gauge]
\label{rmk:Uhlenbeck_theorem_1-3_Wkp}
We note that if $A$ is of class $W^{k,p}$, for an integer $k \geq 1$ and $p \geq 2$, then the gauge transformation, $u$, in Theorem \ref{thm:Uhlenbeck_Lp_1-3} is of class $W^{k+1,p}$; see \cite[page 32]{UhlLp}, the proof of \cite[Lemma 2.7]{UhlLp} via the Implicit Function Theorem for smooth functions on Banach spaces, and our proof of \cite[Theorem 1.1]{FeehanSlice} --- a global version of Theorem \ref{thm:Uhlenbeck_Lp_1-3}.
\end{rmk}

An examination of the proof of Theorem \ref{thm:Uhlenbeck_Lp_1-3} in \cite{UhlLp} yields the

\begin{cor}[Existence of a local Coulomb gauge and \apriori estimate for a Sobolev connection with $L^{d/2}$-small curvature]
\label{cor:Uhlenbeck_theorem_1-3}
\cite[Corollary 4.4]{Feehan_yangmillsenergygapflat}
Assume the hypotheses of Theorem \ref{thm:Uhlenbeck_Lp_1-3}, but allow $p$ in the range $d/2 \leq p < \infty$. Then the conclusions of Theorem \ref{thm:Uhlenbeck_Lp_1-3} continue to hold.
\end{cor}

\begin{proof}
We need only consider the case $d \leq p < \infty$. Writing $u(A) = \Theta + \tilde a$, we have $F_{u(A)} = F_\Theta + d_\Theta\tilde a + \tilde a \wedge \tilde a$ and $F_\Theta = 0$ and so, because $d_\Theta^*\tilde a = 0$ by Theorem \ref{thm:Uhlenbeck_Lp_1-3},
\begin{align*}
\|(d_\Theta  + d_\Theta^*)\tilde a\|_{L^p(B)}
&=
\|d_\Theta\tilde a\|_{L^p(B)} = \|F_{u(A)} - \tilde a \wedge \tilde a\|_{L^p(B)}
\\
&\leq
\|F_{u(A)}\|_{L^p(B)} + \|\tilde a \wedge \tilde a\|_{L^p(B)}
\\
&\leq \|F_A\|_{L^p(B)} + 2\|\tilde a\|_{L^{2p}(B)}^2
\\
&\leq \|F_A\|_{L^p(B)} + C_0\|\tilde a\|_{W^{1,r}(B)}^2,
\end{align*}
for $C_0 = C_0(d,p)$, noting that we have a continuous Sobolev embedding, $W^{1,r}(B) \hookrightarrow L^{r^*}(B)$ for $r^* = dr/(d-r)$ and any $r\in [1,d)$, by \eqref{eq:Sobolev_embedding_domain}; we set $r^* = 2p$ to determine $r = 2pd/(d + 2p)$ and thus $r \in [2d/3, d)$ when $p \in [d,\infty)$. But Theorem \ref{thm:Uhlenbeck_Lp_1-3} implies that
\[
\|\tilde a\|_{W^{1,r}(B)} \leq C_1\|F_A\|_{L^r(B)},
\]
for $C_1 = C_1(d,G,p)$, since $\|F_A\|_{L^{d/2}(B)} \leq \eps$ by \eqref{eq:Ldover2_ball_curvature_small}. From the interpolation inequality \cite[Equation (7.9)]{GilbargTrudinger} we have
\[
\|F_A\|_{L^r(B)} \leq \|F_A\|_{L^{d/2}(B)}^\lambda \|F_A\|_{L^p(B)}^{1-\lambda},
\]
noting that $d/2 < r < d \leq p$ and $1/r = \lambda/(d/2) + (1-\lambda)/p$ for suitable $\lambda \in (0,1)$. Thus $\lambda$ obeys
\[
\frac{1}{r} = \frac{1}{2p} + \frac{1}{d} = \frac{2\lambda}{d} + \frac{1-\lambda}{p},
\]
and we see that $\lambda = 1/2$. By combining the preceding inequalities, we obtain
\begin{align*}
\|(d_\Theta  + d_\Theta^*)\tilde a\|_{L^p(B)}
&\leq
\|F_A\|_{L^p(B)} + C_0C_1^2\|F_A\|_{L^r(B)}^2
\\
&\leq \|F_A\|_{L^p(B)} + C_0C_1^2 \|F_A\|_{L^{d/2}(B)} \|F_A\|_{L^p(B)}
\\
&\leq (1 + C_0C_1^2\eps) \|F_A\|_{L^p(B)}.
\end{align*}
Combining the preceding inequality with the elliptic estimate, for $C_2 = C_2(d,p,G)$,
\[
\|\tilde a\|_{W_\Theta^{1,p}(B)}
\leq
C_2\|(d_\Theta  + d_\Theta^*)\tilde a\|_{L^p(B)},
\]
employed by Uhlenbeck in her proof of Theorem \ref{thm:Uhlenbeck_Lp_1-3} now completes the proof of Corollary \ref{cor:Uhlenbeck_theorem_1-3}.
\end{proof}

When $p > d/2$ in Corollary \ref{cor:Uhlenbeck_theorem_1-3}, we have a continuous Sobolev embedding $W^{2,p}(B) \hookrightarrow C(\bar B)$ by \eqref{eq:Sobolev_embedding_domain}, and hence a useful analogue for an annulus instead of a ball, albeit at the cost of a stronger hypothesis on the curvature. Let $\Omega(r,R) := B_R - \bar B_r \subset \RR^d$ denote the open annulus with center at the origin and radii $r<R$. The following result is a special case of \cite[Corollary 4.3]{UhlChern} due to Uhlenbeck.

\begin{thm}[Existence of a local Coulomb-gauge trivialization and \apriori estimate for a Sobolev connection with $L^p$-small curvature over an annulus]
\label{thm:Uhlenbeck_Chern_corollary_4-3_annulus}
Let $d \geq 2$ be an integer, $G$ be a compact Lie group, $p \in (d/2,\infty)$, and $\rho \in (0, 1)$. Then there are constants, $C=C(d,G,p,\rho) \in [1,\infty)$ and $\eps=\eps(d, G, p, \rho) \in (0,1]$, with the following significance. Let $A$ be a connection of class $W^{1,p}$ on $\Omega(\rho/2,2)\times G$ such that
\begin{equation}
\label{eq:Lp_annulus_curvature_small}
\|F_A\|_{L^p(\Omega(\rho/2,2))} \leq \eps.
\end{equation}
Then there is a gauge transformation, $u:\Omega(\rho/2,2)\to G$, of class $W_\loc^{2,p}(\Omega(\rho/2,2))$ such that the following holds. If $A = \Theta + a$, where $\Theta$ is the product connection on $\Omega(\rho/2,2)\times G$, and $u(A) = \Theta + u^{-1}au + u^{-1}d_\Theta u$, then
\[
d_\Theta^*(u(A) - \Theta) = 0 \quad \text{a.e. on } \Omega(\rho,1),
\]
and
\begin{subequations}
\label{eq:W1p_norm_connection_one-form_leq_constant_Lp_norm_curvature_annulus}
\begin{align}
\|u(A) - \Theta\|_{W^{1,p}(\Omega(\rho,1))} &\leq C\|F_A\|_{L^p(\Omega(\Omega(\rho/2,2)))},
\\
\|u(A) - \Theta\|_{W^{1,d}(\Omega(\rho,1))} &\leq C\|F_A\|_{L^{d/2}(\Omega(\Omega(\rho/2,2)))}.
\end{align}
\end{subequations}
\end{thm}

\begin{proof}
See the proof of \cite[Theorem 5.1]{Feehan_yangmillsenergygapflat} for a detailed justification.
\end{proof}

\begin{thm}[\Apriori interior estimate for the curvature of a Yang-Mills connection]
\label{thm:Uhlenbeck_3-5}
\cite[Theorem 3.5]{UhlRem}
If $d\geq 2$ is an integer, then there are constants, $K_0=K_0(d) \in [1,\infty)$ and $\eps_0=\eps_0(d) \in (0,1]$, with the following significance. Let $G$ be a compact Lie group, $\rho>0$ be a constant, and $A$ be a $C^\infty$ Yang-Mills connection with respect to the standard Euclidean metric on $B_{2\rho}(0)\times G$, where $B_r(x_0) \subset \RR^d$ is the open ball with center at $x_0 \in \RR^d$ and radius $r>0$. If
\begin{equation}
\label{eq:Uhlenbeck_theorem_3-5_FA_Ldover2_small_ball}
\|F_A\|_{L^{d/2}(B_{2\rho}(0))} \leq \eps_0,
\end{equation}
then, for all $B_r(x_0) \subset B_\rho(0)$,
\begin{equation}
\label{eq:Uhlenbeck_theorem_3-5_Linfty_norm_FA_leq_constant_L2_norm_FA_ball}
\|F_A\|_{L^{\infty}(B_r(x_0))} \leq K_0r^{-d/2}\|F_A\|_{L^2(B_r(x_0))}.
\end{equation}
\end{thm}

As Uhlenbeck notes in \cite[Section 3, first paragraph]{UhlRem}, Theorem \ref{thm:Uhlenbeck_3-5} continues to hold for geodesic balls in a manifold $X$ endowed a non-flat Riemannian metric, $g$. The only difference in this more general situation is that the constants $K$ and $\eps$ will depend on bounds on the Riemann curvature tensor, $R$, over $B_{2\rho}(x_0)$ and the injectivity radius at $x_0\in X$; see Remark \ref{rmk:Adjustment_to_pointwise_estimates_for_arbitrary_Riemannian_metrics} and Aubin \cite[Chapter 1]{Aubin}. We can use Theorem \ref{thm:Uhlenbeck_3-5} to obtain the following stronger version of Theorem \ref{thm:Uhlenbeck_Chern_corollary_4-3_annulus}.

\begin{cor}[Existence of a local Coulomb-gauge trivialization and \apriori estimate for a Yang-Mills connection with $L^2$-small curvature over an annulus]
\label{cor:Uhlenbeck_Chern_corollary_4-3_annulus_Yang-Mills}
Let $d \geq 2$ be an integer, $G$ be a compact Lie group, $p \in (d/2,\infty)$, and $\rho \in (0, 1)$. Then there are constants, $C=C(d,G,p,\rho) \in [1,\infty)$ and $\eps=\eps(d, G, p, \rho) \in (0,1]$, with the following significance. Let $A$ be a Yang-Mills connection of class $W^{1,p}$ on $\Omega(\rho/2,2)\times G$ such that
\begin{equation}
\label{eq:L2_annulus_curvature_small}
\|F_A\|_{L^2(\Omega(\rho/2,2))} \leq \eps.
\end{equation}
Then the conclusions of Theorem \ref{thm:Uhlenbeck_Chern_corollary_4-3_annulus} continue to hold.
\end{cor}

We shall need an interior \apriori estimate and regularity result for a Yang-Mills connection in Coulomb gauge with respect to the product connection, analogous to \cite[Theorem 2.3.8]{DK} for the case of an anti-self-dual connection. Although we apply Theorem \ref{thm:Uhlenbeck_3-5} to simplify and shorten an otherwise longer proof, it is possible (as in \cite{DK}) to avoid appealing to Theorem \ref{thm:Uhlenbeck_3-5} and instead rely more on Sobolev Embedding and Multiplication Theorems \cite[pp. 95--96]{FU}.

\begin{prop}[Interior \apriori estimate and regularity result for a Yang-Mills connection in Coulomb gauge]
\label{prop:Donaldson-Kronheimer_Theorem_2-3-8_Yang-Mills_d_geq_2}
Let $d\geq 2$ be an integer, $G$ be a compact Lie group, $k \geq 1$ be an integer, and $p \in [1, \infty)$ be a constant. Then there are constants, $C=C(d,G,k,p) \in [1,\infty)$ and $\eps=\eps(d,G) \in (0,\eps_0]$, for $\eps_0$ as in \eqref{eq:Uhlenbeck_theorem_3-5_FA_Ldover2_small_ball}, with the following significance. Let $\rho$ be a positive constant and $A$ be a Yang-Mills connection of class $W_\loc^{2,2}$ on $B_{2\rho}\times G \subset \RR^d\times G$ whose curvature, $F_A$, obeys \eqref{eq:Uhlenbeck_theorem_3-5_FA_Ldover2_small_ball} and is in Coulomb gauge over $B_\rho$ with respect to the product connection, $\Theta$ on $\RR^d\times G$, so $d_\Theta^*(A-\Theta) = 0$ a.e. on $B_\rho$, and $\RR^d$ has its standard Euclidean metric. Then $A$ is of class $W_\loc^{2,2}(B_{2\rho})\cap W_\loc^{k,p}(B_{2\rho})$ and
\begin{equation}
\label{eq:Donaldson-Kronheimer_Theorem_2-3-8_Yang-Mills_d_geq_2}
\|\nabla_\Theta^k(A - \Theta)\|_{L^p(B_{\rho/2})} \leq C\rho^{d(1/p - 1/2) - k + 1}\|F_A\|_{L^2(B_\rho)}.
\end{equation}
\end{prop}

\begin{proof}
We denote $a := A - \Theta$ for brevity. We first reduce to the case $\rho = 1$ and then apply rescaling, $y = x/\rho$, for $x \in B_\rho$ and $y \in B_1$, to obtain the responding result for a ball of arbitrary radius $\rho > 0$, writing $a(x) = \tilde a(y)$ and $F_A(x) = F_{\tilde A}(y)$. Thus,
\begin{align*}
\|\nabla_\Theta^k a\|_{L^p(B_{\rho/2})}
&=
\left( \int_{B_{\rho/2}} |\nabla_\Theta^k a|^p(x)\,dx \right)^{1/p}
\\
&= \left(\rho^{d - (k+1)p}\right)^{1/p} \left( \int_{B_{1/2}} |\nabla_\Theta^k \tilde a|^p(y)\,dy \right)^{1/p}
\\
&= \rho^{d/p - k - 1}\|\nabla_\Theta^k \tilde a\|_{L^p(B_{1/2})}
\\
&\leq C\rho^{d/p - k - 1} \|F_{\tilde A}\|_{L^2(B_1)} \quad\text{(by \eqref{eq:Donaldson-Kronheimer_Theorem_2-3-8_Yang-Mills_d_geq_2} with $\rho=1$)}
\\
&= C\rho^{d/p - k - 1} \left( \int_{B_1} |F_{\tilde A}|^2(y)\,dy \right)^{1/2}
\\
&= C\rho^{d/p - k - 1} \left(\rho^{-d + 4}\right)^{1/2} \left( \int_{B_\rho} |F_A|^2(x)\,dx \right)^{1/2}
\\
&= C\rho^{d(1/p - 1/2) - k + 1} \|F_A\|_{L^2(B_\rho)}.
\end{align*}
For the remainder of the proof, we may therefore assume that $\rho=1$.

Moreover, we may restrict our attention to the case $d/2 < p < d$ without loss of generality, since the estimate \eqref{eq:Donaldson-Kronheimer_Theorem_2-3-8_Yang-Mills_d_geq_2} with values of $p \geq d$ can be obtained with the Sobolev Embedding \cite[Theorem 4.12]{AdamsFournier} and the estimate \eqref{eq:Donaldson-Kronheimer_Theorem_2-3-8_Yang-Mills_d_geq_2} with values of $p \in [1, d/2]$ follows from the H\"older Inequality.

We first observe that $a \in W_\Theta^{2,2}(B_2;\Lambda^1\otimes\fg)$ is a strong solution to the quasi-linear elliptic system,
\[
d_{\Theta+a}^*F_{\Theta+a} = 0 \quad\text{and}\quad d_\Theta d_\Theta^* a = 0 \quad\text{a.e. on } B_2,
\]
where the first equation asserts that $A = \Theta+a$ is a Yang-Mills connection and the second equation follows from the fact that $A$ is in Coulomb gauge with respect to $\Theta$ over $B$. Expanding $F_{\Theta+a} = d_\Theta a + a\wedge a$, we see that $a \in W_\Theta^{2,2}(B_2;\Lambda^1\otimes\fg)$ is a strong solution to
\begin{equation}
\label{eq:Quasi-linear_Yang-Mills_Coulomb_gauge_system}
\Delta_\Theta a + a\times\nabla_\Theta a + a\times a\times a = 0 \quad\text{a.e. on } B_2,
\end{equation}
where $\Delta_\Theta = d_\Theta^*d_\Theta + d_\Theta d_\Theta^*$ and we use `$\times$' to denote a bilinear operation with universal constant coefficients. Theorem \ref{thm:Uhlenbeck_3-5} (with $r=1/2$ and $\rho = 1$ and $x_0 = 0 \in \RR^d$) yields
\[
\|F_A\|_{L^{\infty}(B)} \leq K_0\|F_A\|_{L^2(B)},
\]
while Theorem \ref{thm:Uhlenbeck_Lp_1-3} (for small enough $\eps \in (0, 1]$ in \eqref{eq:Ldover2_ball_curvature_small}) implies that, for $d/2 < p < d$,
\[
\|a\|_{W_\Theta^{1,p}(B)} \leq C\|F_A\|_{L^p(B)},
\]
for $C = C(d,G,p)$, noting that $W^{1,p}(B) \hookrightarrow L^q(B)$ by the Sobolev Embedding
\eqref{eq:Sobolev_embedding_domain} when $q=dp/(d-p)$, for $1\leq p < d$, and $1 \leq q < \infty$ when $p=d$, and $q=\infty$ when $p > d$. The interior \apriori estimate for $b \in W_\Theta^{k+2,p}(B;\Lambda^1\otimes\fg)$,
\begin{equation}
\label{eq:Gilbarg_Trudinger_9-36_system_k_geq_0}
\|b\|_{W_\Theta^{k+2,p}(B_{1/2})} \leq C\left(\|f\|_{W_\Theta^{k,p}(B)} + \|b\|_{L^p(B)} \right),
\end{equation}
for $C = C(d,G,k,p)$ and regularity theory \cite{Morrey} for a solution, $b \in W_\loc^{2,2}(B;\Lambda^1\otimes\fg)$, to a linear elliptic system,
\[
\Delta_\Theta b = f \in W_\Theta^{k,p}(B;\Lambda^1\otimes\fg),
\]
and the Sobolev Multiplication Theorems \cite[pp. 95--96]{FU} then yield the conclusions by adapting \mutatis the proof of \cite[Theorem 2.3.8]{DK}. For example, \cite[Theorems 13.29 and 13.33]{Feehan_yang_mills_gradient_flow} provides a global version of the estimate \eqref{eq:Gilbarg_Trudinger_9-36_system_k_geq_0} (with a homogeneous Dirichlet boundary condition) for a second-order elliptic system on a bounded open subset $\Omega \subset \RR^d$ when $k=0$ and that estimate may be converted to an \apriori interior estimate like \eqref{eq:Gilbarg_Trudinger_9-36_system_k_geq_0} using the Krylov localization technique \cite[Section 2.4, Lemma 4]{Krylov_LecturesSobolev} or \cite[Theorem 8.11.1]{Krylov_LecturesSobolev}. The \apriori estimate \eqref{eq:Gilbarg_Trudinger_9-36_system_k_geq_0} for $k \geq 1$ follows from the case with $k=0$ by induction (see, for example, the derivation of \cite[Theorem 8.10]{GilbargTrudinger} from \cite[Theorem 8.8]{GilbargTrudinger}).
\end{proof}

Lastly, we shall need the

\begin{thm}[Removability of singularities for finite-energy Yang-Mills connections]
\label{thm:Uhlenbeck_removable_singularity_Yang-Mills}
\cite[Theorem 4.1]{UhlRem}
Let $G$ be a compact Lie group, $k \geq 1$ be an integer and $p\in [2,\infty)$ be a constant such that $(k+1)p>4$, and $A$ be a Yang-Mills connection, of class $W^{k,p}$, on a principal $G$-bundle over $B\less\{0\}$, where $B\subset \RR^4$ is the unit ball centered at the origin, such that
\[
\|F_A\|_{L^2(B)} < \infty.
\]
Then there is a $W^{k+1,p}$ gauge transformation, $u \in \Aut(P\restriction B\less\{0\})$, such that the gauge-transformed connection, $u(A)$, extends to a $W^{k,p}$ Yang-Mills connection, $\bar A$, on a principal $G$-bundle $\bar P$ over $B$.
\end{thm}

\begin{rmk}[On the regularity of the connections in Uhlenbeck's Removable Singularities Theorem for Yang-Mills connections]
\label{rmk:Uhlenbeck_removable_singularity_Yang-Mills}
Uhlenbeck assumes that the given connection, $A$, in \cite[Theorem 4.1]{UhlRem} is $C^\infty$. This assumption implies no loss of generality and reduces notational clutter in the statement and proof of her theorem. Nevertheless, as we wish to keep track of the regularity of all connections and gauge transformations, we provide the minimal regularity assertions in Theorem \ref{thm:Uhlenbeck_removable_singularity_Yang-Mills}.
\end{rmk}

\begin{rmk}[Uhlenbeck's Removable Singularities Theorem for finite-energy Sobolev connections]
\label{rmk:Uhlenbeck_removable_singularity_Sobolev}
Uhlenbeck provides a generalization \cite[Theorems 2.1 and 4.5 and Corollary 4.6]{UhlChern} of Theorem \ref{thm:Uhlenbeck_removable_singularity_Yang-Mills} which removes the requirement that $A$ be Yang-Mills.
\end{rmk}

We can now quickly dispose of the

\begin{proof}[Proof of Theorem \ref{thm:Sedlacek_4-3_Yang-Mills}]
The argument follows \mutatis that of \cite[Theorem 4.4.4]{DK}, with the roles of
\begin{inparaenum}[\itshape a\upshape)]
\item \cite[Theorem 2.3.7]{DK} replaced by Theorem \ref{thm:Uhlenbeck_Lp_1-3};
\item \cite[Theorem 2.3.8]{DK} replaced by Proposition \ref{prop:Donaldson-Kronheimer_Theorem_2-3-8_Yang-Mills_d_geq_2}; and
\item \cite[Theorem 4.4.12]{DK} replaced by Theorem \ref{thm:Uhlenbeck_removable_singularity_Yang-Mills}. The inequalities, $p_1^i(\ad P)[X] \leq p_1^i(\ad P_\infty)[X] \leq 0$ for all $i$, are consequences of the Chern-Weil formula \cite[Section 9]{Feehan_yang_mills_gradient_flow}.
\end{inparaenum}
(The detailed arguments required in the proof of Theorem \ref{thm:Sedlacek_4-3_Yang-Mills} are very similar, indeed simpler, than those required in the proofs of \cite[Theorems 29.3 and 29.4]{Feehan_yang_mills_gradient_flow}.)
\end{proof}

\subsection{Uhlenbeck compactification for the moduli space of anti-self-dual connections}
\label{subsec:Uhlenbeck_compactification_moduli_space_ASD_connections}
In this section, we provide a few simple extensions of the concepts of ideal connection, Uhlenbeck convergence, and Uhlenbeck topology from the setting of anti-self-dual unitary connections on a complex Hermitian vector bundle over $X$ described in \cite[Section 4.4.1]{DK}.

\begin{defn}[Moduli space of $g$-anti-self-dual connections]
\label{defn:Moduli_space_anti-self-dual_connections}
Let $G$ be a compact Lie group and $P$ be a principal $G$-bundle over a closed, connected, four-dimensional, smooth manifold, $X$, and endowed with a Riemannian metric, $g$. Let $k \geq 1$ be a integer and $p \geq 2$ obey $(k+1)p> 4$ and $\sA(P,g)$ denote the affine space of $W^{k,p}$ connections on $P$. Then
\[
M(P,g) := \left\{A \in \sA(P,g): F_A^{+,g} = 0 \quad\text{a.e. on } X \right\}/\Aut(P),
\]
is the \emph{moduli space} of gauge-equivalence classes of $g$-anti-self-dual connections on $P$, where $\Aut(P)$ denotes the group of $W^{k+1,p}$ gauge transformations of $P$.
\end{defn}

We extend the description \cite[Definition 4.1.1 and Condition 4.1.2]{DK} of Uhlenbeck convergence for a sequence of anti-self-dual unitary connections on a complex Hermitian vector bundle.

\begin{defn}[$W_\loc^{k,p}$ Uhlenbeck convergence of a sequence of Sobolev connections]
\label{defn:Donaldson_Kronheimer_4-4-2_Sobolev_connections}
Let $G$ be a compact Lie group and $P$ be a principal $G$-bundle over a closed, connected, four-dimensional, smooth manifold, $X$, and endowed with a Riemannian metric, $g$. Let $\{A^m\}_{m\in\NN}$ be a sequence of connections on $P$ of class $W^{\bar k,\bar p}$, with $\bar p\geq 2$ and integer $\bar k\geq 1$ obeying $(\bar k+1)\bar p>4$ and
\[
E := \limsup_{m\to\infty} \sE_g(A^m) < \infty.
\]
Let $A^\infty$ be a $W^{\bar k,\bar p}$ connection on a principal $G$-bundle $P_\infty$ over $X$ with $\eta(P_\infty)=\eta(P)$. Let $\ell \geq 0$ be an integer and, if $\ell \geq 1$, let $(\bx,\bE) \in \Sym^\ell(X\times [0,E])$. For $2 \leq p \leq \bar p$ and integer $0\leq k\leq \bar k$, we say that $\{A^m\}_{m\in\NN}$ \emph{converges in the $W_\loc^{k,p}$ Uhlenbeck topology} to an \emph{ideal $W^{\bar k,\bar p}$ connection}, $(A^\infty,\bx,\bE)$, on $P$ if there exists a sequence, $\{u_m\}_{m\in\NN}$, of principal $G$-bundle isomorphisms of class $W^{\bar k+1,\bar p}$,
\[
u_m: P\restriction X\less\bx \cong P_\infty\restriction X\less\bx,
\]
such that the following hold:
\begin{enumerate}
  \item The sequence $\{u_m(A^m)\}_{m\in\NN}$ converges as $m\to\infty$ in the sense of $W_\loc^{k,p}(X\less\bx)$ to $A^\infty$;

  \item The sequence of measures, $|F_{A^m}|^2\,d\vol_g$, converges as $m\to\infty$ in the weak-star topology to the measure, $|F_{A^\infty}|^2\,d\vol_g + \sum_{x\in\bx} \sE_x\,\delta_x$, that is,
\[
|F_{A^m}|^2\,d\vol_g \rightharpoonup |F_{A^\infty}|^2\,d\vol_g + \sum_{x\in\bx} \sE_x\,\delta_x
\quad\text{in } (C(X;\RR))'  \quad\text{as } m \to \infty,
\]
where $\delta_x$ is the Dirac delta measure centered at a point $x \in X$ and $(x,\sE_x) \in (\bx,\bE)$.
\end{enumerate}
\end{defn}

\begin{rmk}[Uhlenbeck convergence of a sequence of pairs of connections and Riemannian metrics]
\label{rmk:Donaldson_Kronheimer_4-4-2_Sobolev_connections_and_metrics}
The Definition \ref{defn:Donaldson_Kronheimer_4-4-2_Sobolev_connections} of Uhlenbeck convergence of a sequence of connections, $\{A^m\}_{m\in\NN}$ on $P$, extends \mutatis to Uhlenbeck convergence of a sequence of pairs of connections and Riemannian metrics, $\{A^m, g^m\}_{m\in\NN}$, where $g^m \to g$ in $C_\loc^n$ on $X\less\bx$ for an integer $n \geq 0$.
\end{rmk}

It is often convenient to represent $(\bx,\bE) \in \Sym^\ell(X\times [0,E])$ using distinct unordered points in $X$ rather than repeating them according to their multiplicity and, depending on the context, we indicate the choice as $\{x_1,\sE_1,\ldots,x_\ell,\sE_\ell\}$ (points $x_i\in X$ counted with multiplicity) or $\{x_1',\sE_1',\ldots,x_l',\sE_l'\}$ (distinct unordered points $x_i' \in X$) with $l\leq \ell$ and $\sE_i' = \sum \sE_j$, where the sum is over all $j$ such that $x_j = x_i'$. When there can be no confusion, we omit the primes.

As in \cite[p. 158]{DK}, the definition of $W_\loc^{k,p}$ Uhlenbeck convergence for a sequence of Sobolev connections in Definition \ref{defn:Donaldson_Kronheimer_4-4-2_Sobolev_connections} extends in an obvious way to a sequence of ideal Sobolev connections.

\begin{defn}[$W_\loc^{k,p}$ Uhlenbeck open neighborhood of an ideal Sobolev connection]
\label{defn:Open_Uhlenbeck_neighborhood}
Continue the notation of Definition \ref{defn:Donaldson_Kronheimer_4-4-2_Sobolev_connections}, let $(A_0,\bx,\bE)$ be an ideal $W^{\bar k,\bar p}$ connection on $P$, and $\rho$ be a positive constant obeying
\[
8\rho < 1 \wedge \Inj(X,g) \wedge \min_{i,j}\dist_g(x_i, x_j),
\]
with $(\bx,\bE)$ represented by $\{x_1,\sE_1,\ldots,x_l,\sE_l\}$ (distinct unordered points $x_i\in X$), and $\eps \in (0,1]$. Let $\sA(P)$ denote the affine space of $W^{\bar k,\bar p}$ connections on $P$. We say that a $W^{\bar k,\bar p}$ connection $A$ on $P$ belongs to a \emph{$W_\loc^{k,p}$ Uhlenbeck $(\eps,\rho)$ open neighborhood $\tilde\sV\subset \sA(P)$} of $(A_0,\bx,\bE)$ if the following hold:
\begin{enumerate}
  \item There is a $W^{\bar k+1,\bar p}$ isomorphism of principal $G$-bundles,
\[
u: P\restriction X\less\bx \cong P_0\restriction X\less\bx,
\]
such that, for some $\eps_\background \in (0,\eps]$,
\begin{equation}
\label{eq:Uhlenbeck_neighborhood_A_Wkp_near_A_0_complement_small_balls}
\|u(A) - A_0\|_{W_{A_0}^{k,p}(X\less B_{\rho/2}(\bx))} < \eps_\background,
\end{equation}
where $B_{\rho/2}(\bx) := \cup_{i=1}^l B_{\rho/2}(x_i)$; and
  \item One has, for some $\eps_\ball \in (0,\eps]$,
\begin{equation}
\label{eq:Uhlenbeck_neighborhood_curvature_density_A_small_balls_near_weight}
\max_{1\leq i\leq l}\left|\int_{B_{2\rho}(x_i)} |F_A|^2\,d\vol_g - \sE_i \right| < \eps_\ball.
\end{equation}
\end{enumerate}
We say that $[A] \in \sB(P,g)$ belongs to a $W_\loc^{k,p}$ \emph{Uhlenbeck $(\eps,\rho)$ open neighborhood $\sV \subset \sB(P)$} if $A$ belongs to a $W_\loc^{k,p}$ bubble-tree $(\eps,\rho)$ open neighborhood $\pi^{-1}(\sV) \subset \sA(P)$. Here, $\sB(P,g) = \sA(P)/\Aut(P)$, where $\Aut(P)$ is the group of $W^{\bar k+1,\bar p}$ gauge transformations of $P$, and $\pi:\sA(P) \to \sB(P,g)$ is the projection map.

If we need to emphasize the individual values of $\eps_\background$ or $\eps_\ball$, then we write $\beps = (\eps_\background, \eps_\ball)$ and refer to a $W_\loc^{k,p}$ Uhlenbeck $(\beps,\rho)$ open neighborhood.
\end{defn}

Observe that it is not necessary to approximate the implied characteristic functions of the balls
with center $x_i$ and radii $2\rho$ by continuous functions on $X$ in \eqref{eq:Uhlenbeck_neighborhood_curvature_density_A_small_balls_near_weight} because the condition \eqref{eq:Uhlenbeck_neighborhood_A_Wkp_near_A_0_complement_small_balls} eliminates the possibility of the curvature density, $|F_A|^2$, concentrating on $X - \cup_{i=1}^l B_\rho(x_i)$.

As with Definition \ref{defn:Donaldson_Kronheimer_4-4-2_Sobolev_connections}, the definition of a $W_\loc^{k,p}$ Uhlenbeck open neighborhood of a Sobolev connection in Definition \ref{defn:Open_Uhlenbeck_neighborhood} extends in an obvious way to a $W_\loc^{k,p}$ Uhlenbeck open neighborhood of an ideal Sobolev connection. We define
\begin{equation}
\label{eq:Minimum_energy_connection_G-bundle_P}
\sE_g(P)
:=
\inf\left\{\sE_g(A): A \text{ is a $W^{1,2}$ connection on } P\to X \right\}
\end{equation}
to be the minimum value\footnote{This topological lower bound is attained by a $g$-anti-self-dual or $g$-self-dual connection on $P$, irrespective of whether such a connection actually exists for $(X,P,g)$.} of the $L^2$-energy $\sE_g(A)$ of a $W^{1,2}$ connection $A$ on $P$ predicted by the Chern-Weil formula \cite[Section 9]{Feehan_yang_mills_gradient_flow}.

By analogy with \cite[p. 158]{DK} (where it is assumed that $G = \U(n)$ for $n \geq 2$), we define the \emph{moduli space of Uhlenbeck ideal $g$-anti-self-dual connections on $P$},
\begin{equation}
\label{eq:Moduli_space_ideal_ASD_connections}
UM(P,g) := \bigcup\left(M(P_\ell,g) \times \Sym^\ell\left(X\times [0,\sE_g(P)]\right)\right),
\end{equation}
where the union on the right-hand side is all over integers $\ell\geq 0$ and isomorphism classes of principal $G$-bundles, $P_\ell$ over $X$, such that the following hold,
\begin{equation}
\label{eq:Sum_rule}
\eta(P_\ell) = \eta(P) \quad\text{and}\quad \sE_g(P) = \sE_g(P_\ell) + \sum_{i=1}^\ell \sE_i,
\end{equation}
where $\sE_i = \sE_g(P_i)$ for some principal $G$-bundle, $P_i$ over $S^4$.

We observe that the collection of Uhlenbeck open neighborhoods given by Definition \ref{defn:Open_Uhlenbeck_neighborhood} form a basis for a topology on $UM(P,g)$, called the \emph{Uhlenbeck topology}. One may check that $UM(P,g)$ is then a Hausdorff, regular, second-countable topological space \cite[Sections 17, 30, 31, and Theorem 17.8]{Munkres2} and thus metrizable by the Urysohn Metrization Theorem \cite[Theorem 34.1]{Munkres2}.

By analogy with \cite[p. 158]{DK} (where it is assumed that $G = \U(n)$ for $n \geq 2$), we define the \emph{Uhlenbeck closure}, $\bar M(P,g)$, to be the closure of $M(P,g)$ in $UM(P,g)$.

Elliptic regularity for solutions to the anti-self-dual and local Coulomb-gauge equations ensure that the definition of Uhlenbeck topology on $\bar M(P,g)$ is independent of the choice of $(\bar k,\bar p)$ obeying $\bar k\geq 1$ and $\bar p\geq 2$ and $(\bar k+1)\bar p > 4$ or $(k,p)$ obeying $0\leq k\leq \bar k$ and $2 \leq p \leq \bar p$. Theorem \ref{thm:Sedlacek_4-3_Yang-Mills} yields

\begin{cor}[Sequential Uhlenbeck compactness for the moduli space of anti-self-dual connections on a principal $G$-bundle]
\label{cor:Donaldson_Kronheimer_4-4-4_G}
(Compare \cite[Theorem 4.4.4]{DK})
Let $G$ be a compact Lie group and $P$ be a principal $G$-bundle over a closed, connected, four-dimensional, oriented, smooth manifold, $X$, and endowed with a Riemannian metric, $g$. Then $\bar M(P,g)$ is sequentially compact.
\end{cor}

Because the Uhlenbeck topology on $\bar M(P, g)$ is metrizable, sequential compactness and compactness coincide for this topology by \cite[Theorem 28.2]{Munkres2}, and thus we have the

\begin{cor}[Uhlenbeck compactness for the moduli space of anti-self-dual connections on a principal $G$-bundle]
\label{cor:Donaldson_Kronheimer_4-4-3_G}
(Compare \cite[Theorem 4.4.3]{DK})
Assume the hypotheses of Corollary \ref{cor:Donaldson_Kronheimer_4-4-4_G}. Then $\bar M(P,g)$ is compact.
\end{cor}

\begin{rmk}[Uhlenbeck compactification and $L^2$-metric completion of the moduli space of anti-self-dual connections]
\label{rmk:Uhlenbeck_compactification_and_L2-metric_completion_moduli_space_anti-self-dual_connections}
We recall \cite[Theorem 1.2]{FeehanGeometry} that the Uhlenbeck compactification and $L^2$-metric completion of the moduli space $M(P,g)$ of anti-self-dual connections coincide, conjecturally always but at least under mild hypotheses, including when the Riemannian metric $g$ is generic and
\begin{inparaenum}[\itshape a\upshape)]
\item $G = \SU(2)$ or $\SO(3)$ and $b^+(X)=0$, or
\item $G = \SO(3)$ and the second Stiefel-Whitney class is non-zero, that is, $w_2(P) \neq 0$.
\end{inparaenum}
\end{rmk}

\subsection{Uhlenbeck compactification for the moduli space of Yang-Mills connections with a uniform $L^2$ bound on curvature}
\label{subsec:Uhlenbeck_compactification_moduli_space_Yang-Mills_connections}
In this section, we extend the definitions and results of Section \ref{subsec:Uhlenbeck_compactification_moduli_space_ASD_connections} to the more general setting of the moduli space of Yang-Mills connections with $L^2$ energies in a compact range, $\sC \subset [0,\infty)$. We begin with a generalization of Definition \ref{defn:Moduli_space_anti-self-dual_connections} (compare \cite[Equation (2.1)]{TauFrame}).

\begin{defn}[Moduli space of $g$-Yang-Mills connections with $L^2$-energy $c$]
\label{defn:Moduli_space_Yang-Mills_connections_L2-energy_c}
Let $G$ be a compact Lie group and $P$ be a principal $G$-bundle over a closed, connected, smooth manifold, $X$, of dimension $d \geq 2$ and endowed with a Riemannian metric, $g$. Let $k \geq 1$ be a integer and $p \geq 2$ obey $(k+1)p> d$, and $\sA(P,g)$ denote the affine space of $W^{k,p}$ connections on $P$, and $c \in [0,\infty)$. Then
\[
\Crit(P,g,c) := \left\{A \in \sA(P,g): \sE_g'(A) = 0 \text{ a.e. on } X \text{ and } \sE_g(A) = c\right\}/\Aut(P),
\]
is the \emph{moduli space} of gauge-equivalence classes of $g$-Yang-Mills connections on $P$ with $L^2$-energy equal to $c$, where $\Aut(P)$ denotes the group of $W^{k+1,p}$ gauge transformations of $P$, and $\sE_g:\sA(P,g) \to [0, \infty)$ is the $L^2$-energy functional \eqref{eq:Yang-Mills_energy_functional}.
\end{defn}

Recall from \eqref{eq:Yang-Mills_equation} that $\sE_g'(A) = d_A^{*,g}F_A$. If $d = 4$ and $c_0$ is the minimum value of $\sE_g$ predicted by topology via the Chern-Weil formula \cite[Section 9]{Feehan_yang_mills_gradient_flow}, \cite[Appendix C]{MilnorStasheff}, \cite[Appendix]{TauSelfDual}, then
\[
\Crit(P,g,c_0) = M(P,g),
\]
that is, $\Crit(P,g,c_0)$ is the moduli space of $g$-(anti-)self-dual connections on $P$.

By analogy with \eqref{eq:Moduli_space_ideal_ASD_connections}, we define the \emph{moduli space of Uhlenbeck ideal $g$-Yang-Mills connections on $P$ with $L^2$-energy $c$},
\begin{equation}
\label{eq:Moduli_space_ideal_Yang-Mills_connections_L2-energy_c}
\UCrit(P,g,c) := \bigcup\left(\Crit(P_\ell,g,c_\ell) \times \Sym^\ell\left(X\times [0,c]\right)\right),
\end{equation}
where the union on the right-hand side is over all
\begin{inparaenum}[\itshape i\upshape)]
\item integers $\ell\geq 0$,
\item isomorphism classes of principal $G$-bundles, $P_\ell$ over $X$, obeying the topological sum rule \eqref{eq:Sum_rule}, and
\item $L^2$ energies $c_\ell \in [0, c]$.
\end{inparaenum}

The collection of Uhlenbeck open neighborhoods given by Definition \ref{defn:Open_Uhlenbeck_neighborhood} form a basis for the \emph{Uhlenbeck topology} on $\UCrit(P,g,c)$. We define the \emph{Uhlenbeck closure}, $\overline{\Crit}(P,g,c)$, to be the closure of $\Crit(P,g,c)$ in $\UCrit(P,g,c)$. Theorem \ref{thm:Sedlacek_4-3_Yang-Mills} yields the

\begin{cor}[Sequential Uhlenbeck compactness for the moduli space of Yang-Mills connections on a principal $G$-bundle with $L^2$-energy $c$]
\label{cor:Donaldson_Kronheimer_4-4-4_Yang-Mills}
Let $G$ be a compact Lie group and $P$ be a principal $G$-bundle over a closed, connected, four-dimensional, smooth manifold, $X$, and endowed with a Riemannian metric, $g$, and $c \in [0,\infty)$. Then $\overline{\Crit}(P,g,c)$ is sequentially compact.
\end{cor}

Because the Uhlenbeck topology on $\overline{\Crit}(P,g,c)$ is metrizable, sequential compactness and compactness coincide for this topology, and thus we have the

\begin{cor}[Uhlenbeck compactness for the moduli space of Yang-Mills connections on a principal $G$-bundle with $L^2$-energy $c$]
\label{cor:Donaldson_Kronheimer_4-4-3_Yang-Mills}
Assume the hypotheses of Corollary \ref{cor:Donaldson_Kronheimer_4-4-4_Yang-Mills}. Then $\overline{\Crit}(P,g,c)$ is compact.
\end{cor}

Lastly, we note that in the context of Yang-Mills connections, it is natural to consider an extension of Definition \ref{defn:Moduli_space_Yang-Mills_connections_L2-energy_c} to the case of a compact range of $L^2$ energy values from that of one single energy value.

\begin{defn}[Moduli space of $g$-Yang-Mills connections with $L^2$ energy in a compact range]
\label{defn:Moduli_space_Yang-Mills_connections_L2-energy_compact_range}
Assume the set-up of Definition \ref{defn:Moduli_space_Yang-Mills_connections_L2-energy_c} and let $\sC \subset [0,\infty)$ be a compact subset. Then
\[
\Crit(P,g,\sC) := \left\{A \in \sA(P,g): \sE_g'(A) = 0 \text{ a.e. on } X \text{ and } \sE_g(A) \in \sC\right\}/\Aut(P),
\]
is the \emph{moduli space} of gauge-equivalence classes of $g$-Yang-Mills connections on $P$ with $L^2$ energy in $\sC$.
\end{defn}

The definition \eqref{eq:Moduli_space_ideal_Yang-Mills_connections_L2-energy_c} of $\UCrit(P,g,c)$ extends almost without change to
\begin{equation}
\label{eq:Moduli_space_ideal_Yang-Mills_connections_L2-energy_compact_range}
\UCrit(P,g,\sC) := \bigcup\left(\Crit(P_\ell,g,c_\ell) \times \Sym^\ell\left(X\times [0,\bar c]\right)\right),
\end{equation}
where $\bar c \in [0,\infty)$ is any constant such that $\sC \subset [0,\bar c)$ and the union on the right-hand side is again over all
\begin{inparaenum}[\itshape i\upshape)]
\item integers $\ell\geq 0$,
\item isomorphism classes of principal $G$-bundles, $P_\ell$ over $X$, obeying the topological sum rule \eqref{eq:Sum_rule}, and
\item $L^2$ energies $c_\ell \in [0, \bar c]$.
\end{inparaenum}

We again define the \emph{Uhlenbeck closure}, $\overline{\Crit}(P,g,\sC)$, to be the closure of $\Crit(P,g,\sC)$ in $\UCrit(P,g,\sC)$ and observe that the proof of Corollary \ref{cor:Donaldson_Kronheimer_4-4-3_Yang-Mills} extends without change to give the

\begin{cor}[Uhlenbeck compactness for the moduli space of Yang-Mills connections on a principal $G$-bundle with $L^2$-energy in a compact range]
\label{cor:Donaldson_Kronheimer_4-4-3_Yang-Mills_L2-energy_compact_range}
Let $G$ be a compact Lie group and $P$ be a principal $G$-bundle over a closed, connected, four-dimensional, smooth manifold, $X$, and endowed with a Riemannian metric, $g$, and $\sC \subset [0,\infty)$ be a compact subset. Then $\overline{\Crit}(P,g,\sC)$ is compact.
\end{cor}

\section{Mass center and scale maps on the quotient space of connections over the sphere}
\label{sec:Mass_center_and_scale_maps_on_space_connections_over_sphere}
In order to describe the rescaling mechanism employed in the definition of bubble-tree convergence provided in Section \ref{sec:Bubble-tree_compactness_Yang-Mills_and_anti-self-dual_connections}, we shall first discuss the concept of a centered connection over $S^4$ and the action of the subgroup of translations and dilations of the group of conformal transformations of $S^4$, namely $\RR^4\times\RR_+ \subset \Conf(S^4)$, on the quotient space of connections over $S^4$.

\subsection{Mass center and scale maps}
\label{subsec:Mass_center_and_scale_maps}
A choice of frame, $f$, in the principal $\SO(4)$-frame bundle, $\Fr(TS^4)$, for $TS^4$, over the north pole $n\in S^4 \cong \RR^4\cup\{\infty\}$ (identified with the origin in $\RR^4$), defines a conformal diffeomorphism,
\begin{equation}
\label{eq:ConformalDiffeo}
\varphi_n:\RR^4 \to S^4\less\{s\},
\end{equation}
that is inverse to a stereographic projection from the south pole $s\in S^4\subset \RR^5$ (identified with the point at infinity in $\RR^4\cup\{\infty\}$). We let $y(\,\cdot\,): S^4\less\{s\}\to\RR^4$ be the corresponding coordinate chart.

\begin{defn}[Center and scale of a connection on a principal $G$-bundle over $S^4$]
\label{defn:Mass_center_scale_connection}
(Compare \cite[Equation (3.10)]{FeehanGeometry} and Taubes \cite[Equation (4.15)]{TauPath}, \cite[pp. 343--344]{TauFrame}.)
Let $G$ be a compact Lie group and $P$ be a principal $G$-bundle over the four-dimensional sphere, $S^4$, with its standard round Riemannian metric of radius one. The \emph{center}, $z=z[A]\in\RR^4$, and the \emph{scale}, $\lambda=\lambda[A] \in \RR_+ = (0,\infty)$ of a non-flat $W^{2,2}$ connection, $A$, on $P$ are defined by
\begin{subequations}
\label{eq:Mass_center_and_scale_connection}
\begin{align}
\label{eq:Mass_center_connection}
\Center[A] &:= \left(\int_{\RR^4}|\varphi_n^*F_A(y)|_\delta^2\,d^4y\right)^{-1}
\int_{\RR^4}y|\varphi_n^*F_A(y)|_\delta^2\,d^4y,
\\
\Scale[A]^2 &:= \left(\int_{\RR^4}|\varphi_n^*F_A(y)|_\delta^2\,d^4y\right)^{-1}
\int_{\RR^4} |y - z[A]|^2|\varphi_n^*F_A(y)|_\delta^2\,d^4y,
\label{eq:Scale_connection}
\end{align}
\end{subequations}
where $\delta$ denotes the standard Euclidean metric on $\RR^4$. The connection, $A$, is \emph{centered} if $\Center[A] = 0$ and $\Scale[A] = 1$. If $A$ is flat, one defines $\Center[A] := 0$ and $\Scale[A] := 0$.
\end{defn}

\begin{rmk}[Normalization constants in the definition of mass center and scale]
\label{rmk:Normalization_constants}
The choice of normalization constant in Definition \ref{defn:Mass_center_scale_connection} is consistent with \cite[Equations (29.44) and (29.45)]{Feehan_yang_mills_gradient_flow} and Definition \ref{defn:Center_scale_positive_Borel_measure} and Taubes \cite[Equation (4.15)]{TauPath}, but differs in general from those of \cite[Equation (3.10)]{FeehanGeometry} or Taubes \cite[pp. 343--344]{TauFrame}.
\end{rmk}

\begin{rmk}[Round versus Euclidean metrics in the definition of mass center and scale]
\label{rmk:Round_vs_Euclidean_metrics_in_definition_mass_center_and_scale}
It is possible, as in Taubes \cite[Equation (4.15)]{TauPath}, to use the pullback to $\RR^4$ of the standard round metric of radius one on $S^4$ when defining the integrals in \eqref{eq:Mass_center_and_scale_connection}. However, is then more difficult to show that it is possible to center a non-centered connection $A$ on $P$ and the relationship between the connection $A$ and the required conformal diffeomorphism of $\RR^4$ is not explicit as it is in Lemma \ref{lem:Centering_connection_over_4sphere}, which may be compared with Taubes \cite[Lemma 4.11]{TauPath}.
\end{rmk}

For any $(z,\lambda)\in\RR^4\times\RR_+$, we define a conformal diffeomorphism of $\RR^4$ by
\begin{equation}
\label{eq:CenterScaleDiffeo}
h_{z,\lambda}:\RR^4\to\RR^4,\qquad y\mapsto (y-z)/\lambda.
\end{equation}
It is convenient to view this as a composition of translation, $\tau_z(y) = y - z$, and dilation, $\delta_\lambda(y) := y/\lambda$, that is, $h_{z,\lambda} = \delta_\lambda(\tau_z(y))$. The group $\RR^4\times\RR_+$ acts on $\sA(\varphi_n^*P)$ and $\sA(P)$ by pullback and composition with $\varphi_n:\RR^4\to S^4\less\{s\}$, that is,
\begin{align*}
\sA(\varphi_n^*P) \times \RR^4\times\RR_+ \ni (\varphi_n^*A, z,\lambda)
&\mapsto
h_{z,\lambda}^*\varphi_n^*A \in \sA(\varphi_n^*P),
\\
\sA(P) \times \RR^4\times\RR_+ \ni (A,z,\lambda)
&\mapsto
\tilde h_{z,\lambda}^*A \in \sA(P),
\end{align*}
where
\[
\tilde h_{z,\lambda}^*A := \varphi_n^{-1,*} h_{z,\lambda}^* \varphi_n^*A.
\]
The group $\RR^4\times\RR_+$ also acts on $\Aut(\varphi_n^*P)$ and $\Aut(P)$ by pullback and descends to an action on $\sB(\varphi_n^*P)$ and $\sB(P)$. We have the following simpler analogue of Taubes \cite[Lemma 4.11]{TauPath}.

\begin{lem}[Centering a connection over $S^4$]
\label{lem:Centering_connection_over_4sphere}
Let $A$ be a non-flat $W^{2,2}$ connection on a principal $G$-bundle over $S^4$ with its standard round metric of radius one. If $(z,\lambda) \in \RR^4\times\RR_+$, then
\[
\Center[\tilde h_{z,\lambda}^*A] = \lambda\Center[A] + z
\quad\text{and}\quad
\Scale[\tilde h_{z,\lambda}^*A] = \lambda\Scale[A].
\]
In particular, if $z = z[A]$ and $\lambda = \lambda[A]$, then $\tilde h_{z,\lambda}^{-1,*}A = (\tilde h_{z,\lambda}^{-1})^*A$ is a centered connection on $P$.
\end{lem}

\begin{proof}
Using the scaling behavior of two-forms and the volume form, the fact that $F(\tilde h_{z,\lambda}^*A) = \tilde h_{z,\lambda}^*F(A)$, and writing $x = h_{z,\lambda}(y) = \delta_\lambda(\tau_z(y)) = (y - z)/\lambda$ and $y = \lambda x + z$, we have
\begin{align*}
\Center[\tilde h_{z,\lambda}^*A]
&=
\left(\int_{\RR^4}|\varphi_n^*F(\tilde h_{z,\lambda}^*A)(y)|_\delta^2\,d^4y\right)^{-1}
\int_{\RR^4}y|\varphi_n^*F(\tilde h_{z,\lambda}^*A)(y)|_\delta^2\,d^4y
\\
&=
\left(\int_{\RR^4}|h_{z,\lambda}^*\varphi_n^*F_A(y)|_\delta^2\,d^4y\right)^{-1}
\int_{\RR^4}y|h_{z,\lambda}^*\varphi_n^*F_A(y)|_\delta^2\,d^4y,
\\
&= \left(\int_{\RR^4}|\varphi_n^*F_A(x)|_\delta^2\,d^4x\right)^{-1}
\int_{\RR^4}(\lambda x + z)|\varphi_n^*F_A(x)|_\delta^2\,d^4x
\\
&= \lambda\Center[A] + z,
\end{align*}
as desired. For the scale,
\begin{align*}
\Scale[\tilde h_{z,\lambda}^*A]^2
&=
\left(\int_{\RR^4}|h_{z,\lambda}^*\varphi_n^*F_A(y)|_\delta^2\,d^4y\right)^{-1}
\int_{\RR^4}|y - \Center[\tilde h_{z,\lambda}^*A]|^2|h_{z,\lambda}^*\varphi_n^*F_A(y)|_\delta^2\,d^4y,
\\
&=
\left(\int_{\RR^4}|\varphi_n^*F_A(y)|_\delta^2\,d^4y\right)^{-1}
\int_{\RR^4}|y - \lambda\Center[A] - z|^2|h_{z,\lambda}^*\varphi_n^*F_A(y)|_\delta^2\,d^4y,
\\
&= \lambda^2\left(\int_{\RR^4}|\varphi_n^*F_A(x)|_\delta^2\,d^4x\right)^{-1}
\int_{\RR^4}|x - \Center[A]|^2|\varphi_n^*F_A(x)|_\delta^2\,d^4x
\\
&= \lambda^2\Scale[A]^2,
\end{align*}
again as desired. Finally, observe that $h_{z,\lambda}^{-1} = h_{-z/\lambda,1/\lambda}$, since $y = h_{z,\lambda}(x) = (x - z)/\lambda$ and $x = h_{z,\lambda}^{-1}(y) = \lambda y + z = \lambda (y + z/\lambda) $. Therefore,
\[
\Center[\tilde h_{z,\lambda}^{-1,*}A]
=
\Center[\tilde h_{-z/\lambda,1/\lambda}^*A]
=
\lambda^{-1}\Center[A] - \lambda^{-1}z = 0,
\]
and
\[
\Scale[\tilde h_{z,\lambda}^{-1,*}A]
=
\Scale[\tilde h_{-z/\lambda,1/\lambda}^*A]
=
\lambda^{-1}\Scale[A] = 1.
\]
This completes the proof.
\end{proof}

\subsection{Continuity of the mass center and scale maps with respect to Sobolev variations of the connections}
\label{subsec:Continuity_mass_center_and_scale_maps_wrt_connections}
According to Taubes \cite[p. 342]{TauFrame}, the integrals in \eqref{eq:Mass_center_and_scale_connection} are well-defined and define a \emph{smooth} map,
\[
\sB'(P)\less[\Theta] \ni [A,p] \to (z[A],\lambda[A]) \in \RR^4\times \RR_+,
\]
where $\Theta$ is the product connection on $P = S^4\times G$ and $\sB'(P) :=  (\sA(P)\times P|_s)/G$ (compare \cite[p. 328]{TauFrame}). However, we shall only need to know that the map
\[
\sB_{2,2}(P)\less[\Theta] \ni [A] \to (z[A],\lambda[A]) \in \RR^4\times \RR_+,
\]
is well-defined and \emph{continuous}, where $\sB_{2,2}(P) = \sA_{2,2}(P)/\Aut_{3,2}(P)$ and $\sA_{k,p}(P)$ is the affine space of $W^{k,p}$ connections on $P$ and $\Aut_{k+1,p}(P)$ is the group of $W^{k+1,p}$ automorphisms of $P$, for $p \geq 1$ and integers $k\geq 1$ obeying $(k+1)p>4$. We shall establish this property, together with a slightly stronger version of continuity in the following

\begin{prop}[$W^{1,2}$ continuity of the mass center and scale maps]
\label{prop:Mass_center_and_scale_connection_W12_continuity}
There is a universal constant $C\in [1,\infty)$ with the following significance. Let $k\geq 2$ be an integer and $A$ be a $W^{k,2}$ connection on a principal $G$-bundle over $S^4$ with its standard round metric of radius one, $g = g_\round$. Suppose $k=2$ and $a \in W_{A,g}^{2,2}(S^4; \Lambda^1\otimes\ad P)$. For the mass center function, $\Center[A] \equiv z[A]$, we have
\begin{multline}
\label{eq:Mass_center_W12_continuity}
|z[A+a] - z[A]|
\\
\leq
C\|F_A\|_{W_{A,g}^{1,2}(S^4)} \left(\|a\|_{W_{A,g}^{1,2}(S^4)} + \|a\|_{W_{A,g}^{1,2}(S^4)}^2 \right)
+ C\|a\|_{W_{A,g}^{1,2}(S^4)} \|a\|_{W_{A,g}^{2,2}(S^4)}^2,
\end{multline}
and $z[A+a] \to z[A]$ uniformly as $\|a\|_{W_{A,g}^{2,2}(S^4)} \to 0$. Moreover, $z[A+a] \to z[A]$ uniformly as $\|a\|_{W_{A,g}^{1,2}(S^4)} \to 0$ while $\|a\|_{W_{A,g}^{2,2}(S^4)}$ remains uniformly bounded; in particular, $z[A]$ is a uniformly continuous function of $[A] \in \sB(P)$ endowed with the $W^{1,2}$ distance function $\dist_{W^{1,2}}$ given by \eqref{eq:Wkp_distance_on_quotient_space_connections} upon restriction to subsets of $\sB(P)$ that are bounded with respect to the $W^{2,2}$ distance function $\dist_{W^{2,2}}$. For the scale function, $\Scale[A] \equiv \lambda[A]$, we have
\begin{multline}
\label{eq:Scale_W22_continuity}
|\lambda[A+a]^2 - \lambda[A]^2|
\\
\leq
C|z[A] - z[A+a]| \left\{\|F_{A+a}\|_{W_{A,g}^{1,2}(S^4)}^2
+ \left(|z[A]| + |z[A+a]|\right) \|F_{A+a}\|_{L^2(S^4)}^2 \right\}
\\
+ C\|F_A\|_{W_{A,g}^{1,2}(S^4)} \left( \|a\|_{W_{A,g}^{2,2}(S^4)}
+  \|a\|_{W_{A,g}^{2,2}(S^4)}^2 \right) + C\|a\|_{W_{A,g}^{2,2}(S^4)}^3,
\end{multline}
and $\lambda[A+a]^2 \to \lambda[A]^2$ uniformly as $\|a\|_{W_{A,g}^{2,2}(S^4)} \to 0$; in particular, $\lambda[A]^2$ is a uniformly continuous function of $[A] \in \sB(P)$ endowed with the $W^{2,2}$ metric. If $k=3$ and $a \in W_{A,g}^{3,2}(S^4; \Lambda^1\otimes\ad P)$, then
\begin{multline}
\label{eq:Scale_W12_continuity}
|\lambda[A+a]^2 - \lambda[A]^2|
\\
\leq C|z[A] - z[A+a]| \left\{\|F_{A+a}\|_{W_{A,g}^{1,2}(S^4)}^2
+ \left(|z[A]| + |z[A+a]|\right) \|F_{A+a}\|_{L^2(S^4)}^2 \right\}
\\
+ C\|F_A\|_{W_{A,g}^{1,2}(S^4)} \left( \|a\|_{W_{A,g}^{1,2}(S^4)}^{1/2} \|_{W_{A,g}^{3,2}(S^4)}^{1/2}
+  \|a\|_{W_{A,g}^{1,2}(S^4)} \|_{W_{A,g}^{3,2}(S^4)} \right)
\\
+ C\|a\|_{W_{A,g}^{1,2}(S^4)}^{3/2} \|_{W_{A,g}^{3,2}(S^4)}^{3/2},
\end{multline}
and $\lambda[A+a]^2 \to \lambda[A]^2$ uniformly as $\|a\|_{W_{A,g}^{1,2}(S^4)} \to 0$ while $\|a\|_{W_{A,g}^{3,2}(S^4)}$ remains uniformly bounded; in particular, $\lambda[A]^2$ is a uniformly continuous function of $[A] \in \sB(P)$ endowed with the $W^{1,2}$ distance function upon restriction to subsets of $\sB(P)$ that are bounded with respect to the $W^{3,2}$ distance function.
\end{prop}

Before proceeding to the proof of Proposition \ref{prop:Mass_center_and_scale_connection_W12_continuity}, we first need to relate Sobolev norms for $\ad P$-valued one-forms over $\RR^4$ and $S^4$ with their standard Riemannian metrics, respectively. Recall that the standard round metric, $g_\round$, of radius one on $S^d$ (for $d \geq 2$) takes the form (see, for example, \cite[p. 26]{Jost_riemannian_geometry_geometric_analysis})
\begin{equation}
\label{eq:Standard_round_metric_radius_one_sphere_stereographic_coordinates}
\begin{aligned}
\varphi_n^*g_\round(x) &= \frac{4}{(1 + |x|^2)^2}(dx)^2, \quad \forall\, x \in \RR^d,
\\
\varphi_s^*g_\round(y) &= \frac{4}{(1 + |y|^2)^2}(dy)^2, \quad \forall\, y \in \RR^d,
\end{aligned}
\end{equation}
where $g_{\euclid}(x) = (dx)^2 = (dx^1)^2 + \cdots + (dx^4)^2$. Define $\varsigma(r) := 2/(1+r^2)$ for $r \in [0,\infty)$, so that $\varphi_n^*g_\round(x) = \varsigma^2(|x|)(dx)^2$ and $d\vol_{\round}(x) = \varsigma^4(|x|)\,dx$, where $d\vol_{\euclid}(x) = dx = dx^1\wedge\cdots \wedge dx^4$.

\begin{lem}[Relating integral norms over $\RR^4$ and $S^4$]
For $w \in \Omega^i(S^4;\ad P)$ and $i \geq 0$ and denoting $g = g_\round$ and $\delta = g_{\euclid}$ for brevity, we have
\begin{subequations}
\label{eq:L2_and_L4_norms_iform_euclid_leq_constant_L2_and_L4_norms_iform_sphere}
\begin{align}
\label{eq:L2norm_iform_euclid_leq_constant_L2norm_iform_sphere}
\|\varphi_n^*w\|_{L^2(\RR^4,\delta)} &\leq 2^{i-2}\|w\|_{L^2(S^4,g)}, \quad\text{for } i \geq 2,
\\
\label{eq:L4norm_iform_euclid_leq_constant_L4norm_iform_sphere}
\|\varphi_n^*w\|_{L^4(\RR^4,\delta)} &\leq 2^{i-1}\|w\|_{L^4(S^4,g)}, \quad\text{for } i \geq 1.
\end{align}
\end{subequations}
\end{lem}

\begin{proof}
We observe that
\begin{align*}
\|\varphi_n^*w\|_{L^2(\RR^4,\delta)}^2
&=
\int_{\RR^4} |\varphi_n^*w|_\delta^2(x)\,dx
\\
&= \int_{\RR^4} \varsigma^{2i-4}(|x|)\varsigma^{-2i}(|x|)|\varphi_n^*w|_\delta^2(x)\varsigma^4(|x|)\,dx
\\
&= \int_{\RR^4} \varsigma^{2i-4}(|x|) |\varphi_n^*w|_g^2(x)\,d\vol_g(x)
\\
&\leq 2^{2i-4}\int_{\RR^4} |\varphi_n^*w|_g^2(x)\,d\vol_g(x) \quad\text{for } i \geq 2,
\end{align*}
since $\varsigma(r) = 2/(1+r^2) \leq 2$ for all $r \geq 0$, and thus we obtain \eqref{eq:L2norm_iform_euclid_leq_constant_L2norm_iform_sphere}. Similarly, we have
\begin{align*}
\|\varphi_n^*w\|_{L^4(\RR^4,\delta)}^4
&=
\int_{\RR^4} |\varphi_n^*w|_\delta^4(x)\,dx
\\
&= \int_{\RR^4} \varsigma^{4i-4}(|x|)\varsigma^{-4i}(|x|)|\varphi_n^*w|_\delta^4(x)\varsigma^4(|x|)\,dx
\\
&= \int_{\RR^4} \varsigma^{4i-4}(|x|)|\varphi_n^*w|_g^4(x)\,d\vol_g(x)
\\
&\leq 2^{4i-4}\int_{\RR^4} |\varphi_n^*w|_g^4(x)\,d\vol_g(x) \quad\text{for } i \geq 1,
\end{align*}
and therefore \eqref{eq:L4norm_iform_euclid_leq_constant_L4norm_iform_sphere} follows.
\end{proof}

We can now proceed to the

\begin{proof}[Proof of Proposition \ref{prop:Mass_center_and_scale_connection_W12_continuity}]
If $a \in W_A^{2,2}(S^4;\Lambda^1\otimes\ad P)$, then $F_{A+a} = F_A + d_Aa + a\wedge a$ over $S^4$, and thus over $\RR^4$,
\begin{equation}
\label{eq:Expansion_integrands_center_A+a_minus_center_A}
|\varphi_n^*F_{A+a}|_\delta^2 - |\varphi_n^*F_A|_\delta^2
=
2\langle \varphi_n^*F_A, \varphi_n^*d_Aa \rangle_\delta
+ 2\langle \varphi_n^*F_A, \varphi_n^*(a\wedge a) \rangle_\delta
+ 2\langle \varphi_n^*d_Aa, \varphi_n^*(a\wedge a) \rangle_\delta.
\end{equation}
We first consider $W_A^{1,2}(S^4)$ continuity of the mass center \eqref{eq:Mass_center_connection} for $W_A^{2,2}(S^4)$ bounded families of connections $A+a$ on $P\to S^4$. We have
\begin{subequations}
\label{eq:Integral_euclidean_space_y_FA_dAa_and_integral_euclidean_space_y_FA_a_wedge_a_prelim_bound}
\begin{align}
\label{eq:Integral_euclidean_space_y_FA_dAa_prelim_bound}
\int_{\RR^4}|y\langle \varphi_n^*F_A, \varphi_n^*d_Aa \rangle_\delta|\,dy
&\leq
\left(\int_{\RR^4}|y|^2|\varphi_n^*F_A|_\delta^2\,dy\right)^{1/2}
\|\varphi_n^*d_Aa\|_{L^2(\RR^4,\delta)},
\\
\label{eq:Integral_euclidean_space_y_FA_a_wedge_a_prelim_bound}
\int_{\RR^4}|y\langle \varphi_n^*F_A, \varphi_n^*(a\wedge a) \rangle_\delta|\,dy
&\leq
2\left(\int_{\RR^4}|y|^2|\varphi_n^*F_A|_\delta^2\,dy\right)^{1/2}
\|\varphi_n^*a\|_{L^4(\RR^4,\delta)}^2.
\end{align}
\end{subequations}
Applying Lemma \ref{lem:Sobolev_embedding_critical_exponent_spaces_Euclidean} yields
\begin{align*}
\left(\int_{\RR^4}|y|^2|\varphi_n^*F_A|_\delta^2\,dy\right)^{1/2}
&\leq
\frac{1}{2}\|\nabla|\varphi_n^*F_A|_\delta\|_{L^2(\RR^4,\delta)}
\\
&\leq \frac{1}{2}\|\nabla_{\varphi_n^*A}^\delta \varphi_n^*F_A\|_{L^2(\RR^4,\delta)} \quad\text{(by the Kato Inequality \eqref{eq:FU_6-20_first-order_Kato_inequality})}
\\
&\leq C\left(\|\nabla_A^g F_A\|_{L^2(S^4,g)} + \|F_A\|_{L^2(S^4,g)}\right)
\quad\text{(by \eqref{eq:L2norm_iform_euclid_leq_constant_L2norm_iform_sphere})},
\end{align*}
to give
\begin{equation}
\label{eq:L2_norm_y_F_A_euclidean_space_leq_W12_norm_FA_sphere}
\left(\int_{\RR^4}|y|^2|\varphi_n^*F_A|_\delta^2\,dy\right)^{1/2}
\leq
C\|F_A\|_{W_{A,g}^{1,2}(S^4)},
\end{equation}
which is finite since $A$ is a $W^{2,2}$ connection on $P\to S^4$, and $C\in [1,\infty)$ is a universal constant.

In the derivation of \eqref{eq:L2_norm_y_F_A_euclidean_space_leq_W12_norm_FA_sphere}, we have used the fact that if $\nabla^g$ and $\nabla^\delta$ are respectively the Levi-Civita covariant derivatives with respect to the Riemannian metrics $g$ on $TS^4$ and $\delta$ on $T\RR^4$, together with their dual and associated bundles, then
\[
\varphi_n^*\nabla^g \eta = \nabla^\delta \varphi_n^* \eta + \Gamma\otimes \varphi_n^*\eta,
\quad\forall\, \eta \in C^\infty(TS^4),
\]
that is (see \cite[Equation (1.25)]{Chavel}), for any $\xi \in C^\infty(TS^4)$,
\[
\varphi_n^*\nabla_\xi^g \eta
=
\xi^\beta\left(\frac{\partial\eta^\mu}{\partial x^\beta}
+ \Gamma_{\alpha\beta}^\mu \eta^\alpha\right)\frac{\partial}{\partial x^\mu},
\]
where $\varphi_n^*\eta = \eta^\mu \partial/\partial x^\mu$ and $\varphi_n^*\xi = \xi^\mu \partial/\partial x^\mu$ and \cite[Equation (1.23)]{Chavel}
\[
\nabla_{\frac{\partial}{\partial x^\beta}}^g\left(\frac{\partial}{\partial x^\alpha}\right)
=
\Gamma_{\alpha\beta}^\mu \frac{\partial}{\partial x^\mu},
\]
and the $\Gamma_{\alpha\beta}^\mu$ are the Christoffel symbols \cite[Equation (1.31)]{Chavel} for the metric $g$ with respect to the coordinates $x^\mu(\cdot)$ defined by $\varphi_n^{-1}:S^4\less\{n\} \to \RR^4$. Similarly,
\[
\varphi_n^*\nabla_A^g v = \nabla_A^\delta \varphi_n^* v + \Gamma\otimes \varphi_n^*v,
\quad\forall\, v \in C^\infty(TS^4\otimes\ad P),
\]
and the analogous identities hold for $v \in C^\infty(\Lambda^i(T^*S^4)\otimes\ad P) = \Omega^i(S^4;\ad P)$ when $i \geq 1$.

Consequently, from \eqref{eq:L2_and_L4_norms_iform_euclid_leq_constant_L2_and_L4_norms_iform_sphere}, \eqref{eq:Integral_euclidean_space_y_FA_dAa_and_integral_euclidean_space_y_FA_a_wedge_a_prelim_bound}, and \eqref{eq:L2_norm_y_F_A_euclidean_space_leq_W12_norm_FA_sphere}, we obtain
\begin{subequations}
\label{eq:Integral_euclidean_space_y_FA_dAa_and_integral_euclidean_space_y_FA_a_wedge_a_bound}
\begin{align}
\label{eq:Integral_euclidean_space_y_FA_dAa_bound}
\int_{\RR^4}|y\langle \varphi_n^*F_A, \varphi_n^*d_Aa \rangle_\delta|\,dy
&\leq
C\|F_A\|_{W_{A,g}^{1,2}(S^4)}\|\nabla_A^ga\|_{L^2(S^4,g)},
\\
\label{eq:Integral_euclidean_space_y_FA_a_wedge_a_bound}
\int_{\RR^4}|y\langle \varphi_n^*F_A, \varphi_n^*(a\wedge a) \rangle_\delta|\,dy
&\leq
C\|F_A\|_{W_{A,g}^{1,2}(S^4)}\|a\|_{L^4(S^4,g)}^2.
\end{align}
\end{subequations}
Furthermore, we have
\begin{align*}
{}&\int_{\RR^4}|y\langle \varphi_n^*d_Aa, \varphi_n^*(a\wedge a) \rangle_\delta|\,dy
\\
&\quad\leq
2\|\varphi_n^*d_Aa\|_{L^2(\RR^4,\delta)}
\left(\int_{\RR^4}|y|^2|\varphi_n^*a|_\delta^4\,dy\right)^{1/2}
\\
&\quad\leq \|\varphi_n^*d_Aa\|_{L^2(\RR^4,\delta)}
\|\nabla|\varphi_n^*a|_\delta^2\|_{L^2(\RR^4,\delta)} \quad\text{(by Lemma \ref{lem:Sobolev_embedding_critical_exponent_spaces_Euclidean})}
\\
&\quad\leq \|\varphi_n^*d_Aa\|_{L^2(\RR^4,\delta)} \|\nabla_A^\delta\varphi_n^*a\|_{L^4(\RR^4,\delta)}^2 \quad\text{(by the Kato Inequality \eqref{eq:FU_6-20_first-order_Kato_inequality})}
\\
&\quad\leq 2\|d_Aa\|_{L^2(S^4,g)} \left(\|\nabla_A^ga\|_{L^4(S^4,g)} + \|a\|_{L^4(S^4,g)}\right)^2
\quad\text{(by \eqref{eq:L2_and_L4_norms_iform_euclid_leq_constant_L2_and_L4_norms_iform_sphere})}.
\end{align*}
Applying the Sobolev Embedding \eqref{eq:Sobolev_embedding_manifold_bounded_geometry} in the form $W_g^{1,2}(S^4)\hookrightarrow L^4(S^4,g)$ and Kato Inequality \eqref{eq:FU_6-20_first-order_Kato_inequality} we obtain
\begin{equation}
\label{eq:Integral_euclidean_space_y_dAa_a_wedge_a_bound}
\int_{\RR^4}|y\langle \varphi_n^*d_Aa, \varphi_n^*(a\wedge a) \rangle_\delta|\,dy
\leq
C\|\nabla_A^ga\|_{L^2(S^4,g)} \|a\|_{W_{A,g}^{2,2}(S^4)}^2.
\end{equation}
Combining the definition \eqref{eq:Mass_center_connection} of $\Center[A]$, the identity \eqref{eq:Expansion_integrands_center_A+a_minus_center_A}, and the inequalities \eqref{eq:Integral_euclidean_space_y_FA_dAa_and_integral_euclidean_space_y_FA_a_wedge_a_bound} and \eqref{eq:Integral_euclidean_space_y_dAa_a_wedge_a_bound} yields \eqref{eq:Mass_center_W12_continuity}.

Second, we consider $W_{A,g}^{1,2}(S^4)$ continuity of the scale \eqref{eq:Scale_connection} and write
\begin{equation}
\label{eq:Scale_connection_difference_terms_1_and_2}
\begin{aligned}
{}&|y-z[A+a]|^2 |\varphi_n^*F_{A+a}|_\delta^2
-
|y-z[A]|^2 |\varphi_n^*F_A|_\delta^2
\\
&\quad =\, \left\{|y-z[A+a]|^2 - |y-z[A]|^2 \right\} |\varphi_n^*F_{A+a}|_\delta^2
+ |y-z[A]|^2 \left\{ |\varphi_n^*F_{A+a}|_\delta^2 - |\varphi_n^*F_A|_\delta^2 \right\}
\\
&\quad =: \text{Term 1} + \text{Term 2}.
\end{aligned}
\end{equation}
For Term 1 in the identity \eqref{eq:Scale_connection_difference_terms_1_and_2}, we observe that
\begin{align*}
|y-z[A+a]|^2 - |y-z[A]|^2
&=
|y|^2 - 2\langle y, z[A+a]\rangle + |z[A+a]|^2
- \left\{ |y|^2 - 2\langle y, z[A]\rangle + |z[A]|^2 \right\}
\\
&= 2\langle y, z[A] - z[A+a]\rangle + |z[A+a]|^2 - |z[A]|^2,
\end{align*}
and so
\[
\left||y-z[A+a]|^2 - |y-z[A]|^2\right|
\leq
|z[A] - z[A+a]| \left( 2|y| + |z[A]| + |z[A+a]| \right).
\]
We calculate, just as in the case of $W_{A,g}^{1,2}(S^4)$ continuity of the mass center,
\begin{align*}
\int_{\RR^4}|y| |\varphi_n^*F_{A+a}|_\delta^2\,dy
&\leq
\int_{\RR^4} (1 + |y|^2) |\varphi_n^*F_{A+a}|_\delta^2\,dy
\\
&\leq C\left(\int_{S^4} |F_{A+a}|_g^2\,d\vol_g + \int_{S^4} |\nabla_A^gF_{A+a}|_g^2\,d\vol_g\right)
\\
&= C\|F_{A+a}\|_{W_{A,g}^{1,2}(S^4)}^2.
\end{align*}
Thus, for Term 1 in \eqref{eq:Scale_connection_difference_terms_1_and_2}, we have
\begin{multline}
\label{eq:Scale_W12_continuity_term_one}
\int_{\RR^4} \left\{|y-z[A+a]|^2 - |y-z[A]|^2 \right\} |\varphi_n^*F_{A+a}|_\delta^2 \,dy
\\
\leq C|z[A] - z[A+a]| \left\{\|F_{A+a}\|_{W_{A,g}^{1,2}(S^4)}^2
+ \left(|z[A]| + |z[A+a]|\right) \|F_{A+a}\|_{L^2(S^4,g)}^2 \right\}.
\end{multline}
For Term 2 in the identity \eqref{eq:Scale_connection_difference_terms_1_and_2} and arguing just as in the case of $W_{A,g}^{1,2}(S^4)$ continuity of the mass center, we see that
\begin{align*}
{}&\int_{\RR^4}|y-z[A]|^2 |\langle \varphi_n^*F_A, \varphi_n^*d_Aa \rangle_\delta|\,dy
\\
&\quad\leq
\left(\int_{\RR^4}|y-z[A]|^2 |\varphi_n^*F_A|_\delta^2\,dy\right)^{1/2}
\left(\int_{\RR^4}|y-z[A]|^2 |\varphi_n^*d_Aa|_\delta^2\,dy\right)^{1/2}
\\
&\quad\leq
\frac{1}{2}\left(\int_{\RR^4} |\nabla|\varphi_n^*F_A|_\delta|^2\,dy\right)^{1/2}
\left(\int_{\RR^4}|\nabla|\varphi_n^*d_Aa|_\delta|^2\,dy\right)^{1/2},
\end{align*}
so that
\begin{equation}
\label{eq:Integral_euclidean_space_y_minus_z_squared_FA_dAa_W22_bound}
\int_{\RR^4}|y-z[A]|^2 |\langle \varphi_n^*F_A, \varphi_n^*d_Aa \rangle_\delta|\,dy
\leq
C\|F_A\|_{W_{A,g}^{1,2}(S^4)} C\|a\|_{W_A^{2,2}(S^4,g)}.
\end{equation}
Also,
\begin{align*}
{}&\int_{\RR^4}|y-z[A]|^2 |\langle \varphi_n^*F_A, \varphi_n^*(a\wedge a) \rangle_\delta|\,dy
\\
&\quad\leq
2\left(\int_{\RR^4}|y-z[A]|^2 |\varphi_n^*F_A|_\delta^2\,dy\right)^{1/2}
\left(\int_{\RR^4}|y-z[A]|^2 |\varphi_n^*a|_\delta^4\,dy\right)^{1/2}
\\
&\quad\leq
\left(\int_{\RR^4} |\nabla|\varphi_n^*F_A|_\delta|^2\,dy\right)^{1/2}
\left(\int_{\RR^4} |\nabla|\varphi_n^*a|_\delta^2|^2\,dy\right)^{1/2}
\\
&\quad\leq
\left(\int_{\RR^4} |\nabla_{\varphi_n^*A}^\delta\varphi_n^*F_A|_\delta^2\,dy\right)^{1/2}
\left(\int_{\RR^4} |\nabla_{\varphi_n^*A}^\delta\varphi_n^*a|_\delta^4\,dy\right)^{1/2}
\\
&\leq C\|F_A\|_{W_{A,g}^{1,2}(S^4)} \|a\|_{W_{A,g}^{1,4}(S^4)}^2,
\end{align*}
applying \eqref{eq:L2_and_L4_norms_iform_euclid_leq_constant_L2_and_L4_norms_iform_sphere}, the Sobolev Embedding \eqref{eq:Sobolev_embedding_manifold_bounded_geometry} in the form $W_g^{1,2}(S^4)\hookrightarrow L^4(S^4,g)$, and Kato Inequality \eqref{eq:FU_6-20_first-order_Kato_inequality} to obtain the last two inequalities, and
\begin{equation}
\label{eq:Integral_euclidean_space_y_minus_z_squared_FA_a_wedge_a_W22_bound}
\int_{\RR^4}|y-z[A]|^2 |\langle \varphi_n^*F_A, \varphi_n^*(a\wedge a) \rangle_\delta|\,dy
\leq
C\|F_A\|_{W_{A,g}^{1,2}(S^4)} \|a\|_{W_{A,g}^{2,2}(S^4)}^2.
\end{equation}
Furthermore, we have
\begin{align*}
{}&\int_{\RR^4} |y-z[A]|^2 |\langle \varphi_n^*d_Aa, \varphi_n^*(a\wedge a) \rangle_\delta|\,dy
\\
&\quad \leq
2\left(\int_{\RR^4} |y-z[A]|^2|\varphi_n^*d_Aa|_\delta^2\,dy\right)^{1/2}
\left(\int_{\RR^4} |y-z[A]|^2|\varphi_n^*a|_\delta^4\,dy\right)^{1/2}
\\
&\quad \leq \|\nabla|\varphi_n^*d_Aa|\|_{L^2(\RR^4,\delta)}
\|\nabla|\varphi_n^*a|_\delta^2\|_{L^2(\RR^4,\delta)} \quad\text{(by Lemma \ref{lem:Sobolev_embedding_critical_exponent_spaces_Euclidean})}
\\
&\quad \leq \|\nabla_A^\delta\varphi_n^*d_Aa\|_{L^2(\RR^4,\delta)} \|\nabla_A^\delta\varphi_n^*a\|_{L^4(\RR^4,\delta)}^2 \quad\text{(by the Kato Inequality \eqref{eq:FU_6-20_first-order_Kato_inequality})}.
\end{align*}
Applying \eqref{eq:L2_and_L4_norms_iform_euclid_leq_constant_L2_and_L4_norms_iform_sphere}, the Sobolev Embedding \eqref{eq:Sobolev_embedding_manifold_bounded_geometry} in the form $W_g^{1,2}(S^4)\hookrightarrow L^4(S^4,g)$, and Kato Inequality \eqref{eq:FU_6-20_first-order_Kato_inequality} yields
\begin{equation}
\label{eq:Integral_euclidean_space_y_minus_z_squared_dAa_a_wedge_a_W22_bound}
\int_{\RR^4} |y-z[A]|^2 |\langle \varphi_n^*d_Aa, \varphi_n^*(a\wedge a) \rangle_\delta|\,dy
\leq
C\|a\|_{W_{A,g}^{2,2}(S^4)}^3.
\end{equation}
Integration by parts gives
\[
\|(\nabla_A^g)^2a\|_{L^2(S^4,g)}^2
=
(\nabla_A^g a, \nabla_A^{g,*}(\nabla_A^g)^2 a\|_{L^2(S^4,g)}
\leq
\|\nabla_A^g a\|_{L^2(S^4,g)} \|\nabla_A^{g,*}(\nabla_A^g)^2 a\|_{L^2(S^4,g)},
\]
and hence the interpolation inequality,
\[
\|(\nabla_A^g)^2a\|_{L^2(S^4,g)} \leq C\|a\|_{W_{A,g}^{1,2}(S^4)}^{1/2} \|a\|_{W_{A,g}^{3,2}(S^4)}^{1/2},
\]
and thus,
\begin{equation}
\label{eq:W22_interpolation_inequality_W12_W32}
\|a\|_{W_{A,g}^{2,2}(S^4)} \leq C\|a\|_{W_{A,g}^{1,2}(S^4)}^{1/2} \|a\|_{W_{A,g}^{3,2}(S^4)}^{1/2}.
\end{equation}
Consequently, applying the interpolation inequality \eqref{eq:W22_interpolation_inequality_W12_W32} to the inequalities \eqref{eq:Integral_euclidean_space_y_minus_z_squared_FA_dAa_W22_bound}, \eqref{eq:Integral_euclidean_space_y_minus_z_squared_FA_a_wedge_a_W22_bound}, and \eqref{eq:Integral_euclidean_space_y_minus_z_squared_dAa_a_wedge_a_W22_bound} gives
\begin{align*}
\int_{\RR^4} |y-z[A]|^2 |\langle \varphi_n^*F_A, \varphi_n^*d_Aa \rangle_\delta|\,dy
&\leq
C\|F_A\|_{W_{A,g}^{1,2}(S^4)} \|a\|_{W_{A,g}^{1,2}(S^4)}^{1/2} \|a\|_{W_{A,g}^{3,2}(S^4)}^{1/2},
\\
\int_{\RR^4}|y-z[A]|^2 |\langle \varphi_n^*F_A, \varphi_n^*(a\wedge a) \rangle_\delta|\,dy
&\leq
C\|F_A\|_{W_{A,g}^{1,2}(S^4)} \|a\|_{W_{A,g}^{1,2}(S^4)} \|a\|_{W_{A,g}^{3,2}(S^4)},
\\
\int_{\RR^4} |y-z[A]|^2 |\langle \varphi_n^*d_Aa, \varphi_n^*(a\wedge a) \rangle_\delta|\,dy
&\leq
C\|a\|_{W_{A,g}^{1,2}(S^4)}^{3/2} \|a\|_{W_{A,g}^{3,2}(S^4)}^{3/2}.
\end{align*}
Combining the inequality \eqref{eq:Scale_W12_continuity_term_one} corresponding to Term 1 in the identity \eqref{eq:Scale_connection_difference_terms_1_and_2} with the preceding inequalities corresponding to Term 2 yields \eqref{eq:Scale_W22_continuity} and \eqref{eq:Scale_W12_continuity} and this completes the proof of the proposition.
\end{proof}

The method of proof of Proposition \ref{prop:Mass_center_and_scale_connection_W12_continuity} also yields boundedness results for the mass center and scale maps.

\begin{lem}[Boundedness of the mass center and scale maps]
\label{lem:Mass_center_and_scale_connection_W12_boundedness}
There is a universal constant $C\in [1,\infty)$ with the following significance. If $A$ is a $W^{2,2}$ connection on a principal $G$-bundle over $S^4$ with its standard round metric of radius one, then
\begin{subequations}
\label{eq:Mass_center_and_scale_bounds}
\begin{align}
\label{eq:Mass_center_bound}
|z[A]| &\leq C\|F_A\|_{W_{A,g}^{1,2}(S^4)}^2,
\\
\label{eq:Scale_bound}
\lambda[A]^2 &\leq C\|F_A\|_{W_{A,g}^{1,2}(S^4)}^2.
\end{align}
\end{subequations}
\end{lem}

\begin{proof}
The proof of Proposition \ref{prop:Mass_center_and_scale_connection_W12_continuity} shows that
\begin{align*}
|z[A]| &\leq \int_{\RR^4}|y| |\varphi_n^*F_A|_\delta^2\,dy
\leq
C\|F_A\|_{W_{A,g}^{1,2}(S^4)}^2,
\\
\lambda[A]^2
&=
\int_{\RR^4}|y-z[A]|^2 |\varphi_n^*F_A|_\delta^2\,dy
\leq
C\|F_A\|_{W_{A,g}^{1,2}(S^4)}^2.
\end{align*}
This completes the proof.
\end{proof}

Finally, we recall the

\begin{defn}[Center and scale of a positive Borel measure over $\RR^4$]
\label{defn:Center_scale_positive_Borel_measure}
(Compare \cite[Equation (4.11)]{FeehanGeometry})
Let $\mu$ be a positive Borel measure over $\RR^4$ with the property that
\[
0 < \int_{\RR^4}(1+|y|^2) \,d\mu(y) < \infty.
\]
The \emph{center}, $z=z[\mu]\in\RR^4$, and the \emph{scale}, $\lambda=\lambda[\mu]$ of $\mu$ are defined by
\begin{align}
z[\mu] &:= \frac{1}{\mu(\RR^4)}\int_{\RR^4}y \,d\mu(y),
\label{eq:Mass_center_measure}
\\
\lambda[\mu]^2 &:= \frac{1}{\mu(\RR^4)}\int_{\RR^4} |y - z[\mu]|^2 \,d\mu(y).
\label{eq:Scale_measure}
\end{align}
The measure, $\mu$, is \emph{centered} if $z[\mu] = 0$ and $\lambda[\mu] = 1$.
\end{defn}

\subsection{Sobolev bounded derivatives and continuity of a family of connections with respect to mass centers and scales}
\label{subsec:Sobolev_continuity_family_connections_wrt_mass_centers_and_scales}
Thus far we have considered continuity of the mass center $z[A]$ and scale $\lambda[A]$ with respect to Sobolev variations of the connection $A$. We consider the opposite question of Sobolev continuity of the family $\tilde h_{z,\lambda}^*A$ with respect to the parameters $(z,\lambda) \in \RR^4\times\RR_+$, given a fixed $W^{2,2}$ connection $A$ on $P$. Specifically, we prove the

\begin{prop}[$W^{1,2}$ boundedness of derivatives and continuity of a family of connections with respect to mass centers and scales]
\label{prop:W12_continuity_family_connections_wrt_mass_centers_and_scales}
Let $A$ be a $W^{2,2}$ connection on a principal $G$-bundle over $S^4$ with its standard round metric of radius one, $g = g_\round$. The family of connections, $h_{z,\lambda}^*A \in \sA(h_{z,\lambda}^*P) \cong \sA(P)$, has \emph{bounded derivatives} with respect to the mass centers, $z \in \RR^4$, in the sense of \eqref{eq:L4norm_bound_derivative_connection_wrt_mass_centers} and \eqref{eq:L2norm_nabla_bound_derivative_connection_wrt_mass_centers}, and bounded derivatives with respect to the scales, $\lambda \in \RR_+$, in the sense of \eqref{eq:L4norm_bound_derivative_connection_wrt_scales} and
\eqref{eq:L2norm_nabla_bound_derivative_connection_wrt_scales}. The family of connections, $h_{z,\lambda}^*A \in \sA(P)$, is \emph{uniformly continuous} with respect to the mass centers, $z \in \RR^4$, in the sense of \eqref{eq:L4distance_connections_wrt_mass_centers} and \eqref{eq:L2nabla_distance_connections_wrt_mass_centers}, and continuous with respect to the scales, $\lambda \in \RR_+$, in the sense of \eqref{eq:L4distance_connections_wrt_scales} and \eqref{eq:L2nabla_distance_connections_wrt_scales} and hence the family of connections, $h_{z,\lambda}^*A \in \sA(P)$, is \emph{uniformly continuous} in the $W_{A,g}^{1,2}(S^4;\Lambda^1\otimes\ad P)$ norm with respect to $(z,\lambda)\in\RR^4\times\RR_+$.
The family of points, $[h_{z,\lambda}^*A] \in \sB(h_{z,\lambda}^*P) \cong \sB(P)$, is \emph{uniformly continuous}, for the distance function $\dist_{W^{1,2}}$ given by \eqref{eq:Wkp_distance_on_quotient_space_connections} on $\sB(P)$, with respect to $(z,\lambda) \in \RR^4\times\RR_+$.
\end{prop}

\begin{proof}
Recall from \cite[Equation (3.9)]{FeehanGeometry} (taking $q=0\in\RR^4$ on the left-hand side and $q = z \in \RR^4$ and $\partial/\partial p = \partial/\partial z^\mu$ and thus $\mathbf{p} = \partial/\partial x^\mu$ on the right-hand side) that\footnote{There is a typographical error (a factor of $1/\lambda$ that should be removed) in the expression \cite[Equation (3.9)]{FeehanGeometry} for the directional derivative $\partial\tau_q^*c_\lambda^*A/\partial p$ which we correct here.}
\begin{align}
\label{eq:Feehan_1995_3-9_derivative_wrt_scale}
\frac{\partial}{\partial\lambda}\tilde\delta_\lambda^*A
&=
-\frac{1}{\lambda}\tilde\delta_\lambda^*\left(\iota_{\varphi_{n*}\frac{\partial}{\partial r}}F_A\right),
\\
\label{eq:Feehan_1995_3-9_derivative_wrt_mass_center}
\frac{\partial}{\partial z^\mu}\tilde\tau_z^*A
&=
-\tilde\tau_z^* \left(\iota_{\varphi_{n*}\frac{\partial}{\partial x^\mu}}F_A\right),
\end{align}
where $\iota_\xi:\Omega^{i+1}(X) \to \Omega^i(X)$ for $i\geq 0$ denotes interior product with respect to a vector field $\xi \in C^\infty(TX)$ over a smooth manifold $X$, and $\{x^\mu\}$ are the standard coordinates on $\RR^4$. These correspond to the expressions over $\RR^4$,
\begin{align*}
\frac{\partial}{\partial\lambda}\delta_\lambda^*\varphi_n^*A
&=
-\frac{1}{\lambda}\delta_\lambda^*\left(\iota_{\frac{\partial}{\partial r}} \varphi_n^*F_A\right),
\\
\frac{\partial}{\partial z^\mu}\tau_z^*\varphi_n^*A
&=
-\tau_z^*\left(\iota_{\frac{\partial}{\partial x^\mu}} \varphi_n^*F_A\right).
\end{align*}
Thus,
\[
\left\|\frac{\partial}{\partial\lambda}\delta_\lambda^*\varphi_n^*A \right\|_{L^4(\RR^4,\delta)}
=
\frac{1}{\lambda} \left\|\delta_\lambda^*\left(\iota_{\frac{\partial}{\partial r}}\varphi_n^*F_A\right) \right\|_{L^4(\RR^4,\delta)},
\]
and hence, by the invariance of the $L^4$ norm on one-forms with respect to conformal changes of the metric and with respect to rescaling,
\begin{equation}
\label{eq:L4norm_bound_derivative_connection_wrt_scales}
\left\|\frac{\partial}{\partial\lambda}\tilde\delta_\lambda^*A \right\|_{L^4(S^4,g)}
=
\frac{1}{\lambda} \left\|\iota_{\varphi_{n*}\frac{\partial}{\partial r}}F_A \right\|_{L^4(S^4,g)}.
\end{equation}
Similarly,
\[
\left\| \frac{\partial}{\partial z^\mu}\tau_z^*\varphi_n^*A \right\|_{L^4(\RR^4,\delta)}
=
\left\| \tau_z^*\left(\iota_{\frac{\partial}{\partial x^\mu}} \varphi_n^*F_A\right) \right\|_{L^4(\RR^4,\delta)}
=
\left\| \iota_{\frac{\partial}{\partial x^\mu}} \varphi_n^*F_A \right\|_{L^4(\RR^4,\delta)},
\]
and thus
\begin{equation}
\label{eq:L4norm_bound_derivative_connection_wrt_mass_centers}
\left\|\frac{\partial}{\partial z^\mu}\tilde\tau_z^*A \right\|_{L^4(S^4,g)}
=
\left\|\iota_{\varphi_{n*}\frac{\partial}{\partial x^\mu}}F_A \right\|_{L^4(S^4,g)}.
\end{equation}
Turning to the covariant derivatives,
\begin{align*}
\left\|\nabla_{\tau_z^*\varphi_n^*A}^\delta\frac{\partial}{\partial z^\mu}\tau_z^*\varphi_n^*A \right\|_{L^2(\RR^4,\delta)}
&=
\left\|\nabla_{\tau_z^*\varphi_n^*A}^\delta
\tau_z^*\left(\iota_{\frac{\partial}{\partial x^\mu}} \varphi_n^*F_A\right) \right\|_{L^2(\RR^4,\delta)}
\\
&= \left\|\tau_z^*\left(\nabla_{\varphi_n^*A}^\delta\iota_{\frac{\partial}{\partial x^\mu}} \varphi_n^*F_A\right) \right\|_{L^2(\RR^4,\delta)}
\\
&= \left\|\nabla_{\varphi_n^*A}^\delta\iota_{\frac{\partial}{\partial x^\mu}} \varphi_n^*F_A \right\|_{L^2(\RR^4,\delta)}
\\
&=
\left\|\nabla_A^\delta\iota_{\varphi_{n*}\frac{\partial}{\partial x^\mu}}F_A \right\|_{L^2(S^4,g)}
\\
&\leq C\left( \left\|\nabla_A^g\iota_{\varphi_{n*}\frac{\partial}{\partial x^\mu}}F_A \right\|_{L^2(S^4,g)}
+ \|\iota_{\varphi_{n*}\frac{\partial}{\partial x^\mu}}F_A\|_{L^2(S^4,g)}\right),
\\
&\leq C\left\|\iota_{\varphi_{n*}\frac{\partial}{\partial x^\mu}}F_A\right\|_{W_{A,g}^{1,2}(S^4)},
\end{align*}
where $C \in [1,\infty)$ is a universal constant. On the other hand,
\begin{align*}
\left\|\nabla_{\tau_z^*\varphi_n^*A}^\delta\frac{\partial}{\partial z^\mu}\tau_z^*\varphi_n^*A\right\|_{L^2(\RR^4,\delta)}
&=
\left\|\nabla_{\tilde\tau_z^*A}^\delta\frac{\partial}{\partial z^\mu}\tilde\tau_z^*A\right\|_{L^2(S^4,g)}
\\
&\geq \left\|\nabla_{\tilde\tau_z^*A}^g\frac{\partial}{\partial z^\mu}\tilde\tau_z^*A\right\|_{L^2(S^4,g)}
- C\left\|\frac{\partial}{\partial z^\mu}\tilde\tau_z^*A\right\|_{L^2(S^4,g)},
\end{align*}
and so
\begin{align*}
\left\|\nabla_{\tilde\tau_z^*A}^g\frac{\partial}{\partial z^\mu}\tilde\tau_z^*A\right\|_{L^2(S^4,g)}
&\leq \left\|\nabla_{\tau_z^*\varphi_n^*A}^\delta\frac{\partial}{\partial z^\mu}\tau_z^*\varphi_n^*A\right\|_{L^2(\RR^4,\delta)}
+ C\left\|\frac{\partial}{\partial z^\mu}\tilde\tau_z^*A\right\|_{L^4(S^4,g)}
\\
&\leq C\left\|\iota_{\varphi_{n*}\frac{\partial}{\partial x^\mu}}F_A\right\|_{W_{A,g}^{1,2}(S^4)}
+ \left\|\iota_{\varphi_{n*}\frac{\partial}{\partial x^\mu}}F_A\right\|_{L^4(S^4,g)},
\end{align*}
and applying the Sobolev Embedding \eqref{eq:Sobolev_embedding_manifold_bounded_geometry} in the form $W_g^{1,2}(S^4)\hookrightarrow L^4(S^4,g)$ and Kato Inequality \eqref{eq:FU_6-20_first-order_Kato_inequality},
\begin{equation}
\label{eq:L2norm_nabla_bound_derivative_connection_wrt_mass_centers}
\left\|\nabla_{\tilde\tau_z^*A}^g\frac{\partial}{\partial z^\mu}\tilde\tau_z^*A\right\|_{L^2(S^4,g)}
\leq
C\left\|\iota_{\varphi_{n*}\frac{\partial}{\partial x^\mu}}F_A\right\|_{W_{A,g}^{1,2}(S^4)}.
\end{equation}
Finally,
\begin{align*}
\left\|\nabla_{\delta_\lambda^*\varphi_n^*A}^\delta\frac{\partial}{\partial\lambda}\delta_\lambda^*\varphi_n^*A \right\|_{L^2(\RR^4,\delta)}
&=
\frac{1}{\lambda}
\left\|\nabla_{\delta_\lambda^*\varphi_n^*A}^\delta\delta_\lambda^*\left(\iota_{\frac{\partial}{\partial r}} \varphi_n^*F_A\right) \right\|_{L^2(\RR^4,\delta)}
\\
&=
\frac{1}{\lambda}
\left\|\delta_\lambda^*\left(\nabla_{\varphi_n^*A}^\delta\left(\iota_{\frac{\partial}{\partial r}} \varphi_n^*F_A\right)\right) \right\|_{L^2(\RR^4,\delta)}
\\
&=
\frac{1}{\lambda}
\left\|\nabla_{\varphi_n^*A}^\delta\left(\iota_{\frac{\partial}{\partial r}} \varphi_n^*F_A\right) \right\|_{L^2(\RR^4,\delta)}
\\
&\leq \frac{C}{\lambda}\left\|\iota_{\varphi_{n*}\frac{\partial}{\partial r}}F_A\right\|_{W_{A,g}^{1,2}(S^4)}.
\end{align*}
On the other hand,
\begin{align*}
\left\|\nabla_{\delta_\lambda^*\varphi_n^*A}^\delta\frac{\partial}{\partial\lambda}\delta_\lambda^*\varphi_n^*A \right\|_{L^2(\RR^4,\delta)}
&=
\left\|\nabla_{\tilde\delta_\lambda^*A}^\delta\frac{\partial}{\partial\lambda}\tilde\delta_\lambda^*A \right\|_{L^2(S^4,g)}
\\
&\geq \left\|\nabla_{\tilde\delta_\lambda^*A}^g\frac{\partial}{\partial\lambda}\tilde\delta_\lambda^*A \right\|_{L^2(S^4,g)}
- C\left\|\frac{\partial}{\partial\lambda}\tilde\delta_\lambda^*A \right\|_{L^2(S^4,g)},
\end{align*}
and so
\begin{align*}
\left\|\nabla_{\tilde\delta_\lambda^*A}^g\frac{\partial}{\partial\lambda}\tilde\delta_\lambda^*A \right\|_{L^2(S^4,g)}
&\leq
\left\|\nabla_{\delta_\lambda^*\varphi_n^*A}^\delta\frac{\partial}{\partial\lambda}\delta_\lambda^*\varphi_n^*A \right\|_{L^2(\RR^4,\delta)}
+ C\left\|\frac{\partial}{\partial\lambda}\tilde\delta_\lambda^*A \right\|_{L^4(S^4,g)}
\\
&\leq \frac{C}{\lambda}\left\|\iota_{\varphi_{n*}\frac{\partial}{\partial r}}F_A\right\|_{W_{A,g}^{1,2}(S^4)}
+ \frac{1}{\lambda}\left\|\iota_{\varphi_{n*}\frac{\partial}{\partial r}}F_A\right\|_{L^4(S^4,g)},
\end{align*}
and hence,
\begin{equation}
\label{eq:L2norm_nabla_bound_derivative_connection_wrt_scales}
\left\|\nabla_{\tilde\delta_\lambda^*A}^g\frac{\partial}{\partial\lambda}\tilde\delta_\lambda^*A \right\|_{L^2(S^4,g)}
\leq
\frac{C}{\lambda}\left\|\iota_{\varphi_{n*}\frac{\partial}{\partial r}}F_A\right\|_{W_{A,g}^{1,2}(S^4)}.
\end{equation}
This completes the $W^{1,2}(S^4)$ bounds for the tangent vectors to the family $h_{z,\lambda}^*A$.

Now let $\lambda_t := (1-t)\lambda_0 + t\lambda_1$ for $\lambda_0, \lambda_1 \in \RR_+$, and observe that
\[
\|\tilde\delta_{\lambda_1}^*A - \tilde\delta_{\lambda_0}^*A\|_{L^4(S^4,g)}
\leq
\int_0^1 \left\|\frac{\partial}{\partial t}\tilde\delta_{\lambda_t}^*A\right\|_{L^4(S^4,g)}\,dt
\]
and so
\begin{equation}
\label{eq:L4distance_connections_wrt_scales}
\|\tilde\delta_{\lambda_1}^*A - \tilde\delta_{\lambda_0}^*A\|_{L^4(S^4,g)}
\leq
|\lambda_1 - \lambda_0|\int_0^1 \left\|\frac{\partial}{\partial \lambda_t}\tilde\delta_{\lambda_t}^*A\right\|_{L^4(S^4,g)}\,dt.
\end{equation}
Similarly, for $z_t = (1-t)z_0 + tz_1 = z_0 + t(z_1-z_0)$ and $z_0, z_1 \in \RR^4$, we observe that
\[
\frac{\partial}{\partial t}\tilde\tau_{z_t}^*A
=
\sum_{\mu=1}^4(z_1-z_0)^\mu\frac{\partial}{\partial z^\mu}\tilde\tau_{z_t}^*A,
\]
and thus
\begin{align*}
\|\tilde\tau_{z_1}^*A - \tilde\tau_{z_0}^*A\|_{L^4(S^4,g)}
&\leq
\int_0^1 \left\|\frac{\partial}{\partial t}\tilde\tau_{z_t}^*A\right\|_{L^4(S^4,g)}\,dt
\\
&\leq \int_0^1
\left\|\sum_{\mu=1}^4(z_1-z_0)^\mu\frac{\partial}{\partial z^\mu}\tilde\tau_{z_t}^*A\right\|_{L^4(S^4,g)}\,dt,
\end{align*}
and therefore,
\begin{equation}
\label{eq:L4distance_connections_wrt_mass_centers}
\|\tilde\tau_{z_1}^*A - \tilde\tau_{z_0}^*A\|_{L^4(S^4,g)}
\leq
|z_1 - z_0| \sum_{\mu=1}^4 \int_0^1
\left\|\frac{\partial}{\partial z^\mu}\tilde\tau_{z_t}^*A\right\|_{L^4(S^4,g)}\,dt.
\end{equation}
For the covariant derivatives, we have
\begin{align*}
{}&\|\nabla_{\tilde\tau_{z_0}^*A}^g(\tilde\tau_{z_1}^*A - \tilde\tau_{z_0}^*A)\|_{L^2(S^4,g)}
\\
&\quad\leq
\int_0^1 \left\|\nabla_{\tilde\tau_{z_0}^*A}^g \frac{\partial}{\partial t}\tilde\tau_{z_t}^*A\right\|_{L^2(S^4,g)}\,dt
\\
&\quad\leq
\int_0^1 \left\|\nabla_{\tilde\tau_{z_t}^*A}^g \frac{\partial}{\partial t}\tilde\tau_{z_t}^*A\right\|_{L^2(S^4,g)}\,dt
+ \int_0^1 \left\|(\tilde\tau_{z_t}^*A - \tilde\tau_{z_0}^*A)\wedge\frac{\partial}{\partial t}\tilde\tau_{z_t}^*A\right\|_{L^2(S^4,g)}\,dt
\\
&\quad\leq
\int_0^1 \left\|\sum_{\mu=1}^4(z_1-z_0)^\mu \nabla_{\tilde\tau_{z_t}^*A}^g \frac{\partial}{\partial z^\mu}\tilde\tau_{z_t}^*A\right\|_{L^2(S^4,g)}\,dt
\\
&\qquad + 2\max_{t\in[0,1]}\|\tilde\tau_{z_t}^*A - \tilde\tau_{z_0}^*A \|_{L^4(S^4,g)}
\int_0^1 \left\|\frac{\partial}{\partial t}\tilde\tau_{z_t}^*A\right\|_{L^4(S^4,g)}\,dt,
\end{align*}
and therefore,
\begin{multline}
\label{eq:L2nabla_distance_connections_wrt_mass_centers}
\|\nabla_{\tilde\tau_{z_0}^*A}^g(\tilde\tau_{z_1}^*A - \tilde\tau_{z_0}^*A)\|_{L^2(S^4,g)}
\\
\leq
|z_1 - z_0| \sum_{\mu=1}^4 \max_{t\in[0,1]}
\left\|\nabla_{\tilde\tau_{z_t}^*A}^g\frac{\partial}{\partial z^\mu}\tilde\tau_{z_t}^*A\right\|_{L^2(S^4,g)}
\\
+ 2|z_1 - z_0|\max_{t\in[0,1]}\|\tilde\tau_{z_t}^*A - \tilde\tau_{z_0}^*A \|_{L^4(S^4,g)}
\sum_{\mu=1}^4 \left\|\frac{\partial}{\partial z^\mu}\tilde\tau_{z_t}^*A\right\|_{L^4(S^4,g)}.
\end{multline}
A similar calculation yields,
\begin{multline}
\label{eq:L2nabla_distance_connections_wrt_scales}
\|\nabla_{\tilde\delta_{\lambda_0}^*A}^g(\tilde\delta_{\lambda_1}^*A - \tilde\delta_{\lambda_0}^*A)\|_{L^2(S^4,g)}
\\
\leq
|\lambda_1 - \lambda_0| \max_{t\in[0,1]}
\left\|\nabla_{\tilde\delta_{\lambda_t}^*A}^g\frac{\partial}{\partial\lambda_t}\tilde\delta_{\lambda_t}^*A
\right\|_{L^2(S^4,g)}
\\
+ 2|\lambda_1 - \lambda_0|\max_{t\in[0,1]}\|\tilde\delta_{\lambda_t}^*A - \tilde\delta_{\lambda_0}^*A \|_{L^4(S^4,g)}
\left\|\frac{\partial}{\partial\lambda_t}\tilde\delta_{\lambda_t}^*A\right\|_{L^4(S^4,g)}.
\end{multline}
We thus obtain $W_{A,g}^{1,2}(S^4;\Lambda^1\otimes\ad P)$ continuity with respect to $(z,\lambda)\in\RR^4\times\RR_+$ of the family of connections $\tilde h_{z,\lambda}^*A$ on the principal $G$-bundle $P\to S^4$ when $A$ is a fixed $W^{2,2}$ connection on $P$ and this completes the proof of Proposition \ref{prop:W12_continuity_family_connections_wrt_mass_centers_and_scales}.
\end{proof}

\section{Bubble-tree convergence for a sequence of Yang-Mills connections with a uniform $L^2$ bound on curvature}
\label{sec:Bubble-tree_compactness_Yang-Mills_and_anti-self-dual_connections}
In Section \ref{subsec:Bubble-tree_convergence_Yang-Mills_connections}, we develop a bubble-tree convergence result in the same setting as Section \ref{subsec:Uhlenbeck_compactness_Yang-Mills_connections_L2_bounds_curvature} (namely, Theorem \ref{thm:Kozono_Maeda_Naito_5-4_YM_sequence} and Remark \ref{rmk:Bubble-tree_convergence_Yang-Mills_connections}). Section \ref{subsec:Bubble_tree_compactification_moduli_space_anti-self-dual_connections} contains a review and some extensions and modifications from our article \cite{FeehanGeometry} of the bubble-tree topology and bubble-tree compactification for the moduli space of anti-self-dual connections, $M(P,g)$. In Section \ref{subsec:Bubble_tree_compactification_moduli_space_Yang-Mills_connections}, we provide a definition of the bubble-tree topology and establish bubble-tree compactness results for the moduli space of Yang-Mills connection, $\Crit(P,g,\sC)$, with $L^2$ energies belonging to a compact range $\sC \Subset [0,\infty)$.

\subsection{Bubble-tree convergence for a sequence of Yang-Mills connections with a uniform $L^2$ bound on curvature}
\label{subsec:Bubble-tree_convergence_Yang-Mills_connections}
The bubble-tree convergence possibility for sequences of harmonic maps from $S^2$ was first noted by Sacks and Uhlenbeck \cite[p. 3]{Sacks_Uhlenbeck_1981}, even before further development by Taubes in the context of sequences of Yang-Mills connections in dimension four \cite[Section 5]{TauFrame}, Parker and Wolfson in the context of sequences of harmonic maps from Riemann surfaces \cite{ParkerHarmonic, ParkerWolfson}, and the author in the context of sequences of anti-self-dual connections \cite[Section 4]{FeehanGeometry} and Yang-Mills gradient flow \cite[Section 29]{Feehan_yang_mills_gradient_flow}. In \cite[Section 4]{FeehanGeometry}, we allowed the Riemannian metric, $g$, on $X$ to vary within its conformal equivalence class, $[g]$ (by analogy with the treatment by Parker and Wolfson \cite{ParkerHarmonic, ParkerWolfson}) but we could equally well have worked with a fixed Riemannian metric, $g$ on $X$, as in Taubes \cite[Section 5]{TauFrame} and \cite[Section 29]{Feehan_yang_mills_gradient_flow}. We shall develop the fixed Riemannian metric scheme in this section while in Sections \ref{sec:Bubble_trees_and_Riemannian_metrics_connected_sums} and \ref{sec:Global W12_metrics_bubble-tree_neighborhoods}, we discuss the varying Riemannian metric scheme.

While the basic concept underlying bubble-tree compactifications is well understood, the detail required for a complete and thorough description can become quite involved in the setting of Yang-Mills or anti-self-dual connections, as one see from \cite{FeehanGeometry, TauFrame}. Iterated rescaling is required to ensure that the rescaled connections over $S^4$ converge to smooth anti-self-dual connections over $S^4$ with no energy loss at point singularities. We shall describe the effect of one level of rescaling carefully. The description of second or higher levels of rescaling does not involve any new technical difficulties, but is cumbersome. We refer the reader to Parker \cite{ParkerHarmonic} for a very informative description in the context of harmonic maps, with many useful illustrations and careful analysis of the `neck' regions.

The proof of Theorem \ref{thm:Kozono_Maeda_Naito_5-4_YM_sequence} below can be obtained \mutatis from that of \cite[Theorem 29.6]{Feehan_yang_mills_gradient_flow} (Yang-Mills gradient flow), by now drawing on Theorem \ref{thm:Sedlacek_4-3_Yang-Mills} (sequence of Yang-Mills connections). Alternatively, one can observe that the ideas involved in the proof of Theorem \ref{thm:Kozono_Maeda_Naito_5-4_YM_sequence} originate with those of Taubes in his proofs of \cite[Proposition 4.4]{TauPath}, \cite[Propositions 5.1 and 5.3]{TauFrame} and which were further developed by the author in our proof of \cite[Theorem 4.15]{FeehanGeometry}, which asserts the existence of a bubble-tree convergent subsequence, given a sequence of anti-self-dual connections on $P$.

To obtain the full bubble-tree convergence for a sequence of Yang-Mills connections with a uniform $L^2$ bound on curvature, one must repeatedly apply Theorem \ref{thm:Kozono_Maeda_Naito_5-4_YM_sequence} to each copy of $S^4$ until there no more points of curvature concentration where sequences of connections remain to be rescaled, just as in the proof of our \cite[Theorem 4.15]{FeehanGeometry}. See Remark \ref{rmk:Bubble-tree_convergence_Yang-Mills_connections}.

One may view Theorem \ref{thm:Kozono_Maeda_Naito_5-4_YM_sequence} below as a sequential analogue of \cite[Theorem 29.6]{Feehan_yang_mills_gradient_flow}, a single-level bubble-tree convergence result for Yang-Mills gradient flow.

\begin{thm}[Uhlenbeck convergence over $S^4$ for a rescaled sequence of Yang-Mills connections]
\label{thm:Kozono_Maeda_Naito_5-4_YM_sequence}
Assume the hypotheses of Theorem \ref{thm:Sedlacek_4-3_Yang-Mills} and, in addition, that $l \geq 1$ with $\bx := \{x_i\in X: 1 \leq i \leq l\} \subset X$. Let $g$ denote the smooth Riemannian metric on $X$ and $\rho = \rho(g, R, \bx, [A^\infty]) \in (0, 1]$ and $R = R(\sE_1,\ldots,\sE_l) \in [1, \infty)$ be constants obeying the conditions \eqref{eq:Kozono_Maeda_Naito_Theorem_5-4_fixed_small_singular_point_radial_parameter} and \eqref{eq:Kozono_Maeda_Naito_Theorem_5-4_proof_choice_large_Euclidean_4ball_radius_R} in the sequel. For $1 \leq i \leq l$, choose oriented $g$-orthonormal frames, $f_i$, for the tangent spaces, $(TX)_{x_i}$, and points $p_i \in P_{x_i}$. If $i \in \{1, \ldots, l\}$ then, after passing to a subsequence $\{m'\} \subset \{m\} \subset \NN$ and relabeling as $\{m\}$, there exist
\begin{itemize}
\item A geodesic normal coordinate chart, $\varphi_i: \RR^4 \supset B_{4\rho}(0) \cong B_{4\rho}(x_i) \subset X$ with $\varphi_i(0) = x_i$, defined by the frame, $f_i$;

\item A sequence of relative mass centers, $\{x_{i,m}\}_{m \in \NN} \subset X$, such that $x_{i,m} \to x_i$ as $m \to \infty$;

\item A sequence of relative scales, $\{\lambda_{i,m}\}_{m \in \NN} \subset (0, \rho)$, such that $\lambda_{i,m} \searrow 0$ as $m \to \infty$;

\item A sequence of geodesic normal coordinate charts, $\varphi_{i,m}: \RR^4 \supset B_{2\rho}(0) \cong B_{2\rho}(x_{i,m}) \subset X$ with $\varphi_{i,m}(0) = x_{i,m}$, defined by the $g$-orthonormal frame $f_{i,m}$ for $(TX)_{x_{i,m}}$ obtained by parallel translation of $f_i$ along the geodesic joining $x_i$ to $x_{i,m}$;

\item A sequence of local sections, $\sigma_{i,m}: B_{2\rho}(x_{i,m}) \to P\restriction B_{2\rho}(x_{i,m})$, constructed via parallel translation with respect to the connections $A^m$ on $P$, of the point $p_i \in P_{x_i}$ to a point $p_{i,m} \in P_{x_{i,m}}$ along the geodesic joining $x_i$ to $x_{i,m}$, followed by parallel translation along radial geodesics emanating from $x_{i,m} \in X$;

\item Sequences of $W^{k+1,p}$ isomorphisms of principal $G$-bundles, $u_{i,m}: P \restriction S^4 \less \{\infty\} \cong G \times \RR^4$.
\end{itemize}
For each $i \in \{1, \ldots, l\}$, define a sequence of rescaled connections, $\{A_i^m\}_{m \in \NN}$, on a sequence of product $G$-bundles, $B_{\rho/\lambda_{i,m}}(0)\times G \subset \RR^4\times G$, by setting\footnote{Compare \cite[p. 516]{FeehanGeometry}; if $g$ is flat, then $\varphi_{i,m}^*\sigma_{i,m}^*A^m(\lambda_{i,m} y) = \varphi_i^*\sigma_{i,m}^*A^m(x_{i,m}+\lambda_{i,m} y)$ and $\varphi_{i,m}^*g(\lambda_{i,m} y) = \varphi_i^*g(x_{i,m}+\lambda_{i,m} y)$.}
\begin{equation}
\label{eq:Schlatter_25_Yang-Mills_theoremstatement}
\begin{aligned}
A_i^m &:= \Theta + a_{i,m} \quad\text{and}
\\
a_{i,m}(y) &:= \varphi_{i,m}^*\sigma_{i,m}^*A^m(\lambda_{i,m} y),
\quad\forall\, y \in B_{\rho/\lambda_{i,m}}(0), \quad m \in \NN,
\end{aligned}
\end{equation}
where $\Theta$ denotes the product connection on $\RR^4\times G$, and define a sequence of smooth Riemannian metrics, $\{g_{i,m}\}_{m\in\NN}$ on $B_{\rho/\lambda_{i,m}}(0)$, by setting
\begin{equation}
\label{eq:Feehan_3-33_Euclidean_space}
g_{i,m}(y) := \lambda_{i,m}^{-2}\varphi_{i,m}^*g(\lambda_{i,m} y), \quad\forall\, y \in B_{\rho/\lambda_{i,m}}(0).
\end{equation}
Then the following hold for each $i \in \{1, \ldots, l\}$:
\begin{enumerate}
\item
\label{item:Theorem_Kozono_Maeda_Naito_5-4_YM_sequence_Cinfinity_loc_convergence_gm}
If $\delta$ denotes the standard Euclidean metric on $\RR^4$, then $g_{i,m} \to \delta$ in $C_\loc^\infty(\RR^4)$ as $m \to \infty$;

\item
\label{item:Theorem_Kozono_Maeda_Naito_5-4_YM_sequence_Am_obeys_Yang-Mills_equation_Euclidean_space}
The sequence of connections, $A_i^m$, obeys the Yang-Mills equation over $B_{\rho/\lambda_{i,m}}(0) \subset \RR^4$ with respect to the metric $g_{i,m}$;

\item
\label{item:Theorem_Kozono_Maeda_Naito_5-4_YM_sequence_H2loc_weak_and_W1ploc_strong_convergence_Am}
There is a smooth Yang-Mills connection, $A_i^\infty$, on a principal $G$-bundle, $P_i$, over $S^4$ with its standard round metric of radius one, an integer $l_i \geq 0$ and, if $l_i \geq 1$, a set of points, $\bx_i := \{y_{i,j}\}_{j=1}^{l_i} \subset \RR^4$ such that
\begin{equation}
\label{eq:Kozono_Maeda_Naito_5-4_YM_sequence_PhimAlm_converges_to_Al_infinity_flat_strongly_in_Wkploc_on_X_away_from_bubble_points}
u_{i,m}(A_i^m) \to A_i^\infty \quad\text{strongly in } W_{\Theta,\loc}^{k,p}(\RR^4\less \bx_i; \Lambda^1\otimes\fg),
\quad\text{as } m \to \infty.
\end{equation}
The ideal limit, $(A_i^\infty, \bx_i)$, is non-trivial in the sense that either $l_i \geq 1$ or, if $l_i = 0$, then $A_i^\infty$ is not flat (that is, gauge-equivalent to the product connection, $\Theta$).

\item
\label{item:Theorem_Kozono_Maeda_Naito_5-4_YM_sequence_weak_convergence_FAm_measures}
If $l_i \geq 1$, there are constants, $0 < \sE_{i,j} \leq \sE_i$ for $1 \leq j \leq l_i$,
given by
\begin{equation}
\label{eq:Limit_r_to_zero_m_to_infinity_energy_over_ball_equals_Elj_Euclidean_space}
\sE_{i,j} = \lim_{r\to 0} \lim_{m\to \infty} \int_{B_r(y_{i,j})} |F_{A_i^m}|^2 \, d^4y, \quad 1 \leq j \leq l_i,
\end{equation}
such that
\[
|F_{A_i^m}|^2\,d\vol
\rightharpoonup
|F_{A_i^\infty}|^2\,d\vol + \sum_{j=1}^{l_i} \sE_{i,j}\, \delta_{y_{i,j}}
\quad\text{in } (C_0(\RR^4;\RR))' \text{ as } m \to \infty.
\]

\item
\label{item:Kozono_Maeda_Naito_5-4_YM_sequence_centered_limit}
The measure, $\mu_i := |F_{A_i^\infty}|^2\,d\vol + \sum_{j=1}^{l_i} \sE_{i,j}\, \delta_{y_{i,j}}$, is centered\footnote{This centering condition reflects the simplest choice and one could alternatively redefine the sequence, $\{A_i^m\}_{m\in\NN}$, in such a way that the connection, $A_i^\infty$, is centered, as in \cite[Lemma 4.7]{FeehanGeometry}.}
in the sense of Definition \ref{defn:Center_scale_positive_Borel_measure} and, if $\bx_i = \emptyset$, then the connection, $A_i^\infty$, is centered in the sense of Definition \ref{defn:Mass_center_scale_connection}.

\item
\label{item:Kozono_Maeda_Naito_5-4_YM_sequence_limit_m_to_infinity_energy_over_increasing_balls_in_Euclidean_space_equals_E}
The positive constant, $\sE_i$, in Theorem \ref{thm:Sedlacek_4-3_Yang-Mills} can be computed by the alternative formula,
\begin{equation}
\label{eq:Limit_m_to_infinity_energy_over_increasing_balls_in_Euclidean_space_equals_E}
\sE_i = \lim_{\zeta \to 0} \lim_{m \to \infty} \int_{B_{\zeta/\lambda_{i,m}}(0)} |F_{A_i^m}(y)|^2 \, d^4y.
\end{equation}

\item
\label{item:Kozono_Maeda_Naito_5-4_YM_sequence_limit_over_sphere_asd}
If in addition the sequence $\{A^m\}_{m\in\NN}$ is absolutely minimizing in the sense of \eqref{eq:Absolutely_minimizing_sequence_connections}, then $A_i^\infty$ is an anti-self-dual connection with respect to the standard round metric of radius one on $S^4$.
\end{enumerate}
\end{thm}

\begin{rmk}[On the choice of points in the fibers of $P$ and local sections of $P$]
\label{rmk:Redundancy_fiber_points_sections_principal_bundle}
While Theorem \ref{thm:Kozono_Maeda_Naito_5-4_YM_sequence} is most easily stated and proved with the aid of local sections of $P$ near the points of curvature concentration defined by a choice of points $p_i \in P|_{x_i}$ for $1 \leq i \leq l$, that choice is not necessary for a description of bubble-tree convergence.
\end{rmk}

\begin{rmk}[Bubble-tree convergence for a sequence of Yang-Mills connections with a uniform $L^2$ bound on curvature]
\label{rmk:Bubble-tree_convergence_Yang-Mills_connections}
As we noted earlier, the full bubble-tree convergence is obtained by repeatedly applying Theorem \ref{thm:Kozono_Maeda_Naito_5-4_YM_sequence} (with $X$ replaced by $S^4$) until all curvature concentration singularities have been resolved, in the sense that the sequences of connections on copies of $S^4$ at terminal nodes of the tree converge to smooth connections over $S^4$ without further curvature concentration, as described in \cite{FeehanGeometry, TauFrame}. The required number of applications of Theorem \ref{thm:Kozono_Maeda_Naito_5-4_YM_sequence} is bounded above by a constant that depends at most on $\sup_{m\in\NN}\|F_{A^m}\|_{L^2(X)}$, the Pontrjagin numbers of $P$, and the Lie group, $G$. For the sake of exposition and notational simplicity, however, we shall give most detail in the case where the application of Theorem \ref{thm:Kozono_Maeda_Naito_5-4_YM_sequence} gives $\bx_i = \emptyset$ for $1 \leq i \leq l$ and thus one level of rescaling suffices to resolve all curvature concentration singularities. The general case follows by an induction and presents no new difficulties.
\end{rmk}

\begin{proof}
We shall just briefly indicate the sources of the proofs of the conclusions of Theorem \ref{thm:Kozono_Maeda_Naito_5-4_YM_sequence}, whose statement is modeled on that of \cite[Theorem 29.6]{Feehan_yang_mills_gradient_flow} for Yang-Mills gradient flow, rather than a sequence of Yang-Mills connections and it is the latter hypothesis that is responsible for the stronger notion of convergence in Item \eqref{item:Theorem_Kozono_Maeda_Naito_5-4_YM_sequence_H2loc_weak_and_W1ploc_strong_convergence_Am}, as in our statement and proof of Theorem \ref{thm:Sedlacek_4-3_Yang-Mills}.

The statement and proof of Theorem \ref{thm:Kozono_Maeda_Naito_5-4_YM_sequence} are implicit in \cite[Section 5]{TauFrame} but, in the form presented here, essentially translate those of \cite[Theorem 4.15]{FeehanGeometry} to the more general setting of Yang-Mills rather than anti-self-dual connections. It will be important in the sequel to have precise definitions of the sequences of mass centers and scales, $\{x_{i,m}\}_{m\in\NN}$ and $\{\lambda_{l.m}\}_{m\in\NN}$ for each $i\in \{1,\ldots,l\}$, and whose existence is asserted by Theorem \ref{thm:Kozono_Maeda_Naito_5-4_YM_sequence}. For this purpose, we recall our definitions from \cite[Theorem 29.6]{Feehan_yang_mills_gradient_flow} (compare \cite[Section 4.2]{FeehanGeometry}).

Item \eqref{item:Theorem_Kozono_Maeda_Naito_5-4_YM_sequence_Cinfinity_loc_convergence_gm} is a consequence (compare \cite[Lemma 3.12]{FeehanGeometry}) of the definition \eqref{eq:Feehan_3-33_Euclidean_space} and more refined $C^0$ and $C^1$ convergence rate estimates follow from standard properties of geodesic normal coordinates \cite[Definition 1.24, Proposition 1.25, and Corollary 1.32]{Aubin}.

We recall that a singular point $x_i \in X$ for the sequence $\{A^m\}_{m\in\NN}$ may be characterized by
\begin{equation}
\label{eq:Kozono_Maeda_Naito_5-1}
\lim_{r\to 0}\limsup_{m\to\infty}\int_{B_r(x_i)} |F_{A^m}|^2 \, d\vol_g \geq \eps_0,
\end{equation}
where $\eps_0 \in (0,1]$ is a \emph{universal} positive constant. (Our proof of Theorem \ref{thm:Kozono_Maeda_Naito_5-4_YM_sequence} and proofs of similar results show that this constant is determined by the minimum energy of a non-flat connection on a non-trivial principal $G$-bundle over $S^4$ with its standard round metric of radius one, namely that of a non-flat anti-self-dual or self-dual connection.) Of course, for each $i \in \{1,\ldots,l\}$, since
$$
\int_{B_r(x_i)} |F_{A^m}|^2 \, d\vol_g \leq \int_X |F_{A^m}|^2 \, d\vol_g \leq \sup_{m\in\NN}\int_X |F_{A^m}|^2 \, d\vol_g < \infty, \quad\forall\, m \in \NN,
$$
the limit in \eqref{eq:Kozono_Maeda_Naito_5-1} is finite. By virtue of \eqref{eq:Kozono_Maeda_Naito_5-1}, there is a subsequence, $\{m'\} \subset \{m\}$, such that (compare \cite[Equation (4.3)]{FeehanGeometry} and \cite[Equation (29.37)]{Feehan_yang_mills_gradient_flow}) after relabeling,
\begin{equation}
\label{eq:Feehan_4-3}
\sE_i :=  \lim_{r\to 0} \lim_{m\to\infty} \int_{B_r(x_i)} |F_{A^m}|^2 \, d\vol_g \geq \eps_0, \quad 1 \leq i \leq l.
\end{equation}
For a sufficiently large positive constant,
\begin{equation}
\label{eq:Kozono_Maeda_Naito_Theorem_5-4_proof_large_Euclidean_4ball_radius_R_dependencies}
R = R(\sE_1, \ldots, \sE_l) \in [1,\infty)
\end{equation}
to be determined in the sequel (see \eqref{eq:Kozono_Maeda_Naito_Theorem_5-4_proof_choice_large_Euclidean_4ball_radius_R} and note that $\eps_0$ is a universal constant), we choose $\rho = \rho(g, R, \bx, [A^\infty]) \in (0, 1]$ small enough that (compare \cite[Equation (4.4)]{FeehanGeometry} and \cite[Equation (29.39)]{Feehan_yang_mills_gradient_flow})
\begin{equation}
\label{eq:Kozono_Maeda_Naito_Theorem_5-4_fixed_small_singular_point_radial_parameter}
8R\rho < 1 \wedge \Inj(X,g) \wedge \min_{i\neq j}\dist_g(x_i,x_j),
\end{equation}
and (compare \cite[Equation (29.40)]{Feehan_yang_mills_gradient_flow})
\begin{equation}
\label{eq:Kozono_Maeda_Naito_Theorem_5-4_proof_small_background_connection}
\int_{B_{8R\rho}(x_i)} |F_{A^\infty}|^2 \, d\vol_g \leq \frac{\eps_0}{16}.
\end{equation}
We now define sequences of relative mass centers and scales by adapting our construction in \cite[Section 4.2]{FeehanGeometry}. Theorem \ref{thm:Sedlacek_4-3_Yang-Mills} implies that, for each $i \in \{1,\ldots,l\}$,
\begin{equation}
\label{eq:Kozono_Maeda_Naito_Theorem_5-3_2_punctured_ball}
F_{A^m} \to F_{A^\infty} \quad\text{strongly in } L^2_\loc(B_\rho(x_i)\less\{x_i\}), \quad\text{as } m \to \infty,
\end{equation}
and thus we also have (compare \cite[Equation (4.5)]{FeehanGeometry} and \cite[Equation (29.42)]{Feehan_yang_mills_gradient_flow})
\begin{multline}
\label{eq:Feehan_4-5}
\sE_i = \lim_{m\to\infty} \int_{B_\rho(x_i)} \left(|F_{A^m}|^2 - |F_{A^\infty}|^2\right)\, d\vol_g
\\
= \lim_{r\to 0} \lim_{m\to\infty} \int_{B_r(x_i)} \left(|F_{A^m}|^2 - |F_{A^\infty}|^2\right) \, d\vol_g
\\
= \lim_{r\to 0} \lim_{m\to\infty} \int_{B_r(x_i)} |F_{A^m}|^2 \, d\vol_g.
\end{multline}
By virtue of \eqref{eq:Feehan_4-5} we may also assume, by restricting to large enough $m$ and then relabeling the sequence if needed, that (compare \cite[Equations (4.9) and (4.10)]{FeehanGeometry} and \cite[Equation (29.43)]{Feehan_yang_mills_gradient_flow})
\begin{equation}
\label{eq:Feehan_4-9_and_4-10_difference_measure}
\frac{7}{8}\sE_i \leq \int_{B_\rho(x_i)} \left||F_{A^m}|^2 - |F_{A^\infty}|^2\right|\, d\vol_g \leq \frac{9}{8}\sE_i, \quad\forall\, m \in \NN.
\end{equation}
For each $i \in \{1,\ldots, l\}$, the data $(\{[A^m]\}_{m\in\NN}, [A^\infty], g, x_i, \rho)$ determine a sequence of \emph{relative local mass centers}, $\{x_{i,m}\}_{m\in\NN} \subset B_{9\rho/8}(x_i)$, with coordinates $x_{i,m}^\nu = (\varphi_i^{-1}(x_{i,m}))^\nu$ (compare \cite[Equation (4.6)]{FeehanGeometry} and \cite[Equation (29.44)]{Feehan_yang_mills_gradient_flow})
\begin{equation}
\label{eq:Feehan_4-6}
x_{i,m}^\nu := \frac{1}{\sE_i}\int_{B_\rho(x_i)} x^\nu\left(|F_{A^m}|^2 - |F_{A^\infty}|^2\right)\, d\vol_g, \quad 1\leq \nu \leq 4, \quad\forall\, m \in \NN,
\end{equation}
and \emph{relative local scales}, $\{\lambda_{i,m}\}_{m \in \NN} \subset (0, 3\rho/(2\sqrt{2}))$ (compare \cite[Equation (4.7)]{FeehanGeometry} and \cite[Equation (29.45)]{Feehan_yang_mills_gradient_flow}),
\begin{equation}
\label{eq:Feehan_4-7}
\lambda_{i,m}^2 := \frac{1}{\sE_i}\int_{B_\rho(x_i)} \dist_g(\cdot, x_{i,m})^2 \left||F_{A^m}|^2 - |F_{A^\infty}|^2\right|\, d\vol_g, \quad\forall\, m \in \NN.
\end{equation}
The upper bound, $3\rho/(2\sqrt{2})$, for the scales, $\lambda_{i,m}$, is due to the upper bound $9\sE_i/8$ in \eqref{eq:Feehan_4-9_and_4-10_difference_measure}.

Equation \eqref{eq:Feehan_4-7} leads to the following Chebychev-type inequality (compare \cite[Equation (4.8)]{FeehanGeometry} and \cite[Equation (29.46)]{Feehan_yang_mills_gradient_flow}) for all $R \geq 1$,
\begin{equation}
\label{eq:Feehan_4-8}
\int_{B_\rho(x_i) \less B_{\lambda_{i,m} R}(x_{i,m})} \left||F_{A^m}|^2 - |F_{A^\infty}|^2\right|\, d\vol_g \leq R^{-2}\sE_i, \quad\forall\, m \in \NN.
\end{equation}
For a verification of \eqref{eq:Feehan_4-8}, we refer to \cite{Feehan_yang_mills_gradient_flow}. Our definitions \eqref{eq:Feehan_4-6} and \eqref{eq:Feehan_4-7} of the relative mass centers and scales ensure that they have the properties described in Claim \ref{claim:Convergence_mass_centers_and_scales_and_no_external_bubbles}, whose proof follows \mutatis that of \cite[Claim 29.7]{Feehan_yang_mills_gradient_flow} with the aid of \eqref{eq:Feehan_4-8}.

\begin{claim}[Convergence of relative mass centers and scales and absence of bubbling over $B_\rho(x_i) \less B_{\lambda_{i,m} R}(x_{i,m})$]
\label{claim:Convergence_mass_centers_and_scales_and_no_external_bubbles}
For each $i \in \{1,\ldots,l\}$, one has
\begin{align}
\label{eq:Mass_centers_converge_to_bubble_point_in_X}
x_{i,m} &\to x_i,
\\
\label{eq:Scales_converge_to_zero}
\lambda_{i,m} &\to 0, \quad\text{as } m \to \infty.
\end{align}
If in addition the constant $R \geq 1$ is large enough that
\begin{equation}
\label{eq:Kozono_Maeda_Naito_Theorem_5-4_proof_choice_large_Euclidean_4ball_radius_R}
R^{-2}\sE_i \leq \frac{\eps_0}{2}, \quad\forall\, x_i \in \bx,
\end{equation}
then no bubbling can occur for the sequence $A^m$ over $B_\rho(x_i) \less B_{\lambda_{i,m} R}(x_{i,m})$ as $m \to\infty$.
\end{claim}

We can now dispose of the proofs of the remaining assertions of Theorem \ref{thm:Kozono_Maeda_Naito_5-4_YM_sequence}. Item \eqref{item:Theorem_Kozono_Maeda_Naito_5-4_YM_sequence_Am_obeys_Yang-Mills_equation_Euclidean_space} is a consequence of the fact that the Yang-Mills equation is invariant with respect to rescaling of the Riemannian metric. (Indeed, the equation is invariant with respect to conformal transformations of the Riemannian metric.) Items \eqref{item:Theorem_Kozono_Maeda_Naito_5-4_YM_sequence_H2loc_weak_and_W1ploc_strong_convergence_Am}
and \eqref{item:Theorem_Kozono_Maeda_Naito_5-4_YM_sequence_weak_convergence_FAm_measures} are proved by adapting \mutatis the method of proof of the corresponding items in the proof of \cite[Theorem 29.6]{Feehan_yang_mills_gradient_flow}, with the role of \cite[Theorems 29.3 and 29.4]{Feehan_yang_mills_gradient_flow} replaced by that of Theorem \ref{thm:Sedlacek_4-3_Yang-Mills}.

The proof of Item \eqref{item:Kozono_Maeda_Naito_5-4_YM_sequence_centered_limit} follows from the expressions \eqref{eq:Feehan_4-5} for $\sE_i$, the definitions \eqref{eq:Feehan_4-6} and \eqref{eq:Feehan_4-7} of the relative local mass centers, $x_{i,m}$, and relative local scales, $\lambda_{i,m}$; the Definition \ref{defn:Center_scale_positive_Borel_measure} of the global mass center and scale of a positive Borel measure on $\RR^4$; the definition \eqref{eq:Feehan_3-33_Euclidean_space} for the sequences of Riemannian metrics, $g_{i,m}$, and their convergence described in Item \eqref{item:Theorem_Kozono_Maeda_Naito_5-4_YM_sequence_Cinfinity_loc_convergence_gm}; and the definition \eqref{eq:Schlatter_25_Yang-Mills_theoremstatement} for the sequences of connections, $A_i^m$, and their convergence described in Item \eqref{item:Theorem_Kozono_Maeda_Naito_5-4_YM_sequence_H2loc_weak_and_W1ploc_strong_convergence_Am}. (Compare the proof of \cite[Lemma 4.7]{FeehanGeometry}.)

Item \eqref{item:Kozono_Maeda_Naito_5-4_YM_sequence_limit_m_to_infinity_energy_over_increasing_balls_in_Euclidean_space_equals_E} follows just as in the proof of the corresponding item in \cite[Theorem 29.6]{Feehan_yang_mills_gradient_flow}. Item \eqref{item:Kozono_Maeda_Naito_5-4_YM_sequence_limit_over_sphere_asd} is an immediate consequence of the definitions.
\end{proof}

\begin{rmk}[On the definitions of the local centers and scales in Theorem \ref{thm:Kozono_Maeda_Naito_5-4_YM_sequence}]
\label{rmk:Normalization_in_definition_local_centers_and_scales}
In the definitions \eqref{eq:Feehan_4-6} and \eqref{eq:Feehan_4-7} of the local centers and scales, we could just as easily have replaced the normalization factors $\sE_i$ by
\[
\int_{B_\rho(x_i)}\left(|F_{A^m}|^2 - |F_{A^\infty}|^2\right)\, d\vol_g,
\]
as one can see from \eqref{eq:Feehan_4-5}. This redefinition would have the advantage that it is notationally more consistent with our definitions of the centers and scales for connections over $S^4$ in Section \ref{subsec:Mass_center_and_scale_maps}.
\end{rmk}

\subsection{Bubble-tree compactification for the moduli space of anti-self-dual connections}
\label{subsec:Bubble_tree_compactification_moduli_space_anti-self-dual_connections}
In this section, we provide bubble-tree analogues of the Definition \ref{defn:Donaldson_Kronheimer_4-4-2_Sobolev_connections} of Uhlenbeck convergence of Sobolev connections, Uhlenbeck compactification of the moduli space of anti-self-dual connections, and Definition \ref{defn:Open_Uhlenbeck_neighborhood} of an open Uhlenbeck neighborhood of a Sobolev connection (more generally, an ideal Sobolev connection). While we essentially follow our previous development \cite[Section 4.3]{FeehanGeometry}, we simplify and streamline that description.

\begin{defn}[$W_\loc^{k,p}$ convergence with one level of rescaling for a sequence of Sobolev connections]
\label{defn:Donaldson_Kronheimer_4-4-2_Sobolev_connections_one_level_rescaling}
Let $G$ be a compact Lie group and $P$ be a principal $G$-bundle over a closed, connected, four-dimensional, smooth manifold, $X$, and endowed with a Riemannian metric, $g$. Let $\{A^m\}_{m\in\NN}$ be a sequence of connections on $P$ of class $W^{\bar k,\bar p}$, with $\bar p\geq 2$ and integer $\bar k\geq 1$ obeying $(\bar k+1)\bar p>4$, and
\[
E := \limsup_{m\to\infty} \sE_g(A^m) < \infty.
\]
Let $A_0$ be a $W^{\bar k,\bar p}$ connection on a principal $G$-bundle $P_0$ over $X$ with $\eta(P_0)=\eta(P)$. Let $\ell \geq 0$ be an integer and, if $\ell \geq 1$, let $\bx \in \Sym^\ell(X)$ be represented by the set $\{x_1,\ldots,x_l\} \subset X$ of distinct unordered points with $l \leq \ell$ and let $f_i$ be an oriented $g$-orthonormal frame for $(TX)_{x_i}$, for $1\leq i\leq l$. Let $A_i$ be a $W^{\bar k,\bar p}$ connection on a principal $G$-bundle $P_i$ over $S^4$, let $l_i \geq 0$ be an integer and, if $l_i \geq 1$, let $\by_i = \{y_{i,1},\ldots,y_{i,l_i}\} \subset S^4\less\{s\}$ be a collection of distinct unordered points, for $1\leq i\leq l$.

Let $R$ and $\rho$ be positive constants obeying \eqref{eq:Kozono_Maeda_Naito_Theorem_5-4_fixed_small_singular_point_radial_parameter} and\footnote{This condition implies the hypothesis \eqref{eq:Kozono_Maeda_Naito_Theorem_5-4_proof_choice_large_Euclidean_4ball_radius_R} in Theorem \ref{thm:Kozono_Maeda_Naito_5-4_YM_sequence}.}
\begin{equation}
\label{eq:Kozono_Maeda_Naito_Theorem_5-4_proof_choice_large_Euclidean_4ball_radius_R_generalized}
R^{-2}E \leq \frac{1}{2}\sE_g(G),
\end{equation}
where we define
\begin{equation}
\label{eq:Minimum_energy_connection_G-bundle_S4}
\sE_g(G)
:=
\inf_{P\, \not\cong\, S^4\times G}
\left\{\sE_{g_\round}(A): A \text{ is a $W^{1,2}$ connection on } P\to S^4 \right\},
\end{equation}
to be the minimum value\footnote{This topological lower bound is attained by a $g_\round$-anti-self-dual or $g_\round$-self-dual connection on $P$, irrespective of whether such a connection actually exists for $(S^4,P,g_\round)$.}, over all non-trivial principal $G$-bundles $P$ over $S^4$ with its standard round metric of radius one, of the $L^2$-energy $\sE_{g_\round}(A)$ of a $W^{1,2}$ connection $A$ on $P$ predicted by the Chern-Weil formula \cite[Section 9]{Feehan_yang_mills_gradient_flow}, \cite[Appendix C]{MilnorStasheff}, \cite[Appendix]{TauSelfDual}.

Suppose that $\ell \geq 1$ and $l_i \geq 1$ for $1 \leq i\leq l$. For $2 \leq p \leq \bar p$ and integer $0 \leq k\leq \bar k$, we say that $\{A^m\}_{m\in\NN}$ \emph{converges in $W_\loc^{k,p}$ with one level of rescaling} to a \emph{level one limit},
\[
(A_0, (A_1,\by_1),\ldots,(A_l,\by_l)),
\]
if, as in Theorem \ref{thm:Kozono_Maeda_Naito_5-4_YM_sequence}, there exist sequences,
\[
\{u_m\}_{m\in\NN},\quad \{x_{i,m}\}_{m\in\NN},\quad \{\lambda_{i,m}\}_{m\in\NN},\quad \{A_i^m\}_{m\in\NN},\quad \{u_{i,m}\}_{m\in\NN},
\]
and sequences $\{f_{i,m}\}_{m\in\NN}$ defined by the choice $\{f_i\}$ such that, for $1 \leq i\leq l$ and as $m \to \infty$, one has the convergence specified by Theorems \ref{thm:Sedlacek_4-3_Yang-Mills} and \ref{thm:Kozono_Maeda_Naito_5-4_YM_sequence}:
\begin{align*}
x_{i,m} &\to x_i, \quad f_{i,m} \to f_i, \quad \lambda_{i,m} \to 0,
\\
u_m(A^m) &\to A_0 \quad\text{in } W_{A_0,\loc}^{k,p}(X\less\bx),
\\
|F_{A^m}|^2\,d\vol_g &\rightharpoonup |F_{A_0}|^2\,d\vol_g + \sum_{i=1}^l \sE_i\, \delta_{x_i}
\quad\text{in } (C(X;\RR))',
\\
u_{i,m}(A_i^m) &\to A_i \quad\text{in } W_{A_i,\loc}^{k,p}(S^4\less\{\by_i,s\}),
\\
|F_{A_i^m}|^2\,d\vol_{\round} &\rightharpoonup |F_{A_i}|^2\,d\vol_{\round} + \sum_{j=1}^{l_i} \sE_{i,j}\, \delta_{y_{i,j}}
\quad\text{in } (C(S^4;\RR))',
\end{align*}
where $d\vol_{\round}$ is the volume form for the standard round metric $g_\round$ of radius one on $S^4$. If $l=0$, the preceding convergence simplifies to
\[
u_m(A^m) \to A_0 \quad\text{in } W_{A_0}^{k,p}(X), \quad\text{for } m \to \infty.
\]
If $l_i = 0$ for some $i$ then the preceding convergence simplifies to
\begin{align*}
u_{i,m}(A_i^m) &\to A_i \quad\text{in } W_{A_i,\loc}^{k,p}(S^4\less\{s\}),
\\
|F_{A_i^m}|^2\,d\vol_{\round} &\rightharpoonup |F_{A_i}|^2\,d\vol_{\round} \quad\text{in } (C(S^4;\RR))',
\quad\text{for } m \to \infty.
\end{align*}
\end{defn}

The definition of the full bubble-tree convergence of a sequence of Sobolev connections is obtained from Definition \ref{defn:Donaldson_Kronheimer_4-4-2_Sobolev_connections_one_level_rescaling}, modeled on the observation that one may repeatedly apply Theorem \ref{thm:Kozono_Maeda_Naito_5-4_YM_sequence} until the sequences of rescaled connections over $S^4\less\{s\}$ converge in $W_\loc^{k,p}(S^4\less\{s\})$ without points of curvature concentration in $S^4\less\{s\}$.

\begin{defn}[$W_\loc^{k,p}$ convergence with $L$ levels of rescaling and $W_\loc^{k,p}$ bubble-tree convergence of a sequence of Sobolev connections]
\label{defn:Donaldson_Kronheimer_4-4-2_Sobolev_connections_bubble_tree}
Continue the notation of Definition \ref{defn:Donaldson_Kronheimer_4-4-2_Sobolev_connections_one_level_rescaling}. We say that the sequence $\{A^m\}_{m\in\NN}$ on $P\to X$ \emph{converges in $W_\loc^{k,p}$ with zero levels of rescaling} if $\{A^m\}_{m\in\NN}$ has the standard Uhlenbeck convergence specified in Definition \ref{defn:Donaldson_Kronheimer_4-4-2_Sobolev_connections}.

Suppose that $\{A^m\}_{m\in\NN}$ converges in $W_\loc^{k,p}$  with one level of rescaling and that $l_i \geq 1$ for at least one $i \in \{1,\ldots,l\}$. We say that $\{A^m\}_{m\in\NN}$ \emph{converges in $W_\loc^{k,p}$ with two levels of rescaling} to a \emph{level-two limit},
\begin{align*}
{}&(A_0, (A_1,(A_{1,1},\by_{1,1})),\ldots,(A_1,(A_{1,l_1},\by_{1,l_1}))),
\\
{}&(A_0, (A_2,(A_{2,1},\by_{2,1})),\ldots,(A_2,(A_{2,l_2},\by_{2,l_2}))),
\\
{}&\vdots
\\
{}&(A_0, (A_l,(A_{l,1},\by_{l,1})),\ldots,(A_l,(A_{l,l_l},\by_{l,l_l}))),
\end{align*}
if the sequence\footnote{We abuse notation since this sequence of connections is only defined on a sequence of balls with increasing radii that provide an exhaustion of $S^4\less\{s\}$.} $\{A_i^m\}_{m\in\NN}$ on $P_i\to S^4\less\{s\}$ converges in $W_\loc^{k,p}$ with one level of rescaling for each $i \in \{1,\ldots,l\}$ such that $l_i \geq 1$ and hence the set $\by_i \subset S^4\less\{s\}$ of points of curvature concentration is non-empty.

Continuing in this way, if $\{A^m\}_{m\in\NN}$ converges in $W_\loc^{k,p}$  with $L$ levels of rescaling (for an integer $L \geq 0$) and at least one non-empty set of points of curvature concentration in $S^4\less\{s\}$, we apply Definition \ref{defn:Donaldson_Kronheimer_4-4-2_Sobolev_connections} to define \emph{convergence in $W_\loc^{k,p}$ with $L+1$ levels of rescaling}.

Finally, we say that $\{A^m\}_{m\in\NN}$ \emph{converges in $W_\loc^{k,p}$ in the sense of bubble trees} to a \emph{bubble-tree limit} if $\{A^m\}_{m\in\NN}$ converges in $W_\loc^{k,p}$ with $L$ levels of rescaling and no remaining points of curvature concentration in $S^4\less\{s\}$.
\end{defn}

As Definition \ref{defn:Donaldson_Kronheimer_4-4-2_Sobolev_connections}, the definitions of $W_\loc^{k,p}$ bubble-tree convergence for a sequence of Sobolev connections in Definition \ref{defn:Donaldson_Kronheimer_4-4-2_Sobolev_connections_bubble_tree} extends in an obvious way to a sequence of bubble-tree Sobolev connections.

Again, is often convenient to work with a base for the bubble-tree topology on $\bar M^\tau(P,g)$ corresponding to the Definition \ref{defn:Donaldson_Kronheimer_4-4-2_Sobolev_connections_bubble_tree} of sequential bubble-tree convergence. Before proceeding to do this, we first extract definitions of local mass centers and scales modeled on the definitions \eqref{eq:Feehan_4-6} and \eqref{eq:Feehan_4-7} and Remark \ref{rmk:Normalization_in_definition_local_centers_and_scales} arising in the proof of Theorem \ref{thm:Kozono_Maeda_Naito_5-4_YM_sequence} and the corresponding global definitions for connections over $S^4$ in Definition \ref{defn:Mass_center_scale_connection}.


\begin{defn}[Local mass center and scale of a connection on a principal $G$-bundle over a four-dimensional Riemannian manifold]
\label{defn:Relative_local_mass_center_scale_connection_over_Riemannian_manifold}
Let $G$ be a compact Lie group, $P$ be a principal $G$-bundle over a four-dimensional, Riemannian, smooth manifold, $(X,g)$, and $\rho \in (0,\Inj(X,g)]$, and $A_0$ be a $W^{2,2}$ connection on $P\restriction B(x_0,\rho)$, and $f_0$ be an oriented orthonormal frame for $(TX)_{x_0}$, given a point $x_0 \in X$. Then the \emph{local relative mass center}, $\bar z_0 =\Center_{A_0,f_0,\rho}[A]\in\RR^4$, and the \emph{local relative scale}, $\lambda=\Scale_{A_0,f_0,\rho}[A] \in \RR_+ = (0,\infty)$ of $W^{2,2}$ connection $A$ on $P\restriction B(x_0,\rho)$ with \emph{positive mass relative to $A_0$} are defined by
\begin{subequations}
\label{eq:Relative_local_mass_center_and_scale_connection}
\begin{align}
\label{eq:Relative_local_mass_connection}
\Mass_{A_0,f_0,\rho}[A] &:= \int_{B(x_0,\rho)} \left(|F_A|^2 - |F_{A_0}|^2\right) \,d\vol_g > 0,
\\
\label{eq:Relative_local_mass_center_connection}
\Center_{A_0,f_0,\rho}[A] &:= \frac{1}{\Mass_{A_0,f_0,\rho}[A]}
\int_{B(x_0,\rho)}x(\cdot)\left(|F_A|^2 - |F_{A_0}|^2\right) \,d\vol_g,
\\
\Scale_{A_0,f_0,\rho}[A]^2 &:= \frac{1}{\Mass_{A_0,f_0,\rho}[A]}
\int_{B(x_0,\rho)} \dist_g(\cdot, \bar x_0)^2 \left| |F_A|^2 - |F_{A_0}|^2 \right|\,d\vol_g,
\label{eq:Relative_local_scale_connection}
\end{align}
\end{subequations}
where $\bar x_0 := \varphi_{f_0}(\bar z_0) \in X$, and $x(\cdot) = \varphi_{f_0}^{-1}$ denotes the local geodesic normal coordinate chart on $B(x_0,\varrho)$, and $\varphi_{f_0} = \exp_{f_0}:\RR^4 \supset B(0,\varrho) \cong B(x_0,\varrho) \subset X$ is the exponential map defined by the frame $f_0$, and $\varrho := \Inj(X,g)$. If $\Mass_{A_0,f_0,\rho}[A] = 0$, one defines $\Center_{A_0,f_0,\rho}[A] := 0$ and $\Scale_{A_0,f_0,\rho}[A] := 0$.
\end{defn}

In the inequality \eqref{eq:Bubble_tree_neighborhood_rescaled_A_Wkp_near_A_i_large_ball_S4} appearing in Definition \ref{defn:Open_bubble_tree_neighborhood}, the points $\bar x_i$ and scales $\lambda_i$ may be regarded as being determined by Definition \ref{defn:Relative_local_mass_center_scale_connection_over_Riemannian_manifold}, although this is not required even if the typical case. We have the

\begin{defn}[$W_\loc^{k,p}$ bubble-tree open neighborhood of a bubble-tree connection]
\label{defn:Open_bubble_tree_neighborhood}
Continue the notation of Definition \ref{defn:Donaldson_Kronheimer_4-4-2_Sobolev_connections_bubble_tree}. Let $\eps\in (0,1]$, and $\rho \in (0, \Inj(X,g))$, and $R \in [1,\infty)$. Let $l \geq 1$ be an integer and $\bx := \{x_1,\ldots,x_l\} \subset X$ be a subset of distinct points. Let $f_0$ be an oriented orthonormal frame for $(TS^4)_n$. We say that a $W^{\bar k,\bar p}$ connection $A$ on $P$ belongs to a $W_\loc^{k,p}$ \emph{bubble-tree $(\eps,\rho, R)$ open neighborhood} of a \emph{level-one bubble-tree connection}, $(A_0, (A_1,x_1),\ldots, (A_l,x_l))$, if the following hold:
\begin{enumerate}
  \item There is a $W^{\bar k+1,\bar p}$ isomorphism of principal $G$-bundles,
\[
u: P\restriction X\less\bx \cong P_0\restriction X\less\bx,
\]
such that, for some $\eps_\background \in (0,\eps]$,
\begin{equation}
\label{eq:Bubble_tree_neighborhood_A_Wkp_near_A_0_complement_small_balls}
\|u(A) - A_0\|_{W_{A_0}^{k,p}(X\less B_{\rho/4}(\bx))} < \eps_\background,
\end{equation}
where $B_r(\bx) := \cup_{i=1}^l B_r(x_i)$ for any $r \in (0,\Inj(X,g)]$; and

  \item There are a set of scales, $\{\lambda_i\}_{i=1}^l \subset (0,1]$, which obeys the following analogue of \eqref{eq:Kozono_Maeda_Naito_Theorem_5-4_fixed_small_singular_point_radial_parameter},
\begin{equation}
\label{eq:Taubes_1988_4-11a_analogue_mass_center_scale_constraint}
0 < R\lambda_i < \rho
\quad\text{and}\quad
4\rho < 1 \wedge \Inj(X,g) \wedge \min_{i\neq j} \dist_g(x_i,x_j),
\end{equation}
a set of points, $\{\bar x_i\}_{i=1}^l \subset X$ with $\bar x_i \in B(x_i,\rho/4)$, and, for $1 \leq i \leq l$, an oriented orthonormal frame $\bar f_i$ for $(TX)_{\bar x_i}$ obtained by parallel translation, using the Levi-Civita connection for $g$, of a choice of oriented orthonormal frame $f_i$ for $(TX)_{x_i}$, and $W^{\bar k+1,\bar p}$ isomorphisms of principal $G$-bundles,
\[
u_i: \varphi_n^{-1,*}\delta_{\lambda_i}^{-1,*}\varphi_{\bar x_i}^*(P\restriction B(\bar x_i,2\rho)) \cong P_i\restriction \varphi_n(B(0,2\rho/\lambda_i)),
\]
where $\varphi_n(B(0,2\rho/\lambda_i)) = S^4\less \varphi_s(B(0,\lambda_i/2\rho))$, such that, for some $\eps_\sphere \in (0,\eps]$,
\begin{equation}
\label{eq:Bubble_tree_neighborhood_rescaled_A_Wkp_near_A_i_large_ball_S4}
\|u_i(\varphi_n^{-1,*}\delta_{\lambda_i}^{-1,*}\varphi_{\bar x_i}^*A) - A_i\|_{W_{A_i}^{k,p}(S^4\less \varphi_s(B(0,1/2R))} < \eps_\sphere,
\quad\text{for } 1 \leq i \leq l,
\end{equation}
where $\varphi_{\bar x_i}^{-1}$ is the geodesic normal coordinate chart on the open ball $B_\varrho(\bar x_i)$ defined by the oriented orthonormal frame $\bar f_i$ for $(TX)_{\bar x_i}$ and $\varrho := \Inj(X,g)$, and $\varphi_n^{-1}:S^4\less\{s\} \cong \RR^4$ is the stereographic projection from the south pole $s\in S^4\subset \RR^5$ defined by $f_0$, and $\varphi_s^{-1}:S^4\less\{n\} \cong \RR^4$ is the corresponding stereographic projection from the north pole $n\in S^4\subset \RR^5$, and $\delta_\lambda:\RR^4 \ni x \to x/\lambda \in \RR^4$ for $\lambda > 0$; and

\item One has, for some $\eps_\annulus \in (0,\eps]$,
\begin{equation}
\label{eq:Bubble_tree_neighborhood_A_L2_small_curvature_annulus}
\|F_A\|_{L^2(\Omega(\bar x_i;R\lambda_i/4,2\rho))} < \eps_\annulus, \quad\text{for } 1 \leq i \leq l.
\end{equation}
\end{enumerate}
We define \mutatis a $W^{\bar k,\bar p}$ connection $A$ on $P$ to be in a $W_\loc^{k,p}$ \emph{bubble-tree $(\eps,\rho, R)$ open neighborhood} of a \emph{level-two bubble-tree connection},
\begin{align*}
{}&(A_0, (A_1,x_1, (A_{1,1},x_{1,1})),\ldots,(A_1,x_1,(A_{1,l_1},x_{1,l_1}))),
\\
{}&(A_0, (A_2,x_2,(A_{2,1},x_{2,1})),\ldots,(A_2,(A_{2,l_2},x_{2,l_2}))),
\\
{}&\quad\vdots
\\
{}&(A_0, (A_l,(A_{l,1},x_{l,1})),\ldots,(A_l,(A_{l,l_l},x_{l,l_l}))).
\end{align*}
Finally, for any number $L \geq 0$ of iterated rescalings, we define \mutatis a $W^{\bar k,\bar p}$ connection $A$ on $P$ to be in a $W_\loc^{k,p}$ \emph{bubble-tree $(\eps,\rho, R)$ open neighborhood $\tilde\sU \subset \sA(P)$} of a \emph{level $L$ bubble-tree connection}, where $\sA(P)$ denotes the affine space of $W^{\bar k,\bar p}$ connections on $P$.

We say that $[A] \in \sB(P,g)$ belongs to a $W_\loc^{k,p}$ \emph{bubble-tree $(\eps,\rho, R)$ open neighborhood $\sU \subset \sB(P)$} if $A$ belongs to a $W_\loc^{k,p}$ bubble-tree $(\eps,\rho, R)$ open neighborhood $\pi^{-1}(\sU) \subset \sA(P)$. Here, $\sB(P,g) = \sA(P)/\Aut(P)$, where $\Aut(P)$ is the group of $W^{\bar k+1,\bar p}$ gauge transformations of $P$, and $\pi:\sA(P) \to \sB(P,g)$ is the projection map.

If we need to emphasize the individual values of $\eps_\background$, $\eps_\sphere$, or $\eps_\annulus$ in the preceding definitions, then we write $\beps = (\eps_\background, \eps_\sphere, \eps_\annulus)$ and refer to $W_\loc^{k,p}$ bubble-tree $(\beps,\rho, R)$ open neighborhoods.
\end{defn}

\begin{defn}[Fine $W_\loc^{k,p}$ bubble-tree open neighborhoods in $\sA(P)$ and $\sB(P,g)$]
\label{defn:Open_bubble_tree_neighborhood_fine}
Let $\tilde\sU \subset \sA(P)$ and $\sU \subset \sB(P,g)$ be as in Definition \ref{defn:Open_bubble_tree_neighborhood}. We call $\tilde\sU$ (respectively, $\sU$) \emph{fine} if there are constants $\eps_\mycenter, \eps_\scale \in (0,\eps]$ such that the following holds for every pair of connections $A, A' \in \tilde\sU$ (respectively, points $[A], [A'] \in \sU$).

If $L=1$, then the pairs of local mass centers, $\{\bar x_i\}_{i=1}^l, \{\bar x_i'\}_{i=1}^l \subset X$, and pairs of local scales, $\{\lambda_i\}_{i=1}^l, \{\lambda_i'\}_{i=1}^l \subset (0,\rho]$, defined by $A$ and $A'$, respectively, and $\rho \in (0,\Inj(X,g)]$ and $\{x_i\}_{i=1}^l \subset X$, obey
\begin{align}
\label{eq:Comparable_lambdaiprime_lambdai}
|\lambda_i' - \lambda_i| &< \eps_\scale \min\{\lambda_i,\lambda_i'\},
\\
\label{eq:Comparable_xi_xiprime}
\dist_g(\bar x_i, \bar x_i') &< \eps_\mycenter \min\{\lambda_i,\lambda_i'\}, \quad\text{for } 1 \leq i \leq l.
\end{align}
If $L \geq 2$, then the analogous relations should hold for pairs of local mass centers in $S^4\less\{s\}$ and pairs of local scales in $(0,\rho]$.
\end{defn}

\begin{defn}[Coarse $W_\loc^{k,p}$ bubble-tree open neighborhoods in $\sA(P)$ and $\sB(P,g)$]
\label{defn:Open_bubble_tree_neighborhood_coarse}
Assume the notation of Definition \ref{defn:Open_bubble_tree_neighborhood_fine}. We call $\tilde\sU$ (respectively, $\sU$) \emph{coarse} if there is a constant $\eps_\mycenter \in (0,\eps]$ such that the following holds for every pair of connections $A, A' \in \tilde\sU$ (respectively, points $[A], [A'] \in \sU$).

If $L=1$, then the pairs of local mass centers, $\{\bar x_i\}_{i=1}^l, \{\bar x_i'\}_{i=1}^l \subset X$, and pairs of local scales, $\{\lambda_i\}_{i=1}^l,  \{\lambda_i'\}_{i=1}^l \subset (0,\rho]$ obey \eqref{eq:Comparable_xi_xiprime}. If $L \geq 2$, then the analogous relations should hold for pairs of local mass centers in $S^4\less\{s\}$ and pairs of local scales in $(0,\rho]$.
\end{defn}

Briefly, $\sU \subset \sB(P,g)$ is a \emph{fine} neighborhood if for each pair of points, $[A], [A'] \in \sU$, their local scales and local mass centers are close to one another relative to the smaller of each pair of local scales, while $\sU \subset \sB(P,g)$ is a \emph{coarse} neighborhood if for each pair of points, $[A], [A'] \in \sU$, we only require that their local mass centers are close to one another relative to the smaller of each pair of local scales.

As with Definition \ref{defn:Donaldson_Kronheimer_4-4-2_Sobolev_connections_bubble_tree}, the definition of a $W_\loc^{k,p}$ bubble-tree open neighborhood of a Sobolev connection in Definition \ref{defn:Open_bubble_tree_neighborhood} extends in an obvious way to a $W_\loc^{k,p}$ bubble-tree open neighborhood of a bubble-tree Sobolev connection.

Suppose in addition that $X$ is oriented. We let $BM(P,g)$ denote the set of gauge-equivalence classes of \emph{bubble-tree $g$-anti-self-dual connections on $P$}, where each point in $BM(P,g)$ consists of a labeled tree $\sT$ (see Diestel \cite{Diestel_graph_theory4} for definitions of terms relating to trees), with a gauge-equivalence class of a $g$-anti-self-dual connection on a principal $G$-bundle $P_0$ over $X$ associated to the root of the tree and a gauge-equivalence class of $g_\round$-anti-self-dual connections on principal $G$-bundles over $S^4$ associated to each other vertex of the tree. Each point in $BM(P,g)$ obeys the constraint that $\eta(P_0) = \eta(P)$ and $\sE_g(P)$ is equal to the sum of the $L^2$ energies of all connections attached to vertices of the tree.

For example, $([A_0],[A_1],\ldots,[A_l]) \in BM(P,g)$ if $A_0$ is a $g$-anti-self-dual connection on a principal $G$-bundle $P_0$ over $X$ obeying $\eta(P_0) = \eta(P)$ and each $A_i$ is a $g_\round$-anti-self-dual connection on a principal $G$-bundle $P_i$ over $S^4$, for $1 \leq i \leq l$, such that
\[
\sum_{i=0}^l \sE_g(P_i) = \sE_g(P).
\]
We observe that the collection of $W_\loc^{k,p}$ bubble-tree open neighborhoods given by Definition \ref{defn:Open_bubble_tree_neighborhood} form a basis for a topology on $BM(P,g)$, called the \emph{bubble-tree topology}. One may check that $BM(P,g)$ is then a Hausdorff, regular, second-countable topological space and thus metrizable.

We define the \emph{bubble-tree closure}, $\bar M^\tau(P,g)$ (compare \cite[Definition 4.3]{FeehanGeometry}), to be the closure of $M(P,g)$ in $BM(P,g)$, with respect to the bubble-tree topology.

Again, elliptic regularity for solutions to the anti-self-dual and local Coulomb-gauge equations ensure that the definition of bubble-tree topology on $\bar M^\tau(P,g)$ is independent of the choice of $(\bar k,\bar p)$ obeying $\bar k\geq 1$ and $\bar p\geq 2$ and $(\bar k+1)p > 4$ or $(k,p)$ obeying $0\leq k\leq \bar k$ and $2\leq \bar p\leq p$. Theorems \ref{thm:Sedlacek_4-3_Yang-Mills} and \ref{thm:Kozono_Maeda_Naito_5-4_YM_sequence} yield the

\begin{cor}[Sequential bubble-tree compactness for the moduli space of anti-self-dual connections on a principal $G$-bundle]
\label{cor:Donaldson_Kronheimer_4-4-4_G_bubble_tree}
(Compare \cite[Theorem 4.15]{FeehanGeometry})
Let $G$ be a compact Lie group and $P$ be a principal $G$-bundle over a closed, connected, four-dimensional, oriented, smooth manifold, $X$, and endowed with a Riemannian metric, $g$. Then $\bar M^\tau(P,g)$ is sequentially compact.
\end{cor}

Sequential compactness and compactness coincide for the bubble-tree topology on $\bar M^\tau(P, g)$ since this topology is metrizable and thus we obtain the

\begin{cor}[Bubble-tree compactness for the moduli space of anti-self-dual connections on a principal $G$-bundle]
\label{cor:Donaldson_Kronheimer_4-4-3_G_bubble_tree}
(Compare \cite[Theorem 4.14]{FeehanGeometry})
Assume the hypotheses of Corollary \ref{cor:Donaldson_Kronheimer_4-4-4_G_bubble_tree}. Then $\bar M^\tau(P,g)$ is compact.
\end{cor}

\subsection{Bubble-tree compactification for the moduli space of Yang-Mills connections}
\label{subsec:Bubble_tree_compactification_moduli_space_Yang-Mills_connections}
Finally, given in addition a compact subset $\sC \subset [0,\infty)$, we define the bubble-tree compactification, $\overline{\Crit}^{\,\tau}(P,g,\sC)$, of the moduli space of Yang-Mills connections on $P$ with $L^2$ energies belonging to $\sC$ by exact analogy with the definition of $\bar M^\tau(P,g)$ in Section \ref{subsec:Bubble_tree_compactification_moduli_space_anti-self-dual_connections}.

We let $\BCrit(P,g,\sC)$ denote the set of gauge-equivalence classes of \emph{bubble-tree $g$-Yang-Mills connections on $P$ with $L^2$ energy in $\sC$}, where each point in $\BCrit(P,g,\sC)$ consists of a labeled tree \cite{Diestel_graph_theory4}, with a gauge-equivalence class of a $g$-Yang-Mills connection on a principal $G$-bundle $P_0$ over $X$ associated to the root of the tree and a gauge-equivalence class of $g_\round$-Yang-Mills connections on principal $G$-bundles over $S^4$ associated to each other vertex of the tree. Each point in $\BCrit(P,g,\sC)$ obeys the constraint that $\eta(P_0) = \eta(P)$ and $\sE_g(P)$ is equal to the sum of the $L^2$ energies of all connections attached to vertices of the tree.

One may check that $\BCrit(P,g,\sC)$ is a Hausdorff, regular, second-countable topological space and thus metrizable. We define the \emph{bubble-tree closure}, $\overline{\Crit}^{\,\tau}(P,g,\sC)$, to be the closure of $\Crit(P,g,\sC)$ in $\BCrit(P,g,\sC)$, with respect to the bubble-tree topology.

Once again, elliptic regularity for solutions to the Yang-Mills and local Coulomb-gauge equations ensure that the definition of bubble-tree topology on $\overline{\Crit}^{\,\tau}(P,g,\sC)$ is independent of the choice of $(\bar k,\bar p)$ obeying $\bar k\geq 1$ and $\bar p\geq 2$ and $(\bar k+1)\bar p > 4$ or $(k,p)$ obeying $0\leq k\leq \bar k$ and $2 \leq p\leq \bar p$. Theorems \ref{thm:Sedlacek_4-3_Yang-Mills} and \ref{thm:Kozono_Maeda_Naito_5-4_YM_sequence} yield the

\begin{cor}[Sequential bubble-tree compactness for the moduli space of Yang-Mills connections on a principal $G$-bundle]
\label{cor:Donaldson_Kronheimer_4-4-4_G_Yang-Mills_bubble_tree}
Let $G$ be a compact Lie group, $P$ be a principal $G$-bundle over a closed, connected, four-dimensional, smooth manifold, $X$, and endowed with a Riemannian metric, $g$, and $\sC \subset [0,\infty)$ be a compact subset. Then $\overline{\Crit}^{\,\tau}(P,g,\sC)$ is sequentially compact.
\end{cor}

Sequential compactness and compactness coincide for the bubble-tree topology on $\overline{\Crit}^{\,\tau}(P,g,\sC)$ since this topology is metrizable and thus we obtain the

\begin{cor}[Bubble-tree compactness for the moduli space of anti-self-dual connections on a principal $G$-bundle]
\label{cor:Donaldson_Kronheimer_4-4-3_G_Yang-Mills_bubble_tree}
Assume the hypotheses of Corollary \ref{cor:Donaldson_Kronheimer_4-4-4_G_Yang-Mills_bubble_tree}. Then $\overline{\Crit}^{\,\tau}(P,g,\sC)$ is compact.
\end{cor}

\section{Bubble trees and smooth Riemannian metrics on connected sums}
\label{sec:Bubble_trees_and_Riemannian_metrics_connected_sums}
Section \ref{subsec:Introduction_bubble_trees_and_Riemannian_metrics_connected_sums} provides an overview of our construction of Riemannian metrics on connected sums of a given four-dimensional manifold, $X$, and copies of $S^4$ , as prescribed by a tree, $\sT$, where $X$ labels the root vertex. In Section \ref{subsec:Construction_Riemannian_metric_connected_sum_for_bubble-tree_data}, we describe in detail the construction of a Riemannian metric, $g^\#$, on a connected sum, $X\#_\sT S^4$, defined by bubble-tree data. Section \ref{subsec:Convergence_Riemannian_metric_connected_sum} contains our analysis of how the metric, $g^\#$, converges as the neck or annulus widths shrink to zero. Lastly, in Section \ref{subsec:Quasi-conformally_invariant_Sobolev_norms} we review results due to Taubes \cite{TauPath, TauFrame} on the quasi-conformal invariance property of the $W_{A,g_\round}^{1,2}(S^4)$ Sobolev norm for sections of $T^*S^4\otimes \ad P$, where $P$ is a principal $G$-bundle over $S^4$, endowed with its standard round metric, $g_\round$, of radius one.

\subsection{Introduction and motivation}
\label{subsec:Introduction_bubble_trees_and_Riemannian_metrics_connected_sums}
In our discussion of bubble-tree convergence in Section \ref{sec:Bubble-tree_compactness_Yang-Mills_and_anti-self-dual_connections}, we have focused on rescaling a $g$-Yang-Mills connection $A$ belonging to a bubble-tree open neighborhood (Definition \ref{defn:Open_bubble_tree_neighborhood}) near points of curvature concentration $x_i \in X$ using the local scales $\lambda_i \in (0,1]$ and oriented $g$-orthonormal frames $f_i$ for $(TX)_{x_i}$. However, taking the example of one level of rescaling, we have not yet indicated how one might construct a smooth Riemannian metric, $g^\#$, on a connected sum, $X\#_{i=1}^l S^4$, that is conformally equivalent to $g$ on $X$ and has the following properties:
\begin{inparaenum}[\itshape a\upshape)]
\item the metric on each copy of $S^4$ is $C^2$-close to the standard round metric, $g_{S^4}$, of radius one on the complement of small balls, $\varphi_s(B(0,4\sqrt{\lambda_i})) \Subset S^4 \less \{s\}$;
\item the metric is equal to $g$ on $X \less B(x_i, 4\sqrt{\lambda_i})$; and
\item each neck joining a copy of $S^4\less\varphi_s(B(0,\sqrt{\lambda_i}/2))$ to $X \less \cup_{i=1}^lB(x_i, \sqrt{\lambda_i}/2)$ has a metric that is approximately cylindrical and is conformally equivalent to $g$ on the annulus, $\Omega(x_i; \sqrt{\lambda_i}/2, 2\sqrt{\lambda_i}) \subset X$.
\end{inparaenum}

Our motivation for the construction is to show that rather than rescale a $g$-Yang-Mills connection, $A$, near points of curvature concentration $x_i \in X$ by the local scales $\lambda_i$, one may equivalently rescale the Riemannian metric, $g$, on $X$ to a conformally equivalent metric, $g^\#$, on $X\#_{i=1}^l S^4$, with respect to which $A$ is Yang-Mills by conformal invariance of the Yang-Mills equation, in turn a consequence of the conformal invariance of the $L^2$-energy functional, $\sE_g:\sA(P)\to[0,\infty)$. This construction goes back to Donaldson \cite[Section IV (ii)]{DonConn} and Donaldson and Kronheimer \cite[Sections 7.2 and 7.3]{DK}, but we provide more detail here that is used in our analysis in Section \ref{sec:Global W12_metrics_bubble-tree_neighborhoods} and especially, Sections \ref{subsec:W12_bubble-tree_convergence_sequences_Yang-Mills_connections_fine}, \ref{subsec:W12_bubble-tree_convergence_sequences_Yang-Mills_connections_coarse}, and \ref{subsec:Global_conformal_blow-up_maps_coarse_W12_bubble-tree_neighborhoods}.

A construction of a smooth Riemannian metric on $X\#_{i=1}^l S^4$ of this type was given by the author in \cite[Section 3.5]{FeehanGeometry}, but our development was notationally complex as we treated the general case of a connected sum of $X$ and trees $\sT$ of arbitrary height with copies of $S^4$ at each vertex other than the root, which is labeled by $X$. In this section, we clarify and simplify our description in \cite[Section 3.5]{FeehanGeometry} by restricting our attention to trees of height one (thus $\sT = \{1,\ldots,l\}$) and provide additional detail for this special case to better explain how it may be iterated in order to accommodate trees of arbitrary height. The description of the general case differs only in notational complexity from the special case we describe in Section \ref{subsec:Construction_Riemannian_metric_connected_sum_for_bubble-tree_data}.

\subsection{Construction of a Riemannian metric on a connected sum defined by bubble-tree data}
\label{subsec:Construction_Riemannian_metric_connected_sum_for_bubble-tree_data}
Our construction of a Riemannian metric is not specific to dimension four, so we now allow the manifold $X$ to have an arbitrary dimension $d \geq 2$. To define the smooth Riemannian metrics over the neck regions, we use geodesic polar coordinates on each ball, $B(x_i,\varrho) \subset X$, where $\varrho := \Inj(X,g)$ and $x_i \in X$ for $1 \leq i \leq l$. Recall from \cite[Propositions 1.33 and 1.53 and p. 12]{Aubin_1998} that the expression for the components of a Riemannian metric $g$ with respect to geodesic polar coordinates centered at a point $x_i$ in a Riemannian manifold $X$ of dimension $d$ is
\[
g(x) = (dr)^2 + r^2h_{ab}(r\theta)\,d\theta^a\, d\theta^b,
\quad r \in (0,\varrho) \text{ and } \theta \in S^{d-1},
\]
where $x=r\theta \in B(0,\varrho)$ with $r := ((x^1)^2 + \cdots + (x^d)^2)^{1/2} \in (0,\varrho)$; the $x^\mu$ (for $1\leq \mu \leq d$) are geodesic normal coordinates centered at $x_i$ and obtained by composing the inverse of the exponential map, $\exp_{f_i}^{-1}: X\supset B(x_i,\varrho) \to B(0,\varrho) \subset (TX)_{x_i} \cong \RR^d$, with the standard coordinates on $\RR^d$; the $f_i$ is an oriented $g$-orthonormal frame for $(TX)_{x_i}$; and the $\theta^a$ (for $1\leq a \leq d-1$) are the induced spherical coordinates on $S^{d-1} \subset (TX)_{x_i} \cong \RR^d$.

\subsubsection{Construction of neck metrics emanating from annuli centered at points in $X$}
\label{subsubsec:Construction_neck_metric_emanating_from_manifold}
Setting $t := \log r$ and $x = e^t\theta$, the expression for the Riemannian metric $g$ with respect to the resulting geodesic cylindrical coordinates near $x_i$ becomes,
\[
g(x) = e^{2t}\left((dt)^2 + h_{ab}(e^t\theta)\,d\theta^a d\theta^b\right),
\quad t \in (-\infty, \log\varrho) \text{ and } \theta \in S^{d-1},
\]
and this is conformally equivalent to a metric corresponding to $x_i$,
\[
ds^2 = (dt)^2 + h_{ab}(e^t\theta)\,d\theta^a d\theta^b,
\quad t \in (-\infty, \log\varrho) \text{ and } \theta \in S^{d-1},
\]
which is $C^2$-close to the standard cylindrical metric corresponding to $x_i$,
\[
ds_{\cyl}^2 = (dt)^2 + (d\theta)^2
\quad\text{on } (-\infty, \log\varrho)\times S^{d-1}.
\]
Therefore, we can form a neck metric extending from $X$ near $x_i$ by choosing a smooth function, $\kappa_i \in C^\infty(\RR)$, such that
\[
\kappa_i(t)
=
\begin{cases}
r = e^t &\text{for } \log (4\sqrt{\lambda_i}) \leq t < \log\varrho,
\\
r_i \equiv e^{t_i} \equiv  2\sqrt{\lambda_i} &\text{for } -\infty < t \leq \log (2\sqrt{\lambda_i}).
\end{cases}
\]
For example, we may choose $\kappa_i(t) := (1-\beta_i(t))e^{t_i} + \beta_i(t) e^t$, where $\beta_i \in C^\infty(\RR)$ obeys $0 \leq \beta_i \leq 1$ on $\RR$ and
\begin{equation}
\label{eq:beta_i_tube}
\beta_i(t)
=
\begin{cases}
1 &\text{for } \log 4 + (\log\lambda_i)/2 \leq t < \infty,
\\
0 &\text{for } -\infty < t \leq \log 2 + (\log\lambda_i)/2.
\end{cases}
\end{equation}
We can define such a function, $\beta_i$, by setting
\[
\beta_i(t) := \alpha(t - (\log\lambda_i)/2),
\]
where $\alpha \in C^\infty(\RR)$ obeys $0\leq  \alpha \leq 1$ on $\RR$ and
\[
\alpha(t)
:=
\begin{cases}
1 &\text{for } t \geq \log 4,
\\
0 &\text{for } t \leq \log 2.
\end{cases}
\]
We thus form a metric that conformally bends \emph{upward} into a tube emanating from a small $(d-1)$-dimensional sphere centered at each $x_i \in X$,
\begin{equation}
\label{eq:Neck_metric_emanating_from_annulus_around_xi_in_X}
g_X^{\tube,i}(x)
:=
\begin{cases}
4\lambda_i\left((dt)^2 + h_{ab}(e^t\theta)\,d\theta^a d\theta^b\right)
&\text{for } x = e^t\theta \in \Omega(x_i; \sqrt{\lambda_i}/2, 2\sqrt{\lambda_i}),
\\
\kappa_i(t)^2\left((dt)^2 + h_{ab}(e^t\theta)\,d\theta^a d\theta^b\right)
&\text{for } x = e^t\theta \in \Omega(x_i; 2\sqrt{\lambda_i}, 4\sqrt{\lambda_i}),
\\
r^2\left((dt)^2 + h_{ab}(e^t\theta)\,d\theta^a d\theta^b\right)
= g(x) &\text{for } x = e^t\theta \in \Omega(x_i; 4\sqrt{\lambda_i}, \varrho).
\end{cases}
\end{equation}
This is a smooth metric on the neck that is conformally equivalent to $g$ on $\Omega(x_i; \sqrt{\lambda_i}/2, \varrho)$, approximately cylindrical on $\Omega(x_i; \sqrt{\lambda_i}/2, 2\sqrt{\lambda_i})$, and equal to $g$ on $\Omega(x_i; 4\sqrt{\lambda_i}, \varrho)$.

\subsubsection{Construction of a neck metric emanating from an annulus centered at the south pole in $S^d$}
\label{subsubsec:Construction_neck_metric_emanating_from_south_pole_sphere}
Now let us consider the corresponding definition of an approximately round, radius-one, smooth metric $g_{S^d}^{\tube,i}$ on $S^d \less \varphi_s(B(0,2\sqrt{\lambda_i}))$ that is conformally equivalent to $g$ on the ball $B(x_i,2\sqrt{\lambda_i})$ and, in particular, that obeys the smooth matching condition,
\[
g_{S^d}^{\tube,i} = g_X^{\tube,i} \quad\text{on } \Omega\left(x_i; \sqrt{\lambda_i}/2, 2\sqrt{\lambda_i}\right) \subset X, \quad\text{for } 1 \leq i \leq l.
\]
We now consider the effect of the orientation-reversing inversion map,
\[
\RR^d\less\{0\} \ni x = r\theta \mapsto z \equiv \iota(x) \equiv x^{-1} = r^{-1}\bar\theta \in \RR^d\less\{0\},
\]
or more precisely its pull back via the geodesic normal coordinate chart, $\varphi_i^{-1} \equiv \exp_{f_i}^{-1}:X\supset B(x_i,\varrho) \to B(0,\varrho) \subset \RR^d$, where $x^{-1} := (x_1,-x_2,\ldots,-x_d)/|x|^2$. Thus,
\[
B\left(0,2\sqrt{\lambda_i}\right)\less \{0\}
\cong
\RR^d\less \bar B\left(0, 1/2\sqrt{\lambda_i}\right),
\quad r\theta \mapsto r^{-1}\bar\theta.
\]
We now compose the preceding inversion with the rescaling map, $\RR^d \ni z \mapsto y \equiv \delta_{\lambda_i}^{-1}(z) \equiv \lambda_i z \in \RR^d$,
\[
\RR^d\less \bar B\left(0,1/2\sqrt{\lambda_i}\right)
\cong
\RR^d\less \bar B\left(0, \sqrt{\lambda_i}/2\right),
\quad z \mapsto \lambda_i z,
\]
to give the rescaled inversion map,
\[
B\left(0,2\sqrt{\lambda_i}\right)\less \{0\}
\cong
\RR^d\less \bar B\left(0, \sqrt{\lambda_i}/2\right),
\quad r\theta \mapsto \lambda_i r^{-1}\bar\theta.
\]
The rescaled inversion map restricts to an orientation-reversing diffeomorphism of the annulus,
\[
\RR^d \supset \Omega\left(0; \sqrt{\lambda_i}/2, 2\sqrt{\lambda_i}\right)
\cong
\Omega\left(0; \sqrt{\lambda_i}/2, 2\sqrt{\lambda_i}\right) \subset \RR^d,
\]
and hence of
\[
X \supset \Omega\left(x_i; \sqrt{\lambda_i}/2, 2\sqrt{\lambda_i}\right)
\cong
\Omega\left(x_i; \sqrt{\lambda_i}/2, 2\sqrt{\lambda_i}\right) \subset X,
\]
after composition with the restriction of the geodesic normal coordinate chart,
\[
\varphi_i^{-1}: X \supset \Omega\left(x_i; \sqrt{\lambda_i}/2, 2\sqrt{\lambda_i}\right)
\to
\Omega\left(0; \sqrt{\lambda_i}/2, 2\sqrt{\lambda_i}\right) \subset \RR^d,
\]
and its inverse.

We now pull back the metric $g$ on the ball $B(x_i, 2\sqrt{\lambda_i})$ by the composition of the inverse of the geodesic normal coordinate chart and the inversion map,
\[
\iota^*\varphi_i^*g(z)
=
g_{\mu\nu}(z)\,dz^\mu dz^\nu,
\quad\text{for } |z| > 1/2\sqrt{\lambda_i}.
\]
Now pulling back by the rescaling map as well,
\[
\delta_{\lambda_i}^{-1,*}\iota^*\varphi_i^*g(y)
=
\lambda_i^2 g_{\mu\nu}(\lambda_i y)\,dy^\mu dy^\nu,
\quad\text{for } |y| > \sqrt{\lambda_i}/2,
\]
we define
\[
g_{\RR^d}^i(y) := \lambda_i^{-2}\delta_{\lambda_i}^{-1,*}\iota^*\varphi_i^*g(y),
\quad\text{for } |y| > \sqrt{\lambda_i}/2,
\]
and therefore,
\[
g_{\RR^d}^i(y)
=
g_{\mu\nu}(\lambda_i y)\,dy^\mu dy^\nu, \quad\text{for } |y| > \sqrt{\lambda_i}/2.
\]
Just as in the construction \eqref{eq:Neck_metric_emanating_from_annulus_around_xi_in_X} of the metric $g_X^{\tube,i}$, over the annulus $X \supset \Omega(x_i; \sqrt{\lambda_i}/2, 2\sqrt{\lambda_i}) \cong \Omega(0; \sqrt{\lambda_i}/2, 2\sqrt{\lambda_i}) \subset \RR^d$, the metric
\[
g_{\RR^d}^i(y) = e^{2t}\left((dt)^2 + h_{ab}(e^t\theta)\,d\theta^a d\theta^b\right)
\quad\text{on }
(-\log 2 + (\log\lambda_i)/2, \,\log 2 + (\log\lambda_i)/2) \times S^{d-1},
\]
for $y = e^t\theta$ is again conformally equivalent to a metric on the cylinder,
\[
ds^2 = (dt)^2 + h_{ab}(e^t\theta)\,d\theta^a d\theta^b
\quad\text{on }
(-\log 2 + (\log\lambda_i)/2, \,\log 2 + (\log\lambda_i)/2) \times S^{d-1},
\]
that approximates the standard metric on the cylinder,
\[
ds_{\cyl}^2 = (dt)^2 + (d\theta)^2
\quad\text{on }
(-\log 2 + (\log\lambda_i)/2, \,\log 2 + (\log\lambda_i)/2) \times S^{d-1}.
\]
Therefore, by analogy with the definition \eqref{eq:Neck_metric_emanating_from_annulus_around_xi_in_X} of the metric $g_X^{\tube,i}$, we conformally bend the metric $g_{\RR^d}^i(y)$ \emph{downwards} into a tube emanating from a small $(d-1)$-dimensional sphere centered at the south pole $s \in S^d$,
\begin{equation}
\label{eq:Almost_euclidean_neck_metric_emanating_from_annulus_around_s_in_Sd}
g_{\RR^d}^{\tube,i}(y)
:=
\begin{cases}
e^{2t}\left((dt)^2 + h_{ab}(e^t\theta)\,d\theta^a d\theta^b\right)
= g_{\RR^d}^i(y) &\text{for } 4\sqrt{\lambda_i} < |y| <\infty,
\\
\kappa_i(t)^2\left((dt)^2 + h_{ab}(e^t\theta)\,d\theta^a d\theta^b\right)
&\text{for } 2\sqrt{\lambda_i} \leq |y| \leq 4\sqrt{\lambda_i},
\\
4\lambda_i\left((dt)^2 + h_{ab}(e^t\theta)\,d\theta^a d\theta^b\right)
&\text{for } \sqrt{\lambda_i}/2 < |y| < 2\sqrt{\lambda_i}.
\end{cases}
\end{equation}
On the other hand, we can construct an approximation to the round metric of radius one on $S^d\less\varphi_s(B(0, \sqrt{\lambda_i}/2))$ by replacing the standard Euclidean metric, $\delta$, in the expression \eqref{eq:Standard_round_metric_radius_one_sphere_stereographic_coordinates} for $g_{S^d}$ by $g_{\RR^d}^i$, that is
\[
g_{S^d}^i(y) := \frac{4}{(1 + |y|^2)^2}g_{\RR^d}^i(y),
\quad\text{for } |y| > \sqrt{\lambda_i}/2.
\]
Over the annulus $\Omega(0; 2\sqrt{\lambda_i}/, 4\sqrt{\lambda_i}) \subset \RR^d$, we may interpolate between the square of the conformal factor, $2/(1 + |y|^2)$, and the constant, $1$, via the function
\[
\zeta_i(t)
:=
\begin{cases}
2/(1 + e^{2t}) &\text{for } \log 4 + (\log\lambda_i)/2 \leq t < \infty,
\\
2\beta_i(t)/(1 + e^{2t}) + (1-\beta_i(t))
&\text{for } \log 2 + (\log\lambda_i)/2 \leq t \leq \log 4 + (\log\lambda_i)/2,
\\
1 &\text{for } -\log 2 + (\log\lambda_i)/2 < t < \log 2 + (\log\lambda_i)/2,
\end{cases}
\]
where $y = e^t\theta$ and $\beta_i$ is as in \eqref{eq:beta_i_tube}. By combining these conformal factors, we obtain a smooth metric for $y = e^t\theta \in \RR^d$ obeying $|y| > \sqrt{\lambda_i}/2$ such that
\begin{equation}
\label{eq:Almost_round_neck_metric_emanating_from_annulus_around_s_in_Sd}
g_{S^d}^{\tube,i}(y)
:=
\begin{cases}
(4e^{2t}/(1 + e^{2t})^2)\left((dt)^2 + h_{ab}(e^t\theta)\,d\theta^a d\theta^b\right)
&\text{for } 4\sqrt{\lambda_i} < |y| <\infty,
\\
\kappa_i(t)^2\zeta_i(t)^2\left((dt)^2 + h_{ab}(e^t\theta)\,d\theta^a d\theta^b\right)
&\text{for } 2\sqrt{\lambda_i} \leq |y| \leq 4\sqrt{\lambda_i},
\\
4\lambda_i\left((dt)^2 + h_{ab}(e^t\theta)\,d\theta^a d\theta^b\right)
&\text{for } \sqrt{\lambda_i}/2 < |y| < 2\sqrt{\lambda_i}.
\end{cases}
\end{equation}
This is a smooth metric that is approximately equal to the round metric of radius one on $S^d\less \varphi_s(B(0, 4\sqrt{\lambda_i}))$, approximately cylindrical on $\Omega(0; \sqrt{\lambda_i}/2, 2\sqrt{\lambda_i})$, and conformally equivalent to $g$ on the ball $X \supset B(x_i; 2\sqrt{\lambda_i}) \cong S^d\less \varphi_s(B(0, \sqrt{\lambda_i}/2)) \subset S^d$.

By construction, the smooth metrics $g_{S^d}^{\tube,i}$ on each copy of $S^d\less\varphi_s(B(0,\sqrt{\lambda_i}/2))$ match $g_X^{\tube,i}$ on the cylinder,
\[
\Omega(0; \sqrt{\lambda_i}/2, 2\sqrt{\lambda_i})
\cong
(-\log 2 + (\log\lambda_i)/2, \,\log 2 + (\log\lambda_i)/2) \times S^{d-1},
\]
and so we have constructed the desired smooth Riemannian metric, $g^\#$, on $X\#_{i=1}^l S^d$.

\subsubsection{Iterated conformal blow up of a Riemannian metric}
\label{subsubsec:Iterated_local_conformal_blow-up_Riemannian_metric}
The construction in Sections \ref{subsubsec:Construction_neck_metric_emanating_from_manifold} and \ref{subsubsec:Construction_neck_metric_emanating_from_south_pole_sphere} may be iterated as often as required to give smooth Riemannian metrics, $g^\#$, on connected sums, $X\#_\sT S^d$, corresponding to trees, $\sT$, of arbitrary height arising in bubble-tree data provided by Definition \ref{defn:Open_bubble_tree_neighborhood}.

\subsection{Equivalence of iterated conformal blow-ups of Sobolev connections and iterated conformal blow-ups of Riemannian metrics}
\label{subsubsec:Iterated_conformal_blowup_Sobolev_connection_Riemannian_metric}
As we noted in Section \ref{subsec:Introduction_bubble_trees_and_Riemannian_metrics_connected_sums}, the process of locally rescaling a sequence of connections, $\{A^m\}_{m\in\NN}$ on $P$, as in our Definition \ref{defn:Donaldson_Kronheimer_4-4-2_Sobolev_connections_bubble_tree} of bubble-tree convergence, is equivalent to the process of constructing a sequence of locally rescaled Riemannian metrics, $\{g^m\}_{m\in\NN}$ on $X \cong X\#_\sT S^4$ from the bubble-tree data for $\{A^m\}_{m\in\NN}$, just as in Section \ref{subsec:Construction_Riemannian_metric_connected_sum_for_bubble-tree_data}. This alternative view of bubbling via the sequence of pairs of Riemannian metrics and connections, $\{(A^m, g^m)\}_{m\in\NN}$ on $(P, X\#_\sT S^4)$, is convenient since each metric $g^m$ is smooth on $X\#_\sT S^4$ by virtue of the explicit tube construction joining copies of $S^4$ to one another and to $X$ as specified by the tree, $\sT$. In particular, each metric $g^m$ is conformally equivalent to $g$ on $X$ via a smooth diffeomorphism, $h^m:X \cong X\#_\sT S^4$, so $g^m = h^{m,*}g$. Away from the neck regions in the connected sum, $X\#_\sT S^4$, each metric $g^m$ coincides with $g$ on $X$ and is $C^2$-close to the standard round metric, $g_\round$, of radius one on each copy of $S^4$.

We can thus use the bubble-tree data to rescale a sequence of smooth connections, $\{A^m\}_{m\in\NN}$, via a sequence of conformal diffeomorphisms, $\{h^m\}_{m\in\NN} \subset \Conf(X,g)$, and produce a sequence of smooth connections, $\{h^{m,*}A^m\}_{m\in\NN}$, on $(h^{m,*}P, X\#_\sT S^4)\cong (P,X)$. For the purpose of keeping track of the data used to rescale $\{A^m\}_{m\in\NN}$ and define $\{g^m\}_{m\in\NN}$, it is convenient to make the

\begin{defn}[Iterated conformal blow-up of a Sobolev connection and a Riemannian metric defined by bubble-tree neighborhood data]
\label{defn:Iterated_local_conformal_blowup_Sobolev_connection_Riemannian_metric}
Let $G$ be a compact Lie group and $P$ be a principal $G$-bundle over a closed, connected, four-dimensional, smooth manifold, $X$, and endowed with a smooth Riemannian metric, $g$. Suppose $\eps, \rho \in (0, 1]$ and $R \in [1,\infty)$. Let $A$ be a Sobolev connection on $P$ of class $W^{\bar k,\bar p}$, with $\bar p\geq 2$ and integer $\bar k\geq 1$ obeying $(\bar k+1)\bar p>4$. If $A$ belongs to a $W_\loc^{1,2}$ bubble-tree $(\eps,\rho,R)$ open neighborhood (Definition \ref{defn:Open_bubble_tree_neighborhood}), let
\begin{align*}
\bz &:= (z_1,z_2,\ldots,z_l,z_{11},z_{12},\ldots),
\\
\blambda &:= (\lambda_1,\lambda_2,\ldots,\lambda_l,\lambda_{11},\lambda_{12},\ldots),
\end{align*}
denote the finite sequences of mass centers and scales associated with $A$, corresponding to all vertices in the finite tree, $\sT$, defined by the bubble-tree neighborhood, $\sU$. This data and the construction in Section \ref{subsec:Construction_Riemannian_metric_connected_sum_for_bubble-tree_data} define a conformal diffeomorphism, $\tilde h_{\bz,\blambda} \in \Conf(X,g)$, and
\[
\tilde h_{\bz,\blambda}^{-1,*}A \quad\text{on}\quad \tilde h_{\bz,\blambda}^{-1,*}P
\quad\text{and}\quad
\tilde h_{\bz,\blambda}^{-1,*}g \quad\text{on}\quad X\#_\sT S^4,
\]
the iterated conformal blow-up of the connection, $A$ on $P$, and Riemannian metric, $g$ on $X$, respectively.

When $L=1$, the construction of the Riemannian metric, $\tilde h_{\bz,\blambda}^{-1,*}g$, can be obtained from Section \ref{subsec:Construction_Riemannian_metric_connected_sum_for_bubble-tree_data} using the reference points, $\{x_i\}_{i=1}^l, \{x_i\}_{i=1}^l \subset X$, local mass centers, $\{\bar x_i\}_{i=1}^l, \{x_i\}_{i=1}^l \subset X$, oriented orthonormal frames, $f_i$ for $(TX)_{x_i}$, and local scales, $\{\lambda_i\}_{i=1}^l \subset (0,1]$, supplied by Definition \ref{defn:Open_bubble_tree_neighborhood}. For $1 \leq i \leq l$, parallel translation of $f_i$ along the geodesic joining $x_i$ to $\bar x_i$ defines an oriented orthonormal frames $\bar f_i$ for $(TX)_{\bar x_i}$ and an exponential map, $\varphi_{\bar x_i} = \exp_{\bar f_i}: \RR^4 \supset B(0,\varrho) \cong B(\bar x_i, \varrho) \subset X$, where $\varrho = \Inj(X,g)$. The definition of $\tilde h_{\bz,\blambda}^{-1,*}g$ when $L \geq 2$ follows by iterating the preceding construction over copies of $S^4$ in place of $X$, corresponding to each vertex in the tree $\sT$ and using the data supplied by Definition \ref{defn:Open_bubble_tree_neighborhood}.

When $L=1$, the connection $\tilde h_{\bz,\blambda}^{-1,*}A$ is defined by the data used to construct $\tilde h_{\bz,\blambda}^{-1,*}g$ and, in addition, parameters $\rho \in (0, \Inj(X,g)]$ and $R \in [1, \infty)$ supplied by Definition \ref{defn:Open_bubble_tree_neighborhood}. If we omit the explicit diffeomorphisms identifying the small annuli $\Omega(\bar x_i; \sqrt{\lambda_i}/2, 2\sqrt{\lambda_i}) \subset X$ with the necks joining $S^4$ to $X$ near each point $\bar x_i \in X$, then the expression for the connection, $\tilde h_{\bz,\blambda}^{-1,*}A$, simplifies to
\begin{equation}
\label{eq:Iterated_conformal_blowup_Sobolev_connection}
\tilde h_{\bz,\blambda}^{-1,*}A
=
\begin{cases}
A &\text{on } P \restriction X \less B(\bar x_i,R\lambda_i/2),
\\
\varphi_n^{-1,*}\delta_{\lambda_i}^{-1,*}\varphi_{\bar x_i}^*A
&\text{on } \varphi_n^{-1,*}\delta_{\lambda_i}^{-1,*}\varphi_{\bar x_i}^*P \restriction S^4 \less \varphi_s(B(0,/2R)), \quad 1 \leq i \leq l,
\end{cases}
\end{equation}
Note that, for $1 \leq i \leq l$,
\[
\varphi_n(B(0,2\rho/\lambda_i)) = S^4 \less \varphi_s(B(0,\lambda_i/2\rho))
\quad\text{and}\quad
\varphi_n^{-1,*}\delta_{\lambda_i}^{-1,*}\varphi_{\bar x_i}^*A
=
\tilde\delta_{\lambda_i}^{-1,*}\varphi_n^{-1,*}\varphi_{\bar x_i}^*A.
\]
The definition of $\tilde h_{\bz,\blambda}^{-1,*}A$ when $L \geq 2$ follows by iterating the preceding construction over copies of $S^4$ in place of $X$, corresponding to each vertex in the tree $\sT$ and using the data supplied by Definition \ref{defn:Open_bubble_tree_neighborhood}

When there can be no confusion as to meaning, we write $P$ for the pulled-back bundle, $\tilde h_{\bz,\blambda}^{-1,*}P$ on $X\#_\sT S^4$, and abbreviate $\tilde h_{\bz,\blambda}^{-1,*}A$ and $\tilde h_{\bz,\blambda}^{-1,*}g$ by $A^\#$ and $g^\#$, respectively.
\end{defn}

\subsection{Convergence of the Riemannian metric on the connected sum}
\label{subsec:Convergence_Riemannian_metric_connected_sum}
When restricted to pre-compact open subsets, $U \Subset S^d\less\{s\}$, of the spheres, $S^d$, in the connected sum $X\#_{i=1}^l S^d$, the Riemannian metric, $g^\#$, will be approximately but not exactly equal to the standard round metric, $g_\round = g_{S^d}$, of radius one on $S^d$. Indeed, that is a key feature of the construction as we want the metric $g^\#$ to remain within the conformal equivalence class of $g$ in order to preserve the Yang-Mills property of $A$, even when $g$ is not locally flat. However, it is not difficult to see that the metric $g^\#$ converges in $C_0^\infty(S^4\less\{s\})$ to $g_{S^d}$ for each copy of $S^d$ as $\lambda_i \searrow 0$, as we now review from \cite[Lemma 3.12]{FeehanGeometry}.\footnote{We correct a typographical error in the statement of \cite[Lemma 3.12]{FeehanGeometry}, where the multiplicative factors of $\lambda_i$ appearing on the right-hand side of the two inequalities should be adjusted.}

With respect to a geodesic local normal coordinate chart, $\varphi^{-1}: X \supset B(x_0,\varrho) \cong B(0,\varrho) \subset \RR^d$, defined by a $g$-orthonormal oriented frame $f$ for $(TX)_{x_0}$, we have
\[
\varphi^*g(x) = g_{\mu\nu}(x)\,dx^\mu\otimes dx^\nu, \quad |x| < \varrho.
\]
Applying the inverse of the rescaling map, $\delta_\lambda: \RR^d \ni x \mapsto y = x/\lambda \in \RR^d$, so $x = \lambda y$, we have
\[
\delta_\lambda^{-1,*}\varphi^*g(y) = \lambda^2 g_{\mu\nu}(\lambda y)\,dy^\mu\otimes dy^\nu, \quad |y| < \varrho/\lambda,
\]
and thus we define the rescaled Riemannian metric on $B(0,\varrho/\lambda)$ by
\[
g^\lambda(y)
:=
\lambda^{-2}\delta_\lambda^{-1,*}\varphi^*g(y)
=
g_{\mu\nu}(\lambda y)\,dy^\mu\otimes dy^\nu, \quad |y| < \varrho/\lambda.
\]
We recall from Aubin \cite[Definition 1.24, Proposition 1.25, and Corollary 1.32]{Aubin_1998} that
\begin{equation}
\label{eq:Riemannian_metric_components_in_geodesic_normal_coordinates}
\Gamma_{\mu\nu}^\gamma(x_0) = 0,  \quad \Gamma_{\mu\nu}^\gamma(x) = O(r), \quad\hbox{and}\quad g_{\mu\nu}(x) = \delta_{\mu\nu} + O(r^2),
\end{equation}
where $r := \dist_g(x, x_0)$ and $\Gamma_{\mu\nu}^\gamma$ are the Christoffel symbols for $g$ with respect to the local coordinates, $x^\mu$. Therefore,
\begin{gather*}
|g_{\mu\nu}(x) - \delta_{\mu\nu}| \leq C_0|x|^2, \quad \left|\frac{\partial g_{\mu\nu}}{\partial x^\alpha}(y)\right| \leq C_1|x|,
\\
\hbox{and}\quad \left|\frac{\partial^k g_{\mu\nu}}{\partial x^{\alpha_1}\cdots \partial x^{\alpha_k}}(x)\right| \leq C_k, \quad |x| < \varrho,
\end{gather*}
for any integer $k \geq 2$, where $C_k = C_k(g) \in [1,\infty)$. Thus,
\begin{gather*}
|(g^\lambda)_{\mu\nu}(y) - \delta_{\mu\nu}| \leq C_0\lambda^2|y|^2, \quad \left|\frac{\partial(g^\lambda)_{\mu\nu}}{\partial y^\alpha}(y)\right| \leq C_1\lambda^2|y|,
\\
\hbox{and}\quad \left|\frac{\partial^k(g^\lambda)_{\mu\nu}}{\partial y^{\alpha_1}\cdots \partial y^{\alpha_k}}(y)\right| \leq C_k\lambda^k,
\quad |y| < \varrho/\lambda,
\end{gather*}
and for any integer $k \geq 3$. Therefore, $g^\lambda$ converges in $C_0^\infty(\RR^d)$ to the standard Euclidean metric, $\delta$, on $\RR^d$ as $\lambda\searrow 0$ and hence the approximately round metric, $g_{S^d}^\lambda$, of radius one on $S^4\less \varphi_s(B(0,\lambda/\varrho))$,
\[
\varphi_n^*g_{S^d}^\lambda(y)
:=
\frac{4}{(1+|y|^2)^2} g_{\mu\nu}(\lambda y)\,dy^\mu\otimes dy^\nu, \quad |y| < \varrho/\lambda,
\]
converges in $C_0^\infty(S^d\less\{s\})$ to the standard round metric, $g_{S^d}$, of radius one on $S^d$, as desired.

\subsection{Quasi-conformally invariant Sobolev norms}
\label{subsec:Quasi-conformally_invariant_Sobolev_norms}
From Taubes \cite[Section 2]{TauPath}, \cite[Section 3, Part 3]{TauFrame}, we recall the definition of certain Sobolev norms for sections of the vector bundle $T^*S^4\otimes\ad P$ over $S^4$ that are equivalent to the usual norm on $W_{A,g_\round}^{1,2}(S^4;\Lambda^1\otimes\ad P)$ but whose behavior under the action of the group $\Conf(S^4)$ of conformal diffeomorphisms of $S^4$ can be more easily described. Such properties were employed by the author in \cite[Sections 3 and 4]{FeehanGeometry}.

We first digress to recall the description of $\Conf(S^d)$ for $d\geq 3$. Let $\RR^{p,q} = (\RR^{p+q}, g^{p,q})$, for $p, q \in \NN$, where $g^{p,q}$ is the bilinear form \cite[Section 1.1]{Schottenloher_2008}
\[
g^{p,q}(\xi,\eta) := \sum_{i=1}^p \xi^i \eta^i - \sum_{i=p+1}^{p+q} \xi^i\eta^i, \quad\forall\, \xi, \eta \in \RR^{p+q}.
\]
The $d$-dimensional sphere, $S^d \subset \RR^{d+1,0}$, obtains its standard round metric of radius one by the inclusion in $\RR^{d+1,0}$, and is the conformal compactification of $\RR^{d,0}$. Recall from \cite[Definition 2.1]{Schottenloher_2008} that the conformal group $\Conf(\RR^{p,q})$ is the connected component containing the identity in the group of conformal diffeomorphisms of the conformal compactification of $\RR^{p,q}$. According to \cite[Theorem 2.9]{Schottenloher_2008}, when $p+q>2$ we have
\[
\Conf(\RR^{p,q}) \cong \SO(p+1,q+1),
\]
and thus $\Conf(S^d) = \Conf(\RR^{d,0}) = \SO(d+1,1)$.

For any $\lambda\in\RR_+ = (0,\infty)$, let $\lambda(\cdot)$ be the dilation of $\RR^d$ given by $x\mapsto \lambda x$ and for any $z\in\RR^d$, let $z(\cdot)$ be the translation of $\RR^d$ defined by $x\mapsto x-z$. If $\lambda(\cdot)$ and $z(\cdot)$ again denote the conformal diffeomorphisms of $S^d$ induced by the chart $\varphi_n^{-1}:S^d\less\{s\}\cong\RR^d$, then the group $\SO(d)\times \RR \times \RR^d$ of rotations, dilations, and translations of $\RR^d$ is identified with the subgroup $\Conf_s(S^d) \subset \Conf(S^d)$ of diffeomorphisms which fix the south pole $s\in S^d$ \cite[p. 346]{TauPath}. Indeed, the finite generators of $\Conf(S^d)$ are dilations, translations, rotations and special conformal transformations, where the latter can be understood as an inversion, followed by a translation, and followed again by an inversion \cite[Section 2.1 and Table 2.1]{Blumenhagen_Plauschinn_2009}, \cite[Theorem 1.9]{Schottenloher_2008}.

Let $P$ be a principal $G$-bundle over $S^d$ with its standard round metric $g = g_\round$ of radius one. The group $\Conf(S^d)$ acts on $\sA(P)$ and $\Aut(P)$ by pullback (compare \cite[p. 342]{TauFrame}),
\begin{align*}
\sA(P) \times \Conf(S^4) \ni (A, h) &\mapsto h^*A \in \sA(P),
\\
\Aut(P) \times \Conf(S^4) \ni (u, h) &\mapsto h^*u \in \Aut(P),
\end{align*}
and thus descends to an action on $\sB(P,g)$,
\[
\sB(P,g) \times \Conf(S^4) \ni ([A], h) \mapsto [h^*A] \in \sB(P,g).
\]
We now specialize to the case $d = 4$ and consider the definition of Sobolev norms for sections of $T^*S^4\otimes\ad P$ that are equivalent to the standard norm, $\|\cdot\|_{W_{A,g}^{1,2}(S^4)}$, but have more easily described conformal invariance properties.

Let $A$ be a $W^{1,2}$ connection on a principal $G$-bundle $P$ over $S^4$ with its standard round metric $g = g_\round$ of radius one. Let $\delta = g_{\euclid}$ be the flat metric on $S^4\less\{s\} \cong \RR^4$ obtained by pullback of the standard Euclidean metric on $\RR^4$ via the conformal diffeomorphism $\varphi_n^{-1}: S^4\less\{s\}\to \RR^4$. Let $\nabla_A^g$ denote the covariant derivative on $T^*S^4\otimes\ad P$ defined by the connection $A$ and metric $g$, while $\nabla_A^\delta$ denotes the covariant derivative on $T^*S^4\otimes\ad P \restriction S^4\less\{s\}$ defined by $A$ and $\delta$. We define the usual $W^{1,2}$ norm on $C^\infty$ sections $a$ of $T^*S^4\otimes\ad P$ by
\[
\|a\|_{W_{A,g}^{1,2}(S^4)} := \|\nabla_A^g a\|_{L^2(S^4,g)} + \|a\|_{L^2(S^4,g)}.
\]
Similarly, if $a$ is a $C^\infty$ section of $T^*S^4\otimes\ad P$ and has compact support in $S^4\less\{s\}$, define
\begin{align*}
|a|_A &:=  \|\nabla_A^\delta a\|_{L^2(S^4,\delta)},
\\
\|a\|_{W_{A,\delta}^{1,2}(S^4)}
&:=
\|\nabla_A^\delta a\|_{L^2(S^4,\delta)} + \|a\|_{L^2(S^4,\delta)}.
\end{align*}
The properties of $|\cdot|_A$ and $\|\cdot\|_{W_{A,\delta}^{1,2}(S^4)}$ are described by two lemmata of Taubes \cite{TauPath, TauFrame}.

\begin{lem}
\label{lem:Taubes_1984_path_proposition_2-4}
\cite[Proposition 2.4]{TauPath}
There is a universal constant $z \in [1,\infty)$ with the following significance. If $A$ is a $W^{1,2}$ connection on a principal $G$-bundle $P$ over $S^4$ with its standard round metric of radius one, then the following hold:
\begin{enumerate}
\item $|\cdot|_A$ extends to a continuous norm on $\Omega^1(S^4; \ad P) = C^\infty(S^4; T^*S^4\otimes\ad P)$.

\item The norm $|\cdot|_A$ is $\Conf_s(S^4)$-invariant:
    \[
    |h^*a|_{h^*A}=|a|_A, \quad\forall\, h\in \Conf_s(S^4) \text{ and } a \in \Omega^1(S^4; \ad P).
    \]

\item If $a \in \Omega^1(S^4;\ad P)$, then
\begin{align*}
z^{-1}\|a\|_{W_{A,g}^{1,2}(S^4)}
&\leq
|a|_A
\leq
z\|a\|_{W_{A,g}^{1,2}(S^4)},
\\
z^{-1}\|a\|_{W_{A,g}^{1,2}(S^4)}
&\leq
\|a\|_{W_{A,\delta}^{1,2}(S^4)}
\leq
z\|a\|_{W_{A,g}^{1,2}(S^4)}.
\end{align*}
\end{enumerate}
\end{lem}

\begin{lem}
\label{lem:Taubes_1988_3-1}
\cite[Lemma 3.1]{TauFrame}
Assume the hypotheses and notation of Lemma \ref{lem:Taubes_1984_path_proposition_2-4}. If $h \in \Conf(S^4)$ and $a\in\Omega^1(S^4;\ad P)$, then
\[
z^{-1}\|a\|_{W_{A,g}^{1,2}(S^4)}
\leq
\|h^*a\|_{W_{h^*A,g}^{1,2}(S^4)}
\leq
z\|a\|_{W_{A,g}^{1,2}(S^4)}.
\]
\end{lem}

\begin{rmk}[Variant of the $W^{1,2}$ norm on sections of $T^*S^4\otimes\ad P$]
\label{rmk:Taubes_1988_3-1}
A combination of the Kato Inequality \eqref{eq:FU_6-20_first-order_Kato_inequality} in the sequel and the Sobolev embedding $W^{1,2}(S^4;\RR) \hookrightarrow L^4(S^4;\RR)$ given by \eqref{eq:Sobolev_embedding_manifold_bounded_geometry} yields a universal constant $z_0 \in [1,\infty)$ such that
\[
\|a\|_{L^4(S^4,g)} \leq z_0\|a\|_{W_{A,g}^{1,2}(S^4)},
\]
and thus, since $\|a\|_{L^2(S^4,g)} \leq \vol(S^4)^{1/4}\|a\|_{L^4(S^4,g)}$, the norm $\|a\|_{W_{A,g}^{1,2}(S^4)}$ may be replaced by the equivalent norm,
\[
\|a\|_{W_{A,g}^{1,2;4}(S^4)} := \|\nabla_A^g a\|_{L^2(S^4,g)} + \|a\|_{L^4(S^4,g)},
\]
in the statements of Lemmas \ref{lem:Taubes_1984_path_proposition_2-4} and \ref{lem:Taubes_1988_3-1}.
\end{rmk}

\section{Global $W^{1,2}$ distance functions and bubble-tree neighborhoods}
\label{sec:Global W12_metrics_bubble-tree_neighborhoods}
We begin in Section \ref{subsec:Pointwise_estimates_local_connection_one-forms_radial_gauge} by reviewing the well-known pointwise estimates for local connection one-forms in radial gauge in terms of their curvatures. In Section \ref{subsec:Pointwise_decay_curvature_Yang-Mills_connection_Ldover2-small_energy_annulus} we derive pointwise decay estimates for the curvature of a Yang-Mills connection with $L^{d/2}$-small energy over an annulus in Euclidean space and hence obtain pointwise estimates for local connection one-forms in radial gauge. The construction of suitable bubble-tree open neighborhoods in $\Crit(P,g,\sC)$ requires a good understanding of the behavior of center and scale maps and certain key properties are discussed in Section \ref{subsec:Centers_and_scales_for_bubble_tree_neighborhood_design}. We estimate the global $W^{1,2}$ distance between pairs of Yang-Mills connections according to whether their gauge-equivalence classes belong to a \emph{fine} or \emph{coarse} $W_\loc^{1,2}$ bubble-tree open neighborhood. In Section \ref{subsec:W12_bubble-tree_convergence_sequences_Yang-Mills_connections_fine}, we discuss the simpler case where the pair belongs to a \emph{fine} $W_\loc^{1,2}$ bubble-tree neighborhood and obtain a global $W_{A,g}^{1,2}(X)$ bound between $u(A')$ and $A$, when $A$ and $A'$ are $g$-Yang-Mills connections on a principal $G$-bundle $P$ over $X$ and $u \in \Aut(P)$. In Section \ref{subsec:W12_bubble-tree_convergence_sequences_Yang-Mills_connections_coarse}, we analyze the more difficult case where the pair belongs to a \emph{coarse} $W_\loc^{1,2}$ bubble-tree neighborhood. We now obtain a global $\overline{W}_{\sU,g}^{1,2}(X\#_\sT S^4)$ bound between \emph{local} conformal blow-ups of $u(A')$ and $A$, where a modification of the standard $W_{A,g}^{1,2}(X)$ norm defined in this section takes into account the fact that each local scale defined by $A'$ is simply bounded between zero and the corresponding local scale of $A$ (and vice versa) rather than being comparable. Finally, in Section \ref{subsec:Global_conformal_blow-up_maps_coarse_W12_bubble-tree_neighborhoods} we apply the main result of Section \ref{subsec:W12_bubble-tree_convergence_sequences_Yang-Mills_connections_coarse} to show that one can build a finite collection of coarse $W_\loc^{1,2}$ bubble-tree open neighborhoods based on the existence of suitable conformal blow-up diffeomorphisms for a closed, connected, four-dimensional, oriented Riemannian, smooth manifold, $(X,g)$.

\subsection{Pointwise estimates for local connection one-forms in radial gauge}
\label{subsec:Pointwise_estimates_local_connection_one-forms_radial_gauge}
Suppose $A$ is a connection on a principal $G$-bundle $P$ over a convex open neighborhood, $U \subset \RR^d$, of the origin and that $\sigma:U\to P$ is a section. If the connection one-form, $\sigma^*A \in \Omega^1(U;\fg)$, is in radial gauge with respect to the origin, so $\sigma^*A(0)=0$ and $x^\mu (\sigma^*A)_\mu(x) = 0$ for $x\in U$, then \cite[p. 14]{UhlRem}
\begin{equation}
\label{eq:Connection_one-form_radial_gauge_curvature_line_integral}
(\sigma^*A)_\nu(x) = \int_0^1 tx^\mu(F_A)_{\mu\nu}(tx)\,dt, \quad 1\leq \nu \leq d \text{ and } x \in U.
\end{equation}
The Mean Value Theorem for integrals implies that, for each $x\in U$, there is a $t(x) \in (0,1)$ such that
\[
(\sigma^*A)_\nu(x) = t(x) x^\mu(F_A)_{\mu\nu}(t(x) x),
\]
and consequently,
\begin{equation}
\label{eq:Connection_one-form_radial_gauge_estimate}
|\sigma^*A(x)| \leq |x||F_A|(t(x) x), \quad\forall\, x \in U.
\end{equation}
Alternatively, we can apply Jensen's Inequality to give
\begin{equation}
\label{eq:Connection_one-form_radial_gauge_estimate_Jensen_inequality}
\varphi((\sigma^*A)_\nu(x)) \leq \int_0^1 \varphi\left(tx^\mu(F_A)_{\mu\nu}(tx)\right)\,dt,
\quad\forall\, x \in U,
\end{equation}
where $\varphi:\RR\to\RR$ is any convex function, such as $\varphi(s) = as^p$ with constants $p\geq 1$ and $a>0$.

\subsection{Pointwise decay estimates for the curvature of Yang-Mills connections with $L^{d/2}$-small energy over an annulus}
\label{subsec:Pointwise_decay_curvature_Yang-Mills_connection_Ldover2-small_energy_annulus}
We begin by observing that Theorem \ref{thm:Uhlenbeck_3-5} leads to a pointwise decay estimate for the curvature of a Yang-Mills connection over an open annulus,
$\Omega(x_0;r_0,R) := B_R(x_0)\less\bar B_{r_0}(x_0) \subset \RR^d$, where $0<r_0<R$. As usual, if $x_0$ is the origin, we may abbreviate $\Omega(r_0,R) = \Omega(0;r_0,R)$. The estimate \eqref{eq:Uhlenbeck_corollary_theorem_3-5_pointwise_norm_FA_leq_constant_Ldover2_norm_FA_annulus} may be compared with those provided by Taubes \cite[Lemmas 9.1 and 9.2]{TauFrame} or Uhlenbeck \cite[Lemma 4.5]{UhlRem} when $d = 4$, but the difference here is that we do not assume that the Yang-Mills connection is defined and smooth over $\RR^d$ or a punctured ball, $B_{2R}\less\{0\}$, but rather over an annulus, $\Omega(r_0/2, 2R) \subset \RR^d$, with $0 < r_0 < R < \infty$.

\begin{cor}[Pointwise decay estimate for the curvature of a Yang-Mills connection with $L^{d/2}$-small energy over an annulus]
\label{cor:Uhlenbeck_theorem_3-5_annulus}
If $d\geq 2$ is an integer, then there are constants, $K_0=K_0(d) \in [1,\infty)$ and $\eps_0=\eps_0(d) \in (0,1]$, with the following significance. Let $G$ be a compact Lie group, $r_0 < R$ be positive constants, and $A$ be a Yang-Mills connection with respect to the standard Euclidean metric on $P = \Omega(r_0/2,2R)\times G$, where $\Omega(r_0/2,2R) \subset \RR^d$. If
\begin{equation}
\label{eq:Uhlenbeck_corollary_theorem_3-5_FA_Ldover2_small_annulus}
\|F_A\|_{L^{d/2}(\Omega(r_0/2,2R))} \leq \eps_0,
\end{equation}
then
\begin{equation}
\label{eq:Uhlenbeck_corollary_theorem_3-5_pointwise_norm_FA_leq_constant_Ldover2_norm_FA_annulus}
|F_A|(x) \leq K_0 |x|^{-d/2}\|F_A\|_{L^{d/2}(\Omega(r_0/2,2R))},
\quad\forall\, x\in \RR^d \text{ with } r_0 \leq |x| \leq R.
\end{equation}
\end{cor}

\begin{proof}
Suppose $x\in \RR^d$ with $r_0 \leq r = |x| \leq R$. Observe that if $r=r_0$, then any $y \in B_{r_0/2}(x)$ obeys $r_0/2<|y|<3r_0/2$ and if $r=R$, then any $y \in B_{R/2}(x)$ obeys $R/2<|y|<3R/2$. In particular, we have
\[
B_{r/2}(x)\subset \Omega(r_0/2, 2R),
\quad\forall\, x \in \RR^d \text{ with } r_0 \leq r=|x| \leq R.
\]
Hence, \eqref{eq:Uhlenbeck_corollary_theorem_3-5_FA_Ldover2_small_annulus} gives $\|F_A\|_{L^2(B_{r/2}(x))} \leq \eps_0$ and Theorem \ref{thm:Uhlenbeck_3-5} (with $\rho = r/4$) implies that
\[
|F_A|(x) \leq \|F_A\|_{L^\infty(B_{r/4}(x))} \leq K_0r^{-d/2}\|F_A\|_{L^2(B_{r/4}(x))},
\quad\forall\, x \in \RR^d \text{ with } r_0 \leq r=|x| \leq R,
\]
and this is the desired inequality \eqref{eq:Uhlenbeck_corollary_theorem_3-5_pointwise_norm_FA_leq_constant_Ldover2_norm_FA_annulus}.
\end{proof}

The pointwise decay estimate in \eqref{eq:Uhlenbeck_corollary_theorem_3-5_pointwise_norm_FA_leq_constant_Ldover2_norm_FA_annulus} is not optimal when $d=4$ and $\Omega(r_0/2,2R)$ is replaced by $B_{2R}\less\{0\}$, since any potential singularity of $F_A$ at the origin is then necessarily removable by Theorem \ref{thm:Uhlenbeck_removable_singularity_Yang-Mills} and so $|F_A|$ is uniformly bounded over $B_R\less\{0\}$.

It is possible to improve Corollary \ref{cor:Uhlenbeck_theorem_3-5_annulus} in two ways that we briefly describe here for the sake of completeness, although we shall not need such improvements in this article. To describe these improvements, we need to restrict to the case $r_0 = 0$ (and $R = 1/2$), so $\Omega(r_0/2,2R) = B\less\{0\}$, the punctured unit ball in $\RR^d$. Nakajima \cite[Lemma 3.4]{Nakajima_1987} applies a version of the monotonicity formula \cite[Lemma 3.1]{Nakajima_1987}, due originally to Price \cite[Theorems 1 and $1'$]{Price_1983} (see also Tian \cite[Theorems 2.1.20 and Remark 3]{Tian_2002}, quoted by Yang as \cite[Proposition 1]{Yang_2003aim}) to show that the factor $|x|^{-d/2}$ in inequality \eqref{eq:Uhlenbeck_corollary_theorem_3-5_pointwise_norm_FA_leq_constant_Ldover2_norm_FA_annulus} may be improved to $|x|^{-2}$; this result is also proved by Sibner as \cite[Theorem 1.1]{Sibner_1984} when the Yang-Mills connection has small $L^{d/2}$-energy over the punctured unit ball.

\begin{lem}[Improved pointwise decay estimate for the curvature of a Yang-Mills connection with $L^2$-small energy over the punctured unit ball]
\label{lem:Nakajima_1987_3-4}
\cite[Lemma 3.4]{Nakajima_1987},
\cite[Theorem 1.1]{Sibner_1984}
If $d\geq 4$ is an integer and $G$ is a compact Lie group, then there are constants, $K=K(d,G) \in [1,\infty)$ and $\alpha=\alpha(d,G) > 0$ and $\eps=\eps(d,G) \in (0,1]$, with the following significance. Let $A$ be a Yang-Mills connection with respect to the standard Euclidean metric on $P = B\less\{0\}\times G$, where $B \subset \RR^d$ is the unit ball centered at the origin. If
\begin{equation}
\label{eq:Nakajima_1987_lemma_3-4_FA_L2_small_punctured_unit_ball}
\|F_A\|_{L^2(B)} \leq \eps,
\end{equation}
then
\begin{equation}
\label{eq:Nakajima_1987_lemma_3-4_pointwise_decay_FA_leq_constant_L2_norm_FA_annulus}
|F_A|(x) \leq K|x|^{-2}\|F_A\|_{L^2(B)},
\quad\forall\, x\in \RR^d \text{ with } 0 < |x| \leq 1/2.
\end{equation}
\end{lem}

Nakajima and Sibner establish a further improvement of Lemma \ref{lem:Nakajima_1987_3-4}. Indeed, Nakajima uses the monotonicity formula \cite[Lemma 3.1]{Nakajima_1987} and Uhlenbeck's broken Hodge gauge \cite[Definition, page 25, and Theorem 4.6]{UhlRem} to show that the factor $|x|^{-2}$ in Lemma \ref{lem:Nakajima_1987_3-4} may be improved to $|x|^{-2+\alpha}$, for some $\alpha > 0$. A result similar to Lemma \ref{lem:Nakajima_1987_3-13}, for $3 \leq d \leq 7$, is proved by Sibner as \cite[Proposition 3.1]{Sibner_1984}. See also the discussion in Uhlenbeck \cite[first paragraph of page 25]{UhlRem} for the case $d=4$. Of course, Theorem \ref{thm:Uhlenbeck_removable_singularity_Yang-Mills} implies that $\alpha=2$ in Lemma \ref{lem:Nakajima_1987_3-13} when $d=4$.

\begin{lem}[Improved pointwise decay estimate for the curvature of a Yang-Mills connection with $L^2$-small energy over the punctured unit ball]
\label{lem:Nakajima_1987_3-13}
\cite[Lemma 3.13]{Nakajima_1987},
\cite[Proposition 3.1]{Sibner_1984}
If $d\geq 4$ is an integer and $G$ is a compact Lie group, then there are constants, $K=K(d,G) \in [1,\infty)$ and $\alpha=\alpha(d,G) > 0$ and $\eps=\eps(d,G) \in (0,1]$, with the following significance. Let $A$ be a Yang-Mills connection with respect to the standard Euclidean metric on $P = B\less\{0\}\times G$, where $B \subset \RR^d$ is the unit ball centered at the origin. If
\begin{equation}
\label{eq:Nakajima_1987_lemma_3-13_FA_L2_small_punctured_unit_ball}
\|F_A\|_{L^2(B)} \leq \eps,
\end{equation}
then
\begin{equation}
\label{eq:Nakajima_1987_lemma_3-13_improved_pointwise_decay_FA_leq_constant_L2_norm_FA_annulus}
|F_A|(x) \leq K|x|^{-2+\alpha}\|F_A\|_{L^2(B)},
\quad\forall\, x\in \RR^d \text{ with } 0 < |x| \leq 1/2.
\end{equation}
\end{lem}

For our application in this article, the simpler Corollary \ref{cor:Uhlenbeck_theorem_3-5_annulus} and immediate consequence below in Corollary \ref{cor:Uhlenbeck_theorem_3-5_annulus_connection_one-form_radial_gauge} will be more than sufficient.

\begin{cor}[Pointwise decay estimate for a local connection one-form in radial gauge for a Yang-Mills connection with small $L^{d/2}$-energy over an annulus]
\label{cor:Uhlenbeck_theorem_3-5_annulus_connection_one-form_radial_gauge}
Assume the hypotheses of Corollary \ref{cor:Uhlenbeck_theorem_3-5_annulus}. If $\sigma^*A$ is in radial gauge with respect to the origin, that is $x^\mu (\sigma^*A)_\mu(x) = 0$ for all $x \in \Omega(r_0/2,2R)$ and $\sigma^*A(0) = 0$, for a section $\sigma:\Omega(r_0/2,2R) \to P$, then
\begin{equation}
\label{eq:Uhlenbeck_theorem_3-5_pointwise_norm_A_radialgauge_leq_constant_Ldover2_norm_FA_annulus}
|\sigma^*A|(x) \leq K_0 |x|^{1-d/2}\|F_A\|_{L^{d/2}(\Omega(r_0/2,2R))},
\quad\forall\, x\in \RR^d \text{ with } r_0 \leq |x| \leq R.
\end{equation}
\end{cor}

\begin{proof}
The conclusion follows by combining the inequalities \eqref{eq:Connection_one-form_radial_gauge_estimate} and \eqref{eq:Uhlenbeck_corollary_theorem_3-5_pointwise_norm_FA_leq_constant_Ldover2_norm_FA_annulus}.
\end{proof}

\begin{rmk}[Adjustments to pointwise decay estimates for the curvature of connections that are Yang-Mills with respect to arbitrary Riemannian metrics]
\label{rmk:Adjustment_to_pointwise_estimates_for_arbitrary_Riemannian_metrics}
Following \cite{Nakajima_1987}, we note that the estimates in Corollaries \ref{cor:Uhlenbeck_theorem_3-5_annulus} and \ref{cor:Uhlenbeck_theorem_3-5_annulus_connection_one-form_radial_gauge} and Lemmas \ref{lem:Nakajima_1987_3-4} and \ref{lem:Nakajima_1987_3-13} may be extended to allow for connections that are Yang-Mills with respect to arbitrary Riemannian metrics on open subsets of $\RR^d$. Indeed, given $0 \leq r_0 < R \leq \infty$, let $\Lambda > 0$ be a constant such that
\[
\sup_{r_0/2 \leq |x| \leq 2R} \left|\frac{\partial^2 g_{ij}}{\partial x^i\partial x^j}\right| \leq \Lambda.
\]
Then the above-referenced estimates continue to hold as stated, the only change being that all constants ($\eps, \eps_0$, $\alpha$, and $K, K_0$) now depend in addition on $\Lambda$.
\end{rmk}

\subsection{Centers and scales in bubble-tree open neighborhoods}
\label{subsec:Centers_and_scales_for_bubble_tree_neighborhood_design}
Our definitions of bubble-tree open neighborhoods in Section \ref{sec:Bubble-tree_compactness_Yang-Mills_and_anti-self-dual_connections} are motivated by Proposition \ref{prop:Mass_center_and_scale_connection_W12_continuity}, which gives continuity of center and scale maps, $\Center[A]$ and $\Scale[A]$, with respect to $W_{A,g}^{1,2}(S^4;\Lambda^1\otimes\ad P)$ variations of a connection $A$ on a principal $G$-bundle $P$ over $S^4$, with its standard round metric $g = g_\round$ of radius one, and by Proposition \ref{prop:W12_continuity_family_connections_wrt_mass_centers_and_scales}, which gives continuity in the $W_{A,g}^{1,2}(S^4;\Lambda^1\otimes\ad P)$ norm of a family of connections, $h_{z,\lambda}^*A$, with respect to center and scale parameters, $(z,\lambda) \in \RR^4\times\RR_+$.

Suppose that $A$ and $A'$ are connections on $P$ that are close in the $W_A^{1,2}(S^4;\Lambda^1\otimes\ad P)$ norm and $A$ is centered in the sense that $(\Center[A],\Scale[A]) = (0,1) \in \RR^4\times\RR_+$.  Proposition \ref{prop:Mass_center_and_scale_connection_W12_continuity}, via Inequality \eqref{eq:Mass_center_W12_continuity}, indicates that
\[
|\Center[A'] - \Center[A]| = |\Center[A']| < C\eps,
\]
for a constant $C = C(K) \in [1,\infty)$, with $K \in [1,\infty)$ and
\[
\|A' - A\|_{W_{A,g}^{1,2}(S^4)} < \eps
\quad\text{and}\quad
\|A' - A\|_{W_{A,g}^{2,2}(S^4)} \leq K,
\]
for a given $\eps \in (0,1]$. Similarly, Proposition \ref{prop:Mass_center_and_scale_connection_W12_continuity}, via Inequality \eqref{eq:Scale_W12_continuity}, indicates that
\[
|\Scale[A'] - \Scale[A]| = |\Scale[A'] - 1| < C\sqrt{\eps},
\]
for a constant $C = C(K) \in [1,\infty)$, provided
\[
\|A' - A\|_{W_{A,g}^{1,2}(S^4)} < \eps^2
\quad\text{and}\quad
\|A' - A\|_{W_{A,g}^{3,2}(S^4)} \leq K.
\]
We next consider the more general case of the preceding inequalities when $A$ is not centered.

Suppose that $(\Center[A],\Scale[A]) = (z,\lambda) \in \RR^4\times\RR_+$ and $(z[A'],\lambda[A']) = (z',\lambda') \in \RR^4\times\RR_+$. Recall from Lemma \ref{lem:Centering_connection_over_4sphere} that
\begin{align*}
\Center[h_{z,\lambda}^*A'] &= \lambda\Center[A'] + z = \lambda z' + z,
\\
\Scale[h_{z,\lambda}^*A'] &= \lambda\Scale[A'] = \lambda\lambda',
\end{align*}
and
\[
\left(\Center[h_{z,\lambda}^{-1,*}A], \Scale[h_{z,\lambda}^{-1,*}A]\right) = (0,1),
\]
For the mass centers, noting that $\Center[h_{z,\lambda}^{-1,*}A] = 0$, we ask that
\[
|\Center[h_{z,\lambda}^{-1,*}A']|
=
\lambda^{-1}|\Center[A'] - z|
<
C\eps,
\]
that is, using $z = \Center[A]$,
\begin{equation}
\label{eq:Two_mass_centers_close_relative_to_scale_when_two_connections_W12_close}
|\Center[A'] - \Center[A]| < C\lambda\eps,
\end{equation}
for a constant $C = C(K) \in [1,\infty)$, when
\[
\|h_{z,\lambda}^{-1,*}A' - h_{z,\lambda}^{-1,*}A\|_{W_{h_{z,\lambda}^{-1,*}A,g}^{1,2}(S^4)} < \eps
\quad\text{and}\quad
\|h_{z,\lambda}^{-1,*}A' - h_{z,\lambda}^{-1,*}A\|_{W_{h_{z,\lambda}^{-1,*}A,g}^{2,2}(S^4)} \leq K,
\]
where the first inequality is equivalent to
\[
\|A' - A\|_{W_{A,g}^{1,2}(S^4)} < \eps/C_0,
\]
where $C_0$ is the universal constant in Lemma \ref{lem:Taubes_1988_3-1}. For the scales, we ask that
\begin{align*}
|\Scale[h_{z,\lambda}^{-1,*}A'] - \Scale[h_{z,\lambda}^{-1,*}A]|
&=
|\lambda^{-1}\Scale[A'] - \lambda^{-1}\Scale[A]|
\\
&=
|\lambda^{-1}\Scale[A'] - 1| < C\sqrt{\eps},
\end{align*}
that is, using $\lambda = \Scale[A]$,
\begin{equation}
\label{eq:Two_scales_close_relative_to_scale_when_two_connections_W12_close}
|\Scale[A'] - \Scale[A]| < C\lambda\sqrt{\eps},
\end{equation}
for a constant $C = C(K) \in [1,\infty)$, provided
\[
\|h_{z,\lambda}^{-1,*}A' - h_{z,\lambda}^{-1,*}A\|_{W_{h_{z,\lambda}^{-1,*}A,g}^{1,2}(S^4)} < \eps^2
\quad\text{and}\quad
\|h_{z,\lambda}^{-1,*}A' - h_{z,\lambda}^{-1,*}A\|_{W_{h_{z,\lambda}^{-1,*}A,g}^{3,2}(S^4)} \leq K,
\]
where again the first inequality is equivalent to
\[
\|A' - A\|_{W_{A,g}^{1,2}(S^4)} < \eps^2/C_0,
\]
where $C_0$ is the universal constant in Lemma \ref{lem:Taubes_1988_3-1}. We employ these observations on centers and scales as a guide to the construction of suitable bubble-tree open neighborhoods in Section \ref{sec:Bubble-tree_compactness_Yang-Mills_and_anti-self-dual_connections}.

\subsection{Estimate for the $W^{1,2}$ distance between conformal blow-ups of a pair of Yang-Mills connections in a fine $W_\loc^{1,2}$ bubble-tree open neighborhood}
\label{subsec:W12_bubble-tree_convergence_sequences_Yang-Mills_connections_fine}
In Sections \ref{subsec:Bubble_tree_compactification_moduli_space_anti-self-dual_connections} and \ref{subsec:Bubble_tree_compactification_moduli_space_Yang-Mills_connections}, we gave a definition of $W_\loc^{1,2}$ bubble-tree convergence of a sequence of connections to a bubble-tree limit. In this section, we prove that local $W^{1,2}$ bubble-tree convergence implies \emph{global} $W^{1,2}$ bubble-tree convergence. Our main result (Proposition \ref{prop:Taubes_1988_proposition_5-4_fine}) is related to a pair of results due to Taubes \cite[Propositions 5.3. and 5.4]{TauFrame}.

\begin{prop}[Global $W^{1,2}$ distance between a pair of Yang-Mills connections in a \emph{fine} $W_\loc^{1,2}$ bubble-tree open neighborhood]
\label{prop:Taubes_1988_proposition_5-4_fine}
Let $G$ be a compact Lie group and $P$ be a principal $G$-bundle over a closed, connected, four-dimensional, smooth manifold, $X$, and endowed with a Riemannian metric, $g$. Then there are constants $\eps, \rho \in (0, 1]$ and $R \in [1,\infty)$ with the following significance. Let $A$ and $A'$ be $g$-Yang-Mills connections on $P$ of class $W^{\bar k,\bar p}$, with $\bar p\geq 2$ and integer $\bar k\geq 1$ obeying $(\bar k+1)\bar p>4$. If the gauge-equivalence classes, $[A]$ and $[A']$, belong to a \emph{fine} $W_\loc^{1,2}$ bubble-tree $(\eps,\rho,R)$ open neighborhood $\sU$, then there is a gauge transformation, $u \in \Aut(P)$, of class $W^{\bar k+1,\bar p}$ such that
\begin{equation}
\label{eq:Taubes_1988_proposition_5-4_fine}
\|u(A') - A\|_{W_{A,g}^{1,2}(X)}
<
C\eps\left(1 + \sum_{\balpha \in \sT}\|F_{A_\balpha}\|_{W_{A_\balpha, g_\round}^{1,2}(S^4)}\right),
\end{equation}
where $C = C(g,G)$ and the multi-indices, $\balpha = \{i\}, \{ij\}, \ldots$, appearing in the right-hand side of \eqref{eq:Taubes_1988_proposition_5-4_coarse} label the vertices of the finite tree $\sT$ defining $\sU$, and $|\beps| \leq \eps$, with
\[
|\beps|^2 := \eps_\background^2 + \eps_\annulus^2 + \eps_\sphere^2 + \eps_\mycenter^2 + \eps_\scale^2,
\]
and the small parameters on the right-hand side define the vector $\beps$ in Definition \ref{defn:Open_bubble_tree_neighborhood_fine}.
\end{prop}

\begin{proof}
For the sake of exposition in the proof of Proposition \ref{prop:Taubes_1988_proposition_5-4_fine}, we shall assume that the tree $\sT$ defining the bubble-tree open neighborhood $\sU$ has height one and that $l \geq 1$ in the Definition \ref{defn:Open_bubble_tree_neighborhood} but $l_i=0$ for $1 \leq i \leq l$; the conclusion for the general case of a tree of height greater than one is obtained by induction.

\begin{step}[Local $W^{1,2}$ estimates for $A$ relative to the bubble-tree limit]
\label{step:proof_theorem_Taubes_1988_proposition_5-4_W12_estimate_A_rel_limit}
To estimate the following differences,
\[
\|u(A') - A\|_{W_A^{1,2}(X)},
\]
for a suitable $W^{\bar k+1,\bar p}$ gauge transformation, $u \in \Aut(P)$, to be determined, we shall separately consider the following subsets of $X$:
\begin{enumerate}
  \item Complement of balls, $X - \cup_{i=1}^l B_{\rho/4}(x_i) \supset X - \cup_{i=1}^l B_{\rho/2}(\bar x_i)$, where we denote $B_r(\bx) := \cup_{i=1}^l B_r(x_i)$ for $r \in (0,\Inj(X,g))$, and
      \[
      \|u_0(A) - A_0\|_{W_{A_0}^{1,2}(X - B_{\rho/4}(\bx))} < \eps_\background,
      \]
      so that
\begin{equation}
\label{eq:Convergence_X_minus_balls_center_x_i_radius_rho}
\|u_0(A) - A_0\|_{W_{A_0}^{1,2}(X - B_{\rho/2}(\bar\bx))} < \eps_\background,
\end{equation}
and the $W^{\bar k+1,\bar p}$ gauge transformation, $u_0$ over $X - \cup_{i=1}^l B_{\rho/4}(x_i)$, is supplied by Definition \ref{defn:Open_bubble_tree_neighborhood}.

  \item Annuli, $\Omega(\bar x_i; R\lambda_i/4, 2\rho)$, where for a fixed $\rho \in (0,1]$ and $R \in [1,\infty)$ and scales $\lambda_i$ obeying the constraint \eqref{eq:Taubes_1988_4-11a_analogue_mass_center_scale_constraint}, namely
\[
0 < R\lambda_i < \rho
\quad\text{and}\quad
4\rho < 1 \wedge \Inj(X,g) \wedge \min_{i\neq j} \dist_g(x_i,x_j),
\]
we have
\begin{equation}
\label{eq:Convergence_annuli_center_x_i_inner_radius_Rlambda_i_outer_radius_2rho}
\|F_A\|_{L^2(\Omega(\bar x_i; R\lambda_i/4, 2\rho))} < \eps_\annulus,
\quad\text{for } 1 \leq i \leq l;
\end{equation}

  \item Balls, $B(\bar x_i, 2R\lambda_i)$ for $1 \leq i \leq l$, where
\begin{equation}
\label{eq:Convergence_ball_center_x_i_radius_Rlambda_i}
\begin{aligned}
{}& \left\|u_i(A) - \varphi_{\bar x_i}^{-1,*}\delta_{\lambda_i}^*\varphi_n^*A_i\right\|_{L^4(B(\bar x_i, 2R\lambda_i))}
\\
&\quad + \left\|\nabla_{\varphi_{\bar x_i}^{-1,*}\delta_{\lambda_i}^*\varphi_n^*A_i}^g
\left(u_i(A) - \varphi_{\bar x_i}^{-1,*}\delta_{\lambda_i}^*\varphi_n^*A_i\right)\right\|_{L^2(B(\bar x_i, 2R\lambda_i))} < \eps_\sphere,
\end{aligned}
\end{equation}
noting that $\varphi_{\bar x_i}^{-1}: X \supset B(\bar x_i, 2R\lambda_i) \cong B(0, 2R\lambda_i) \subset \RR^4$, where
\[
\varphi_{\bar x_i} \equiv \exp_{\bar f_i}:\RR^4\supset B_\varrho(0) \cong B_\varrho(\bar x_i) \subset X
\]
is the inverse geodesic normal coordinate chart defined by a frame $\bar f_i$ for $(TX)_{\bar x_i}$ obtained by parallel translation via the $g$-Levi-Civita connection of the frame $f_i$ for $(TX)_{x_i}$, and $\varrho := \Inj(X,g)$, and $\varphi_n:\RR^4\cong S^4\less\{s\}$ is given by \eqref{eq:ConformalDiffeo}, and recalling from \eqref{eq:CenterScaleDiffeo} that $\delta_{\lambda_i} \in \Conf(\RR^4)$ and $\tilde\delta_{\lambda_i} = \varphi_n\circ\delta_{\lambda_i}\circ\varphi_n^{-1} \in \Conf_s(S^4)$, and the $W^{\bar k+1,\bar p}$ gauge transformations, $u_i$ for $1\leq i\leq l$, are supplied by Definition \ref{defn:Open_bubble_tree_neighborhood}. In writing \eqref{eq:Convergence_ball_center_x_i_radius_Rlambda_i}, we have applied Lemma \ref{lem:Taubes_1988_3-1} to relate the $W^{1,2}$ norms for $\ad P$-valued one-forms over $B(\bar x_i, 2R\lambda_i) \subset X$ in \eqref{eq:Convergence_ball_center_x_i_radius_Rlambda_i} and their conformal blow-ups over $S^4 \less \varphi_s(B(0,1/2R))$, as described by Definition \ref{defn:Open_bubble_tree_neighborhood}.
\end{enumerate}

The crux of the proof is therefore to derive Sobolev estimates for $A$ (and $A'$) over the \emph{thick} annuli, $\Omega(\bar x_i; R\lambda_i, \rho) \subset X$, for $1 \leq i \leq l$. Corollary \ref{cor:Uhlenbeck_theorem_3-5_annulus} provides the pointwise decay estimate, for $1 \leq i \leq l$, over those annuli,
\[
|\varphi_{\bar x_i}^*F_A|(x)\leq K_0|x|^{-2}\|F_A\|_{L^2(\Omega(\bar x_i; R\lambda_i/2, 2\rho))},
\quad\forall\, x\in \RR^4 \text{ with } R\lambda_i \leq |x| \leq \rho.
\]
This is the where we use the assumption that $A$ is Yang-Mills.

By \eqref{eq:Connection_one-form_radial_gauge_estimate} and Corollary \ref{cor:Uhlenbeck_theorem_3-5_annulus}, noting that \eqref{eq:Convergence_annuli_center_x_i_inner_radius_Rlambda_i_outer_radius_2rho} holds, we have
\begin{align*}
|\varphi_{\bar x_i}^*\sigma_{\bar x_i}^*A|(x)
&\leq
K_0|x|^{-1}\|F_A\|_{L^2(\Omega(\bar x_i; R\lambda_i/4, 2R\lambda_i))}
\\
&\leq \frac{1}{2}K_0R^{-1}\lambda_i^{-1}
\|F_A\|_{L^2(\Omega(\bar x_i; R\lambda_i/4, 2R\lambda_i))},
\quad\forall\, x\in \RR^4 \text{ with } R\lambda_i/2 \leq |x| \leq R\lambda_i,
\end{align*}
where we require that the local connection one-form, $\sigma_{\bar x_i}^*A$, be in radial gauge with respect to the point $\bar x_i \in X$. Hence,
\begin{equation}
\label{eq:L4norm_sigmai_A_annulus_xi_Rlambda_i_intermsof_L2norm_FA}
\|\sigma_{\bar x_i}^*A\|_{L^4(\Omega(\bar x_i; R\lambda_i/2, R\lambda_i))}
\leq C\|F_A\|_{L^2(\Omega(\bar x_i; R\lambda_i/4, 2R\lambda_i))},
\end{equation}
for a universal positive constant, $C \in [1,\infty)$.

Given a point, $x_0 \in X$, and a constant, $r \in (0,\Inj(X,g)]$, we define a smooth cut-off function, $\chi_{x_0,r}:X\to [0,1]$, by setting
\begin{equation}
\label{eq:Cutoff_function_zero_on_half-ball_one_on_manifold_minus_ball}
\chi_{x_0,r}(x)
:=
\chi(\dist_g(x,x_0)/r), \quad\forall\, x\in X,
\end{equation}
where $\chi:\RR\to [0,1]$ is a smooth function such that $\chi(t)=1$ for $t\geq 1$ and $\chi(t)=0$ for $t\leq 1/2$.  Thus, we have
\[
\chi_{x_0,r}(x)
=
\begin{cases}
1 &\text{for } x\in X - B_r(x_0),
\\
0 &\text{for } x\in B_{r/2}(x_0).
\end{cases}
\]
An elementary calculation yields (see \cite[Lemma 5.8]{FLKM1} for example),
\begin{equation}
\label{eq:L4_bound_dchi}
\|d\chi_{x_0,r}\|_{L^4(X)} \leq C,
\end{equation}
where the constant $C=C(g) \in [1, \infty)$ is independent of $x_0 \in X$ and $r \in (0,\varrho)$, but rather depends only on the fixed universal choice of $\chi$ via $|d\chi| \leq 4$, and the injectivity radius, $\varrho = \Inj(X,g)$.

We denote $\chi_i = \chi_{\bar x_i, R\lambda_i}$ as in \eqref{eq:Cutoff_function_zero_on_half-ball_one_on_manifold_minus_ball} and define a cut-off connection on $P$ by
\begin{equation}
\label{eq:Cutoff_connection_equal_to_A_over_complement_balls_and_product_over_balls}
A_\chi
:=
\begin{cases}
A,  &\text{over } X - \cup_{i=1}^l B(\bar x_i, R\lambda_i),
\\
\Theta + \chi_i\sigma_{\bar x_i}^*A,  &\text{over } \Omega(\bar x_i; R\lambda_i/2, R\lambda_i),
\text{ for } 1 \leq i \leq l,
\\
\Theta,  &\text{over } B(\bar x_i, R\lambda_i/2), \text{ for } 1 \leq i \leq l,
\end{cases}
\end{equation}
where $\Theta$ is the product connection on $B(\bar x_i,\varrho)\times G$. Over the \emph{thin} annuli, $\Omega(\bar x_i; R\lambda_i/2, R\lambda_i)$, we have
\[\
F_{A_\chi} = \chi_iF_A + d\chi_i \wedge \sigma_{\bar x_i}^*A + \chi_i(\chi_i - 1) \sigma_{\bar x_i}^*A \wedge \sigma_{\bar x_i}^*A.
\]
Therefore,
\begin{align*}
\|F_{A_\chi}\|_{L^2(\Omega(\bar x_i; R\lambda_i/2, R\lambda_i))}
&\leq
\|F_A\|_{L^2(\Omega(\bar x_i; R\lambda_i/2, R\lambda_i))}
\\
&\quad + \|d\chi_i\|_{L^4(X)}
\|\sigma_{\bar x_i}^*A\|_{L^4(\Omega(\bar x_i; R\lambda_i/2, R\lambda_i))}
+ 2\|\sigma_{\bar x_i}^*A\|_{L^4(\Omega(\bar x_i; R\lambda_i/2, R\lambda_i))}^2,
\end{align*}
and thus by \eqref{eq:L4norm_sigmai_A_annulus_xi_Rlambda_i_intermsof_L2norm_FA} and \eqref{eq:L4_bound_dchi},
\begin{equation}
\label{eq:L2norm_FAchii_annulus_xi_Rlambda_i_intermsof_L2norm_FA}
\|F_{A_\chi}\|_{L^2(\Omega(\bar x_i; R\lambda_i/2, R\lambda_i))}
\leq
C\|F_A\|_{L^2(\Omega(\bar x_i; R\lambda_i/2, R\lambda_i))},
\end{equation}
where $C \in [1,\infty)$ is a universal constant, noting that
\[
\|F_A\|_{L^2(\Omega(\bar x_i; R\lambda_i/2, R\lambda_i))}
\leq
\|F_A\|_{L^2(\Omega(\bar x_i; R\lambda_i/2, 2\rho))} < \eps_\annulus \in (0, 1]
\quad\text{(by \eqref{eq:Convergence_annuli_center_x_i_inner_radius_Rlambda_i_outer_radius_2rho})}.
\]
Since $F_{A_\chi} = 0$ on $B(\bar x_i, R\lambda_i/2)$, for $1 \leq i \leq l$, then \eqref{eq:L2norm_FAchii_annulus_xi_Rlambda_i_intermsof_L2norm_FA} implies that
\[
\|F_{A_\chi}\|_{L^2(B(\bar x_i, 2\rho))} \leq C\|F_A\|_{L^2(B(\bar x_i, 2\rho))} < C\eps_\annulus.
\]
Hence, for small enough $\eps_\annulus = \eps_\annulus(G) \in (0,1]$, we can apply Theorem \ref{thm:Uhlenbeck_Lp_1-3} and Remark \ref{rmk:Uhlenbeck_theorem_1-3_Wkp} to find a $W^{\bar k+1,\bar p}$ gauge transformation, $u_i^{\coul} \equiv u_i^{\coulomb}$, of $B(\bar x_i, 2\rho)\times G$ taking $\sigma_{\bar x_i}^*A_\chi$ to Coulomb gauge with respect to the product connection, so $d_\Theta^*u_i^{\coul}(\sigma_{\bar x_i}^*A_\chi) = 0$ over $B(\bar x_i, 2\rho)$, and obeying
\begin{align*}
\|u_i^{\coul}(\sigma_{\bar x_i}^*A_\chi)\|_{L^4(B(\bar x_i, 2\rho))}
+ \|\nabla_\Theta u_i^{\coul}(\sigma_{\bar x_i}^*A_\chi)\|_{L^2(B(\bar x_i, 2\rho))}
&\leq
C\|F_{A_\chi}\|_{L^2(B(\bar x_i, 2\rho))}
\\
&\leq C\|F_A\|_{L^2(\Omega(\bar x_i; R\lambda_i/2, 2\rho))}
\\
&< C\eps_\annulus \quad\text{(by \eqref{eq:Convergence_annuli_center_x_i_inner_radius_Rlambda_i_outer_radius_2rho})},
\end{align*}
where the second inequality follows from \eqref{eq:L2norm_FAchii_annulus_xi_Rlambda_i_intermsof_L2norm_FA} and the fact that $F_{A_\chi} = 0$ on $B(\bar x_i, R\lambda_i/2)$. In particular, over the annuli, $\Omega(\bar x_i; R\lambda_i, \rho)$, where $A_\chi = A$, we have
\begin{equation}
\label{eq:Convergence_annulus_center_x_i_inner_radius_Rlambda_i_outer_radius_rho}
\|u_i^{\coul}(\sigma_{\bar x_i}^*A)\|_{L^4(\Omega(\bar x_i; R\lambda_i, \rho))}
+ \|\nabla_\Theta u_i^{\coul}(\sigma_{\bar x_i}^*A)\|_{L^2(\Omega(\bar x_i; R\lambda_i, \rho))}
<
C\eps_\annulus, \quad\text{for } 1 \leq i \leq l.
\end{equation}
This completes Step \ref{step:proof_theorem_Taubes_1988_proposition_5-4_W12_estimate_A_rel_limit}.
\end{step}

\begin{step}[Local $W^{1,2}$ estimates for $A'$ relative to the bubble-tree limit]
\label{step:proof_theorem_Taubes_1988_proposition_5-4_W12_estimate_Aprime_rel_limit}
Naturally, for suitable gauge transformations, we also obtain the analogous estimates for the connection $A'$, with its mass centers $\bar x_i' \in X$ and scales $\lambda_i'$ obeying $0 < R\lambda_i' < \rho$ as in \eqref{eq:Taubes_1988_4-11a_analogue_mass_center_scale_constraint}.

First, over the complement of balls, $X - \cup_{i=1}^l B_{\rho/4}(x_i) \supset X - \cup_{i=1}^l B_{\rho/2}(\bar x_i')$, we have
\[
\|v_0(A') - A_0\|_{W_{A_0}^{1,2}(X - B_{\rho/4}(\bx))} < \eps_\background,
\]
so that
\begin{equation}
\label{eq:Convergence_X_minus_balls_center_x_i_radius_rho_prime}
\|v_0(A') - A_0\|_{W_{A_0}^{1,2}(X - B_{\rho/2}(\bar\bx'))} < \eps_\background,
\end{equation}
where the $W^{\bar k+1,\bar p}$ gauge transformation, $v_0$ over $X - \cup_{i=1}^l B_{\rho/4}(x_i')$, is supplied by Definition \ref{defn:Open_bubble_tree_neighborhood}.

Second, over the annuli $\Omega(\bar x_i'; R\lambda_i', \rho)$, we have
\begin{equation}
\label{eq:Convergence_annulus_center_x_i_inner_radius_Rlambda_i_outer_radius_rho_prime}
\|v_i^{\coul}(\sigma_{\bar x_i'}^*A')\|_{L^4(\Omega(\bar x_i'; R\lambda_i', \rho))}
+ \|\nabla_\Theta v_i^{\coul}(\sigma_{\bar x_i'}^*A')\|_{L^2(\Omega(\bar x_i'; R\lambda_i', \rho))}
<
C\eps_\annulus, \quad\text{for } 1 \leq i \leq l,
\end{equation}
with $W^{\bar k+1,\bar p}$ gauge transformations $v_i^{\coul}$ over $\Omega(\bar x_i'; R\lambda_i/2, 2\rho)$, and $\sigma_{\bar x_i'}^*A'$ is in radial gauge with respect to the point $\bar x_i' \in X$, the Coulomb gauge condition, $d_\Theta^*v_i^{\coul}(\sigma_{\bar x_i'}^*A') = 0$, holds, and $C \in [1,\infty)$ is a universal constant.

Third, over the small balls, $B(\bar x_i', 2R\lambda_i')$ for $1 \leq i \leq l$, we have
\begin{equation}
\label{eq:Convergence_ball_center_x_i_radius_Rlambda_i_prime}
\begin{aligned}
{}& \left\|v_i(A') - \varphi_{\bar x_i'}^{-1,*}\delta_{\lambda_i'}^*\varphi_n^*A_i\right\|_{L^4(B(\bar x_i', 2R\lambda_i'))}
\\
&\quad + \left\|\nabla_{\varphi_{\bar x_i'}^{-1,*}\delta_{\lambda_i'}^*\varphi_n^*A_i}
\left(v_i(A') - \varphi_{\bar x_i'}^{-1,*}\delta_{\lambda_i'}^*\varphi_n^*A_i\right)\right\|_{L^2(B(\bar x_i', 2R\lambda_i'))} < \eps_\sphere,
\end{aligned}
\end{equation}
noting that $\varphi_{\bar x_i'}^{-1}: X \supset B(\bar x_i', 2R\lambda_i') \cong B(0, 2R\lambda_i') \subset \RR^4$ and
\[
\varphi_{\bar x_i'} \equiv \exp_{\bar f_i'}:\RR^4\supset B_\varrho(0) \cong B_\varrho(\bar x_i') \subset X
\]
is the geodesic normal coordinate chart defined by a frame $\bar f_i'$ for $(TX)_{\bar x_i'}$ obtained by parallel translation via the $g$-Levi-Civita connection of the frame $f_i$ for $(TX)_{x_i}$. The $W^{\bar k+1,\bar p}$ gauge transformations $v_i$ over $B(\bar x_i', 2R\lambda_i')$ are supplied by Definition \ref{defn:Open_bubble_tree_neighborhood}. This completes Step \ref{step:proof_theorem_Taubes_1988_proposition_5-4_W12_estimate_Aprime_rel_limit}.
\end{step}

\begin{step}[Local $W^{1,2}$ estimates for $A$ and $A'$ relative to the bubble-tree limit over common open subsets]
\label{step:proof_theorem_Taubes_1988_proposition_5-4_W12_estimate_A_and_Aprime_common}
According to Definition \ref{defn:Open_bubble_tree_neighborhood}, we have
\begin{subequations}
\label{eq:Distance_bar_xi_and_bar_xiprime}
\begin{align}
\dist_g(\bar x_i', \bar x_i) &< \eps_\mycenter\lambda_i,
\\
|\lambda_i' - \lambda_i| &< \eps_\scale\lambda_i,
\quad\text{for } i = 1,\ldots,l,
\end{align}
\end{subequations}
where $\eps_\mycenter, \eps_\scale \in (0, 1]$. For example, choosing $\eps_\mycenter = \eps_\scale = 1/16$ will suffice. Therefore, we can slightly shrink the open subsets used in Steps \ref{step:proof_theorem_Taubes_1988_proposition_5-4_W12_estimate_A_rel_limit} and \ref{step:proof_theorem_Taubes_1988_proposition_5-4_W12_estimate_Aprime_rel_limit} to obtain estimates for both $A$ an $A'$ and gauge transformations defined over common open subsets.

When $g$ is locally flat near the points $x_i \in X$ and we may isometrically identify $\psi_i: X \supset B(x_i,\varrho) \cong B(0,\varrho) \subset \RR^4$ via the choice of frame $f_i$ for $(TX)_{x_i}$ and inverse geodesic normal local coordinate chart $\psi_i$, then the inverse geodesic normal local coordinate charts, $\varphi_{\bar x_i}$ centered at $\bar x_i$ and $\varphi_{\bar x_i'}$ centered at $\bar x_i'$, are related by (compare \cite[Equation (3.18)]{FeehanGeometry})
\begin{equation}
\label{eq:Charts_as_geodesic_normal_coordinate_charts_and_translations}
\varphi_{\bar x_i}(x) = \psi_i(x + \psi_i^{-1}(\bar x_i))
\quad\text{and}\quad
\varphi_{\bar x_i'}(x) = \psi_i(x + \psi_i^{-1}(\bar x_i')), \quad\forall\, x \in B(0,\varrho/2) \subset \RR^4,
\end{equation}
and thus, setting $\bar z_i := \psi_i^{-1}(\bar x_i) \in \RR^4$ and $\bar z_i' := \psi_i^{-1}(\bar x_i') \in \RR^4$, we have
\[
\varphi_{\bar x_i}(x) = \psi_i\circ \tau_{-\bar z_i}(x)
\quad\text{and}\quad
\varphi_{\bar x_i'}(x) = \psi_i\circ \tau_{-\bar z_i'}(x), \quad\forall\, x \in B(0,\varrho/2) \subset \RR^4.
\]
When $g$ is not locally flat near $x_i$, we may either take the preceding relations as definitions of the inverse coordinate charts, $\varphi_{\bar x_i}$ and $\varphi_{\bar x_i'}$, or note that these relations hold approximately with a small error depending on $\dist_g(\bar x_i, \bar x_i')$. We shall take the first approach and use the coordinate charts \eqref{eq:Charts_as_geodesic_normal_coordinate_charts_and_translations} in Definition \ref{defn:Open_bubble_tree_neighborhood_fine}.

First, over the complement of balls, $X - \cup_{i=1}^l B_{\rho/2}(\bar x_i)$, from the inequalities
from \eqref{eq:Convergence_X_minus_balls_center_x_i_radius_rho} and \eqref{eq:Convergence_X_minus_balls_center_x_i_radius_rho_prime} (which also holds over a larger subset), we obtain
\begin{subequations}
\label{eq:Convergence_X_minus_balls_center_x_i_radius_rho_common}
\begin{align}
\label{eq:Convergence_X_minus_balls_center_x_i_radius_rho_common_A}
\|u_0(A) - A_0\|_{W_{A_0}^{1,2}(X - B_{\rho/2}(\bar\bx))} &< \eps_\background,
\\
\label{eq:Convergence_X_minus_balls_center_x_i_radius_rho_common_Aprime}
\|v_0(A') - A_0\|_{W_{A_0}^{1,2}(X - B_{\rho/2}(\bar\bx))} &< \eps_\background.
\end{align}
\end{subequations}
Second, over the annuli $\Omega(\bar x_i; 3R\lambda_i/2, 3\rho/4)$, from \eqref{eq:Convergence_annulus_center_x_i_inner_radius_Rlambda_i_outer_radius_rho}, \eqref{eq:Convergence_annulus_center_x_i_inner_radius_Rlambda_i_outer_radius_rho_prime}, and \eqref{eq:Distance_bar_xi_and_bar_xiprime}, we obtain
\begin{subequations}
\label{eq:Convergence_annulus_center_x_i_inner_radius_Rlambda_i_outer_radius_rho_common}
\begin{multline}
\label{eq:Convergence_annulus_center_x_i_inner_radius_Rlambda_i_outer_radius_rho_common_A}
\|u_i^{\coul}(\sigma_{\bar x_i}^*A)\|_{L^4(\Omega(\bar x_i; 3R\lambda_i/2, 3\rho/4))}
\\
+ \|\nabla_\Theta u_i^{\coul}(\sigma_{\bar x_i}^*A)\|_{L^2(\Omega(\bar x_i; 3R\lambda_i/2, 3\rho/4))}
<
C\eps_\annulus, \quad\text{for } 1 \leq i \leq l,
\end{multline}
\begin{multline}
\label{eq:Convergence_annulus_center_x_i_inner_radius_Rlambda_i_outer_radius_rho_common_Aprime}
\|v_i^{\coul}(\sigma_{\bar x_i'}^*A')\|_{L^4(\Omega(\bar x_i; 3R\lambda_i/2, 3\rho/4))}
\\
+ \|\nabla_\Theta v_i^{\coul}(\sigma_{\bar x_i'}^*A')\|_{L^2(\Omega(\bar x_i; 3R\lambda_i/2, 3\rho/4))}
<
C\eps_\annulus, \quad\text{for } 1 \leq i \leq l.
\end{multline}
\end{subequations}
Third, over the small balls, $B(\bar x_i, 7R\lambda_i/4)$ for $1 \leq i \leq l$, from
\eqref{eq:Convergence_ball_center_x_i_radius_Rlambda_i}, \eqref{eq:Convergence_ball_center_x_i_radius_Rlambda_i_prime}, and \eqref{eq:Distance_bar_xi_and_bar_xiprime}, we obtain
\begin{subequations}
\label{eq:Convergence_ball_center_x_i_radius_Rlambda_i_common}
\begin{multline}
\label{eq:Convergence_ball_center_x_i_radius_Rlambda_i_common_A}
\left\|u_i(A) - \varphi_{\bar x_i}^{-1,*}\delta_{\lambda_i}^*\varphi_n^*A_i\right\|_{L^4(B(\bar x_i, 7R\lambda_i/4))}
\\
+ \left\|\nabla_{\varphi_{\bar x_i}^{-1,*}\delta_{\lambda_i}^*\varphi_n^*A_i}
\left(u_i(A) - \varphi_{\bar x_i}^{-1,*}\delta_{\lambda_i}^*\varphi_n^*A_i\right)\right\|_{L^2(B(\bar x_i, 7R\lambda_i/4))} < \eps_\sphere,
\end{multline}
\begin{multline}
\label{eq:Convergence_ball_center_x_i_radius_Rlambda_i_common_Aprime}
\left\|v_i(A') - \varphi_{\bar x_i'}^{-1,*}\delta_{\lambda_i'}^*\varphi_n^*A_i\right\|_{L^4(B(\bar x_i, 7R\lambda_i/4))}
\\
+ \left\|\nabla_{\varphi_{\bar x_i'}^{-1,*}\delta_{\lambda_i'}^*\varphi_n^*A_i}
\left(v_i(A') - \varphi_{\bar x_i'}^{-1,*}\delta_{\lambda_i'}^*\varphi_n^*A_i\right)\right\|_{L^2(B(\bar x_i, 7R\lambda_i/4))} < \eps_\sphere.
\end{multline}
\end{subequations}
The estimate \eqref{eq:Convergence_ball_center_x_i_radius_Rlambda_i_common_Aprime} can be improved using \eqref{eq:Distance_bar_xi_and_bar_xiprime}, our redefinition \eqref{eq:Charts_as_geodesic_normal_coordinate_charts_and_translations} of the local coordinate charts, and Proposition \ref{prop:W12_continuity_family_connections_wrt_mass_centers_and_scales} which provides $W^{1,2}$ estimates caused by variations in the scale and mass center parameters in a family of $W^{2,2}$ connections, $\tilde h_{z_i,\lambda_i}^*A_i = (\tilde\delta_{\lambda_i}\circ \tilde\tau_{z_i})^*A_i$, over $S^4$ to give
\begin{align*}
\left\|\varphi_{\bar x_i}^{-1,*}\delta_{\lambda_i}^*\varphi_n^*A_i
-  \varphi_{\bar x_i'}^{-1,*}\delta_{\lambda_i'}^*\varphi_n^*A_i\right\|_{L^4(B(\bar x_i, 7R\lambda_i/4))}
&\leq C_1(\eps_\mycenter + \eps_\scale),
\\
\left\| \nabla_{\varphi_{\bar x_i}^{-1,*}\delta_{\lambda_i}^*\varphi_n^*A_i}
\left(\varphi_{\bar x_i}^{-1,*}\delta_{\lambda_i}^*\varphi_n^*A_i
- \varphi_{\bar x_i'}^{-1,*}\delta_{\lambda_i'}^*\varphi_n^*A_i \right) \right\|_{L^2(B(\bar x_i, 7R\lambda_i/4))}
&\leq C_2(\eps_\mycenter + \eps_\scale),
\end{align*}
for $1 \leq i \leq l$, where
\[
C_1 = C\|F_{A_i}\|_{L^4(S^4,g_\round)}
\quad\text{and}\quad
C_2 = C\|F_{A_i}\|_{W_{A_i,g_\round}^{1,2}(S^4)},
\]
for a universal constant $C \in [1,\infty)$, using Inequalities \eqref{eq:L4norm_bound_derivative_connection_wrt_scales},
\eqref{eq:L4norm_bound_derivative_connection_wrt_mass_centers},
\eqref{eq:L2norm_nabla_bound_derivative_connection_wrt_mass_centers},
\eqref{eq:L2norm_nabla_bound_derivative_connection_wrt_scales},
\eqref{eq:L4distance_connections_wrt_scales},
\eqref{eq:L4distance_connections_wrt_mass_centers},
\eqref{eq:L2nabla_distance_connections_wrt_mass_centers},
and \eqref{eq:L2nabla_distance_connections_wrt_scales}.

Therefore, we may replace \eqref{eq:Convergence_ball_center_x_i_radius_Rlambda_i_common_Aprime} by
\begin{multline}
\label{eq:Convergence_ball_center_x_i_radius_Rlambda_i_common_Aprime_original_center_scale}
\left\|v_i(A') - \varphi_{\bar x_i}^{-1,*}\delta_{\lambda_i}^*\varphi_n^*A_i\right\|_{L^4(B(\bar x_i, 7R\lambda_i/4))}
\\
+ \left\|\nabla_{\varphi_{\bar x_i}^{-1,*}\delta_{\lambda_i}^*\varphi_n^*A_i}
\left(v_i(A') - \varphi_{\bar x_i}^{-1,*}\delta_{\lambda_i}^*\varphi_n^*A_i\right)\right\|_{L^2(B(\bar x_i, 7R\lambda_i/4))}
\\
< C(\eps_\mycenter + \eps_\scale)\|F_{A_i}\|_{W_{A_i,g_\round}^{1,2}(S^4)} + C\eps_\sphere,
\end{multline}
for $1 \leq i \leq l$. By combining \eqref{eq:Convergence_X_minus_balls_center_x_i_radius_rho_common}, \eqref{eq:Convergence_ball_center_x_i_radius_Rlambda_i_common_A}, and \eqref{eq:Convergence_ball_center_x_i_radius_Rlambda_i_common_Aprime_original_center_scale}, we obtain
\begin{equation}
\label{eq:Convergence_manifold_complement_balls_common_A_minus_Aprime}
\|u_0(A) - v_0(A')\|_{W_{A_0}^{1,2}(X - B_{\rho/2}(\bar\bx))} < 2\eps_\background,
\end{equation}
and
\begin{multline}
\label{eq:Convergence_small_balls_common_A_minus_Aprime}
\|u_i(A) - v_i(A')\|_{L^4(B(\bar x_i, 7R\lambda_i/4))}
\\
+ \left\|\nabla_{\varphi_{\bar x_i}^{-1,*}\delta_{\lambda_i}^*\varphi_n^*A_i}
\left(u_i(A) - v_i(A')\right)\right\|_{L^2(B(\bar x_i, 7R\lambda_i/4))}
\\
< C(\eps_\mycenter + \eps_\scale)\|F_{A_i}\|_{W_{A_i,g_\round}^{1,2}(S^4)} + C\eps_\sphere,
\end{multline}
while the local connection one-forms $u_i^{\coul}(\sigma_{\bar x_i}^*A)$ and $v_i^{\coul}(\sigma_{\bar x_i'}^*A')$ over the annuli $\Omega(\bar x_i; 3R\lambda_i/2, 3\rho/4)$ are each small individually by \eqref{eq:Convergence_annulus_center_x_i_inner_radius_Rlambda_i_outer_radius_rho_common}. This completes Step \ref{step:proof_theorem_Taubes_1988_proposition_5-4_W12_estimate_A_and_Aprime_common}.
\end{step}

This completes our derivation of the norm bounds of the differences.

\begin{step}[Construction of the global gauge transformation $u$]
\label{step:proof_theorem_Taubes_1988_proposition_5-4_construction_global_gauge_transformation}
It remains to patch the local gauge transformations employed in the local estimates \eqref{eq:Convergence_manifold_complement_balls_common_A_minus_Aprime}, \eqref{eq:Convergence_small_balls_common_A_minus_Aprime}, and \eqref{eq:Convergence_annulus_center_x_i_inner_radius_Rlambda_i_outer_radius_rho_common} and construct the $W^{\bar k+1,\bar p}$ global gauge transformation, $u \in \Aut(P)$, as the global estimate \eqref{eq:Taubes_1988_proposition_5-4_fine} then follows almost immediately.

Indeed, we may use \eqref{eq:Convergence_X_minus_balls_center_x_i_radius_rho}
to replace the covariant derivative, $\nabla_{A_0}$, by the nearby covariant derivative $\nabla_{u_0(A)}$ in \eqref{eq:Convergence_manifold_complement_balls_common_A_minus_Aprime} to give
\begin{equation}
\label{eq:Convergence_manifold_complement_balls_common_A_minus_Aprime_W12A}
\|u_0(A) - v_0(A')\|_{W_{u_0(A)}^{1,2}(X - B_{\rho/2}(\bar\bx))} < C\eps_\background,
\end{equation}
for a universal constant $C \in [1, \infty)$. Second, we may use \eqref{eq:Convergence_ball_center_x_i_radius_Rlambda_i} to replace the covariant derivative, $\nabla_{\varphi_{\bar x_i}^{-1,*}\delta_{\lambda_i}^*\varphi_n^*A_i}$, by the nearby covariant derivative $\nabla_{u_i(A)}$ in \eqref{eq:Convergence_small_balls_common_A_minus_Aprime} to give
\begin{multline}
\label{eq:Convergence_small_balls_common_A_minus_Aprime_W12A}
\|u_i(A) - v_i(A')\|_{L^4(B(\bar x_i, 7R\lambda_i/4))}
\\
+ \left\|\nabla_{u_i(A)}
\left(u_i(A) - v_i(A')\right)\right\|_{L^2(B(\bar x_i, 7R\lambda_i/4))}
\\
< C(\eps_\mycenter + \eps_\scale)\|F_{A_i}\|_{W_{A_i,g_\round}^{1,2}(S^4)} + C\eps_\sphere.
\end{multline}
Third, we may replace the covariant derivative, $\nabla_\Theta$, by the nearby covariant derivative, $\nabla_{u_i^{\coul}(\sigma_{\bar x_i}^*A)}$, in \eqref{eq:Convergence_annulus_center_x_i_inner_radius_Rlambda_i_outer_radius_rho_common} to give
\begin{subequations}
\label{eq:Convergence_annulus_center_x_i_inner_radius_Rlambda_i_outer_radius_rho_common_W12A}
\begin{multline}
\label{eq:Convergence_annulus_center_x_i_inner_radius_Rlambda_i_outer_radius_rho_common_A_W12A}
\|u_i^{\coul}(\sigma_{\bar x_i}^*A)\|_{L^4(\Omega(\bar x_i; 3R\lambda_i/2, 3\rho/4))}
\\
+ \|\nabla_{u_i^{\coul}(\sigma_{\bar x_i}^*A)} u_i^{\coul}(\sigma_{\bar x_i}^*A)\|_{L^2(\Omega(\bar x_i; 3R\lambda_i/2, 3\rho/4))}
<
C\eps_\annulus, \quad\text{for } 1 \leq i \leq l,
\end{multline}
\begin{multline}
\label{eq:Convergence_annulus_center_x_i_inner_radius_Rlambda_i_outer_radius_rho_common_Aprime_W12A}
\|v_i^{\coul}(\sigma_{\bar x_i'}^*A')\|_{L^4(\Omega(\bar x_i; 3R\lambda_i/2, 3\rho/4))}
\\
+ \|\nabla_{u_i^{\coul}(\sigma_{\bar x_i}^*A)} v_i^{\coul}(\sigma_{\bar x_i'}^*A')\|_{L^2(\Omega(\bar x_i; 3R\lambda_i/2, 3\rho/4))}
<
C\eps_\annulus, \quad\text{for } 1 \leq i \leq l.
\end{multline}
\end{subequations}
Observe that all covariant derivatives in the preceding estimates are defined by the connection $A$, as desired for the estimate \eqref{eq:Taubes_1988_proposition_5-4_fine} to give, modulo the construction of global gauge transformation, $u \in \Aut(P)$,
\[
\|\nabla_A^g(u(A') - A)\|_{L^2(X,g)} + \|u(A') - A\|_{L^4(X,g)}
<
C\eps\left(1 + \sum_{\balpha \in \sT}\|F_{A_\balpha}\|_{W_{A_\balpha, g_\round}^{1,2}(S^4)}\right).
\]
As usual, replacement of the $L^4(X,g)$ norm in the preceding bound by the $W_{A,g}^{1,2}(X)$ norm is obtained via the standard Sobolev embedding, $W_g^{1,2}(X) \hookrightarrow L^4(X,g)$ given by \eqref{eq:Sobolev_embedding_W_1_2_into_L4} and the Kato Inequality \eqref{eq:FU_6-20_first-order_Kato_inequality}, and thus we obtain the desired estimate \eqref{eq:Taubes_1988_proposition_5-4_fine}.

The global gauge transformation, $u$ over $X$, is obtained by splicing the following local gauge transformations over the indicated overlapping open subsets of $X$:
\begin{enumerate}
\item $u_0^{-1}\circ v_0$ over $X \less \cup_{i=1}^l B_{\rho/2}(\bar x_i)$;

\item $\tau_{\bar x_i}^{-1}\circ (u_i^{\coul})^{-1}\circ v_i^{\coul}\circ \tau_{\bar x_i'}$ over $\Omega(\bar x_i; 3R\lambda_i/2, 3\rho/4)$, for $1 \leq i \leq l$, where $\tau_{\bar x_i}: P\restriction B(\bar x_i,\varrho) \cong B(\bar x_i,\varrho) \times G$ is the local trivialization corresponding to the radial-gauge local section, $\sigma_{\bar x_i}: B(\bar x_i,\varrho) \to P$ and $\tau_{\bar x_i'}: P\restriction B(\bar x_i',\varrho) \cong B(\bar x_i',\varrho) \times G$ is the local trivialization corresponding to the radial-gauge local section, $\sigma_{\bar x_i'}: B(\bar x_i',\varrho) \to P$.

\item $u_i^{-1}\circ v_i$ over $B(\bar x_i, 7R\lambda_i/4)$, for $1 \leq i \leq l$.
\end{enumerate}
This splicing is possible, just as in the construction of gauge transformations of $P$ over $X\less\{x_1,\ldots,x_l\}$ in the proof of Theorem \ref{thm:Sedlacek_4-3_Yang-Mills}, because over the annuli $\Omega(\bar x_i; 7R\lambda_i/4, 3\rho/4)$, the $L^2$ norms of the curvatures, $F_A$ and $F_{A'}$, of the Yang-Mills connections, $A$ and $A'$, are bounded by a small constant, $\eps_\annulus$, by virtue of  \eqref{eq:Convergence_annuli_center_x_i_inner_radius_Rlambda_i_outer_radius_2rho}, and hence each of the preceding local gauge transformations are close to the identity map over the overlaps. Replacement of the above local gauge transformations by the spliced gauge transformation, $u$, introduces small error in the estimates on overlaps but those are controlled by a universal multiple (say $C$) of the constants $\eps_\background$, $\eps_\annulus$, or $\eps_\sphere$ (for example, see \cite[Lemma 6.5]{Feehan_yangmillsenergygapflat} and its proof). This completes Step \ref{step:proof_theorem_Taubes_1988_proposition_5-4_construction_global_gauge_transformation}.
\end{step}

This completes the proof of Proposition \ref{prop:Taubes_1988_proposition_5-4_fine}.
\end{proof}

\subsection{Estimate for the $W^{1,2}$ distance between conformal blow-ups of a pair of Yang-Mills connections in a coarse $W_\loc^{1,2}$ bubble-tree open neighborhood}
\label{subsec:W12_bubble-tree_convergence_sequences_Yang-Mills_connections_coarse}
In Proposition \ref{prop:Taubes_1988_proposition_5-4_fine}, we consider the case where the gauge-equivalence classes of $g$-Yang-Mills connections, $[A]$ and $[A']$, belong to a \emph{fine} bubble-tree open neighborhood $\sU$ (Definition \ref{defn:Open_bubble_tree_neighborhood_fine}). Thus, if the tree $\sT$ defining that neighborhood has height one then, for each $i \in \{1,\ldots,l\}$, the local mass centers, $\bar x_i$ and $\bar x_i'$, and local scales, $\lambda_i$ and $\lambda_i'$, are close relative to $\min\{\lambda_i, \lambda_i'\}$ (as discussed in Section \ref{subsec:Centers_and_scales_for_bubble_tree_neighborhood_design}). We now turn to the case where $[A]$ and $[A']$ belong to a \emph{coarse} bubble-tree open neighborhood (Definition \ref{defn:Open_bubble_tree_neighborhood_coarse}), and so the local mass centers, $\bar x_i$ and $\bar x_i'$, are close relative to $\min\{\lambda_i, \lambda_i'\}$ for each $i \in \{1,\ldots,l\}$, but the local scales $\lambda_i$ and $\lambda_i'$ are not assumed to be close relative to $\min\{\lambda_i, \lambda_i'\}$.

We begin by recording estimates for the $g$-Yang-Mills connections $A$ and $A'$ in Proposition \ref{prop:Taubes_1988_proposition_5-4_fine} after rescaling with respect to their local scales to give connections over the complement in $S^4$ of small balls around the south pole. The proofs of these estimates are consequences of the proof of Proposition \ref{prop:Taubes_1988_proposition_5-4_fine}.

\begin{lem}[$W^{1,2}$ estimates for rescaled Yang-Mills connections over $S^4$]
\label{lem:Taubes_1988_proposition_5-4_W12_estimates_A_and_Aprime_sphere}
There is a universal constant $C \in [1,\infty)$ with the following significance. Continue the hypotheses and notation of the proof of Proposition \ref{prop:Taubes_1988_proposition_5-4_fine}, except that we instead allow $A$ and $A'$ to represent points $[A]$ and $[A']$ belonging to a \emph{coarse} $W_\loc^{1,2}$ bubble-tree $(\eps,\rho,R)$ neighborhood $\sU$. Denote $B_r = B(0,r) \subset \RR^4$. If the tree $\sT$ defining $\sU$ has height one, then over  $S^4-\varphi_s(B_{1/2R})$ we have,
\begin{multline}
\label{eq:Convergence_sphere_ball_center_x_i_radius_Rlambda_i}
\|\varphi_n^{-1,*}\delta_{\lambda_i}^{-1,*}\varphi_{\bar x_i}^*u_i(A) - A_i\|_{L^4(S^4-\varphi_s(B_{1/2R}))}
\\
+ \left\|\nabla_{A_i}^{g_\round}
\left(\varphi_n^{-1,*}\delta_{\lambda_i}^{-1,*}\varphi_{\bar x_i}^*u_i(A) - A_i\right)\right\|_{L^2(S^4-\varphi_s(B_{1/2R}))}
< C\eps_\sphere,
\end{multline}
with $W^{\bar k+1,\bar p}$ gauge transformations $u_i$ over the balls $X \supset B(\bar x_i, 2R\lambda_i) \cong S^4-\varphi_s(B_{1/2R})$, and
\begin{multline}
\label{eq:Convergence_sphere_ball_center_x_i_radius_Rlambda_i_prime}
\|\varphi_n^{-1,*}\delta_{\lambda_i'}^{-1,*}\varphi_{\bar x_i'}^*v_i(A') - A_i\|_{L^4(S^4-\varphi_s(B_{1/2R}))}
\\
+ \left\|\nabla_{A_i}^{g_\round}
\left(\varphi_n^{-1,*}\delta_{\lambda_i'}^{-1,*}\varphi_{\bar x_i'}^*v_i(A') - A_i\right)\right\|_{L^2(S^4-\varphi_s(B_{1/2R}))}
< C\eps_\sphere,
\end{multline}
with $W^{\bar k+1,\bar p}$ gauge transformations $v_i$ over the balls $X \supset B(\bar x_i', 2R\lambda_i') \cong S^4-\varphi_s(B_{1/2R})$. The analogous estimates hold over each copy of $S^4$ corresponding to a vertex in a tree $\sT$ of height greater than one.
\end{lem}

\begin{proof}
We apply \eqref{eq:Convergence_ball_center_x_i_radius_Rlambda_i} to give estimates over large balls in $S^4$,
\begin{equation}
\label{eq:Diffeomorphism_small_ball_in_X_around_xi_large_ball_in_S4_around_n}
\begin{aligned}
X \supset B(\bar x_i, 2R\lambda_i) \cong B(0,2R\lambda_i) &\cong B(0,2R)
\\
&\cong \varphi_n(B(0,2R)) = S^4-\varphi_s(B(0,1/2R)) \subset S^4,
\end{aligned}
\end{equation}
given by the diffeomorphisms $\varphi_{\bar x_i}^{-1}$, $\delta_{\lambda_i}$, and $\varphi_n$, respectively. We observe that \eqref{eq:Convergence_ball_center_x_i_radius_Rlambda_i} is equivalent, for $1 \leq i \leq l$, to the inequality,
\begin{multline}
\label{eq:Convergence_sphere_ball_center_x_i_radius_Rlambda_i_nabla_euclidean}
\|\varphi_n^{-1,*}\delta_{\lambda_i}^{-1,*}\varphi_{\bar x_i}^*u_i(A) - A_i\|_{L^4(S^4-\varphi_s(B_{1/2R}))}
\\
+ \left\|\nabla_{A_i}^\delta
\left(\varphi_n^{-1,*}\delta_{\lambda_i}^{-1,*}\varphi_{\bar x_i}^*u_i(A) - A_i\right)\right\|_{L^2(S^4-\varphi_s(B_{1/2R}))}
< \eps_\sphere,
\end{multline}
where we exploit the conformal invariance of the $L^4$ norm on sections of $T^*X$ and $L^2$ norm on sections of $T^*X\otimes T^*X$ and the behavior of the Levi-Civita connections with respect to the standard Euclidean metric $\delta = g^{\euclid}$ on $\RR^4$. Replacing $\nabla_{A_i}^\delta$ in \eqref{eq:Convergence_sphere_ball_center_x_i_radius_Rlambda_i_nabla_euclidean} by $\nabla_{A_i}^{g_\round}$, defined by the Levi-Civita connection for the standard round metric $g_\round$ of radius one on $S^4$, we obtain \eqref{eq:Convergence_sphere_ball_center_x_i_radius_Rlambda_i} where $C \in [1,\infty)$ is a universal constant. By applying \eqref{eq:Convergence_ball_center_x_i_radius_Rlambda_i_prime}, a similar argument yields the estimate \eqref{eq:Convergence_sphere_ball_center_x_i_radius_Rlambda_i_prime}.
\end{proof}

By combining the argument of Step \ref{step:proof_theorem_Taubes_1988_proposition_5-4_W12_estimate_A_and_Aprime_common} of the proof of Proposition \ref{prop:Taubes_1988_proposition_5-4_fine} and the proof of Lemma \ref{lem:Taubes_1988_proposition_5-4_W12_estimates_A_and_Aprime_sphere} and remembering that we now drop the assumption that $|\lambda_i'/\lambda_i - 1| < \eps_\scale$ (because this property does not hold for local scales defined by points $[A]$ and $[A']$ in a coarse bubble-tree neighborhood), we can adjust for the fact that the local centers $\bar x_i$ of $A$ and $\bar x_i'$ of $A'$ do not coincide but their distance is bounded by $\eps_\mycenter$. We thus obtain

\begin{lem}[$W^{1,2}$ estimates for the difference between a pair of rescaled Yang-Mills connections over $S^4$]
\label{lem:Taubes_1988_proposition_5-4_W12_estimates_A_and_Aprime_sphere_difference}
There is a universal constant $C \in [1,\infty)$ with the following significance. Continue the hypotheses and notation of the proof of Lemma \ref{lem:Taubes_1988_proposition_5-4_W12_estimates_A_and_Aprime_sphere}. If the tree $\sT$ defining $\sU$ has height one, then over $S^4-\varphi_s(B_{4/7R})$ we have
\begin{multline}
\label{eq:Convergence_sphere_common_A_minus_Aprime}
\|\varphi_n^{-1,*}\delta_{\lambda_i}^{-1,*}\varphi_{\bar x_i}^*u_i(A) - \varphi_n^{-1,*}\delta_{\lambda_i'}^{-1,*}\varphi_{\bar x_i}^*v_i(A')\|_{L^4(S^4-\varphi_s(B_{4/7R}))}
\\
+ \left\|\nabla_{A_i} \left(\varphi_n^{-1,*}\delta_{\lambda_i}^{-1,*}\varphi_{\bar x_i}^*u_i(A)
- \varphi_n^{-1,*}\delta_{\lambda_i'}^{-1,*}\varphi_{\bar x_i}^*v_i(A')\right)\right\|_{L^2(S^4-\varphi_s(B_{4/7R}))}
\\
< C\eps_\mycenter\|F_{A_i}\|_{W_{A_i,g_\round}^{1,2}(S^4)} + C\eps_\sphere,
\end{multline}
with $W^{\bar k+1,\bar p}$ gauge transformations $u_i$ over the balls $X \supset B(\bar x_i, 7R\lambda_i/4) \cong S^4-\varphi_s(B_{4/7R})$ and $v_i$ over the balls $X \supset B(\bar x_i, 7R\lambda_i'/4) \cong S^4-\varphi_s(B_{4/7R})$, for $1 \leq i \leq l$. The analogous estimates hold over each copy of $S^4$ corresponding to a vertex in a tree $\sT$ of height greater than one.
\end{lem}

We now establish the definitions we shall need to state and prove a version of Proposition \ref{prop:Taubes_1988_proposition_5-4_fine} for coarse bubble-tree neighborhoods. In our development of a generalization of Proposition \ref{prop:Taubes_1988_proposition_5-4_fine} for a pair of gauge-equivalence classes of $g$-Yang-Mills connections, $[A]$ and $[A']$, belonging to a \emph{coarse} $W_\loc^{1,2}$ bubble-tree neighborhood, there are different possible choices for a suitable $W^{1,2}$ Sobolev norm to measure the $W^{1,2}$ distance between the locally conformally blown-up Yang-Mills connections, $\tilde h_{\bz,\blambda}^{-1,*}A$ and $\tilde h_{\bz',\blambda'}^{-1,*}A'$:
\begin{enumerate}
  \item The standard $W^{1,2}$ Sobolev norm defined by the connection $A^\# = \tilde h_{\bz,\blambda}^{-1,*}A$ and Riemannian metric $g^\# = h_{\bz,\blambda}^{-1,*}g$ on $X\#_\sT S^4$ and its Levi-Civita connection $\nabla^{g^\#}$. The metric $g^\#$ is conformally equivalent to $g$ via the diffeomorphism $h_{\bz,\blambda}: X \cong X\#_\sT S^4$, as described in Section \ref{subsec:Construction_Riemannian_metric_connected_sum_for_bubble-tree_data}.
  \item An \adhoc $W^{1,2}$ Sobolev norm defined by the connection $A_0$ on $P_0$ over $X$ and its metric $g$ and the connections $A_i$ on $P_i$ over $S^4$ and its metric $g_\round$ for $1 \leq i \leq l$ in the case of height-one tree, $\sT = \{1,\ldots,l\}$, with similar comments in the case of trees, $\sT$, of height greater than one underlying a bubble-tree neighborhood $\sU$.
\end{enumerate}
The standard $W^{1,2}$ norm for elements of $\Omega^1(X;\ad P)$ defined by $(A^\#, g^\#)$ is more elegant, but the \adhoc $W^{1,2}$ norm defined by the bubble-tree data underlying $\sU$ is more convenient for certain calculations and estimates. For example, the standard $W^{1,2}$ norm,
\[
\|h_{\bz,\blambda}^{-1,*}a\|_{W_{A^\#, g^\#}^{1,2}(X)}, \quad\text{for } a \in \Omega^1(X;\ad P),
\]
depends on the choices of diffeomorphisms in Section \ref{subsec:Construction_Riemannian_metric_connected_sum_for_bubble-tree_data} that conformally identify the small annuli in $X$ centered at points $\bar x_i \in X$ with the thin necks joining copies of $S^4$ to $X$, for $1 \leq i \leq l$, with similar remarks applying to trees, $\sT$, of height greater than one; the \adhoc Sobolev $W^{1,2}$ norm does not depend on such choices.

More seriously, we shall need to compare $A$ over the annulus $\Omega(\bar x_i; R\lambda_i/2, 2\rho)$ with $A'$ over the annulus $\Omega(\bar x_i; R\lambda_i'/2, 2\rho)$, for $1 \leq i \leq l$; although concentric, these annuli are \emph{not} conformally equivalent (even when $g$ is flat). When $0 < \lambda_i' \ll \lambda_i$ for $1 \leq i \leq l$, as is the case when $A$ and $A'$ belong to a coarse bubble-tree neighborhood, the required comparison becomes more complicated than in the much simpler case where $|\lambda_i' - \lambda_i| < \eps_\scale\min\{\lambda_i, \lambda_i'\}$ for $1 \leq i \leq l$, when $A$ and $A'$ belong to a fine bubble-tree neighborhood and which we explored in Section \ref{subsec:W12_bubble-tree_convergence_sequences_Yang-Mills_connections_fine}. Fortunately, because the $L^2$ energies of $A$ and $A'$ over the necks $\Omega(\bar x_i; R\lambda_i/2, 2\rho)$ and $\Omega(\bar x_i; R\lambda_i'/2, 2\rho)$ are both small for each $1 \leq i \leq l$, the different possible choices of \adhoc $W^{1,2}$ norm used to compare $A$ and $A'$ will be equivalent and will lead to the same conclusions.

\begin{defn}[An \adhoc Sobolev $W^{1,2}$ norm for $\ad P$-valued one-forms defined by bubble-tree neighborhood data and conformal blow-ups]
\label{defn:W12_norm_via_bubble-tree_data_connected-sum_local_conformal_blow-up}
Continue the notation of Definition \ref{defn:Iterated_local_conformal_blowup_Sobolev_connection_Riemannian_metric}. Suppose first that the tree, $\sT$, specified by the bubble-tree neighborhood data has height one, so $\sT = \{1,\ldots,l\}$ and $\lambda_i \in (0,1]$ and $z_i = \varphi_i^{-1}(\bar x_i) \in \RR^4$ for $1 \leq i \leq l$. We suppress explicit notational dependence of the maps $\tilde h_{z_i,\lambda_i}$ on the oriented, orthonormal frames, $f_i$ for $(TX)_{x_i}$ for $1 \leq i \leq l$ and $f_0$ for $(TS^4)_n$, and parameters $R \in [1,\infty)$ and $\rho \in (0,1]$ prescribed by the data for the bubble-tree neighborhood, $\sU$, and define
\[
X\#_{h_{\bz,\blambda}}S^4 := X\mathop{\#}\limits_{i=1}^l {}_{h_{z_i,\lambda_i}}S^4,
\]
which we may abbreviate by $X\#_{i=1}^l S^4$ or even $X\#_\sT S^4$. We continue to denote by $P$ the pull back of the bundle, $P$, from $X$ to $X\#_{h_{\bz,\blambda}}S^4$ via the diffeomorphisms between those manifolds. For $a \in \Omega^1(X; \ad P)$ and $\rho \in (0,\Inj(X,g))$, we define
\begin{multline}
\label{eq:W12_norm_via_bubble-tree_data_connected-sum_local_conformal_blow-up}
\|h_{\bz,\blambda}^{-1,*}a\|_{\overline{W}_{\sU, g}^{1,2}(X\#_{i=1}^l S^4)}^2
:=
\|a\|_{W_{A_0, g}^{1,2}(X \less \cup_{i=1}^l B(\bar x_i,\rho/2))}^2
\\
+
\sum_{i=1}^l\|\tilde h_{z_i,\lambda_i}^{-1,*}a \|_{W_{A_i, g_\round}^{1,2}(S^4 \less \varphi_s(B(0,\lambda_i/2\rho)))}^2.
\end{multline}
More generally, the definition of the norm $\|h_{\bz,\blambda}^{-1,*}a\|_{\overline{W}_{\sU, g}^{1,2}(X\#_\sT S^4)}$, corresponding to connected sums $X\#_\sT S^4 = X\#_{h_{\bz,\blambda}}S^4$ defined by trees, $\sT$, of height greater than one, is obtained inductively from the definition \eqref{eq:W12_norm_via_bubble-tree_data_connected-sum_local_conformal_blow-up} for the case of a tree of height one.
\end{defn}

\begin{defn}[An \adhoc Sobolev $W^{1,2}$ distance between conformal blow-ups of a pair of connections in a coarse $W_\loc^{1,2}$ bubble-tree neighborhood]
\label{defn:W12_norm_via_bubble-tree_data_connected-sum_Aprime_minus_A_local_conformal_blow-up}
Continue the notation of Definition \ref{defn:W12_norm_via_bubble-tree_data_connected-sum_local_conformal_blow-up}. Suppose first that the tree, $\sT$, specified by the bubble-tree neighborhood data has height one and $\sT = \{1,\ldots,l\}$. For a pair of $W^{\bar k,\bar p}$ connections $A$ and $A'$ on a principal $G$-bundle over $X$ that represent points in a coarse bubble-tree neighborhood, $\sU \subset \sB(P,g)$, we define a corresponding pair of locally conformally blown-up connections, $\tilde h_{\bz,\blambda}^{-1,*}A$ and $\tilde h_{\bz',\blambda'}^{-1,*}A'$ over $X\#_{i=1}^l S^4$, via \eqref{eq:Iterated_conformal_blowup_Sobolev_connection} in Definition \ref{defn:Iterated_local_conformal_blowup_Sobolev_connection_Riemannian_metric}.

We define a $W^{1,2}$ distance between the pair of locally conformally blown-up connections over $X\#_{i=1}^l S^4$ by
\begin{multline}
\label{eq:W12_norm_via_bubble-tree_data_connected-sum_Aprime_minus_A_local_conformal_blow-up}
\|\tilde h_{\bz',\blambda'}^{-1,*}A' - \tilde h_{\bz,\blambda}^{-1,*}A\|_{\overline{W}_{\sU,g}^{1,2}(X\#_{i=1}^l S^4)}^2
:=
\|A' - A\|_{W_{A_0, g}^{1,2}(X \less \cup_{i=1}^l B(\bar x_i,\rho/2))}^2
\\
+ \sum_{i=1}^l\|\tilde h_{z_i',\lambda_i'}^{-1,*}A' - \tilde h_{z_i,\lambda_i}^{-1,*}A
\|_{W_{A_i, g_\round}^{1,2}(S^4 \less \{s\})}^2,
\end{multline}
where the pulled-back connections, $\tilde h_{z_i,\lambda_i}^{-1,*}A$ and $\tilde h_{z_i',\lambda_i'}^{-1,*}A'$, are defined, respectively, on open subsets of $S^4\less\{s\}$ that are conformally equivalent to $B(\bar x_i, 2\rho)$, namely when $\lambda_i' < \lambda_i$,
\begin{multline*}
B(0,2\rho/\lambda_i) \cong \varphi_n(B(0,2\rho/\lambda_i)) = S^4 \less \varphi_s(B(0,\lambda_i/2\rho))
\\
\subset S^4 \less \varphi_s(B(0,\lambda_i'/2\rho)) = \varphi_n(B(0,2\rho/\lambda_i')) \cong B(0,2\rho/\lambda_i'),
\end{multline*}
and with the reverse inclusion when $\lambda_i' > \lambda_i$. Over the complement of the ball,
\[
\RR^4\less B(0,2\rho/\lambda_i) \cong \varphi_n(\RR^4\less B(0,2\rho/\lambda_i)) = \varphi_s(B(0,\lambda_i/2\rho)) \less\{s\},
\]
where at most one of the two pulled-back connections is defined when $\lambda_i' < \lambda_i$, they may be compared by one of three methods described in Definition \ref{defn:Comparing_local_conformal_blow-ups_pair_connections_near_southern_poles}.

More generally, the definition of the norm
\[
\|\tilde h_{\bz',\blambda'}^{-1,*}A' - \tilde h_{\bz,\blambda}^{-1,*}A\|_{\overline{W}_{\sU,g}^{1,2}(X\#_\sT S^4)},
\]
corresponding to connected sums $X\#_\sT S^4$ defined by trees, $\sT$, of height greater than one, is obtained inductively from the definition \eqref{eq:W12_norm_via_bubble-tree_data_connected-sum_Aprime_minus_A_local_conformal_blow-up} for the case of a tree of height one.
\end{defn}

To complete Definition \ref{defn:W12_norm_via_bubble-tree_data_connected-sum_local_conformal_blow-up}, we need to explain the meaning of the last term in \eqref{eq:W12_norm_via_bubble-tree_data_connected-sum_Aprime_minus_A_local_conformal_blow-up}, as neither connection is defined over all of $S^4\less\{s\}$. There are several equivalent methods, as we explain in the

\begin{defn}[Comparing conformal blow-ups of a pair of connections near the southern poles of $S^4$ in the \adhoc Sobolev $W^{1,2}$ norm]
\label{defn:Comparing_local_conformal_blow-ups_pair_connections_near_southern_poles}
Continue the notation of Definition \ref{defn:W12_norm_via_bubble-tree_data_connected-sum_Aprime_minus_A_local_conformal_blow-up}. Suppose first that the tree, $\sT$, specified by the bubble-tree neighborhood data has height one and so $\sT = \{1,\ldots,l\}$.

\begin{method}[Cutting off]
Choose fiber points $p_i \in P|_{x_i}$, for $1 \leq i \leq l$. Over the annulus $\varphi_s(\Omega(0;\sqrt{\lambda_i'}/2, 2\sqrt{\lambda_i'})) \subset S^4 \less\{s\}$, apply a cutting-off construction similar to \eqref{eq:Cutoff_connection_equal_to_A_over_complement_balls_and_product_over_balls} to replace the pulled-back connection, $\tilde h_{z_i',\lambda_i'}^{-1,*}A'$, by a cut-off connection that is equal to $\tilde h_{z_i',\lambda_i'}^{-1,*}A'$ on $S^4 \less \varphi_s(B(0; 2\sqrt{\lambda_i'}))$ and equal to the product connection, $\Theta$, on $\varphi_s(B(0;\sqrt{\lambda_i'}/2))$.

Similarly, over the annulus $\varphi_s(\Omega(0;\sqrt{\lambda_i}/2, 2\sqrt{\lambda_i})) \subset S^4 \less\{s\}$, apply \eqref{eq:Cutoff_connection_equal_to_A_over_complement_balls_and_product_over_balls} to replace the pulled-back connection, $\tilde h_{z_i,\lambda_i}^{-1,*}A$, by a cut-off connection that is equal to $\tilde h_{z_i,\lambda_i}^{-1,*}A$ on $S^4 \less \varphi_s(B(0; 2\sqrt{\lambda_i}))$ and equal to the product connection, $\Theta$, on $\varphi_s(B(0;\sqrt{\lambda_i}/2))$.

The norm \eqref{eq:W12_norm_via_bubble-tree_data_connected-sum_Aprime_minus_A_local_conformal_blow-up} is now well-defined since the difference of the two cut-off connections is defined on $\RR^4 = S^4\less\{s\}$.
\end{method}

\begin{method}[Coning off] We adapt a construction employed by Parker \cite[Equation (1.12)]{ParkerHarmonic} in the context of harmonic maps of a Riemann surface into a target Riemannian manifold.

Let $x_0 \in X$, choose $p_0 \in P_{x_0}$, and use the connection, $A$, to parallel translate $p_0$ to a fiber point $\bar p_0 \in P_{\bar x_0}$. Let $\sigma_{\bar x_0}:B(\bar x_0,\varrho) \to P$ be the corresponding local section of $P$ defined by parallel translation via $A$ of points in the fiber, $P_{\bar x_0}$, along radial geodesics emanating from $\bar x_0$.

We apply rescaling to $\varphi_{x_0}^*\sigma_{x_0}^*A(x)$ for $|x| < \varrho$, so $x = \delta_\lambda^{-1}(\tilde x) = \lambda \tilde x$ for $\tilde x \in \RR^4$ and we rescale $\varphi_{x_0}^*\sigma_{x_0}^*A$ on $B(0,\varrho)$ as
\[
\delta_\lambda^{-1,*}\varphi_{x_0}^*\sigma_{x_0}^*A(\tilde x),
\quad |\tilde x| < \varrho/\lambda.
\]
Let $\tilde y := \iota(\tilde x) = \tilde x^{-1} = \lambda x^{-1} = \lambda y$, where $y  := \iota(x) := x^{-1}$ for $x \in \RR^4\less\{0\}$. Therefore, $\tilde x = \tilde y^{-1} = \iota(\tilde y)$ and the rescaled local connection one-form, $\varphi_{x_0}^*\sigma_{x_0}^*A$ on $B(0,\varrho)$, is equivalent to
\[
\iota^*\delta_\lambda^{-1,*}\varphi_{x_0}^*\sigma_{x_0}^*A(\tilde y)
=
\delta_\lambda^*\iota^*\varphi_{x_0}^*\sigma_{x_0}^*A(\tilde y),
\quad |\tilde y| > \lambda/\varrho,
\]
using the fact that $\iota(\delta_\lambda(z)) = \lambda z^{-1} = \delta_\lambda^{-1}(\iota(z))$, for $z\in \RR^4$.

Given $a \in \Omega^1(\RR^4\less B(0,r_0); \fg)$ for some $r_0 > 0$ and employing polar coordinates, $y = r\theta$ with $r \in \RR_+$ and $\theta \in S^3$, we define
\begin{equation}
\label{eq:Coned-off_local_one-form}
\hat a(y)
:=
\begin{cases}
r r_0^{-1}a(r_0\theta), & r \leq r_0,
\\
a(r\theta), & r > r_0.
\end{cases}
\end{equation}
It follows from \cite[Theorem 7.8]{GilbargTrudinger} that if $a \in W_{\loc}^{1,p}(\RR^4\less B(0,r_0); \Lambda^1\otimes\fg)$ for $p \geq 1$, then $\hat a \in W_{\loc}^{1,p}(\RR^4; \Lambda^1\otimes\fg)$.

We now apply the preceding coning off procedure to the local connection one-forms, $a_i := \iota^*\delta_{\lambda_i}^{-1,*}\varphi_{\bar x_i}^*\sigma_{\bar x_i}^*A$ on $\RR^4\less B(0,2\lambda_i/\varrho)$, and $a_i' := \iota^*\delta_{\lambda_i'}^{-1,*}\varphi_{\bar x_i}^*\sigma_{\bar x_i'}^*A'$ on $\RR^4\less B(0,2\lambda_i'/\varrho)$ to define $\hat a_i \in W_{\loc}^{1,2}(\RR^4; \Lambda^1\otimes\fg)$ and $\hat a_i' \in W_{\loc}^{1,2}(\RR^4; \Lambda^1\otimes\fg)$, respectively. (Note that we use the same inverse geodesic normal coordinate chart, $\varphi_{\bar x_i}$, to pull back each local connection one-form from $B(\bar x_i,\varrho/2) \subset X$ to $B(0,\varrho/2) \subset \RR^4$.) The norm \eqref{eq:W12_norm_via_bubble-tree_data_connected-sum_Aprime_minus_A_local_conformal_blow-up} is now well-defined since the two coned-off local connection one-forms are defined on $\RR^4 = S^4\less\{s\}$.
\end{method}

\begin{method}[Non-conformal stretching]
For a given index, $i \in \{1,\ldots,l\}$, we may suppose without loss of generality that $\lambda_i' < \lambda_i$. Choose constants $\alpha_i \in [1,\infty)$ such that $(R\lambda_i/4\rho)^{\alpha_i} = R\lambda_i'/4\rho$, where $0 < R\lambda_i'/4\rho \leq R\lambda_i/4\rho \leq 1$, and define local stretching maps on $\RR^4\less\{0\}$ by $\cS_i(x) := 2\rho(r/2\rho)^{\alpha_i}\theta$, where $r = |x|$ and $\theta  = r^{-1}x \in S^3$ for $x \in \RR^4 \less \{0\}$. Then $\cS_i$ is a diffeomorphism from $\Omega(0; R\lambda_i/2, 2\rho)$ onto $\Omega(0; R\lambda_i'/2,2\rho)$. Transfer these local diffeomorphisms to $X$ via the local coordinate charts, $\varphi_{\bar x_i}$, and denote the resulting maps by $\tilde\cS_i$, for $1 \leq i \leq l$. Replace \eqref{eq:W12_norm_via_bubble-tree_data_connected-sum_Aprime_minus_A_local_conformal_blow-up} by
\begin{multline}
\label{eq:W12_norm_via_bubble-tree_data_connected-sum_Aprime_minus_A_local_conformal_blow-up_and_stretching}
\|\tilde h_{\bz',\blambda'}^{-1,*}A' - \tilde h_{\bz,\blambda}^{-1,*}A\|_{\overline{W}_{\sU,g}^{1,2}(X\#_{i=1}^l S^4)}^2
\\
:=
\|A' - A\|_{W_{A_0, g}^{1,2}(X \less \cup_{i=1}^l B(\bar x_i,\rho/2))}^2
+
\sum_{i=1}^l\|A' - \tilde\cS_i^{-1,*}A\|_{W_{A_0, g}^{1,2}(\Omega(\bar x_i; R\lambda_i'/2,2\rho))}^2
\\
+ \sum_{i=1}^l\|\tilde h_{z_i',\lambda_i'}^{-1,*}A' - \tilde h_{z_i,\lambda_i}^{-1,*}A
\|_{W_{A_i, g_\round}^{1,2}(S^4 \less \varphi_s(B(0,1/2R)))}^2.
\end{multline}
\end{method}

The preceding constructions extend \mutatis to the case of trees, $\sT$, of height greater than one.
\end{defn}

In practice, the three methods described in Definition \ref{defn:Comparing_local_conformal_blow-ups_pair_connections_near_southern_poles} for completing Definition \ref{defn:W12_norm_via_bubble-tree_data_connected-sum_Aprime_minus_A_local_conformal_blow-up} are equivalent in our application since the connections $A$ and $A'$ will be $g$-Yang-Mills and have small $L^2$ energies, respectively, over the annuli $\Omega(\bar x_i; R\lambda_i/2,2\rho)$ and $\Omega(\bar x_i'; R\lambda_i'/4,4\rho)$ and therefore over the concentric annuli, $\Omega(\bar x_i; R\lambda_i/2,2\rho) \subset \Omega(\bar x_i; R\lambda_i'/2,2\rho)$. This follows because $\dist_g(\bar x_i', \bar x_i) < \eps_\mycenter\lambda_i'$ and $|\bar z_i' - \bar z_i| < \eps_\mycenter\lambda_i'$ by definition of the coarse $W_\loc^{1,2}$ bubble-tree open neighborhood when $0 < \lambda_i' \leq \lambda_i$. Recall that $\bar z_i' = \varphi_{x_i}^{-1}(\bar x_i')$ and $\bar z_i = \varphi_{x_i}^{-1}(\bar x_i)$, where the local geodesic normal coordinate chart, $\varphi_{x_i}^{-1} \equiv \exp_{f_i}^{-1}: X \supset B(x_i,\varrho) \cong B(0,\varrho) \subset \RR^4$, is defined by the frame, $f_i$, for $(TX)_{x_i}$, and its exponential map.

\begin{prop}[Estimate for the Sobolev $W^{1,2}$ distance between conformal blow-ups of a pair of connections in a coarse  $W_\loc^{1,2}$ bubble-tree neighborhood]
\label{prop:Taubes_1988_proposition_5-4_coarse}
Let $G$ be a compact Lie group and $X$ be a closed, connected, four-dimensional, oriented, smooth manifold endowed with a smooth Riemannian metric, $g$. Then there is a constant, $C = C(g,G) \in [1, \infty)$, with the following significance. Let $P$ be a principal $G$-bundle over $X$. Then there are constants $\eps, \rho \in (0, 1]$ and $R \in [1,\infty)$ with the following significance. Let $A$ and $A'$ be $g$-Yang-Mills connections on $P$ of class $W^{\bar k,\bar p}$, with $\bar p\geq 2$ and integer $\bar k\geq 1$ obeying $(\bar k+1)\bar p>4$. If the gauge-equivalence classes, $[A]$ and $[A']$, belong to a \emph{coarse} $W_\loc^{1,2}$ bubble-tree $(\eps,\rho,R)$ open neighborhood, $\sU \subset \sB(P,g)$, then there are a gauge transformation, $u \in \Aut(P)$, of class $W^{\bar k+1,\bar p}$ and conformal blow-up diffeomorphisms, $\tilde h_{\bz,\blambda}$ and $\tilde h_{\bz',\blambda'}$, of $(X,g)$ such that
\begin{equation}
\label{eq:Taubes_1988_proposition_5-4_coarse}
\|\tilde h_{\bz',\blambda'}^{-1,*}u(A') - \tilde h_{\bz,\blambda}^{-1,*}A\|_{\overline{W}_{\sU,g}^{1,2}(X\#_\sT S^4)}
\\
<
C\eps\left(1 + \sum_{\balpha \in \sT}\|F_{A_\balpha}\|_{W_{A_\balpha, g_\round}^{1,2}(S^4)}\right),
\end{equation}
where the norm on the left-hand side of \eqref{eq:Taubes_1988_proposition_5-4_coarse} is defined by \eqref{eq:W12_norm_via_bubble-tree_data_connected-sum_Aprime_minus_A_local_conformal_blow-up} in Definition \ref{defn:W12_norm_via_bubble-tree_data_connected-sum_Aprime_minus_A_local_conformal_blow-up}, and the multi-indices, $\balpha = \{i\}, \{ij\}, \ldots$, appearing in the right-hand side of \eqref{eq:Taubes_1988_proposition_5-4_coarse} label the vertices of the finite tree $\sT$ defined by $\sU$. The pulled-back connection, $\tilde h_{\bz,\blambda}^{-1,*}A$, on the pull-back of $P$ over $X \#_\sT S^4$ by $\tilde h_{\bz,\blambda}:X \cong X \#_\sT S^4$ is Yang-Mills with respect to the pulled-back Riemannian metric, $\tilde h_{\bz,\blambda}^{-1,*}g$, that is conformally equivalent to $g$ and the analogous comments apply to $\tilde h_{\bz',\blambda'}^{-1,*}A$ and $\tilde h_{\bz',\blambda'}^{-1,*}g$.
\end{prop}

\begin{proof}
The argument is the same as that for Proposition \ref{prop:Taubes_1988_proposition_5-4_fine}, except that we apply Lemma \ref{lem:Taubes_1988_proposition_5-4_W12_estimates_A_and_Aprime_sphere_difference} to estimate the norms of the differences between the pairs of connections over each copy of $S^4$ rather than each ball in $X$ and appeal to the Definition \ref{defn:W12_norm_via_bubble-tree_data_connected-sum_Aprime_minus_A_local_conformal_blow-up} of the norm on the left-hand-side of the inequality \eqref{eq:Taubes_1988_proposition_5-4_coarse}.
\end{proof}

\subsection{Conformal blow-up diffeomorphisms and coarse $W_\loc^{1,2}$ bubble-tree neighborhoods}
\label{subsec:Global_conformal_blow-up_maps_coarse_W12_bubble-tree_neighborhoods}
Given Proposition \ref{prop:Taubes_1988_proposition_5-4_coarse}, we can build a finite collection of coarse $W_\loc^{1,2}$ bubble-tree open neighborhoods based on the existence of suitable conformal blow-up diffeomorphisms for a closed, connected, four-dimensional, oriented Riemannian, smooth manifold, $(X,g)$.

\begin{cor}[Conformal blow-up diffeomorphisms and coarse $W_\loc^{1,2}$ bubble-tree open neighborhoods]
\label{cor:Taubes_1988_proposition_5-4_coarse_conformal}
Let $G$ be a compact Lie group and $P$ be a principal $G$-bundle over a closed, connected, four-dimensional, oriented, smooth manifold, $X$, endowed with a smooth Riemannian metric, $g$. Let $\sU \subset \sB(P,g)$ be a coarse $W_\loc^{1,2}$ bubble-tree $(\eps,\rho,R)$ open neighborhood of $[A]$ in $\sB(P,g)$, with finite tree $\sT$ defined by $\sU$. Then the open subset $\sU' \subset \sU$ of points $[A'] \in \sB(P,g)$ such that
\begin{equation}
\label{eq:Coarse_W12_bubble-tree_neighborhood_via_conformal_blow-up_map}
\|h^{\prime,-1,*}u(A') -h^{-1,*} A\|_{\overline{W}_{\sU, g}^{1,2}(X\#_\sT S^4)} < \eps,
\end{equation}
for conformal blow-up diffeomorphisms, $h, h' \in \Conf(X,g)$, and $C^\infty$ gauge transformations, $u \in \Aut(P)$, is a coarse $W_\loc^{1,2}$ bubble-tree open neighborhood of $[A]$ in $\sB(P,g)$. Moreover, if $\sC \Subset [0,\infty)$ is a compact subset, then $\Crit(P,g,\sC)$ has a finite cover by neighborhoods of the form $\sU'\cap\Crit(P,g,\sC)$ defined by \eqref{eq:Coarse_W12_bubble-tree_neighborhood_via_conformal_blow-up_map}.
\end{cor}

\begin{proof}
The fact that $\sU'$ as specified by \eqref{eq:Coarse_W12_bubble-tree_neighborhood_via_conformal_blow-up_map} is a coarse $W_\loc^{1,2}$ bubble-tree open neighborhood follows from Definitions \ref{defn:Open_bubble_tree_neighborhood} and \ref{defn:Open_bubble_tree_neighborhood_coarse}. Proposition \ref{prop:Taubes_1988_proposition_5-4_coarse} ensures that every pair of points $[A], [A'] \in \Crit(P,g,\sC)$ belonging to a coarse $W_\loc^{1,2}$ bubble-tree $(\eps',\rho,R)$ open neighborhood obey an inequality of the form \eqref{eq:Coarse_W12_bubble-tree_neighborhood_via_conformal_blow-up_map} with $\eps = C\eps'$. Corollary \ref{cor:Donaldson_Kronheimer_4-4-3_Yang-Mills_L2-energy_compact_range} ensures that $\overline{\Crit}^{\,\tau}(P,g,\sC)$, and thus $\Crit(P,g,\sC)$, has a finite cover by coarse $W_\loc^{1,2}$ bubble-tree open neighborhoods.
\end{proof}

\section{{\L}ojasiewicz-Simon gradient inequality and the bubble-tree compactification}
\label{sec:Lojasiewicz-Simon gradient inequality_bubble-tree_compactification}
We begin in Section \ref{subsec:Lojasiewicz-Simon gradient inequality_Banach} by reviewing the abstract {\L}ojasiewicz-Simon gradient inequality \cite[Theorem 2.4.5]{Huang_2006} for an analytic potential function on an open subset of a Banach space. In Section \ref{subsec:Lojasiewicz-Simon gradient inequality_Yang-Mills_arbitrary_dimensional_manifolds}, we review the {\L}ojasiewicz-Simon gradient inequality from our monograph \cite{Feehan_yang_mills_gradient_flow} for the Yang-Mills $L^2$-energy functional over closed, Riemannian, smooth manifolds of arbitrary dimension. We discuss in Section \ref{subsec:Lojasiewicz-Simon gradient inequality_Yang-Mills_4-manifolds_and_conformal_invariance} the specialization of the {\L}ojasiewicz-Simon gradient inequality for the Yang-Mills $L^2$-energy functional to the case of four-dimensional manifolds and the weaker, quasi-conformally invariant version that will suffice for our proof of the main results of this article. In Section \ref{subsec:Lojasiewicz-Simon constants_Yang-Mills_W12_variation_connections}, we examine the variation of the {\L}ojasiewicz-Simon constants with respect to a $W^{1,2}$ variation of the Yang-Mills connection over a four-dimensional manifold. Finally, in Section \ref{subsec:Completion_proof_main_theorems}, we complete the proofs of the main results of this article by applying our {\L}ojasiewicz-Simon gradient inequality and our conclusion that one can choose a single {\L}ojasiewicz-Simon triple of constants, $(Z,\sigma,\theta)$, that is valid for all points $[A] \in \Crit(P,g,\sC)$, using the bubble-tree compactification of $\Crit(P,g,\sC)$ discussed in Sections \ref{subsec:Bubble-tree_convergence_Yang-Mills_connections}, \ref{subsec:Bubble_tree_compactification_moduli_space_anti-self-dual_connections}, and \ref{subsec:Bubble_tree_compactification_moduli_space_Yang-Mills_connections}.

\subsection{{\L}ojasiewicz-Simon gradient inequality for an analytic potential function on an open subset of a Banach space}
\label{subsec:Lojasiewicz-Simon gradient inequality_Banach}
We recall the following generalization of Simon's infinite-dimensional version \cite[Theorem 3]{Simon_1983} of the {\L}ojasiewicz gradient inequality \cite{Lojasiewicz_1965}.

\begin{thm}[Abstract {\L}ojasiewicz-Simon gradient inequality with dual Banach space gradient norm]
\label{thm:Huang_2-4-5}
\cite[Proposition 3.3]{Huang_Takac_2001}
\cite[Theorem 2.4.5]{Huang_2006}
Let $\sX$ be a Banach space and $\sH$ a Hilbert space such that the embeddings $\sX \hookrightarrow \sH \hookrightarrow \sX'$ are continuous. Let $\sE:\sX\to\RR$ be an analytic function and let $\varphi$ be a critical point of $\sE$, that is, $\sE'(\varphi) = 0$. Assume that $\sE''(\varphi):\sX\to \sX'$ is a linear, Fredholm operator of index zero. Then there are positive constants, $c$, $\sigma$, and
$\theta \in [1/2, 1)$ such that
\begin{equation}
\label{eq:Simon_2-2_dualspacenorm}
\|\sE'(u)\|_{\sX'} \geq c|\sE(u) - \sE(\varphi)|^\theta, \quad \forall\, u \in \sU \hbox{ such that } \|u-\varphi\|_\sX < \sigma.
\end{equation}
\end{thm}

We next recall our application of Theorem \ref{thm:Huang_2-4-5} to the case of the Yang-Mills $L^2$-energy functional and generalization of R\r{a}de's \cite[Proposition 7.2]{Rade_1992}, where a {\L}ojasiewicz-Simon gradient inequality is proved for the Yang-Mills $L^2$-energy functional over closed Riemannian manifolds of dimensions two or three.

\subsection{{\L}ojasiewicz-Simon gradient inequality for the Yang-Mills $L^2$-energy functional over closed manifolds of arbitrary dimension}
\label{subsec:Lojasiewicz-Simon gradient inequality_Yang-Mills_arbitrary_dimensional_manifolds}
Our {\L}ojasiewicz-Simon gradient inequality for the Yang-Mills $L^2$-energy functional is one of the key technical ingredients underlying the proof of Theorem \ref{mainthm:Discreteness_Yang-Mills_energies}. We begin by recalling its statement from our monograph \cite{Feehan_yang_mills_gradient_flow}.

\begin{thm}[{\L}ojasiewicz-Simon gradient inequality for the Yang-Mills $L^2$-energy functional]
\label{thm:Rade_proposition_7-2}
\cite[Theorem 21.8]{Feehan_yang_mills_gradient_flow}
Let $(X,g)$ be a closed, Riemannian, smooth manifold of dimension $d$, and $G$ be a compact Lie group, $A_\myref$ a connection of class $C^\infty$, and $A_\ym$ a $g$-Yang-Mills connection of class $W^{\bar k,\bar p}$, with integer $\bar k\geq 1$ and $\bar p \geq 1$ obeying $(\bar k+1)\bar p > d$, on a principal $G$-bundle, $P$, over $X$. If $d\geq 2$ and $p \in (1,\infty)$ obey one of the following conditions,
\begin{enumerate}
\item $d \leq 2 \leq 4$ and $p = 2$, or
\item $d \geq 5$ and $p \geq \max\{d/3, 4d/(d+4)\}$,
\end{enumerate}
then there are positive constants $Z \in [1, \infty)$, $\sigma \in (0,1]$, and $\theta \in [1/2,1)$, depending on the gauge-equivalence classes $[A_\myref]$ and $[A_\ym]$, $g$, $G$, $p$, $P$, and $X$ with the following significance.  If $A$ is a $W^{\bar k,\bar p}$ Sobolev connection on $P$ and
\begin{equation}
\label{eq:Rade_7-1_neighborhood}
\|A - A_\ym\|_{W^{1,p}_{A_\myref}(X)} < \sigma,
\end{equation}
then
\begin{equation}
\label{eq:Rade_7-1}
\|d_A^{*,g}F_A\|_{W^{-1,p'}_{A_\myref}(X)} \geq Z|\sE_g(A) - \sE_g(A_\ym)|^\theta,
\end{equation}
where $p'\in (1,\infty)$ is the dual exponent defined by $1/p+1/p'=1$ and $\sE_g(A)$ is given by \eqref{eq:Yang-Mills_energy_functional}.
\end{thm}

\subsection{The {\L}ojasiewicz-Simon gradient inequality for the Yang-Mills $L^2$-energy functional over closed four-dimensional manifolds and conformal invariance}
\label{subsec:Lojasiewicz-Simon gradient inequality_Yang-Mills_4-manifolds_and_conformal_invariance}
We now specialize Theorem \ref{thm:Rade_proposition_7-2} to the case of four-dimensional manifolds and make observations concerning the quasi-invariance property of the {\L}ojasiewicz-Simon triple of constants with respect to conformal diffeomorphisms of the manifold or conformal changes in the Riemannian metric. We first take $d=4$ and $p=2$ and thus $p'=2$ in Theorem \ref{thm:Rade_proposition_7-2} to give

\begin{thm}[{\L}ojasiewicz-Simon gradient inequality for the Yang-Mills $L^2$-energy functional in dimension four]
\label{thm:Rade_proposition_7-2_d_is_4}
Let $(X,g)$ be a closed, four-dimensional, Riemannian, smooth manifold, and $G$ be a compact Lie group, $A_\myref$ a connection of class $C^\infty$, and $A_\ym$ a $g$-Yang-Mills connection of class $W^{\bar k,\bar p}$, with integer $\bar k\geq 1$ and $\bar p \geq 2$ obeying $(\bar k+1)\bar p > 4$, on a principal $G$-bundle, $P$, over $X$. Then there are positive constants $Z \in [1,\infty)$, $\sigma\in (0,1]$, and $\theta \in [1/2,1)$, depending on the gauge-equivalence classes, $[A_\myref]$ and $[A_\ym]$, and $g$, $G$, $P$, and $X$ with the following significance.  If $A$ is a $W^{\bar k,\bar p}$ Sobolev connection on $P$ and
\begin{equation}
\label{eq:Rade_7-1_neighborhood_d_is_4}
\|A - A_\ym\|_{W_{A_\myref, g}^{1,2}(X)} < \sigma,
\end{equation}
then
\begin{equation}
\label{eq:Rade_7-1_d_is_4}
\|d_A^{*,g}F_A\|_{W_{A_\myref,g}^{-1,2}(X)} \geq Z|\sE_g(A) - \sE_g(A_\ym)|^\theta.
\end{equation}
\end{thm}

The norm on the left-hand side of the inequality \eqref{eq:Rade_7-1_d_is_4} is \emph{not} invariant (or quasi-invariant) with respect to conformal diffeomorphisms of $(X,g)$, but it is bounded above by a norm that is conformally invariant. To see this and keep track of the constants depending on the Riemannian metric, $g$, we recall that $W_g^{1,2}(X) \hookrightarrow L^4(X,g)$ is a continuous embedding by \eqref{eq:Sobolev_embedding_manifold_bounded_geometry} and thus $(L^4(X,g))' = L^{4/3}(X,g) \hookrightarrow (W_g^{1,2}(X))' = W^{-1,2}(X,g)$ and
\[
\|f\|_{W_g^{-1,2}(X)} \leq z_1\|f\|_{L^{4/3}(X,g)} \leq z_1(\vol_g(X))^{1/2}\|f\|_{L^4(X,g)},
\]
where $z_1 = z_1(g) \in [1,\infty)$ is the Sobolev embedding constant for $W_g^{1,2}(X) \hookrightarrow L^4(X,g)$, namely
\begin{equation}
\label{eq:Sobolev_embedding_W_1_2_into_L4}
\|f\|_{W_g^{1,2}(X)} \leq z_1\|f\|_{L^4(X,g)}.
\end{equation}
Consequently, for the same constants, we have
\begin{equation}
\label{eq:Sobolev_embedding_W_minus_1_2_into_L4over3_into_L4}
\|a\|_{W_{A,g}^{-1,2}(X)} \leq z_1\|a\|_{L^{4/3}(X,g)} \leq z_1(\vol_g(X))^{1/2}\|a\|_{L^4(X,g)},
\quad \forall\, a \in W_{A,g}^{-1,2}(X; \Lambda^1\otimes\ad P).
\end{equation}
Taking the supremum below over all $b \in W_{A,g}^{1,2}(X; \Lambda^1\otimes\ad P)\less\{0\}$,
\begin{align*}
\|a\|_{W_{A,g}^{-1,2}(X)}
&:= \sup_{b \neq 0} \frac{(a, b)_{L^2(X,g)}}{\|b\|_{W_{A,g}^{1,2}(X)}}
\\
&\leq \sup_{b \neq 0} \frac{\|a\|_{L^{4/3}(X,g)} \|b\|_{L^4(X,g)}}{\|b\|_{W_{A,g}^{1,2}(X)}}
\\
&\leq \sup_{b \neq 0} \frac{z_1\|a\|_{L^{4/3}(X,g)} \||b|\|_{W_g^{1,2}(X)}}{\|b\|_{W_{A,g}^{1,2}(X)}}
\quad\text{(by \eqref{eq:Sobolev_embedding_W_1_2_into_L4})}
\\
&\leq \sup_{b \neq 0} \frac{z_1\|a\|_{L^{4/3}(X,g)} \|b\|_{W_{A,g}^{1,2}(X)}}{\|b\|_{W_{A,g}^{1,2}(X)}}
\quad\text{(by the Kato Inequality \eqref{eq:FU_6-20_first-order_Kato_inequality})}
\\
&= z_1\|a\|_{L^{4/3}(X,g)}
\\
&\leq z_1(\vol_g(X))^{3/4 - 1/4} \|a\|_{L^4(X,g)} \quad\text{(by \cite[Equation (7.8)]{GilbargTrudinger})},
\end{align*}
and this proves \eqref{eq:Sobolev_embedding_W_minus_1_2_into_L4over3_into_L4}. The latter inequality immediately yields the following special case of Theorem \ref{thm:Rade_proposition_7-2_d_is_4}.

\begin{cor}[{\L}ojasiewicz-Simon gradient inequality for the Yang-Mills $L^2$-energy functional in dimension four]
\label{cor:Rade_proposition_7-2_d_is_4_conformally_invariant}
Assume the hypotheses of Theorem \ref{thm:Rade_proposition_7-2_d_is_4}, let $(Z,\sigma,\theta)$ denote the corresponding {\L}ojasiewicz-Simon triple of constants, and let
\[
Z_0 := Zz_1^{-1}(\vol_g(X))^{-1/2}
\]
If $A$ is a $W^{\bar k,\bar p}$ Sobolev connection on $P$ that obeys \eqref{eq:Rade_7-1_neighborhood_d_is_4}, then
\begin{equation}
\label{eq:Rade_7-1_conformally_invariant}
\|d_A^{*,g}F_A\|_{L^4(X,g)} \geq Z_0|\sE_g(A) - \sE_g(A_\ym)|^\theta.
\end{equation}
\end{cor}

This is also a convenient point at which to record the following consequence of the Kato Inequality \eqref{eq:FU_6-20_first-order_Kato_inequality} and Sobolev Embedding \eqref{eq:Sobolev_embedding_domain} in the form $W^{1,2}(U) \hookrightarrow L^4(U)$, where $U \subset \RR^4$ is an open subset obeying the interior cone condition.

\begin{lem}[Embedding into $L^4$ for the Sobolev $W^{1,2}$ norm for $\ad P$-valued one-forms defined by bubble-tree neighborhood data and conformal blow-ups]
\label{lem:W12_norm_via_bubble-tree_data_connected-sum_local_conformal_blow-up_L4_embedding}
Let $(X,g)$ be a closed, four-dimensional, Riemannian, smooth manifold. Then there is a constant, $C = C(g) \in [1,\infty)$ with the following significance. Assume the notation of Definition \ref{defn:W12_norm_via_bubble-tree_data_connected-sum_local_conformal_blow-up}. If $a \in W_{A,g}^{1,2}(X;\Lambda^1\otimes\ad P)$, then
\begin{equation}
\label{eq:W12_norm_via_bubble-tree_data_connected-sum_local_conformal_blow-up_L4_embedding}
\|a\|_{L^4(X,g)}
=
\|h_{\bz,\blambda}^{-1,*}a\|_{L^4(X\#_\sT S^4, h_{\bz,\blambda}^{-1,*}g)}
\leq
C\|h_{\bz,\blambda}^{-1,*}a\|_{\overline{W}_{\sU, g}^{1,2}(X\#_\sT S^4)}.
\end{equation}
\end{lem}

\subsection{Variation of the {\L}ojasiewicz-Simon constants for the Yang-Mills $L^2$-energy functional with respect to a $W^{1,2}$ variation of the Yang-Mills connection}
\label{subsec:Lojasiewicz-Simon constants_Yang-Mills_W12_variation_connections}
Though not required for the proof of Theorem \ref{mainthm:Discreteness_Yang-Mills_energies}, it is interesting to examine the behavior of the {\L}ojasiewicz-Simon triple of constants with respect to a small $W^{1,2}$ variation of the Yang-Mills connection $A_\ym$ on a principal $G$-bundle over a closed four-dimensional manifold, such as hold within fine $W_\loc^{1,2}$ bubble-tree open neighborhoods in $\Crit(P,g,\sC)$.

\begin{lem}[{\L}ojasiewicz-Simon constants and $W^{1,2}$ variations of Yang-Mills connections]
\label{lem:Lojasiewicz-Simon_constants_dependence_W12_variation_Yang-Mills_connection}
Assume the hypotheses of Corollary \ref{cor:Rade_proposition_7-2_d_is_4_conformally_invariant} and let $(Z,\sigma,\theta)$ be the corresponding {\L}ojasiewicz-Simon triple of constants. If $A_\ym'$ is another $g$-Yang-Mills connection of class $W^{\bar k,\bar p}$ on $P$ such that
\[
\|A_\ym' - A_\ym\|_{W_{A_\myref,g}^{1,2}(X)} < \frac{1}{2\sqrt{2}} (1\wedge \sigma),
\]
then $\sE_g(A_\ym') = \sE_g(A_\ym)$ and $(Z, \sigma/(2\sqrt{2}), \theta)$ is a {\L}ojasiewicz-Simon triple for $A_\ym'$.
\end{lem}

\begin{proof}
Let $A$ be a connection of class $W^{\bar k,\bar p}$ on $P$. Suppressing explicit dependence on the Riemannian metric, $g$, for brevity, we have
\begin{align*}
\|A - A_\ym\|_{W_{A_\myref}^{1,2}(X)}^2
&=
\|\nabla_{A_\myref}(A - A_\ym)\|_{L^2(X)}^2 + \|A -A_\ym\|_{L^2(X)}^2
\\
&\leq \left(\|\nabla_{A_\myref}(A - A_\ym')\|_{L^2(X)} + \|\nabla_{A_\myref}(A_\ym' - A_\ym)\|_{L^2(X)}\right)^2
\\
&\quad + \left(\|A - A_\ym'\|_{L^2(X)} + \|A_\ym' - A_\ym\|_{L^2(X)} \right)^2
\\
&\leq 2\|\nabla_{A_\myref}(A - A_\ym')\|_{L^2(X)}^2 + 2\|A - A_\ym'\|_{L^2(X)}^2
\\
&\quad + 2\|\nabla_{A_\myref}(A_\ym' - A_\ym)\|_{L^2(X)}^2 + \|A_\ym' - A_\ym\|_{L^2(X)}^2
\\
&= 2\|A - A_\ym'\|_{W_{A_\myref}^{1,2}(X)}^2 + 2\|A_\ym' - A_\ym\|_{W_{A_\myref}^{1,2}(X)}^2.
\end{align*}
Suppose that
\[
\|A_\ym' - A_\ym\|_{W_{A_\myref}^{1,2}(X)} < \eps,
\]
for some $\eps \in (0,1]$ to be determined. The previous inequality then gives
\[
\|A - A_\ym\|_{W_{A_\myref}^{1,2}(X)}^2
<
2\|A - A_\ym'\|_{W_{A_\myref}^{1,2}(X)}^2 + 2\eps^2,
\]
and thus, taking square roots,
\[
\|A - A_\ym\|_{W_{A_\myref}^{1,2}(X)}
<
\sqrt{2}\|A - A_\ym'\|_{W_{A_\myref}^{1,2}(X)} + \sqrt{2}\eps.
\]
Hence, choosing $\eps := (1 \wedge \sigma)/(2\sqrt{2})$ and supposing that $\|A - A_\ym'\|_{W_{A_\myref}^{1,2}(X)} < \sigma/(2\sqrt{2})$ yields,
\[
\|A - A_\ym\|_{W_{A_\myref}^{1,2}(X)}
\leq
\sqrt{2}\|A - A_\ym'\|_{W_{A_\myref}^{1,2}(X)} + \frac{\sigma}{2}
<
\sigma.
\]
Applying Corollary \ref{cor:Rade_proposition_7-2_d_is_4_conformally_invariant} for $A_\ym$ and the {\L}ojasiewicz-Simon radius $\sigma[A_\ym]$ gives
\[
\|d_A^{*,g}F_A\|_{L^4(X)} \geq Z|\sE_g(A) - \sE_g(A_\ym)|^\theta,
\]
with constants $Z = Z[A_\ym]$ and $\theta = \theta[A_\ym]$. If $A = A_\ym'$, then $d_{A_\ym'}^{*,g}F_{A_\ym'} = 0$ and so the preceding inequality implies that $\sE_g(A_\ym') = \sE_g(A_\ym)$. Consequently, for a general Sobolev connection, $A$, the preceding inequality gives
\[
\|d_A^{*,g}F_A\|_{L^4(X)} \geq Z|\sE_g(A) - \sE_g(A_\ym')|^\theta,
\]
and thus $(Z,\sigma/(2\sqrt{2}),\theta)$ is a {\L}ojasiewicz-Simon triple for $[A_\ym']$.
\end{proof}

\subsection{{\L}ojasiewicz-Simon gradient inequality for the Yang-Mills $L^2$-energy functional and coarse $W_\loc^{1,2}$ bubble-tree neighborhoods}
\label{subsec:Lojasiewicz-Simon gradient_inequality_Yang-Mills_functional_coarse_bubble-tree_neighborhood}
We have the following version of the {\L}ojasiewicz-Simon gradient inequality for the Yang-Mills $L^2$-energy functional in Corollary \ref{cor:Rade_proposition_7-2_d_is_4_conformally_invariant}, with norm 
$\|\cdot\|_{W_{A_\myref, g^\#}^{1,2}(X\#_\sT S^4)}$ on $\Omega^1(X\#_\sT S^4;\ad P)$ replaced by the norm $\|\cdot\|_{\overline{W}_{\sU, g}^{1,2}(X\#_\sT S^4)}$.

\begin{prop}[{\L}ojasiewicz-Simon gradient inequality adapted to a coarse $W_\loc^{1,2}$ bubble-tree neighborhood]
\label{prop:Lojasiewicz-Simon_gradient_inequality_coarse_bubble-tree_neighborhood_conformal_blow-up}
Let $G$ be a compact Lie group and $P$ be a principal $G$-bundle over a closed, connected, four-dimensional, oriented, smooth manifold, $X$, endowed with a smooth Riemannian metric, $g$. Let $A_\ym$ be a $g$-Yang-Mills connection of class $W^{\bar k,\bar p}$, with integer $\bar k\geq 1$ and $\bar p \geq 2$ obeying $(\bar k+1)\bar p > 4$, on a principal $G$-bundle, $P$, over $X$, let $\sC \Subset [0,\infty)$ be a compact subset, and let $\sU \subset \Crit(P,g,\sC)$ be a coarse $W_\loc^{1,2}$ bubble-tree open neighborhood of $[A_\ym]$ provided by Corollary \ref{cor:Taubes_1988_proposition_5-4_coarse_conformal}. Then there are positive constants $Z \in [1,\infty)$, $\sigma\in (0,1]$, and $\theta \in [1/2,1)$, depending on $\sU$, the gauge-equivalence class, $[A_\ym]$, and $g$, $G$, $P$, and $X$ with the following significance. Suppose $A$ is a $W^{\bar k,\bar p}$ Sobolev connection on $P$ and that $[A]$ and $[A_\ym]$ belong to a coarse $W_\loc^{1,2}$ bubble-tree open neighborhood $\sU \subset \sB(P,g)$ such that there are conformal blow-up diffeomorphisms, $f, h \in \Conf(X,g)$, and a $C^\infty$ gauge transformation, $u \in \Aut(P)$ with the property that
\begin{equation}
\label{eq:Rade_7-1_coarse_bubble-tree_neighborhood}
\|f^{-1,*}u(A) - h^{-1,*} A_\ym\|_{\overline{W}_{\sU, g}^{1,2}(X\#_\sT S^4)} < \sigma,
\end{equation}
where $\sT$ is the finite tree defined by $\sU$. Then $A$ obeys
\begin{equation}
\label{eq:Rade_7-1_coarse_bubble-tree_gradient_inequality}
\|d_A^{*,g}F_A\|_{L^4(X,g)} \geq Z|\sE_g(A) - \sE_g(A_\ym)|^\theta.
\end{equation}
\end{prop}

\begin{proof}
The gradient inequality,
\[
\|d_{f^{-1,*}u(A)}^{*,f^{-1,*}g}F_{f^{-1,*}u(A)}\|_{L^4(X\#_\sT S^4,f^{-1,*}g)}
\geq
Z|\sE_{f^{-1,*}g}(f^{-1,*}u(A)) - \sE_{h^{-1,*} g}(h^{-1,*} A_\ym)|^\theta,
\]
follows from the proof of Theorem \ref{thm:Rade_proposition_7-2} in \cite{Feehan_yang_mills_gradient_flow} using Theorem \ref{thm:Huang_2-4-5}, Theorem \ref{thm:Rade_proposition_7-2_d_is_4}, and Corollary \ref{cor:Rade_proposition_7-2_d_is_4_conformally_invariant}, together with Lemma \ref{lem:W12_norm_via_bubble-tree_data_connected-sum_local_conformal_blow-up_L4_embedding}. The invariance of the $L^4$ norm on one-forms and $L^2$ norm on two-forms with respect to conformal diffeomorphisms of a four-dimensional Riemannian manifold then yield \eqref{eq:Rade_7-1_coarse_bubble-tree_gradient_inequality}.
\end{proof}

\subsection{Completion of the proofs of Theorems \ref{mainthm:Discreteness_Yang-Mills_energies} and \ref{mainthm:Separation_strata_Yang-Mills_connections}}
\label{subsec:Completion_proof_main_theorems}
As a consequence of Corollary \ref{cor:Taubes_1988_proposition_5-4_coarse_conformal} and Proposition \ref{prop:Lojasiewicz-Simon_gradient_inequality_coarse_bubble-tree_neighborhood_conformal_blow-up} we can conclude the

\begin{proof}[Proof of Theorem \ref{mainthm:Discreteness_Yang-Mills_energies}]
Corollary \ref{cor:Taubes_1988_proposition_5-4_coarse_conformal} ensures that $\Crit(P,g,\sC)$ has a finite cover by coarse $W_\loc^{1,2}$ bubble-tree open neighborhoods, $\sU'\cap\Crit(P,g,\sC)$ defined by the inequality \eqref{eq:Rade_7-1_coarse_bubble-tree_neighborhood}, of points $[A_\ym] \in \Crit(P,g,\sC)$. If $[A] \in \sU'\cap\Crit(P,g,\sC)$, then $A$ satisfies \eqref{eq:Rade_7-1_coarse_bubble-tree_neighborhood} and hence Proposition \ref{prop:Lojasiewicz-Simon_gradient_inequality_coarse_bubble-tree_neighborhood_conformal_blow-up} implies that $A$ obeys the {\L}ojasiewicz-Simon gradient inequality \eqref{eq:Rade_7-1_coarse_bubble-tree_gradient_inequality}. But $A$ also obeys the Yang-Mills equation, $d_A^{*,g}F_A = 0$, and hence the inequality \eqref{eq:Rade_7-1_coarse_bubble-tree_gradient_inequality} implies that $\sE_g(A) = \sE_g(A_\ym)$. This proves the main result of this article, Theorem \ref{mainthm:Discreteness_Yang-Mills_energies}.
\end{proof}

\begin{proof}[Proof of Theorem \ref{mainthm:Separation_strata_Yang-Mills_connections}]
From inequality \eqref{eq:W12_distance_between_critical_sets_connections_small}, there exist $[A] \in \Crit(P,g,c)$ and $[A'] \in \Crit(P,g,c')$ such that
\[
\|u(A') - A\|_{W_{A,g}^{1,2}(X)} < \delta,
\]
for some $u \in \Aut(P)$. If $\delta = \sigma[A]$, then the {\L}ojasiewicz-Simon gradient inequality given by Corollary \ref{cor:Rade_proposition_7-2_d_is_4_conformally_invariant} implies that $c' = \sE_g(A') = \sE_g(A) = c$. Hence, the conclusion would follow if there were a positive constant, $\sigma_0$, with the stated dependencies such that $\sigma[A] \geq \sigma_0$ for all $[A] \in \Crit(P,g,\sC)$.

To reach the desired conclusion, we shall instead apply the {\L}ojasiewicz-Simon gradient inequality given by Proposition \ref{prop:Lojasiewicz-Simon_gradient_inequality_coarse_bubble-tree_neighborhood_conformal_blow-up}, where the dependencies of its {\L}ojasiewicz-Simon triple of constants can be tracked more easily. For $\eps[A_\ym] > 0$, Corollary \ref{cor:Taubes_1988_proposition_5-4_coarse_conformal} implies that $\Crit(P,g,\sC)$ has a finite cover by coarse $W_\loc^{1,2}$ bubble-tree open neighborhoods, $\sU'\cap\Crit(P,g,\sC)$ defined by the inequality \eqref{eq:Coarse_W12_bubble-tree_neighborhood_via_conformal_blow-up_map}, of points $[A_\ym] \in \Crit(P,g,\sC)$. If $[A] \in \sU'\cap\Crit(P,g,\sC)$, then there are conformal blow-up diffeomorphisms, $f, h \in \Conf(X,g)$, and a $C^\infty$ gauge transformation, $u \in \Aut(P)$ with the property that
\[
\|f^{-1,*}u(A) - h^{-1,*} A_\ym\|_{\overline{W}_{\sU, g}^{1,2}(X\#_\sT S^4)} < \eps,
\]
where $\sT$ is the finite tree defined by $\sU'$. Applying the proof of Lemma \ref{lem:Lojasiewicz-Simon_constants_dependence_W12_variation_Yang-Mills_connection}, it follows \mutatis that if $(Z[A_\ym], \sigma[A_\ym], \theta[A_\ym])$ is a {\L}ojasiewicz-Simon triple for $[A_\ym]$ in Proposition \ref{prop:Lojasiewicz-Simon_gradient_inequality_coarse_bubble-tree_neighborhood_conformal_blow-up}, then $(Z[A_\ym], \sigma[A_\ym]/4, \theta[A_\ym])$ serves as a {\L}ojasiewicz-Simon triple for $[A]$ in Proposition \ref{prop:Lojasiewicz-Simon_gradient_inequality_coarse_bubble-tree_neighborhood_conformal_blow-up} when $\eps[A_\ym] = (1 \wedge \sigma[A_\ym])/4$ in the preceding inequality. Consequently,
\[
\sigma[A] \geq \frac{\sigma[A_\ym]}{4}, \quad\forall\, [A] \in \sU' \cap \Crit(P,g,\sC).
\]
From Definitions \ref{defn:W12_norm_via_bubble-tree_data_connected-sum_Aprime_minus_A_local_conformal_blow-up} and \ref{defn:Comparing_local_conformal_blow-ups_pair_connections_near_southern_poles}, we find that
\[
\|u(A') - A\|_{\overline{W}_{\sU,g}^{1,2}(X\#_\sT S^4)} \leq C\|u(A') - A\|_{W_{A,g}^{1,2}(X)},
\]
where $C = C(g) \in [1,\infty)$ and the identity map on $X$ in place of the pair of conformal blow-up diffeomorphisms in those definitions. If we choose $\delta = (1 \wedge \sigma[A_\ym])/(4C)$, then
 \[
\|u(A') - A\|_{\overline{W}_{\sU,g}^{1,2}(X\#_\sT S^4)} < \frac{\sigma[A_\ym]}{4},
\]
and Proposition \ref{prop:Lojasiewicz-Simon_gradient_inequality_coarse_bubble-tree_neighborhood_conformal_blow-up} (with $f$ and $h$ replaced by the identity map on $X$) implies that $c' = \sE_g(A') = \sE_g(A) = c'$. The conclusion in Theorem \ref{mainthm:Separation_strata_Yang-Mills_connections} now follows by finiteness of the open cover for $\overline{\Crit}^{\,\tau}(P,g,\sC)$ by coarse $W_\loc^{1,2}$ bubble-tree open neighborhoods of the form provided by Corollary \ref{cor:Taubes_1988_proposition_5-4_coarse_conformal}.
\end{proof}

\appendix

\section{A weighted Sobolev embedding on Euclidean space}
\label{sec:Weighted_Sobolev_embedding_Euclidean_space}
The following Sobolev embedding result generalizes that of \cite[Lemma 5.2]{FeehanSlice}, which assumes $p=2$ and $d=4$; similar embeddings appear as \cite[Equation (3.4)]{TauStable} for $p=2$ and $d=4$ and \cite[Definition 4.2 and Lemma 5.4 (a)]{ParkerTaubes} for $d = 3$. Our proof of \cite[Lemma 5.2]{FeehanSlice} contained a small typographical error\footnote{The integrals over $\RR$ should be replaced by integrals over the half-line, $[0,\infty)$.} and so we include the details here for the sake of completeness.

\begin{lem}[Weighted Sobolev embedding on Euclidean space]
\label{lem:Sobolev_embedding_critical_exponent_spaces_Euclidean}
Let $d \geq 3$ be an integer and $\RR^d$ have its standard Euclidean metric. If $f\in W^{1,d-2}(\RR^d;\CC)$, then
\[
\sup_{x_0\in\RR^d}\||x_0 - \cdot\,|^{-1} f\|_{L^{d-2}(\RR^d)} \leq \frac{d-2}{2}\|\cov f\|_{L^{d-2}(\RR^d)}.
\]
\end{lem}

\begin{proof}
Suppose first that $f\in C_0^\infty(\RR^d)$ and $d \geq 3$ and $f \geq 0$ on $\RR^d$. Let $x = (r,\theta)$ denote polar coordinates centered at a point $x_0\in\RR^d$, so $r = |x-x_0|$ and $dx = r^{d-1}\,dr d\theta$. Then
\begin{align*}
\int_{\RR^d}r^{2-d} f^{d-2} \,dx
&=
\int_{S^{d-1}}\int_0^\infty r f^{d-2}\,dr d\theta
=
\frac{1}{2}\int_{S^{d-1}}\int_0^\infty \frac{dr^2}{dr }f^{d-2} \,dr d\theta
\\
&= - \left(\frac{d-2}{2}\right)\int_{S^{d-1}}\int_0^\infty r^2 f^{d-3} \frac{\partial f}{\partial r}\,dr d\theta,
\end{align*}
via integration by parts and hence,
\[
\int_{\RR^d}r^{2-d} f^{d-2} \,dx
=
- \left(\frac{d-2}{2}\right)\int_{\RR^d} r^{3-d} f^{d-3} \frac{\partial f}{\partial r}\,dx.
\]
If $d = 3$, then the preceding identity simplifies to
\[
\int_{\RR^3}r^{-1} f \,dx
=
- \frac{1}{2}\int_{\RR^3} \frac{\partial f}{\partial r}\,dx,
\]
and the conclusion follows since $r = |x-x_0|$, for this class of function, $f$, by taking the supremum over all $x_0 \in \RR^3$.

We now consider the case $d \geq 4$. By applying the H\"older Inequality with exponent $\mu \in (1, 2]$ such that $(d-3)\mu = d-2$, thus $\mu = (d-2)/(d-3)$, and dual exponent $\mu' \in [2,\infty)$ defined by $1/\mu + 1/\mu' = 1$, thus $1/\mu' = 1 - (d-3)/(d-2) = 1/(d-2)$ and $\mu' = d-2$, we see that
\begin{align*}
\int_{\RR^d} r^{2-d}f^{d-2}\,dx
&=
- \left(\frac{d-2}{2}\right)\int_{\RR^d} r^{3-d} f^{d-3} \frac{\partial f}{\partial r}\, dx
\\
&\leq \frac{d-2}{2}\left(\int_{\RR^d} r^{(3-d)\mu} f^{(3-d)\mu} \,dx\right)^{1/\mu}
\left(\int_{\RR^d} |\cov f|^{\mu'}\,dx\right)^{1/\mu'}
\\
&= \frac{d-2}{2}\left(\int_{\RR^d} r^{d-2} f^{d-2} \,dx\right)^{(d-3)/(d-2)}
\left(\int_{\RR^d} |\cov f|^{d-2}\,dx\right)^{1/(d-2)}.
\end{align*}
Therefore, provided $f \not\equiv 0$ on $\RR^d$, we have
\[
\left( \int_{\RR^d} r^{2-d}f^{d-2}\,dx \right)^{1/(d-2)}
\leq
\frac{d-2}{2}\left(\int_{\RR^d} |\cov f|^{d-2}\,dx\right)^{1/(d-2)}.
\]
Again for this class of function, $f$, the conclusion follows by taking the supremum over all $x_0 \in \RR^d$. For the general case of $f \in W^{1, d-2}(\RR^d;\RR)$, one splits $f = f^+ - f^-$, where $f^+ = \max\{f,0\}$ and $f^- = \max\{-f, 0\}$, and proceeds as in the proof of \cite[Theorem 7.8]{GilbargTrudinger}, using \cite[Lemmas 7.5, 7.6, and 7.8]{GilbargTrudinger}. The extension to the case of $f \in W^{1, d-2}(\RR^d;\CC)$ is trivial.
\end{proof}

%
%

\bibliography{master,mfpde}
\bibliographystyle{amsplain}

\end{document}